\documentclass[11pt,a4paper,leqno]{article}
\usepackage[latin1]{inputenc}
\usepackage{amsmath,amsthm}
\usepackage{amsfonts,amssymb,latexsym}
\usepackage{graphicx}
\usepackage[french]{babel}
\usepackage[all]{xy}

 \usepackage[T1]{fontenc}

\DeclareMathOperator{\Spec}{Spec}

\DeclareMathOperator{\et}{et}
\DeclareMathOperator{\Tr}{Tr}
\DeclareMathOperator{\ou}{ou}

\DeclareMathOperator{\rg}{rg}
\DeclareMathOperator{\Fr}{Fr}

\DeclareMathOperator{\NS}{NS}

\DeclareMathOperator{\nor}{nor}
\DeclareMathOperator{\Gal}{Gal}

\DeclareMathOperator{\Lip}{Lip}

\DeclareMathOperator{\eff}{eff}

\DeclareMathOperator{\Val}{Val}

\DeclareMathOperator{\Card}{Card}

\DeclareMathOperator{\PGCD}{pgcd}
\DeclareMathOperator{\PPCM}{ppcm}

\DeclareMathOperator{\Pic}{Pic}

\DeclareMathOperator{\di}{div}

\DeclareMathOperator{\card}{card}

\DeclareMathOperator{\admissible}{admissible}
\DeclareMathOperator{\pour}{pour}

\DeclareMathOperator{\elements}{\acute{e}l\acute{e}ments}

\newcommand{\dd}{\boldsymbol{d}}

\DeclareMathOperator{\Id}{Id}

\DeclareMathOperator{\Vol}{Vol}
\newcommand{\CC}{\mathbf{C}}
\newcommand{\RR}{\mathbf{R}}
\newcommand{\ZZ}{\mathbf{Z}}
\newcommand{\NN}{\mathbf{N}}
\newcommand{\QQ}{\mathbf{Q}}
\newcommand{\XX}{\mathbf{X}}
\newcommand{\YY}{\mathbf{Y}}
\newcommand{\VV}{\mathbf{V}}
\newcommand{\WW}{\mathbf{W}}
\newcommand{\FF}{\mathbf{F}}
\newcommand{\UU}{\mathbf{U}}
\newcommand{\xx}{\boldsymbol{x}}
\newcommand{\aalpha}{\boldsymbol{\alpha}}
\newcommand{\bbeta}{\boldsymbol{\beta}}
\newcommand{\xxi}{\boldsymbol{\xi}}
\newcommand{\hh}{\boldsymbol{h}}
\newcommand{\ee}{\boldsymbol{e}}
\renewcommand{\ss}{\boldsymbol{s}}
\renewcommand{\l}{\boldsymbol{l}}
\newcommand{\x}{\mathbf{x}}
\newcommand{\pt}{\langle \tt \rangle}
\newcommand{\px}{\langle \xx \rangle}

\newcommand{\yy}{\boldsymbol{y}}
\newcommand{\kk}{\boldsymbol{k}}
\newcommand{\zz}{\boldsymbol{z}}

\newcommand{\rr}{\boldsymbol{r}}
\newcommand{\ww}{\boldsymbol{w}}
\renewcommand{\tt}{\boldsymbol{t}}
\newcommand{\jj}{\boldsymbol{j}}
\renewcommand{\1}{\boldsymbol{1}}

\newcommand{\cc}{\boldsymbol{c}}
\renewcommand{\aa}{\boldsymbol{a}}

\newcommand{\0}{\boldsymbol{0}}

\newcommand{\uu}{\boldsymbol{u}}
\newcommand{\vv}{\boldsymbol{v}}
\newcommand{\bb}{\boldsymbol{b}}
\newcommand{\ra}{\rightarrow}
\newcommand{\mt}{\mapsto}

\renewcommand{\AA}{\mathbf{A}}
\newcommand{\OO}{\mathcal{O}}
\newcommand{\MM}{\mathcal{M}}
\newcommand{\BB}{\mathcal{B}}
\newcommand{\EE}{\mathcal{E}}

\newcommand{\II}{\mathcal{I}}

\newcommand{\CCC}{\mathcal{C}}
\newcommand{\FFF}{\mathcal{F}}
\newcommand{\PPP}{\mathcal{P}}
\newcommand{\GG}{\mathcal{G}}
\DeclareMathOperator{\Div}{Div}

  \newtheorem{thm}{Th\'eor\`eme}[section]
 \newtheorem{prop}[thm]{Proposition}
 \newtheorem{lemma}[thm]{Lemme}
  \newtheorem{notation}[thm]{Notations}
 \newtheorem{cor}[thm]{Corollaire}
 \newtheorem{Def}[thm]{D\'efinition}
  \newtheorem{rem}[thm]{Remarque}

 \theoremstyle{definition}

  \usepackage[refpage]{nomencl}

\makenomenclature

\usepackage{makeidx}
\makeindex

\begin{document}

\title{POINTS DE HAUTEUR BORN\'EE \\ SUR LES HYPERSURFACES LISSES \\
DES VARI\'ET\'ES TORIQUES \\ CAS G\'EN\'ERAL}

\author{ Teddy Mignot}

\maketitle

\begin{abstract}
Nous d\'emontrons ici la conjecture de Batyrev/Manin pour le nombre de points de hauteur born\'ee des hypersurfaces de certaines vari\'et\'es toriques. La m\'ethode utilis\'ee est inspir\'ee de celle d\'evelopp\'ee par Schindler pour le cas des hypersurfaces de l'espace biprojectif, qui elle-m\^eme s'inspire de la m\'ethode du cercle de Hardy-Littlewood. La constante obtenue dans la formule asymptotique finale est exactement celle conjectur\'ee par Peyre. 
\end{abstract}

\tableofcontents

\section{Introduction}

On consid\`ere une vari\'et\'e torique d\'eploy\'ee compl\`ete lisse $ X=X(\Delta) $ de dimension $ n $ d\'efinie par le r\'eseau $ N=\ZZ^{n} $ et un \'eventail $ \Delta $ ayant $ n+r $ ar\^etes engendr\'ees par des vecteurs not\'es $ v_{1},v_{2},...,v_{n+r}\in \RR^{n+r} $ (avec $ r\geqslant 2 $). Nous supposerons que le groupe de Picard et le c\^one effectif de $ X $ sont engendr\'es par les classes de diviseurs associ\'es aux ar\^etes de $ r $ vecteurs g\'en\'erateurs de l'\'eventail, disons $ v_{n+1},...,v_{n+r} $. On note $ D_{n+1},...,D_{n+r} $ les diviseurs associ\'es, et $ [D_{n+1}],...,[D_{n+r}] $ leurs classes dans $ \Pic(X) $.  
On peut alors \'ecrire 

\[ \Pic(X)=\bigoplus_{j=1}^{r}\ZZ[D_{n+j}], \]
\[ C^{1}_{\eff}(X)=\sum_{j=1}^{r}\RR^{+}[D_{n+j}], \]
et la classe du diviseur anticanonique de $ X $ est de la forme \[ [-K_{X}]=\sum_{j=1}^{r}n_{j}[D_{n+j}] \] avec $ n_{1},n_{2},...,n_{r}\in \ZZ $. D'autre part, pour $ d_{1},...,d_{r}\in \NN $ fix\'es (tels que $ n_{j}>d_{j} $ pour tout $ j $) consid\'erons un diviseur de classe $ \sum_{j=1}^{r}d_{j}[D_{n+j}] $ et une hypersurface $ Y $ d\'efinie par une section de ce diviseur. On supposera que l'hypersurface choisie est lisse et de dimension sup\'erieure ou \'egale \`a $ 3 $. La classe du diviseur anticanonique de $ Y $ est alors donn\'ee par \[  [-K_{Y}]=\sum_{j=1}^{r}(n_{j}-d_{j})[\tilde{D}_{n+j}], \] o\`u les $ \tilde{D}_{n+j} $ d\'esignent les diviseurs induits par les diviseurs $ D_{n+j} $ sur $ Y $. En utilisant par exemple la construction d\'ecrite par Salberger dans \cite[\S 10]{Sa}, on peut construire une hauteur $ H $ sur $ X $ associ\'ee \`a $ \sum_{j=1}^{r}(n_{j}-d_{j})[D_{n+j}] $. Elle induit une hauteur sur $ Y $ qui est une hauteur associ\'ee \`a $ [-K_{Y}] $, et que l'on notera encore $ H $. L'objectif est alors de donner une formule asymptotique pour le nombre \[ \mathcal{N}_{V}(B)=\Card\{ P\in Y(\QQ)\cap V \; | \; H(P)\leqslant B \}, \] pour un ouvert $ V $ bien choisi.  \\

La vari\'et\'e torique $ X $ peut \^etre d\'efinie comme le quotient de  \[ X_{1}=\{\xx\in \AA^{n+r}\; |\; \forall \sigma \in \Delta_{\max}, \prod_{i\; |\; v_{i}\notin \sigma}x_{i}\neq 0\} \] par l'action du tore $ (\CC^{\ast})^{r} $ de la forme :\[ \forall \tt=(t_{1},...,t_{r})\in (\CC^{\ast})^{r}, \; \forall \xx\in X_{1}, \; \tt.\xx=(\prod_{j=1}^{r}t_{j}^{a_{i,j}}x_{i})_{i\in \{1,...,r\}},  \] o\`u les $ a_{i,j} $ ont \'et\'e d\'efinis par la formule\;\eqref{formuleaij2} (voir par exemple\;\cite[\S 2]{Co} pour plus de d\'etails). Notons $ \pi : X_{1} \ra X $ la projection canonique. L'hypersurface $ Y $ de $ X $ est alors $ \pi(Y_{1}) $ o\`u $ Y_{1} $ est l'hypersurface de $ X_{1} $ donn\'ee par une \'equation $ F(\xx)=0 $, o\`u $ F $ est un polyn\^ome v\'erifiant  \[ \forall \tt=(t_{1},...,t_{r})\in (\CC^{\ast})^{r}, \; \forall \xx\in X_{1}, \;  F(\tt.\xx)=\left(\prod_{j=1}^{r}t_{j}^{d_{j}}\right)F(\xx). \] 
 
Pour tous $ m\in \{1,...,r\} $ et $ \tau\in \mathfrak{S}_{r} $, on note \[ \mathcal{C}_{m,\tau}=\{(j,k)\; |\; j\in \{\tau(1),...,\tau(m)\}, k\in \{1,...,d_{j}\}\} \] \[ \mathcal{D}_{m,\tau}= \NN^{\mathcal{C}_{m,\tau}},  \] \[  \forall (j,k)\in \mathcal{C}_{m,\tau},\; \; J(j,k)=\{i\in \{1,...,n+r\}\; |\; a_{i,j}=k\}.  \]Le polyn\^ome $ F $ peut \^etre d\'ecompos\'e sous la forme : \[ F(\xx)=\sum_{\dd=(d_{j,k})_{(j,k)\in\mathcal{C}_{m,\tau} }\in\mathcal{D}_{m,\tau} }F_{\dd}(\xx), \] o\`u $ F_{\dd} $ est un polyn\^ome homog\`ene de degr\'e $ d_{j,k} $ en les variables $ (x_{i})_{i\in J(j,k)} $ pour tout $ (j,k)\in \mathcal{C}_{m,\tau}   $. On pose alors pour tout $ \dd\in \mathcal{D}_{m,\tau} $, et tout $ (j,k)\in \mathcal{C}_{m,\tau} $, \[  V^{\ast}_{\tau,m,\dd,(j,k)}=\left\{ \xx\in \AA^{n+r}\; |\; \forall i\in J(j,k), \;\frac{\partial F_{\dd}}{\partial x_{i}}(\xx)=0 \right\}. \] Posons enfin \[ n(F)= n+r-\max_{\substack{m\in \{1,...,r\} \\ \tau\in \mathfrak{S}_{r} }}\min_{\dd\in \mathcal{D}^{\ast}_{m,\tau}}\max_{(j,k)\in \mathcal{C}_{m,\tau} }\dim V^{\ast}_{\tau,m,\dd,(j,k)},   \] o\`u $ \mathcal{D}^{\ast}_{m,\tau}=\{\dd=(d_{j,k})_{(j,k)\in\mathcal{C}_{m,\tau} }\in \NN^{\mathcal{C}_{m,\tau}}\;|\; F_{\dd}\neq 0 \} $. On pose par ailleurs $ I_{1},...,I_{N} $ les ensembles $ I $ minimaux pour l'inclusion tels que \[  \forall \sigma \in \Delta, \; \sum_{i\in I} \RR^{+}v_{i}\nsubseteq \sigma.  \] Nous d\'emontrons alors dans cette partie le th\'eor\`eme ci-dessous :

 \begin{thm}\label{thm3b3}
Si l'on suppose que $ \Delta $ est tel que, pour tout $ k\in \{1,...,N\} $, $ \Card I_{k}\geqslant 6 $, qu'une puissance de $ \omega_{Y}^{-1} $ est engendr\'ee par ses sections globales et que \[  n(F)\geqslant r(6.2^{\sum_{j=1}^{r}d_{j}}+4)\left(\prod_{j=1}^{r}(3+10d_{j})\right)\left(\sum_{j=1}^{r} d_{j}\right), \]  alors il existe un ouvert $ V $ de $ X $ de la forme $ \pi(U) $ o\`u $ U $ est un ouvert de $ X_{1} $ tel que :
\begin{equation*}\mathcal{N}_{V}(B)=CB(\log B)^{r-1} +O(B(\log B)^{r-2}), \end{equation*} o\`u $ C $ est la constante conjectur\'ee par Peyre.
\end{thm} 
\begin{rem}\label{remouvU}
L'ouvert $ U $ est d\'efinit par la formule\;\eqref{definitionouvertU3}. Nous verrons en particulier que cet ouvert de $ X_{1} $ v\'erifie \[ \xx\in U \Rightarrow \forall \tt\in (\CC^{\ast})^{r}, \; \tt.\xx\in U. \] On a donc en particulier $ \pi^{-1}(V)=U $. 
\end{rem} Dans la section $ 2 $ nous fixons pr\'ecis\'ement le cadre de notre \'etude. Nous y d\'ecrivons entre autres les vari\'et\'es toriques auxquelles nous nous int\'eressons, l'expression de la hauteur, et la forme des \'equations d\'efinissant les hypersurfaces. Nous montrons par ailleurs que le calcul de $ \mathcal{N}_{U}(B) $ peut se ramener \`a celui de \begin{multline*} N_{\ee,U}(B)= \left\{ \xx\in (\ZZ\setminus\{0\})^{n+r}\cap U\; | \;  \forall i\in \{1,...,n+r\}\; e_{i}|x_{i}, \;  \right. \\ \left. \; F(\xx)=0, \; \forall i\in \{1,...,n+r\} \; |x_{i}|\leqslant \prod_{j=1}^{r}|M_{j}(\xx)|^{a_{i,j}}, \; \prod_{j=1}^{r}|M_{j}(\xx)|^{n_{j}-d_{j}}\leqslant B \right\}. \end{multline*}
pour tout $ \ee\in \NN^{n+r} $ fix\'e, et o\`u $ F, M_{1},...,M_{r} $ sont des mon\^omes et les $ a_{i,j},n_{j},d_{j} $ des entiers que nous pr\'eciserons. La m\'ethode utilis\'ee pour \'evaluer les $ N_{\ee,U}(B )$ est inspir\'ee de celle d\'evelopp\'ee par Schindler dans \cite{S2} pour traiter le cas des hypersurfaces des espaces biprojectifs. Cette m\'ethode consiste dans un premier temps \`a donner une formule asymptotique pour le nombre $ N_{\ee,U}(P_{1},...,P_{r}) $ de points $ \xx  $ de $ U\cap \ZZ^{n+r}$ tels que $ |M_{j}(\xx)|\leqslant P_{j} $ pour tout $ j\in \{1,...,n+r\} $ pour des bornes $ P_{1},P_{2},...,P_{r} $ fix\'ees. Dans la section $ 3 $, en utilisant des arguments issus de la m\'ethode du cercle, on \'etablit une formule asymptotique pour $ N_{\ee,U}(P_{1},...,P_{r}) $ lorsque $ P_{1},...,P_{r} $ sont \og relativement proches \fg \ en un sens que nous pr\'eciserons. Dans la section $ 4 $, pour tout sous-ensemble $ J\subset \{ 1,...,r\}  $ on donne, en utilisant \`a nouveau la m\'ethode du cercle, une formule asymptotique pour $ N_{\ee,U}(P_{1},...,P_{r}) $ lorsque les $ (P_{j})_{j\notin J} $ sont \og petits \fg \ et les $ (P_{j})_{j\in J} $ sont \og grands \fg\ et \og relativement proches \fg. Les r\'esultats obtenus combin\'es avec ceux de la section $ 2 $ nous permettrons dans la section $ 5 $ d'\'etablir une formule asymptotique pour $ N_{d,U}(P_{1},...,P_{r}) $ avec $ P_{1},P_{2},...,P_{r} $ quelconques. Dans la section $ 6 $, nous utilisons une g\'en\'eralisation des r\'esultats \'etablis par Blomer et Br\"{u}dern dans \cite{BB} pour conclure quant \`a la valeur de $ N_{\ee,U}(B) $ \`a partir des estimations obtenues dans les sections pr\'ec\'edentes. Enfin, dans la section $ 7 $, on conclut en d\'emontrant le th\'eor\`eme\;\ref{thm3b3}.

\section{Pr\'eliminaires}\label{preliminaires}

\subsection{Notations et premi\`eres propri\'et\'es}\label{premprop}

Rappelons les d\'efinitions suivantes : 
\begin{Def}
\'Etant donn\'e un r\'eseau $ N $, un \emph{\'eventail} est un ensemble $ \Delta $ de c\^ones poly\'edriques de $ N_{\RR}=N\otimes \RR $ v\'erifiant : 
\begin{enumerate}
\item Pour tout c\^one $ \sigma\in \Delta $, on a $ 0\in \sigma $;
\item Toute face d'un c\^one de $ \Delta $ est un c\^one de $ \Delta $;
\item L'intersection de deux c\^ones de $ \Delta  $ est une face de chacun de ces deux c\^ones. 
\end{enumerate}
On dit de plus que l'\'eventail est \begin{itemize}
\item \emph{complet} si $ \bigcup_{\sigma\in \Delta}\sigma=N_{\RR} $,
\item \emph{r\'egulier} si chaque c\^one de $ \Delta $ est engendr\'e par une famille de vecteurs pouvant \^etre compl\'et\'ee en une base du $ \ZZ $-module $ N $. 
\end{itemize}
\end{Def}
Pour tout \'eventail $ \Delta $ nous noterons $ \Delta_{\max} $ l'ensemble des c\^ones de dimension maximale, et pour tout c\^one $ \sigma\in \Delta $, on notera $ \sigma(1) $ l'ensemble des vecteurs g\'en\'erateurs des ar\^etes de $ \sigma $ (i.e les vecteurs $ v_{1},...,v_{k}\in N $ tels que $ \sigma=\{\sum_{i=1}^{k}\lambda_{i}v_{i}\; |\; \lambda_{i}\geqslant 0\} $ et tels que les coefficients de chaque $ v_{i} $ soient premiers entre eux). Pour un c\^one polyh\'edrique $ \sigma $ de $ N_{\RR} $ donn\'e on d\'efinit un semi-groupe \[ S_{\sigma}=\sigma^{\vee}\cap N^{\vee}, \]
o\`u $ \sigma^{\vee} $ (resp. $ N^{\vee}=M $) d\'esigne le c\^one (resp. r\'eseau) dual de $ \sigma $ (resp. $ N $). La \emph{ vari\'et\'e torique affine } sur un corps $ k $ associ\'ee \`a $ \sigma $ est la vari\'et\'e affine : \begin{equation}
U_{\sigma}=\Spec(k[S_{\sigma}])
\end{equation}
On remarque que si $ \sigma, \tau $ sont deux c\^ones de $ N_{\RR} $, alors \[ \tau \subset \sigma \Rightarrow U_{\tau}\subset U_{\sigma}. \]
\'Etant donn\'e un r\'eseau $ N $ et un \'eventail $ \Delta $, on d\'efinit une vari\'et\'e alg\'ebrique $ X=X(\Delta) $ sur $ k $ par recollement des ouverts $ U_{\sigma} $ pour $ \sigma\in \Delta $ le long de $ U_{\sigma\cap \tau} $. Nous renvoyons le lecteur \`a \cite[\S 1,2,3]{F} pour plus de d\'etails sur les vari\'et\'es toriques. Remarquons que la vari\'et\'e $ X(\Delta) $ est lisse (resp. compl\`ete) si $ \Delta $ est r\'egulier (resp. complet).  \\

Dans ce qui va suivre nous allons consid\'erer $ X $ une vari\'et\'e torique de dimension $ n $ d\'efinie par un \'eventail $ \Delta $ \`a $ d=n+r $ ar\^etes dont les g\'en\'erateurs seront not\'es, $ v_{1},v_{2},...,v_{n}, v_{n+1},...,v_{n+r}\in \ZZ^{n} $, et un r\'eseau $ N=\ZZ^{n} $. On note $ D_{1},...,D_{n},...,D_{n+r} $ les diviseurs invariants sous l'action du tore associ\'es aux vecteurs g\'en\'erateurs (voir \cite[\S 3.3]{F}). Rappelons que dans le cas o\`u la vari\'et\'e torique $ X $ est lisse, le groupe de Picard de $ X $ est de rang $ r $. Pour simplifier nous allons imposer une premi\`ere condition aux vari\'et\'es toriques que nous consid\'erons : nous nous int\'eressons exclusivement aux vari\'et\'es toriques compl\`etes lisses dont le c\^one effectif est simplicial et pour laquelle tout diviseur effectif est combinaison lin\'eaire \`a coefficients positifs de $ r $ diviseurs $ D_{i} $, disons $ [D_{n+1}],...,[D_{n+r}] $. Une premi\`ere question naturelle est de se demander comment traduire ceci en termes de propri\'et\'es sur les c\^ones de l'\'eventail. Nous allons r\'epondre \`a cette question dans ce qui va suivre. \\

On souhaite donc avoir, pour tout $ i\in \{1,...,n\} $  \begin{equation}\label{formuleaij2} [D_{i}]=\sum_{j=1}^{r}a_{i,j}[D_{n+j}] \end{equation} avec $ a_{i,j}\in \NN $ pour tous $ i,j $. Ceci \'equivaut \`a dire que les diviseurs $ D_{i}-\sum_{j=1}^{r}a_{i,j}D_{n+j} $ sont principaux pour tout $ i\in \{1,...,n\} $. Rappelons que les diviseurs principaux stables sous l'action du tore de $ X $ sont exactement les diviseurs $ \di(\chi^{u}) $ associ\'es aux caract\`eres $ \chi^{u} $ du tore de $ X $ (voir \cite{F}) 
 pour $ u\in M=N^{\vee}=\ZZ^{n} $ donn\'es par : \[  \di(\chi^{u})=\sum_{k=1}^{n+r}\langle u,v_{k}\rangle D_{k}. \] On cherche donc des vecteurs $ u_{1},...,u_{n}\in \ZZ^{n} $ tels que pour tous $ i,j \in \{1,...,n\} $, \begin{equation}\label{Cond12}   \langle u_{i},v_{j}\rangle=\delta_{i,j}  \end{equation} (i.e $ (u_{1},...,u_{n}) $ est la base duale de $ (v_{1},...,v_{n}) $ au sens des espaces vectoriels) et \begin{equation}\label{Cond2}
 \langle u_{i},v_{k}\rangle\leqslant 0 
 \end{equation} pour tout $ k\in \{n+1,...,n+r\} $. Ceci implique en particulier que $ (v_{1},...,v_{n}) $ est une famille g\'en\'eratrice d'un c\^one maximal (i.e. de dimension $ n $) de $ \Delta $. En effet, supposons que $ C\langle v_{1},...,v_{n}\rangle $ n'est pas un c\^one de $ \Delta $, alors puisque le vecteur $ a=\sum_{i=1}^{n}v_{i} $ appartient \`a un c\^one de $ \delta $ ($ \Delta $ \'etant complet), on peut \'ecrire \[ a=\sum_{i\in I_{1}}\alpha_{i}v_{i}+\sum_{i\in I_{2}}\alpha_{i}v_{i}\] avec $ \alpha_{i}> 0 $ pour tout $ i $, et $ I_{1}\varsubsetneq\{1,...,n\} $, $ I_{2}\subset \{n+1,...,n+r\} $. Soit $ i_{0}\in \{1,...,n\} $ tel que $ i_{0}\notin I_{1} $. On a alors \[ \langle u_{i_{0}},a\rangle=\sum_{i\in I_{2}}\alpha_{i}\langle u_{i_{0}},v_{i}\rangle\leqslant 0, \] d'apr\`es\;\eqref{Cond12}. Or, par d\'efinition de $ a $, $  \langle u_{i_{0}},a\rangle=1 $, d'o\`u contradiction. Donc $  C\langle v_{1},...,v_{n}\rangle  $ est bien un c\^one maximal de $ \Delta $. \\

Puisque l'on a suppos\'e que $ X $ est lisse, $ (v_{1},...,v_{n}) $ est alors une base du r\'eseau $ \ZZ^{n} $ dont $ (u_{1},...,u_{n}) $ est la base duale (au sens des r\'eseaux). La condition \eqref{Cond2} impose d'autre part que pour cette base duale $ (u_{1},...,u_{n}) $ : \[ \forall k \in \{n+1,...,n+r\}, \; \;  \langle u_{i},v_{k}\rangle\leqslant 0. \] Une condition n\'ecessaire et suffisante pour que ceci soit v\'erifi\'e est que : \[ v_{n+1},...,v_{n+r}\in C\langle -v_{1},-v_{2},...,-v_{n}\rangle \] o\`u $ C\langle -v_{1},-v_{2},...,-v_{n}\rangle $ d\'esigne le c\^one de $ \RR^{n} $ engendr\'e par $ -v_{1},...,-v_{n} $. 
 \begin{rem}
Inversement, \'etant donn\'e un tel \'eventail, si l'on note \[ \forall k\in \{1,...,r\}, \; \;  v_{n+k}=-\sum_{i=1}^{n}a_{i,k}v_{i}, \] avec $ a_{i,k}\in \NN $, on obtient alors : \[ \forall i\in \{1,...,n\}, \; \;  [D_{i}]=\sum_{k=1}^{r}a_{i,k}[D_{n+k}], \] comme on le souhaite.
 \end{rem}

\subsection{Hauteurs sur les hypersurfaces des vari\'et\'es toriques}\label{hauteur}

\'Etant donn\'ee une vari\'et\'e torique compl\`ete lisse $ X $ d\'efinie par un \'eventail $ \Delta $ \`a $ n+r $ ar\^etes et un r\'eseau $ N=\ZZ^{n} $, dont le groupe de Picard et le c\^one effectif sont engendr\'es par $ [D_{n+1}],...,[D_{n+r}] $ (cf. section pr\'ec\'edente), on consid\`ere la classe du diviseur anticanonique de $ X $ qui sera de la forme : \[ [-K_{X}]=\sum_{i=1}^{n+r}[D_{i}]=\sum_{k=1}^{r}n_{k}[D_{n+k}], \] avec, pour tout $ k\in \{1,...,r\} $,  \[n_{k}=1+\sum_{i=1}^{n}a_{i,k}. \]On consid\`ere alors un diviseur de classe $ \sum_{k=1}^{r}d_{k}[D_{n+k}] $, avec $ d_{1},...,d_{r}\in \NN $. Une section globale $ s $ du fibr\'e en droites associ\'e \`a ce diviseur sur $ X $ permet de d\'efinir une hypersurface de $ X $ que l'on notera $ Y $. La classe du diviseur anticanonique sur $ Y $ sera induite par la classe du diviseur \begin{equation}\label{DiviseurD02} D_{0}=\sum_{k=1}^{r}(n_{k}-d_{k})D_{n+k}. \end{equation} Nous allons donner une construction de la hauteur associ\'ee \`a $ \OO(D_{0}) $ sur $ X $. Pour cela, nous utiliserons la construction des hauteurs sur les vari\'et\'es toriques d\'ecrite par Salberger dans \cite[\S 10]{Sa}. \\

Soit $ \nu $ une place sur $ \QQ $, et $ |.|_{\nu} : \QQ^{\ast}\ra \RR^{+} $  la valeur absolue associ\'ee. On consid\`ere la carte affine de $ X $ associ\'ee au c\^one $ \{\0\} $ de l'\'eventail $ \Delta $. Cette carte est un ouvert $ T $ de $ X $ canoniquement \'egal au tore $ \Spec(\QQ[N]) $. L'application $ \log|.|_{\nu} : \QQ_{\nu}^{\ast}\ra \RR $ induit une application \[ L : T(\QQ_{\nu}) \ra N_{\RR}= \RR^{n}. \] Pour tout $ \sigma\in \Delta $, $ L^{-1}(-\sigma) $ est un sous-ensemble ferm\'e de $ W(\QQ_{\nu}) $. On note alors $ C_{\sigma,\nu} $ l'adh\'erence de $ L^{-1}(-\sigma)  $ dans $ X(\QQ_{\nu}) $. On utilise ces ensembles $ C_{\sigma,\nu} $ pour construire une norme $ ||.||_{D,\nu} $ sur $ \OO(D) $ pour tout diviseur de Weil $ D $ sur $ X $, via la proposition suivante : 

\begin{prop}
Soit $ D=\sum_{i=1}^{n+r}a_{i}D_{i} $ un diviseur de Weil sur $ X $ et $ s $ une section locale analytique de $ \OO(D) $ d\'efinie en $ P\in X(\QQ_{\nu}) $. Le point $ P\in X(\QQ_{\nu}) $ appartient \`a $ C_{\sigma,\nu} $ pour un certain $ \sigma \in \Delta $. Soit $ \chi^{u(\sigma)} $ un caract\`ere sur $ T $ repr\'esentant le diviseur de Cartier correspondant \`a $ D $ sur $ U_{\sigma} $ (i.e. $ \langle u(\sigma),v_{i}\rangle=-a_{i} $ pour tout $ v_{i}\in \sigma(1) $). On pose alors : \[ ||s(P)||_{D,\nu}=|s(P)\chi^{u(\sigma)}(P)|_{\nu}, \] et cette expression est ind\'ependante du choix de $ \sigma\in \Delta $ tel que $ P\in C_{\sigma,\nu} $. \end{prop}
\begin{proof} Voir \cite[Proposition 9.2]{Sa}. \end{proof}

On a alors la proposition suivante qui nous sera utile par la suite. 

\begin{prop}\label{descriptC2}
Soit $ D=\sum_{i=1}^{n+r}a_{i}D_{i} $ un diviseur de Weil sur $ X $ tel que $ \OO(D) $ est engendr\'e par ses sections globales. Alors, pour $ \sigma\in \Delta_{\max} $, si $ \chi^{-u(\sigma)} $ d\'esigne l'unique caract\`ere sur $ T $ qui engendre $ \OO(D) $ sur $ U_{\sigma} $ (i.e. $ \langle -u(\sigma),v_{i}\rangle=-a_{i} $ pour tout $ v_{i}\in \sigma(1) $), alors $ \chi^{-u(\sigma)} $ s'\'etend en une section globale de $ \OO(D) $ et $ \chi^{-u(\sigma)}(P)\neq 0 $ pour tout $ P\in U_{\sigma}(\QQ_{\nu}) $. Si $ s $ est une section locale de $ \OO(D) $ d\'efinie en $ P\in X(\QQ_{\nu}) $, alors \[ ||s(P)||_{D,\nu}=\inf_{\sigma\in \Delta_{\max}}|s(P)\chi^{u(\sigma)}(P)|_{\nu}, \] o\`u $ \Delta_{\max} $ d\'esigne l'ensemble des c\^ones de $ \Delta $ de dimension $ n $. De plus, si $ D $ est ample et $ \sigma\in \Delta_{\max} $, alors $ C_{\sigma,\nu} $ est l'ensemble des $ P\in X(\QQ_{\nu})  $ tels que $ |\chi^{u(\sigma)-u(\tau)}(P)|_{\nu}\leqslant 1 $ pour tout $ \tau\in \Delta_{\max} $. \end{prop}
\begin{proof} Voir \cite[Proposition 9.8]{Sa}. \end{proof}

On peut alors d\'efinir la hauteur associ\'ee \`a un diviseur $ D $. Si $ D=\sum_{i=1}^{n+r}a_{i}D_{i} $ est un diviseur de Weil sur $ X $ et $ P\in X(\QQ) $, la hauteur associ\'ee \`a $ D $ est l'application $ H_{D} : X(\QQ)\ra [0,\infty[ $ d\'efinie par \[ H_{D}(P)=\prod_{\nu\in \Val(\QQ)}||s(P)||_{D,\nu}^{-1}, \] o\`u $ \Val(\QQ) $ d\'esigne l'ensemble des places de $ \QQ $, et $ s $ une section locale de $ \OO(D) $ d\'efinie en $ P $ telle que $ s(P)\neq 0 $. 

\begin{rem}
Comme on peut le voir dans \cite[Proposition 10.12]{Sa}, pour tout $ P\in T(\QQ) $, $ H_{D}(P) $ ne d\'epend que de la classe de $ D $ dans $ \Pic(X) $. \end{rem}

Par la suite, on notera $ H $ la hauteur sur $ X $ associ\'ee au diviseur $ D_{0} $ d\'efini par \eqref{DiviseurD02}. Notre objectif sera alors d'\'evaluer \[ \mathcal{N}_{V}(B)=\Card\{P\in V(\QQ)\cap Y(\QQ) \; | \; H_{D_{0}}(P)\leqslant B\}, \] pour un certain ouvert dense $ V $ de $ X $. Pour \'evaluer cette quantit\'e, il est plus pratique de se ramener \`a compter le nombre de points de hauteur born\'ee sur un torseur universel (voir \cite[\S 3]{Sa} pour la d\'efinition de torseurs universels) associ\'e \`a $ X $. Pour les vari\'et\'es toriques, la construction du torseur universel est relativement simple et est donn\'ee dans \cite[\S 8]{Sa}. Nous allons rappeler cette construction. \\

On consid\`ere le r\'eseau $ N_{0}=\ZZ^{n+r} $ et $ M_{0}=N_{0}^{\vee}=\ZZ^{n+r} $. \`A tout g\'en\'erateur $ v_{i} $ d'une ar\^ete de l'\'eventail $ \Delta $ on associe l'\'el\'ement $ e_{i} $ de la base canonique de $ N_{0}=\ZZ^{n+r} $. On pose alors $ N_{1}=N_{0} $ et $ \Delta_{1} $ l'\'eventail constitu\'e de tous les c\^ones engendr\'es par les $ e_{i} $. La vari\'et\'e torique $ X_{1} $ d\'etermin\'ee par $ (N_{1},\Delta_{1}) $ est alors l'espace affine $ \AA^{n+r} $. Pour tout $ \sigma \in \Delta $, on note d'autre part $ \sigma_{0} $ le c\^one de $ N_{0,\RR} $ engendr\'e par les $ e_{i} $ pour $ i $ tel que $ v_{i}\in \sigma $. Les c\^ones $ \sigma_{0} $ ainsi associ\'es forment alors un \'eventail r\'egulier $ \Delta_{0} $ de $ N_{0,\RR} $ (cf. \cite[Proposition 8.4]{Sa}), et $ (\Delta_{0},N_{0}) $ d\'efinit une vari\'et\'e torique $ X_{0} $ qui est un ouvert de $ X_{1} $. Soit $ U_{0,\sigma}=\Spec(\QQ[S_{\sigma_{0}}]) $ o\`u $ S_{\sigma_{0}}=\sigma_{0}^{\vee}\cap M_{0} $. Les morphismes toriques $ \pi_{\sigma} : U_{0,\sigma}\ra U_{\sigma} $ d\'efinies par les applications naturelles de $ \sigma_{0} $ sur $ \sigma $ se recollent en un morphisme $ \pi : X_{0}\ra X $ qui est alors un torseur universel sur $ X $ (cf. \cite[Proposition 8.5]{Sa}).\\

\'Etant donn\'e que $ X_{0}\subset X_{1}=\AA^{n+r}_{\QQ} $ les points de $ X_{0} $ s'\'ecrivent sous forme de $ (n+r) $-uplets de coordonn\'ees $ \xx=(x_{1},...,x_{n},x_{n+1},...,x_{n+r}) $. On notera alors pour tout diviseur $ D=\sum_{i=1}^{n+r}a_{i}D_{i}  $ : \[ \xx^{D}=\prod_{i=1}^{n+r}x_{i}^{a_{i}}. \] 

\begin{rem}
Si $ \sigma\in \Delta $ , on note \[ \underline{\sigma}=\sum_{i\; |\; v_{i}\notin \sigma(1)}D_{i}, \] alors $ U_{0,\sigma} $ est l'ouvert de $ X_{1} $ d\'etermin\'e par $ \xx^{\underline{\sigma}}\neq 0 $, et donc $ X_{0} $ est l'ouvert de $ X_{1} $ d\'efini par  : \[ \xx\in X_{0} \Leftrightarrow \exists \sigma\in \Delta_{\max} \; | \; \xx^{\underline{\sigma}}\neq 0. \] 
\end{rem}

En rappelant que $ D_{0}=\sum_{k=1}^{r}(n_{k}-d_{k})D_{n+k} $, on d\'efinit alors les diviseurs $ D(\sigma) $ associ\'es : 
\begin{Def}
Soit $ \sigma\in \Delta_{\max} $, et soit $ \chi^{u(\sigma)} $ le caract\`ere de $ U $ tel que $ \chi^{-u(\sigma)} $ engendre $ \OO(D_{0}) $ sur $ U_{\sigma} $. On pose alors \[ D(\sigma)=D_{0}+\sum_{v_{i}\in \sigma(1)}\langle -u(\sigma),v_{i}\rangle D_{i}. \] 
\end{Def}
\begin{rem}
Les diviseurs $ D(\sigma) $ ne d\'ependent que de la classe de $ D_{0} $ dans $ \Pic(X) $. 
\end{rem}
\begin{lemma}\label{effectif2}
Soit $ \sigma\in \Delta_{\max} $. Si $ \OO(D_{0}) $ est engendr\'e par ses section globales, alors $ \chi^{-u(\sigma)} $ s'\'etend en une section globale de $ \OO(D_{0}) $, et $ D(\sigma) $ est un diviseur effectif \`a support contenu dans $ \bigcup_{v_{i}\notin \sigma(1)}D_{i} $. \end{lemma}\begin{proof}
Voir la d\'emonstration de \cite[Proposition 8.7.(a)]{Sa}.
\end{proof}

\begin{prop}
On suppose qu'il existe $ m\in \NN^{\ast} $ tel que $ \OO(mD_{0}) $ est engendr\'e par ses sections globales. Avec les notation ci-dessus, on a : \[ \forall \xx\in X_{0}(\QQ), \; \; H_{0}(\xx)=\prod_{\nu\in \Val(\QQ)}\sup_{\sigma\in \Delta_{\max}}|\xx^{D(\sigma)}|_{\nu}. \]
\end{prop}
\begin{proof}
La d\'emonstration de cette proposition est directement inspir\'ee de la preuve de \cite[Proposition 10.14]{Sa}. On consid\`ere un point $ \xx\in X_{0}(\QQ) $, $ P=\pi(\xx) $, et $ \tau \in \Delta_{\max} $ tel que $ P\in U_{\tau} $. On a alors que $ \chi^{-u(\tau)} $ est une section locale d\'efinie en $ P\in U_{\tau} $, et  \[||\chi^{-u(\tau)}(P)||_{D_{0},\nu}  =\inf_{\sigma\in \Delta_{\max}}|\chi^{u(\sigma)-u(\tau)}|_{\nu}.\]

Remarquons que puisque $ P\in U_{\tau} $, d'apr\`es le lemme \ref{effectif2}, $ \xx^{D(\tau)}\neq 0 $ (\'etant donn\'e que $ D(\tau) $ est effectif \`a support contenu dans $ \bigcup_{v_{i}\notin \sigma(1)}D_{i} $), et que \[ \frac{\xx^{D(\sigma)}}{\xx^{D(\tau)}}=\chi^{u(\tau)-u(\sigma)}(P). \]Par cons\'equent, si $ s $ d\'esigne la section locale $ \chi^{-u(\tau)} $, on a alors :  \[||s(P)||_{D_{0},\nu}^{-1}  =\sup_{\sigma\in \Delta_{\max}}\left|\frac{\xx^{D(\sigma)}}{\xx^{D(\tau)}}\right|_{\nu}.\] De plus, par la formule du produit, on a \[ \prod_{\nu\in \Val(\QQ)}|\xx^{D(\tau)}|_{\nu}=1. \] D'o\`u le r\'esultat. 

\end{proof}

Par extension de la construction des vari\'et\'es toriques par les \'eventails, il est possible de construire une vari\'et\'e $ \tilde{X} $ sur $ \ZZ $ par recollement des ouverts affines $ \tilde{U}_{\sigma}=\Spec(\ZZ[S_{\sigma}]) $ (voir par exemple\;\cite[Chapitre 2]{M}). De la m\^eme mani\`ere que nous avons construit $ X_{0} $, on peut construire un $ \ZZ $-torseur universel sur la vari\'et\'e $ \tilde{X} $ (voir \cite[p. 207]{Sa}). On notera ce torseur $ \tilde{\pi} : \tilde{X}_{0}\ra \tilde{X}  $. On consid\`ere alors la proposition suivante (issue de \cite[Proposition 11.3]{Sa}) : 

\begin{prop}
Soit $ \xx\in X_{0}(\QQ)  $ qui se rel\`eve en un $ \ZZ $-point $ \tilde{\xx} $ de $ \tilde{X}_{0} $ (v\'erifiant alors $ \PGCD_{\sigma\in \Delta_{\max}}\tilde{\xx}^{\underline{\sigma}}=1 $). On a alors \[ H_{0}(\xx)=\sup_{\sigma\in \Delta_{\max}}|\tilde{\xx}^{D(\sigma)}|, \] o\`u $ |.| $ d\'esigne la valeur absolue usuelle sur $ \RR $. \end{prop}

\begin{proof}
D'apr\`es le lemme\;\ref{effectif2}, pour tout $ \sigma\in \Delta_{\max} $, $ D(\sigma) $ a un support contenu dans $ \bigcup_{v_{i}\notin \sigma(1)}D_{i} $. La condition $ \PGCD\tilde{\xx}^{\underline{\sigma}}=1 $ implique donc que $ \PGCD\tilde{\xx}^{D(\sigma)}=1 $. Par cons\'equent, on a $ \sup_{\tau\in \Delta_{\max}}|\tilde{\xx}^{D(\tau)}|_{p}=1 $, et ainsi \[ H_{0}(\xx)=\prod_{\nu\in \Val(\QQ)}\sup_{\tau\in \Delta_{\max}}|\tilde{\xx}^{D(\tau)}|_{\nu}=\sup_{\sigma\in \Delta_{\max}}|\tilde{\xx}^{D(\sigma)}|. \] \end{proof}
 
Plut\^ot que de compter les $ \QQ $-points de hauteur born\'ee de $ X $, nous allons compter les $ \ZZ $-points de $ \tilde{X}_{0} $ en utilisant le lemme ci-dessous : 

\begin{lemma}\label{zpoints2} Pour $ m\in \NN $, soient \[ c(m)=\Card\{ P\in T(\QQ)\; |\; H(P)=m\}, \] \[ c_{0}(m)=\Card\{ P\in \tilde{X}_{0}\cap T_{0}(\QQ)\; |\; H_{0}(P_{0})=m\}  \] (o\`u $ T_{0}=\pi^{-1}(T) $). Alors $ c(m)=c_{0}(m)/2^{r} $. \end{lemma}
\begin{proof} Voir la d\'emonstration de \cite[Lemme 11.4.a)]{Sa}. \end{proof}

Ainsi, \'etant donn\'e un ouvert de Zariski $ V $ de $ X $, si l'on note \[ \mathcal{N}_{0,V}(B)=\Card\{P_{0}\in \tilde{Y}_{0}(\ZZ)\cap T_{0}(\QQ)\cap \pi^{-1}(V)\; |\; H_{0}(P_{0})\leqslant B \}\] (o\`u $ \tilde{Y}_{0} $ est l'hypersurface de $ \tilde{X}_{0} $ correspondant \`a l'hypersurface $ Y $ de $ X $), on a alors \[ \mathcal{N}_{V}(B)=\mathcal{N}_{0,V}(B)/2^{r}. \] Nous chercherons donc dor\'enavant \`a \'evaluer $ \mathcal{N}_{0,V}(B) $.
\`A partir d'ici, pour les g\'en\'eralit\'es sur les vari\'et\'es toriques et sur la construction de la hauteur sur l'hypersurface $ Y $ nous renvoyons le lecteur aux sections\;\ref{premprop} et\;\ref{hauteur}. Nous adopterons en particulier les notations de ces sections dans tout ce qui va suivre. \\

Dans le cas pr\'esent, $ \tilde{X}_{0}(\ZZ)\subset \ZZ^{n+r} $ peut \^etre d\'ecrit comme l'ensemble des $ (n+r) $-uplets d'entiers not\'es $ \xx $, tels que (cf. \cite[11.5]{Sa}) : 
\begin{equation}\label{condexistence3}
\exists \sigma \in \Delta_{\max} \; | \; \ \xx^{ \underline{\sigma}}\neq 0,
\end{equation}  

 \begin{equation}\label{condprim3}
\PGCD_{\sigma \in \Delta_{\max}}(\xx^{\underline{\sigma}})=1,
\end{equation} o\`u \begin{equation}
\xx^{ \underline{\sigma}}=\prod_{i\notin \sigma(1)}x_{i}.
\end{equation}

Nous allons \`a pr\'esent exprimer la hauteur $ H_{0} $ de fa\c{c}on plus explicite. Rappelons avant tout que, d'apr\`es\;\eqref{formuleaij2}, nous avons pour tout $ i\in \{1,...,n+r\} $,  \[ [D_{i}]=\sum_{j=1}^{r}a_{i,j}[D_{n+j}] \] avec $ a_{i,j}\in \NN $ pour tous $ i,j $. On a alors \[ [K_{X}^{-1}]=\sum_{i=1}^{n+r}[D_{i}]=\sum_{j=1}^{r}\underbrace{\left(\sum_{i=1}^{n+r}a_{i,j}\right)}_{n_{j}}[D_{n+j}]. \]
On remarque par ailleurs que l'action du tore de N\'eron-Severi $ T_{\NS} $ sur les points $ \xx=(x_{i})_{i\in \{1,...,n+r\}} $ du torseur universel $ X_{0} $ est donn\'ee par : \[ \forall \tt=(t_{1},...,t_{r})\in T_{\NS}, \; \tt.\xx=\left(\left(\prod_{j=1}^{r}t_{j}^{a_{i,j}}\right)x_{i}\right)_{i\in \{1,...,n+r\}}. \] Autrement dit, la variable $ x_{i} $ a pour poids $ (a_{i,1},a_{i,2},...,a_{i,r}) $. Consid\'erons \`a pr\'esent une hypersurface $ Y $ de $ X $ donn\'ee par une section globale $ s $ de $ \OO(D) $ o\`u $ D = \sum_{j=1}^{r}d_{j}[D_{n+j}] $. Une telle section $ s $ se rel\`eve en une unique fonction polynomiale $ F:\tilde{X}_{0}\ra \RR $ \`a coefficients rationnels v\'erifiant, pour tout $ \xx\in \QQ^{n+r} $ :  \[ \forall \tt=(t_{1},...,t_{r})\in T_{\NS}, \; \\ F(\tt.\xx) = \left(\prod_{j=1}^{r}t_{j}^{d_{j}}\right)F(\xx). \]   On a par ailleurs la proposition ci-dessous 
\begin{prop}
Si $ V $ est une intersection compl\`ete lisse dans une vari\'et\'e torique compl\`ete et lisse, et si de plus la restriction de $ \Pic (X) $ \`a $ \Pic(V) $ est un isomorphisme, alors le torseur universel au dessus de $ V $ est obtenu en prenant l'inverse de $ V $ dans le torseur universel au-dessus de $ X $.
\end{prop}
\begin{proof}
Voir \cite[Remarque 2.1.2 et Exemple 2.1.16]{Pe2}
\end{proof}
En particulier, dans le cas pr\'esent, d'apr\`es le th\'eor\`eme de Lefschetz, l'hypersurface de $ \tilde{X}_{0} $ d\'efinie par l'annulation de la fonction $ F $ correspond alors au torseur universel au-dessus de $ Y $. Par cons\'equent, en utilisant le lemme \ref{zpoints2}, les $ \QQ $-points de $ Y $ correspondent, modulo l'action des points de torsion de $ T_{\NS} $, aux $ \ZZ $-points $ \xx $  de $ \tilde{X}_{0} $ tels que $ F(\xx)=0 $.
\begin{rem}
Dans tout ce qui va suivre on supposera que l'hypersurface $ Y $ d\'efinie par $ F(\xx)=0 $ est lisse. En fait cette propri\'et\'e est vraie pour un ouvert dense de Zariski de coefficients $ (\alpha_{u})_{u\in P_{D}\cap \ZZ^{n} } $.
\end{rem}
Rappelons que l'on souhaite \'evaluer \begin{multline*} \mathcal{N}_{0,V}(B)  =\Card\{P_{0}\in \tilde{Y}_{0}(\ZZ)\cap T_{0}(\QQ)\cap \pi^{-1}(V)\; |\; H_{0}(P_{0})\leqslant B \} \\ =\Card\{\xx\in \ZZ^{n+r}\cap \pi^{-1}(V) \; |\; \forall i\in \{1,...,n+r\}\; x_{i}\neq 0, \\ \; F(\xx)=0, \; \PGCD_{\sigma\in \Delta_{\max}}(\xx^{\underline{\sigma}})=1, \; H_{0}(\xx)\leqslant B \} \end{multline*} pour un certain ouvert $ V $. Quitte \`a appliquer une inversion de M\"{o}bius appropri\'ee (cf. \cite[\S 11]{Sa} et section\;\ref{sommemobius}), on se ram\`ene \`a \'evaluer, pour $ \ee=(e_{1},...,e_{n+r})\in \NN^{n+r} $ de : \begin{multline*}
N_{\ee,U}(B)=\Card\{\xx\in (\ZZ\setminus\{0\})^{n+r}\cap U \; |\; \forall i\in \{1,...,n+r\}\; e_{i}|x_{i}, \\ \; F(\xx)=0, \; H_{0}(\xx)\leqslant B \},
\end{multline*} 
(o\`u $ U=\pi^{-1}(V) $ est l'ouvert de $ X_{1} $ cit\'e dans la remarque\;\ref{remouvU}) ce qui revient encore \`a estimer pour tout $ \sigma\in \Delta_{\max} $ (cf. \cite[\S 9, 11]{Sa}) : \begin{multline*}
N_{\ee,\sigma,U}(B)=\Card\{\xx\in (\ZZ\setminus\{0\})^{n+r}\cap U\cap C_{0,\sigma}(\RR) \; |\; \forall i\in \{1,...,n+r\}\; e_{i}|x_{i}, \\ \; F(\xx)=0, \; H_{0}(\xx)\leqslant B \},
\end{multline*} o\`u $ C_{0,\sigma}(\RR)=\pi^{-1}(C_{\sigma,\infty}(\RR)) $.\\

Consid\'erons un c\^one maximal $ \sigma\in\Delta_{\max}  $ engendr\'e par des \'el\'ements $ (v_{i})_{i\notin\{i_{1},...,i_{r}\}} $ (qui forment alors une base de $ \ZZ^{n} $ puisque $ \Delta $ est suppos\'e lisse et complet) pour $ i_{1},...,i_{r}\in \{1,...,n+r\} $ fix\'es. Par ailleurs, pour tout $ i\in \{1,...,n+r\}\setminus \{i_{1},...,i_{r}\} $, on note $ u_{i} $ l'unique vecteur tel que $ \langle u_{i},v_{k}\rangle=\delta_{i,k} $ pour tout $ k\notin\{i_{1},...,i_{r}\} $, et on prend d'autre part $ u_{i}=\0 $ pour tout $ i \in \{i_{1},...,i_{r}\} $. On pose alors pour tout $ i $ : \[ E(i)= D_{i}-\sum_{k=1}^{n+r}\langle u_{i},v_{k}\rangle D_{k}. \] On a en particulier $ E(i_{j})=D_{i_{j}} $ pour tout $ j\in \{1,...,r\} $. 
\begin{rem}\label{reminv3}
D\'efinissons la matrice $ (\beta_{n+j,k})_{1\leqslant j,k \leqslant r}\in \MM_{r}(\ZZ) $ comme la matrice inverse de $ (a_{i_{l},j})_{1\leqslant l,j \leqslant r} $. On remarque alors que pour tout $ j\in\{1,...,r\} $, \[ E(n+j)=\sum_{k=1}^{r}\beta_{n+j,k}D_{i_{k}}. \]
\end{rem}
\begin{rem}\label{remE(i)3}
Si $ (e_{1},...,e_{n+r}) $ d\'esigne la base duale de $ (D_{1},...,D_{n+r}) $ (vue comme une base de $ \Div_{T}(X) $), les diviseurs $ E(i) $ sont caract\'eris\'es par $ [E(i)]=[D_{i}] $ et $ \langle E(i),e_{k}\rangle=0 $ pour tout $ k\in \sigma(1) $.
\end{rem}
On a alors que (par la description de $ C_{\sigma,\infty} $ donn\'ee par la proposition\;\ref{descriptC2}) : \[ C_{0,\sigma}(\RR)=\pi^{-1}(C_{\sigma,\infty}(\RR))=\left\{ \xx\in \RR^{n+r}\; |\;  \forall i\in \{1,...,n+r\}, \; |x_{i}|\leqslant \left|\xx^{E(i)}\right|\right\}. \] On observe d'autre part que, pour tout $ \xx\in C_{0,\sigma}(\RR) $, \[ \max_{\tau\in \Delta_{\max}}|\xx^{D(\tau)}|=|\xx^{D(\sigma)}|. \] Par cons\'equent : \begin{multline*} N_{\ee,\sigma,U}(B)= \left\{ \xx\in (\ZZ\setminus\{0\})^{n+r}\cap U\; | \;  \forall i\in \{1,...,n+r\}\; e_{i}|x_{i}, \;  \right. \\ \left. |x_{i}|\leqslant \left|\xx^{E(i)}\right|,\; F(\xx)=0 \; \et \; \left|\xx^{D(\sigma)}\right|\leqslant B \right\}. \end{multline*}
Remarquons que, $ D(\sigma) $ \'etant caract\'eris\'e par $ [D(\sigma)]=\sum_{j=1}^{r}(n_{j}-d_{j})[D_{n+j}] $ et $ \langle D(\sigma),e_{k}\rangle=0 $ pour tout $ k\in \sigma(1) $, on a, d'apr\`es la remarque\;\ref{remE(i)3} : \begin{equation*}  D(\sigma)=\sum_{j=1}^{r}(n_{j}-d_{j})E(n+j). \end{equation*}
Ainsi, \[ \xx^{D(\sigma)}=\prod_{j=1}^{r}\left(\xx^{E(n+j)}\right)^{n_{j}-d_{j}}. \]
On remarque par ailleurs que, toujours d'apr\`es la remarque\;\ref{remE(i)3}  : \[\forall i\in \{1,...,n+r\}, \; E(i)=\sum_{j=1}^{r}a_{i,j}E(n+j). \] On a donc \[ \xx^{E(i)}=\prod_{j=1}^{r}\left(\xx^{E(n+j)}\right)^{a_{i,j}}. \] Ainsi,   \begin{multline*} N_{\ee,\sigma,U}(B)= \left\{ \xx\in (\ZZ\setminus\{0\})^{n+r}\cap U\; | \;  \forall i\in \{1,...,n+r\}\; e_{i}|x_{i}, \right. \\ \left.\; |x_{i}|\leqslant \prod_{j=1}^{r}\left|\xx^{E(n+j)}\right|^{a_{i,j}}, \; F(\xx)=0, \; \prod_{j=1}^{r}\left|\xx^{E(n+j)}\right|^{n_{j}-d_{j}}\leqslant B \right\}, \end{multline*}
que nous pouvons encore r\'ecrire : \begin{multline*} N_{\ee,\sigma,U}(B)= \left\{ \xx\in (\ZZ\setminus\{0\})^{n+r} \; | \;  \ee.\xx\in U, \;F(\ee.\xx)=0, \;\right. \\ \left.  \forall i\in \{1,...,n+r\},  \; |x_{i}|\leqslant \frac{1}{e_{i}}\prod_{j=1}^{r}\left|(\ee.\xx)^{E(n+j)}\right|^{a_{i,j}}, \; \prod_{j=1}^{r}\left|(\ee.\xx)^{E(n+j)}\right|^{n_{j}-d_{j}}\leqslant B \right\}, \end{multline*}
o\`u l'on a pos\'e $ \ee.\xx=(e_{1}x_{1},....,e_{n+r}x_{n+r}) $. Cette r\'ecriture de $ N_{\ee,\sigma,U}(B)$ en termes de $ (\ee.\xx)^{E(n+j)} $ s'av\`erera cruciale pour ce qui va suivre. \\

Posons \`a pr\'esent pour tout $ (k_{1},...,k_{r})\in \NN^{r} $ \begin{multline*} \underline{h}_{\ee,U}(k_{1},...,k_{r})=\Card\left\{ 
 \xx\in (\ZZ\setminus\{0\})^{n+r} \; | \;  \ee.\xx\in U, \;F(\ee.\xx)=0, \;\right. \\ \left. \forall j\in \{1,...,r\}, \;  \left\lfloor|(\ee.\xx)^{E(n+j)}|\right\rfloor+1=k_{j}, \;  \forall i\in \{1,...,n+r\},  \; |x_{i}|\leqslant \frac{1}{e_{i}}\prod_{j=1}^{r}k_{j}^{a_{i,j}} \right\},
\end{multline*} et 
 \begin{multline*}\overline{h}_{\ee,U}(k_{1},...,k_{r})=\Card\left\{ 
 \xx\in (\ZZ\setminus\{0\})^{n+r} \; | \;  \ee.\xx\in U, \;F(\ee.\xx)=0, \;\right. \\ \left. \forall j\in \{1,...,r\}, \;  \left\lfloor|(\ee.\xx)^{E(n+j)}|\right\rfloor=k_{j}, \;  \forall i\in \{1,...,n+r\},  \; |x_{i}|\leqslant \frac{1}{e_{i}}\prod_{j=1}^{r}(k_{j}+1)^{a_{i,j}} \right\},
\end{multline*}
(on peut observer que $ \overline{h}_{\ee,U}(\kk)=\underline{h}_{\ee,U}(\kk-\1) $) et on remarque alors que : \[ \sum_{\prod_{j=1}^{r}k_{j}^{n_{j}-d_{j}}\leqslant B}\underline{h}_{\ee,U}(k_{1},...,k_{r})\leqslant N_{\ee,\sigma,U}(B) \leqslant \sum_{\prod_{j=1}^{r}k_{j}^{n_{j}-d_{j}}\leqslant B}\overline{h}_{\ee,U}(k_{1},...,k_{r}). \] 
En appliquant un analogue du r\'esultat de Blomer et Br\"{u}dern sur les sommation hyperboliques (cf. \cite{BB}), nous allons montrer qu'il existe une constante $ C_{\sigma,\ee} $ et des r\'eels $ \beta_{1},...,\beta_{n+r} $ tels que \[ \sum_{\prod_{j=1}^{r}k_{j}^{n_{j}-d_{j}}\leqslant B}\underline{h}_{\ee,U}(k_{1},...,k_{r})=C_{\sigma,\ee}B\log(B)^{r-1}+O\left(\left(\prod_{i=1}^{n+r}e_{i}^{\beta_{i}}\right) B\log(B)^{r-2}\right), \] \[ \sum_{\prod_{j=1}^{r}k_{j}^{n_{j}-d_{j}}\leqslant B}\overline{h}_{\ee,U}(k_{1},...,k_{r})=C_{\sigma,\ee}B\log(B)^{r-1}+O\left(\left(\prod_{i=1}^{n+r}e_{i}^{\beta_{i}}\right) B\log(B)^{r-2}\right), \]
et nous en d\'eduirons donc que \[ N_{\ee,\sigma,U}(B)=C_{\sigma,\ee}B\log(B)^{r-1}+O\left(\left(\prod_{i=1}^{n+r}e_{i}^{\beta_{i}}\right) B\log(B)^{r-2}\right). \] 
Pour cela, il est n\'ecessaire de calculer, pour tout $ J\subset \{1,...,r\} $, pour $ (t_{j})_{j\in \{1,...,r\}\setminus J}\in \NN^{r-\Card J} $ et pour $ (k_{j})_{j\in J} $ fix\'es, une formule asymptotique de la somme \[ \sum_{\forall j\notin J, \; k_{j}\leqslant t_{j}}h_{\ee}(k_{1},...,k_{r}). \] Nous allons commencer par traiter, dans la section suivante, le cas o\`u $ J=\emptyset $ avec des bornes $ t_{1},...,t_{r} $ \og relativement proches \fg.

\section{Premi\`ere \'etape}

\subsection{Pr\'eliminaires}

En pr\'eliminaire aux in\'egalit\'es de Weyl dans ce cadre plus g\'en\'eral, nous construisons, dans cette section, des op\'erateurs alg\'ebriques que nous appliquerons au polyn\^ome $ F $.\\

\`A chaque entier $ i\in \{1,...,n+r\} $, nous associons un poids $ \aa_{i}=(a_{i,1},...,a_{i,r})\in \NN^{r} $ ainsi qu'un ensemble fini $ E_{i} $. Posons \[ \CCC=\{(j,k)\; |\; j\in \{1,...,r\}, \; k\neq 0,\; \exists i \in \{1,...,n+r\}, \; a_{i,j}=k\}, \] et fixons un ensemble de degr\'es $ (t_{j,k})_{(j,k)\in \CCC}\in \NN^{\CCC} $. Posons alors \[ I_{0}=\{i\in \{1,...,n+r\}\; |\; \forall j\in \{1,...,r\}, \; t_{j,a_{i,j}}\neq 0 \ou a_{i,j}=0\}, \] \[ \CCC_{0}=\{(j,k)\; |\; j\in \{1,...,r\}, \; k\neq 0,\; \exists i \in I_{0}, \; a_{i,j}=k\}\subset \CCC, \] \[ \tt=(t_{j,k})_{(j,k)\in \CCC_{0}}, \] \[ \forall (j,k)\in \CCC_{0}, \; J(j,k)=\{i \in I_{0} \; |\; a_{i,j}=k \}. \] On consid\`ere alors un ensemble de variables $ \XX=(x_{i}^{(l)})_{(i,l)\in \EE} $, o\`u $ \EE=\{(i,l)\; |\; i\in I_{0}, \; l\in E_{i}\} $. Pour tout $ (j,k)\in \CCC_{0} $, notons \[ \EE(j,k)=\{(i,l)\in \EE \; |\; i\in J(j,k)\}=\{(i,l)\in \EE \; |\; a_{i,j}=k\}. \] Consid\'erons un polyn\^ome $ \FFF\in \ZZ[\XX] $. Pour tout $ (j,k)\in \CCC_{0} $, on d\'efinit l'op\'erateur $ \Delta_{(j,k)}^{(t_{j,k})} $ par : \begin{multline*} \Delta_{(j,k)}^{(t_{j,k})}\FFF\left( (x_{i}^{(l)})_{(i,l)\notin \EE(j,k)},(x_{i}^{(l,h)})_{\substack{(i,l)\in \EE(j,k) \\ h\in \{1,...,t_{j,k}\}  }}\right) \\ =\FFF_{t_{j,k}+1}\left( (x_{i}^{(l)})_{(i,l)\notin \EE(j,k)},(x_{i}^{(l,h)})_{\substack{(i,l)\in \EE(j,k) \\ h\in \{1,...,t_{j,k}+1\}  }}\right) \\ -\FFF_{t_{j,k}}\left( (x_{i}^{(l)})_{(i,l)\notin \EE(j,k)},(x_{i}^{(l,h)})_{\substack{(i,l)\in \EE(j,k) \\ h\in \{1,...,t_{j,k}\}  }}\right), \end{multline*} o\`u, pour tout $ t\in \NN $ : \begin{multline*}  \FFF_{t}\left( (x_{i}^{(l)})_{(i,l)\notin \EE(j,k)},(x_{i}^{(l,h)})_{\substack{(i,l)\in \EE(j,k) \\ h\in \{1,...,t\}  }}\right) \\ = \sum_{(\varepsilon_{1},...,\varepsilon_{t})\in \{0,1\}^{t}}(-1)^{\sum_{h=1}^{t}\varepsilon_{h}}\FFF\left( (x_{i}^{(l)})_{(i,l)\notin \EE(j,k)},(\sum_{h=1}^{t}\varepsilon_{h}x_{i}^{(l,h)})_{\substack{(i,l)\in \EE(j,k)  }}\right). \end{multline*}

\begin{rem}
Pour un $ j $ fix\'e, les op\'erateurs $ \Delta_{(j,k)}^{(t_{j,k)}} $ commutent deux \`a deux.  
\end{rem}

\begin{prop}\label{propprelim13}
Consid\'erons un ensemble d'entiers naturels $ \dd=(d_{j,k})_{(j,k)\in \CCC_{0}}\in \NN^{\CCC_{0}} $ et un polyn\^ome $ \GG_{\dd}((x_{i}^{(l)})_{(i,l)\in \EE}) $ homog\`ene de degr\'e $ d_{j,k} $ en les variables $ (x_{i}^{(l)})_{(i,l)\in \EE(j,k)} $. Alors pour tout $ (j_{0},k_{0})\in \CCC_{0} $ on a : \begin{enumerate}
\item Si $ d_{j_{0},k_{0}}<t_{j_{0},k_{0}} $, alors $ \Delta_{(j_{0},k_{0})}^{(t_{j_{0},k_{0}})}\GG_{\dd}=0 $,
\item Si $ d_{j_{0},k_{0}}=t_{j_{0},k_{0}} $, alors \begin{multline*} \Delta_{(j_{0},k_{0})}^{(t_{j_{0},k_{0}})}\GG_{\dd}\left( (x_{i}^{(l)})_{(i,l)\notin \EE(j_{0},k_{0})},(x_{i}^{(l,h)})_{\substack{(i,l)\in \EE(j_{0},k_{0}) \\ h\in \{1,...,t_{j_{0},k_{0}}\}  }}\right) \\ =-\GG_{\dd,t_{j_{0},k_{0}}}\left( (x_{i}^{(l)})_{(i,l)\notin \EE(j_{0},k_{0})},(x_{i}^{(l,h)})_{\substack{(i,l)\in \EE(j_{0},k_{0}) \\ h\in \{1,...,t_{j_{0},k_{0}}\}  }}\right) \end{multline*} et est lin\'eaire en chaque $ (x_{i}^{(l,h)})_{\substack{(i,l)\in \EE(j_{0},k_{0})}} $ pour tout $ h\in \{1,...,t_{j_{0},k_{0}}\} $ si $ \dd=\tt $.
\end{enumerate}
De plus, quel que soit $ \dd $, $ \Delta_{(j_{0},k_{0})}^{(t_{j_{0},k_{0}})}\GG_{\dd}\left( (x_{i}^{(l)})_{(i,l)\notin \EE(j_{0},k_{0})},(x_{i}^{(l,h)})_{\substack{(i,l)\in \EE(j_{0},k_{0}) \\ h\in \{1,...,t_{j_{0},k_{0}}\}  }}\right) $ est homog\`ene de degr\'e $ d_{j,k} $ en les variables $ (x_{i}^{(l)})_{(i,l)\in \EE(j,k)} $ pour tout $ (j,k)\neq (j_{0},k_{0}) $ et homog\`ene de degr\'e $ d_{j_{0},k_{0}} $ en les variables $ (x_{i}^{(l,h)})_{\substack{(i,l)\in \EE(j_{0},k_{0})\\ h\in\{1,...,t_{j_{0},k_{0}}\} }} $. 
\end{prop}
\begin{proof}
Ces r\'esultats d\'ecoulent de\;\cite[Lemme\;11.2]{Sm} (les r\'esultats de Schmidt, s'appliquant \`a n'importe quel anneau, y compris $ \ZZ[(x_{i}^{(l)})_{(i,l)\notin\EE(j_{0},k_{0})} $).
\end{proof}

On d\'efinit alors : \[ \Delta_{j}^{\tt_{j}}=\Delta_{(j,k_{1})}^{(t_{j,k_{1}})}\circ ... \circ\Delta_{(j,k_{m_{j}})}^{(t_{j,k_{m_{j}}})} \] o\`u $ \{k_{1},...,k_{m_{j}}\}=\{k\in \NN\setminus\{0\}\; |\; \exists i \in I_{0} \; |\; a_{i,j}=k\} $. Par d\'efinition on a \begin{multline*} \Delta_{j}^{\tt_{j}}\FFF\left( (x_{i}^{(l,h)})_{\substack{(i,l)\in \EE \\ h\in \{1,...,t_{j,a_{i,j}}\}  }}\right) \\ =(-1)^{\Card\PPP(j)}\sum_{(\varepsilon_{k})_{k\in \PPP(j)}\in \{0,1\}^{\PPP(j)}}(-1)^{\sum_{k\in \PPP(j)}\varepsilon_{k}}\FFF_{(t_{j,k}+\varepsilon_{k})_{k\in \PPP(j)}} \left( (x_{i}^{(l,h)})_{\substack{(i,l)\in \EE \\ h\in \{1,...,t_{j,a_{i,j}+\varepsilon_{a_{i,j}}}\}  }}\right), \end{multline*} o\`u \[ \PPP(j)=\{k\in \NN\setminus\{0\} \; |\; \exists i\in I_{0} \; |\; a_{i,j}=k\}, \] et pour tout $ (b_{k})_{k\in \PPP(j)}\in \NN^{\PPP(j)} $, \begin{multline*} 
 \FFF_{(b_{k})_{k\in \PPP(j)}}\left( (x_{i}^{(l,h)})_{\substack{(i,l)\in \EE \\ h\in \{1,...,b_{a_{i,j}}\}  }}\right)\\ = \sum_{(\varepsilon_{k,1},...,\varepsilon_{k,b_{k}})_{k\in \PPP(j)}\in \{0,1\}^{\sum_{k\in \PPP(j)}b_{k}}}(-1)^{\sum_{\substack{ k\in \PPP(j) \\ h\in\{1,...,b_{k}\}}}\varepsilon_{k,h}} \\ \FFF\left( \left(\left(\sum_{h=1}^{b_{k}}\varepsilon_{k,h}x_{i}^{(l,h)}\right)_{\substack{(i,l)\in \EE(j,k) \\h\in\{1,...,b_{k}\}  }}\right)_{k\in \PPP(j)}\right).
\end{multline*}
\begin{rem}
Les op\'erateurs $ \Delta^{\tt_{j}} $ commutent deux \`a deux.  
\end{rem}
On peut alors d\'efinir l'op\'erateur : \[ \Delta^{\tt}=\Delta^{\tt_{1}}\circ \Delta^{\tt_{2}}\circ...\circ \Delta^{\tt_{r}}. \]

Consid\'erons alors un ensemble de variables $ \xx=(x_{1},...,x_{n+r}) $. Tout polyn\^ome $ \FFF(\xx)\in \ZZ[\xx] $ peut \^etre d\'ecompos\'e de fa\c{c}on unique sous la forme : \[ \FFF(\xx)=\sum_{\dd=(d_{j,k})_{(j,k)\in \CCC_{0}}\in \NN^{\CCC_{0}}}\FFF_{\dd}(\xx), \] o\`u $ \FFF_{\dd}(\xx) $ est un polyn\^ome homog\`ene de degr\'e $ d_{j,k} $ en les variables $ (x_{i})_{i\in J(j,k)} $ pour tout $ (j,k)\in \CCC_{0} $. Introduisons \`a pr\'esent les notations suivantes :\\
\begin{notation}
 \begin{enumerate}
\item $ \xx_{0}=(x_{i})_{i\notin I_{0}} $ ,
\item $ \tilde{L}=\{(i,\l)\; |\; i\in I_{0}, \; \l=(l_{1},...,l_{r})\in\prod_{j=1}^{r}\{1,...,t_{j,a_{i,j}}+1\}\}  $,
\item $ L=\{(i,\l)\; |\; i\in I_{0}, \; \l=(l_{1},...,l_{r})\in\prod_{j=1}^{r}\{1,...,t_{j,a_{i,j}}\}\}  $,
\item $ \forall (j,k)\in \CCC_{0} $, $ L(j,k)=\{(i,\l)\in L \; |\; i\in J(j,k)\}  $,
\item  $ \forall (j,k)\in \CCC_{0} $, $ \widehat{L}(j,k)=\{(i,\l)\in L(j,k)\; |\; l_{j}\neq t_{j,k} \} $,
\item  $ \forall (j,k)\in \CCC_{0} $, $ E(j,k)=L(j,k)\setminus \widehat{L}(j,k)=\{(i,\l)\in L(j,k)\; |\; l_{j}=t_{j,k} \} $,
\item  $ \forall (j,k)\in \CCC_{0} $, $ \forall h\in \{1,...,t_{j,k}\} $, $ L(j,k,h)=\{(i,\l)\in L(j,k)\; |\; l_{j}=h\} $,
\item $ \widetilde{\XX}=(x_{i}^{\l})_{(i,\l)\in \tilde{L}} $,
\item $ \XX=(x_{i}^{\l})_{(i,\l)\in L} $,
\item  $ \forall (j,k)\in \CCC_{0} $, $ \widehat{\XX}_{(j,k)}=(x_{i}^{\l})_{(i,\l)\in \widehat{L}(j,k)} $.\\

\end{enumerate}
\end{notation}
Avec ces notations et d'apr\`es la proposition\;\ref{propprelim13}, nous avons alors le r\'esultat suivant : 
\begin{prop}\label{propprelim23}
Pour tout ensemble de degr\'es $ \dd=(d_{j,k})_{(j,k)\in \CCC_{0}}\in \NN^{\CCC_{0}} $, le polyn\^ome $ \Delta^{\tt}\FFF_{\dd}(\xx_{0},\widetilde{\XX}) $ vaut : \begin{enumerate}
\item $ 0 $ s'il existe $ (j,k)\in \CCC_{0} $ tel que $ d_{j,k}<t_{j,k} $,
\item un polyn\^ome $ \Gamma(\XX) $ en $ \XX $ de degr\'e $ d_{j,k} $ en $ (x_{i}^{\l})_{(i,\l)\in L(j,k)} $, et lin\'eaire en chaque $ (x_{i}^{\l})_{(i,\l)\in L(j,k,h)} $ pour tout $ (j,k)\in \CCC_{0} $ et $ h\in \{1,...,t_{j,k}\} $ si $ \dd=\tt $.
\end{enumerate}
\end{prop}
\begin{cor}\label{cordeltat3}
Pour tout polyn\^ome $ \FFF(\xx)\in \ZZ[\xx] $, on a \[ \Delta^{\tt}\FFF(\xx_{0},\widetilde{\XX})=\Gamma(\XX)+ \sum_{\dd>\tt}\Delta^{\tt}\FFF_{\dd}(\xx_{0},\widetilde{\XX}), \]
o\`u $ \dd>\tt $ signifie que $ d_{j,k}\geqslant t_{j,k} $ pour tout $ (j,k)\in \CCC_{0} $ et $ \dd\neq \tt $.
\end{cor}
Introduisons la d\'efinition suivante :  
\begin{Def}
Un polyn\^ome $ F\in \ZZ[\xx] $ sera dit \emph{quasi-$ r $-homog\`ene} de $ r $-degr\'e $ (d_{1},...,d_{r})\in \NN^{r} $ si \[ \forall \ss=(s_{1},...,s_{r})\in \CC^{r}, \; \; F(\ss.\xx)=(\prod_{j=1}^{r}s_{j}^{d_{j}})F(\xx), \] o\`u $ \ss.\xx=\left((\prod_{j=1}^{r}s_{j}^{a_{i,j}})x_{i}\right)_{i\in \{1,...,n+r\}} $.
\end{Def}
\begin{prop}\label{propdeltat3}
Si $ \FFF $ est un polyn\^ome quasi-$ r $-homog\`ene de $ r $-degr\'e $ (d_{1},...,d_{r}) $ et si l'ensemble de degr\'es $ (t_{j,k})_{(j,k)\in \CCC} $ est tel que \[ \forall j\in \{1,...,r\}, \;\sum_{k\geqslant 1}kt_{j,k}=d_{j}, \] alors,   \[ \Delta^{\tt}\FFF(\xx_{0},\widetilde{\XX})=\Gamma(\XX). \]  
\end{prop}
\begin{proof}
D'apr\`es le corollaire\;\ref{cordeltat3},  \[ \Delta^{\tt}\FFF(\xx_{0},\widetilde{\XX})=\Gamma(\XX)+\sum_{\dd>\tt}\Delta^{\tt}\FFF_{\dd}(\xx_{0},\widetilde{\XX}). \]
Par d\'efinition, pour tout $ \dd $, $ \FFF_{\dd} $ est quasi-$ r $-homog\`ene de $ r $-degr\'e \linebreak $ (\sum_{k\geqslant 1}kd_{j,k})_{j\in \{1,...,r\}} $. Or ce polyn\^ome est \'egalement quasi-$ r $-homog\`ene de $ r $-degr\'e $ (d_{1},...,d_{r}) $ (car $ \FFF $ l'est). Le polyn\^ome $ \FFF_{\dd} $ est donc nul si la condition \[ \forall j\in \{1,...,r\}, \; \;  \sum_{k\geqslant 0}kd_{j,k}=d_{j} \] n'est pas v\'erifi\'ee. Si $ \dd>\tt $, on a alors que $ d_{j,k}\geqslant t_{j,k} $ pour tout $ (j,k)\in \CCC_{0} $ et il existe $ (j_{0},k_{0}) $ tel que $ d_{j_{0},k_{0}}> t_{j_{0},k_{0}} $. On a alors \[  \sum_{k\in \PPP(j_{0})}kd_{j_{0},k}> \sum_{k\in \PPP(j_{0})}kt_{j_{0},k}=d_{j_{0}}. \] Par cons\'equent, si $ \dd>\tt $, alors $ \FFF_{\dd}=0 $ et la proposition est d\'emontr\'ee. 
\end{proof}

\subsection{Une in\'egalit\'e de Weyl :}\label{sectionweyl3}

Plut\^ot que de consid\'erer directement la fonction $ \overline{h}_{\ee,V} $, posons \begin{multline*}  \overline{h}_{\ee}(k_{1},...,k_{r})=\Card\left\{ 
 \xx\in \ZZ^{n+r} \; | \; F(\ee.\xx)=0, \;\right. \\ \left. \forall j\in \{1,...,r\}, \;  \left\lfloor|(\ee.\xx)^{E(n+j)}|\right\rfloor=k_{j}, \;  \forall i\in \{1,...,n+r\},  \; |x_{i}|\leqslant \frac{1}{e_{i}}\prod_{j=1}^{r}(k_{j}+1)^{a_{i,j}} \right\} \end{multline*}
et cherchons \`a obtenir une formule asymptotique pour \begin{multline*} N_{\ee}(P_{1},...,P_{r})=\sum_{\forall j\notin J, \; k_{j}\leqslant P_{j}}\overline{h}_{\ee}(k_{1},...,k_{r}) \\ = \Card\left\{ 
 \xx\in \ZZ^{n+r} \; | \; F(\ee.\xx)=0, \; \forall j\in \{1,...,r\}, \;  \left\lfloor|(\ee.\xx)^{E(n+j)}|\right\rfloor\leqslant P_{j}, \; \right. \\ \left. \forall i\in \{1,...,n+r\},  \; |x_{i}|\leqslant \frac{1}{e_{i}}\prod_{j=1}^{r}(\left\lfloor|(\ee.\xx)^{E(n+j)}|\right\rfloor+1)^{a_{i,j}} \right\}. \end{multline*}
en utilisant la m\'ethode du cercle (la restriction \`a l'ouvert $ \pi^{-1}(V) $ se fera ult\'erieurement et induira un terme d'erreur lorsque les bornes $ P_{1},...,P_{r} $ sont \og assez proches\fg). Quitte \`a permuter les variables, nous pouvons supposer \[ P_{1}\geqslant P_{2}\geqslant ...\geqslant P_{r}, \] et nous poserons $ P_{j}=P_{r}^{b_{j}} $, avec $ b_{j}\geqslant 1 $ pour tout $ j\in \{1,...,r\} $. \\

Notons $ e $ la fonction $ x \mt \exp(2i\pi x) $ et remarquons que \[ N_{\ee}(P_{1},...,P_{r})=\int_{0}^{1}S_{\ee}(\alpha)d\alpha, \] o\`u $ S_{\ee}(\alpha) $ est la s\'erie g\'en\'eratrice : \begin{equation}\label{Salpha3}
S_{\ee}(\alpha)=\sum_{\substack{\xx\in \ZZ^{n+r} \\  \left\lfloor|(\ee.\xx)^{E(n+j)}|\right\rfloor\leqslant P_{j} \\ |x_{i}|\leqslant \frac{1}{e_{i}}\prod_{j=1}^{r}(\left\lfloor|(\ee.\xx)^{E(n+j)}|\right\rfloor+1)^{a_{i,j}} }}e\left(\alpha F(\ee.\xx)\right).
\end{equation}

L'objectif des sections\;\ref{sectionweyl3} et\;\ref{sectionreseau3} est de d\'emontrer la proposition ci-dessous ($ K $ \'etant un param\`etre bien choisi que nous pr\'eciserons), $ \mathfrak{M}(\theta) $ d\'esignant une r\'eunion d'arcs majeurs :

\begin{prop}\label{dilemme43}
Pour tout $ \alpha\in [0,1[ $ et tout $ \varepsilon>0 $ arbitrairement petit, l'une au moins des assertions suivantes est vraie : \begin{enumerate}
\item  on a la majoration : \[|S_{\ee}(\alpha)|\ll \left(\prod_{i=1}^{n+r}e_{i}\right)^{-1+\frac{1}{2^{r}}}\left(\prod_{j=1}^{r}P_{j}^{n_{j}+1}\right)P^{-K\theta+\varepsilon},\]
\item le r\'eel $ \alpha $ appartient \`a $ \mathfrak{M}(\theta) $.
\end{enumerate}
\end{prop}

Rappelons que chaque variable $ x_{i} $ admet un poids $ (a_{i,1},...,a_{i,r}) $ relatif \`a l'action de $ T_{\NS} $. 
Choisissons \`a pr\'esent, comme dans la proposition\;\ref{propdeltat3}, un ensemble de degr\'es $ (t_{j,k})_{\substack{j\in \{1,...,r\} \\ k\in \{1,...,d_{j}\} }} $ v\'erifiant \[ \forall j\in \{1,...,r\}, \;\sum_{k\geqslant 1}kt_{j,k}=d_{j}. \]

Dans ce qui va suivre nous allons appliquer une forme de diff\'erenciation de Weyl appropri\'ee, en utilisant les r\'esultats pr\'esent\'es dans le paragraphe pr\'ec\'edent, afin de donner une majoration assez pr\'ecise de $ S_{\ee}(\alpha) $. Remarquons que la majoration obtenue ne d\'ependra que de $ F_{\tt} $ pour l'ensemble de degr\'es $ \tt $ choisi.\\

Consid\'erons la majoration triviale : \begin{multline*}
|S_{\ee}(\alpha)|  \ll \sum_{k_{1}=1}^{P_{1}}...\sum_{k_{r}=1}^{P_{r}}\big|\sum_{\substack{\xx\in \ZZ^{n+r} \\ \left\lfloor|(\ee.\xx)^{E(n+j)}|\right\rfloor=k_{j}  \\ |x_{i}|\leqslant \frac{1}{e_{i}}\prod_{j=1}^{r}(k_{j}+1)^{a_{i,j}} }}e\left(\alpha F(\ee.\xx)\right)\big|
\\  = \sum_{k_{1}=1}^{P_{1}}...\sum_{k_{r}=1}^{P_{r}}\big|\sum_{\substack{\xx\in \ZZ^{n+r} \\ k_{j}\leqslant|(\ee.\xx)^{E(n+j)}| <k_{j}+1  \\ |x_{i}|\leqslant \frac{1}{e_{i}}\prod_{j=1}^{r}(k_{j}+1)^{a_{i,j}}}}e\left(\alpha F(\ee.\xx)\right)\big|.
\end{multline*}
Or, \'etant donn\'e que $ (\ee.\xx)^{E(n+j)}=\prod_{k=1}^{r}(e_{i_{k}}x_{i_{k}})^{\beta_{n+j,k}} $ et puisque, d'apr\`es la remarque\;\ref{reminv3}, la matrice $ (\beta_{n+j,k})_{1\leqslant j,k \leqslant r} $ est la matrice inverse de $ (a_{i_{l},j})_{1\leqslant l,j \leqslant r} $, on a que \[ \left\{ \begin{array}{l}k_{j}\leqslant|(\ee.\xx)^{E(n+j)}| <k_{j}+1  \\ |x_{i}|\leqslant \frac{1}{e_{i}}\prod_{j=1}^{r}(k_{j}+1)^{a_{i,j}}
\end{array}\right. \Leftrightarrow \left\{ \begin{array}{l}  \forall i, \; |x_{i}|\leqslant \frac{1}{e_{i}}\prod_{j=1}^{r}(k_{j}+1)^{a_{i,j}} \\ \forall i\in \{i_{1},...,i_{r}\}, \; \frac{1}{e_{i}}\prod_{j=1}^{r}k_{j}^{a_{i,j}}\leqslant|x_{i}|< \frac{1}{e_{i}}\prod_{j=1}^{r}(k_{j}+1)^{a_{i,j}}.
\end{array}\right.  \]
Posons \begin{multline*} \BB_{\kk}=\{\xx\in \ZZ^{n+r}\; |\;  \forall i, \; |x_{i}|\leqslant \frac{1}{e_{i}}\prod_{j=1}^{r}(k_{j}+1)^{a_{i,j}} \\ \forall i\in \{i_{1},...,i_{r}\}, \; \frac{1}{e_{i}}\prod_{j=1}^{r}k_{j}^{a_{i,j}}\leqslant|x_{i}|< \frac{1}{e_{i}}\prod_{j=1}^{r}(k_{j}+1)^{a_{i,j}}\},
\end{multline*}
et on a alors que \[  \BB_{\kk}=\bigcup_{\II\subset \{1,...,n+r\}}\BB_{\kk,\II}, \]
o\`u $ \BB_{\kk,\II} $ est la bo\^ite \[ \BB_{\kk,\II}=\{\xx\in \BB_{\kk} \; |\; \forall i\in \II, \; x_{i}\geqslant 0 \; \et\; \forall i\notin \II, \; x_{i}<0\} \]
que l'on peut noter \[ \BB_{\kk,\II}=\prod_{i=1}^{n+r} \BB_{\kk,\II,i}, \]   $ \BB_{\kk,\II,i} $ \'etant alors un intervalle de taille $ O(\frac{1}{e_{i}}\prod_{j=1}^{r}(k_{j}+1)^{a_{i,j}}) $. L'estimation de $ |S_{\ee}(\alpha)| $ se ram\`ene \`a celle de \[ \sum_{k_{1}=1}^{P_{1}}...\sum_{k_{r}=1}^{P_{r}}|S_{\ee,\kk,\II}(\alpha)| \]
o\`u $ \kk=(k_{1},...,k_{r}) $, $ \II\subset \{1,...,n+r\} $ et \begin{equation}
S_{\ee,\kk,\II}(\alpha)=\sum_{\substack{\xx\in \BB_{\kk,\II}}}e\left(\alpha F(\ee.\xx)\right).
\end{equation}

Consid\'erons \`a pr\'esent le plus grand entier $ k\in \{1,...,d_{r}\} $ tel que $ t_{r,k}\geqslant 1 $. Par une in\'egalit\'e de H\"{o}lder, nous obtenons : 
\begin{multline*}
\left|S_{\ee,\kk,\II}(\alpha)\right|^{2^{t_{r,k}}}=\left|\sum_{\substack{(x_{i})_{i\notin J(r,k)}\in\prod_{i\notin J(r,k)}\BB_{\kk,\II,i} }} \sum_{\substack{(x_{i})_{i\in J(r,k)} \in \prod_{i\in J(r,k)}\BB_{\kk,\II,i}}} e\left(\alpha F(\ee.\xx)\right)\right|^{2^{t_{r,k}}} \\ \ll \prod_{i\notin J(r,k)}\left( \frac{1}{e_{i}}\prod_{j=1}^{r}(k_{j}+1)^{a_{i,j}}\right)^{2^{t_{r,k}}-1}\sum_{\substack{(x_{i})_{i\notin J(r,k)} \\  |x_{i}|\leqslant \frac{1}{e_{i}}\prod_{j=1}^{r}(k_{j}+1)^{a_{i,j}}}}\left|S_{\ee,\kk,\II,(r,k)}(\alpha)\right|^{2^{t_{r,k}}}
\end{multline*}
avec \begin{equation}
S_{\ee,\kk,\II,(r,k)}(\alpha)=\sum_{\substack{(x_{i})_{i\in J(r,k)}\in \prod_{i\in J(r,k)}\BB_{\kk,\II,i}}}e\left(\alpha F(\ee.\xx)\right).
\end{equation}
Pour simplifier les notations nous noterons : \[ \forall i \in \{1,...,n+r\}, \; \; B_{i}=\frac{1}{e_{i}} \prod_{j=1}^{r}(k_{j}+1)^{a_{i,j}}, \; \; T_{i}=\frac{1}{e_{i}} \prod_{j=1}^{r}(P_{j}+1)^{a_{i,j}}, \] et nous supposerons que $ T_{i}\geqslant 1 $ pour tout $ i $.  
On pose par ailleurs \[ \mathcal{U}=\prod_{i\in J(r,k)}\BB_{\kk,\II,i} \] et on d\'efinit \[ \mathcal{U}^{D}=\mathcal{U}-\mathcal{U}, \] \begin{multline*} \mathcal{U}((x_{i}^{(1)})_{i\notin J(r,k)},...,(x_{i}^{(t)})_{i\notin J(r,k)}) \\ =\bigcap_{(\varepsilon_{1},...,\varepsilon_{t})\in \{0,1\}^{t}}(\mathcal{U}_{N}-\varepsilon_{1}(x_{i}^{(1)})_{i\notin J(r,k)}-...-\varepsilon_{t}(x_{i}^{(t)})_{i\notin J(r,k)}). \end{multline*}

Notons $ \mathcal{F}(\xx)= F(\ee.\xx) $ (pour $ \ee $ fix\'e). On remarque que ce polyn\^ome est quasi-$ r $-homog\`ene de $ r $-degr\'e $ (d_{1},...,d_{r}) $ (car $ F $ l'est). En utilisant l'\'equation $ (11.2) $ de \cite{Sm}, on obtient la majoration :
\begin{multline*}
|S_{\ee,\kk,\II,(r,k)}(\alpha)|^{2^{t_{r,k}}}\ll \left(\prod_{i\in J(r,k)}B_{i}\right)^{2^{t_{r,k}}-t_{r,k}-1}\sum_{\substack{(x_{i}^{(1)})_{i\in J(r,k)}\\ |x_{i}^{(1)}|\leqslant B_{i}}}...\sum_{\substack{(x_{i}^{(t_{r,k}-1)})_{i\in J(r,k)}\\ |x_{i}^{(t_{r,k}-1)}|\leqslant B_{i}}} \\ \left|\sum_{\substack{(x_{i}^{(t_{r,k})})_{i\in J(r,k)}\in \mathcal{U}\left((x_{i}^{(l)})_{\substack{l\in \{1,...,t_{r,k}-1\} \\ i\in J(r,k)}}\right)}} e\left( \alpha\mathcal{F}_{t_{r,k}}((x_{i}^{(l)})_{\substack{l\in \{1,...,t_{r,k}\} \\ i\in J(r,k)}})\right)\right|^{2}.
\end{multline*}
Or, on a 
\begin{multline*}
 \left|\sum_{\substack{(x_{i}^{(t_{r,k})})_{i\in J(r,k)}\in \mathcal{U}\left((x_{i}^{(l)})_{\substack{l\in \{1,...,t_{r,k}-1\} \\ i\in J(r,k)}}\right)}} e\left( \alpha\mathcal{F}_{t_{r,k}}((x_{i}^{(l)})_{\substack{l\in \{1,...,t_{r,k}\} \\ i\in J(r,k)}})\right)\right|^{2} \\ =\sum_{\substack{(y_{i}^{(t_{r,k})})_{i\in J(r,k)}\in \mathcal{U}\left((x_{i}^{(l)})_{\substack{l\in \{1,...,t_{r,k}-1\} \\ i\in J(r,k)}}\right)}}\sum_{\substack{(z_{i}^{(t_{r,k})})_{i\in J(r,k)}\in \mathcal{U}\left((x_{i}^{(l)})_{\substack{l\in \{1,...,t_{r,k}-1\} \\ i\in J(r,k)}}\right)}} \\ e\left( \alpha\left(\mathcal{F}_{t_{r,k}}((x_{i}^{(l)})_{\substack{l\in \{1,...,t_{r,k}-1\} \\ i\in J(r,k)}},(y_{i}^{(t_{r,k})})_{i\in J(r,k)})-\mathcal{F}_{t_{r,k}}((x_{i}^{(l)})_{\substack{l\in \{1,...,t_{r,k}-1\} \\ i\in J(r,k)}},(z_{i}^{(t_{r,k})})_{i\in J(r,k)})\right)\right).
\end{multline*}
et si l'on pose \[\begin{array}{l}
 (x_{i}^{(t_{r,k}+1)})_{i\in J(r,k)}=(y_{i}^{(t_{r,k})})_{i\in J(r,k)} \\  (x_{i}^{(t_{r,k})})_{i\in J(r,k)}=(z_{i}^{(t_{r,k})}-y_{i}^{(t_{r,k})})_{i\in J(r,k)}, 
\end{array} \] on observe que : \begin{multline*}
\mathcal{F}_{t_{r,k}}\left((x_{i}^{(l)})_{\substack{l\in \{1,...,t_{r,k}-1\} \\ i\in J(r,k)}},(y_{i}^{(t_{r,k})})_{i\in J(r,k)}\right)-\mathcal{F}_{t_{r,k}}\left((x_{i}^{(l)})_{\substack{l\in \{1,...,t_{r,k}-1\} \\ i\in J(r,k)}},(z_{i}^{(t_{r,k})})_{i\in J(r,k)}\right) \\ =\mathcal{F}_{t_{r,k}+1}\left((x_{i}^{(l)})_{\substack{l\in \{1,...,t_{r,k}+1\} \\ i\in J(r,k)}}\right)-\mathcal{F}_{t_{r,k}}\left((x_{i}^{(l)})_{\substack{l\in \{1,...,t_{r,k}\} \\ i\in J(r,k)}}\right) \\ =\Delta_{(r,k)}^{(t_{r,k})}\FFF\left((x_{i})_{i\notin  J(r,k)}, (x_{i}^{(l)})_{\substack{ i\in J(r,k)\\ l\in \{1,...,t_{r,k}+1\}}}\right).
\end{multline*}

On a ainsi : 
\begin{multline*}
|S_{\ee,\kk,\II}(\alpha)|^{2^{t_{r,k}}}\ll \left(\prod_{i\notin J(r,k)}T_{i}\right)^{2^{t_{r,k}}-1}\left(\prod_{\substack{i\in J(r,k)}}T_{i}\right)^{2^{t_{r,k}}-t_{r,k}-1} \\ \sum_{\substack{(x_{i})_{i\notin J(r,k)} \\ |x_{i}|\leqslant T_{i}}} \sum_{(x_{i}^{(l)})_{\substack{i\in J(r,k)\\ l\in \{1,...,t_{r,k}+1\}}}}e\left(\alpha\Delta_{(r,k)}^{(t_{r,k})}\FFF\left((x_{i})_{i\notin  J(r,k)}, (x_{i}^{(l)})_{\substack{ l\in \{1,...,t_{r,k}+1\}\\ i\in J(r,k)}}\right) \right)  \\ = \left(\prod_{i=1}^{n+r}T_{i}\right)^{2^{t_{r,k}}-1}\left(\prod_{i\in J(r,k)}T_{i}\right)^{-t_{r,k}}\sum_{\substack{(x_{i})_{i\notin J(r,k)} \\ |x_{i}|\leqslant T_{i}}} \sum_{(x_{i}^{(l)})_{\substack{l\in \{1,...,t_{r,k}+1\}\\ i\in J(r,k)}}} \\ e\left(\alpha \Delta_{(r,k)}^{(t_{r,k})}\FFF\left((x_{i})_{i\notin J(r,k)}, (x_{i}^{(l)})_{\substack{ l\in \{1,...,t_{r,k}+1\}\\i\in J(r,k) }}\right) \right).\end{multline*} Par la suite, on proc\`ede de m\^eme avec les variables $ (x_{i})_{i\in J(r,k')} $ o\`u $ k' $ est le plus grand entier naturel non nul tel que $ k'<k $ et $ t_{r,k'}\neq 0 $. En utilisant la majoration pr\'ec\'edente, par une in\'egalit\'e de H\"{o}lder on trouve : 

\begin{multline}\label{deuxSek3}
|S_{\ee,\kk,\II}(\alpha)|^{2^{t_{r,k}+t_{r,k'}}}\ll \left(\prod_{i=1}^{n+r}T_{i}\right)^{2^{t_{r,k}+t_{r,k'}}-2^{t_{r,k'}}}\left(\prod_{i\in J(r,k)}T_{i}\right)^{-t_{r,k}2^{t_{r,k'}}} \\ \left(\prod_{i\notin J(r,k)\cup J(r,k')}T_{i}\right)^{2^{t_{r,k'}}-1}\left(\prod_{\substack{i\in J(r,k)\\ l\in \{1,...,t_{r,k}+1\}}}T_{i}\right)^{2^{t_{r,k'}}-1}\sum_{\substack{(x_{i})_{i\notin J(r,k)\cup J(r,k')} }}\sum_{(x_{i}^{(l)})_{\substack{l\in \{1,...,t_{r,k}+1\}\\i\in J(r,k) }}} \\ \left| \sum_{\substack{(x_{i})_{i\in J(r,k')} \\ |x_{i}|\leqslant T_{i}}}  e\left(\alpha\Delta_{(r,k)}^{(t_{r,k})}\FFF\left((x_{i})_{i\notin  J(r,k)}, (x_{i}^{(l)})_{\substack{l\in \{1,...,t_{r,k}+1\} \\ i\in J(r,k)}}\right) \right)\right|^{2^{t_{r,k'}}}
\\ = \left(\prod_{i=1}^{n+r}T_{i}\right)^{2^{t_{r,k}+t_{r,k'}}-2^{t_{r,k'}}}\left(\prod_{i\notin J(r,k)\cup J(r,k') }T_{i}\right)^{2^{t_{r,k'}}-1}\left(\prod_{i\in J(r,k)}T_{i}\right)^{2^{t_{r,k'}}-t_{r,k}-1} \\ \sum_{\substack{(x_{i})_{i\notin J(r,k)\cup J(r,k')} }}\sum_{(x_{i}^{(l)})_{\substack{l\in \{1,...,t_{r,k}+1\}\\ i\in J(r,k)}}} |S_{\ee,\kk,(k',r)}(\alpha)|^{2^{t_{r,k'}}}
 \end{multline}
o\`u \begin{equation}
S_{\ee,\kk,(k',r)}(\alpha)= \sum_{\substack{(x_{i})_{i\in J(r,k')} \\ |x_{i}|\leqslant T_{i}}}  e\left(\alpha\Delta_{(r,k)}^{(t_{r,k})}\FFF\left((x_{i})_{i\notin  J(r,k)}, (x_{i}^{(l)})_{\substack{i\in J(r,k) \\ l\in \{1,...,t_{r,k}+1\} }}\right) \right).
\end{equation} 
De la m\^eme mani\`ere que pr\'ec\'edemment, on trouve : 
\begin{multline*}
|S_{\ee,\kk,(k',r)}(\alpha)|^{2^{t_{r,k'}}}\ll \left(\prod_{i\in J(r,k')}T_{i}\right)^{2^{t_{r,k'}}-t_{r,k'}-1}\sum_{\substack{(x_{i}^{(1)})_{i\in J(r,k)}}}...\sum_{\substack{(x_{i}^{(t_{r,k'+1})})_{i\in J(r,k')}}}\sum_{\substack{(x_{i}^{(t_{r,k'})})_{i\in J(r,k')}}} \\  e\left( \alpha \Delta_{(r,k')}^{(t_{r,k'})}\circ\Delta_{(r,k)}^{(t_{r,k})}\FFF\left((x_{i})_{i\notin  J(r,k)\cup J(r,k')}, (x_{i}^{(l)})_{\substack{ i\in J(r,k)\\l\in \{1,...,t_{r,k}\} }},(x_{i}^{(l)})_{\substack{i\in J(r,k')\\ l\in \{1,...,t_{r,k'}+1\} }}\right) \right)
\end{multline*}
On obtient alors, par la majoration \eqref{deuxSek3}, \begin{multline}|S_{\ee,\kk}(\alpha)|^{2^{t_{r,k}+t_{r,k'}}}\ll\left(\prod_{i=1}^{n+r}T_{i}\right)^{2^{t_{r,k}+t_{r,k'}}-1}\left(\prod_{i\in J(r,k)}T_{i}\right)^{-t_{r,k}}\left(\prod_{i\in J(r,k')}T_{i}\right)^{-t_{r,k'}}  \\ \sum_{\substack{(x_{i})_{i\notin J(r,k)\cup J(r,k') } }}\sum_{(x_{i}^{(l)})_{\substack{i\in J(r,k)\\ l\in \{1,...,t_{r,k}+1\}}}} \sum_{(x_{i}^{(l)})_{\substack{i\in J(r,k')\\ l\in \{1,...,t_{r,k'}+1\}}}} \\  e\left( \alpha\Delta_{(r,k')}^{(t_{r,k'})}\circ\Delta_{(r,k)}^{(t_{r,k})}\FFF\left((x_{i})_{i\notin J(r,k)\cup J(r,k') }, (x_{i}^{(l)})_{\substack{  i\in J(r,k)\\l\in \{1,...,t_{r,k}+1\}}},(x_{i}^{(l)})_{\substack{i\in J(r,k') \\ l\in \{1,...,t_{r,k'}+1\}}}\right) \right)
 \end{multline}
Par r\'ecurrence, on obtient finalement, en posant pour tout $ j\in \{1,...,r\} $ \begin{equation} D_{j}=\sum_{k\geqslant 1 }t_{j,k}, \end{equation}

\begin{equation} J(j)=\bigcup_{k\geqslant 1 }J(j,k), \end{equation}\begin{multline}
|S_{\ee,\kk,\II}(\alpha)|^{2^{D_{r}}}\ll\left(\prod_{i=1}^{n+r}T_{i}\right)^{2^{D_{r}}-1}\left(\prod_{i\in J(r)}T_{i}^{-t_{r,a_{i,r}}}\right) \sum_{\substack{(x_{i})_{i\notin J(r)} }}\\ \sum_{(x_{i}^{(l)})_{\substack{i\in J(r)\\ l\in \{1,...,t_{r,a_{i,r}}+1\}}}}  e\left( \alpha\Delta^{\tt_{r}}\FFF\left((x_{i})_{i\notin  J(r)}, (x_{i}^{(l)})_{\substack{l\in \{1,...,t_{r,a_{i,r}}+1\} \\ i\in J(r)}}\right) \right).
\end{multline}
Si $ i\in I_{0} $ et $ i\notin J(r) $ (i.e si $ a_{i,r}=0 $), nous noterons dor\'enavant $ x_{i}^{(1)}=x_{i} $. Nous poserons par ailleurs, pour tout $ j\in \{1,...,r\} $, \[ t_{j,k}'=\left\{\begin{array}{lcr} t_{r,k}+1 & \mbox{si} & k\neq 0 \\ 1 & \mbox{si} & k=0. 
\end{array}\right.\]
La majoration pr\'ec\'edente peut alors \^etre r\'e\'ecrite :
\begin{multline}\label{formeDr3}
|S_{\ee,\kk,\II}(\alpha)|^{2^{D_{r}}}\ll\left(\prod_{i=1}^{n+r}T_{i}\right)^{2^{D_{r}}}\left(\prod_{i\notin I_{0}}T_{i}\right)^{-1}\left(\prod_{i\notin I_{0}}T_{i}^{-t_{r,a_{i,r}}'}\right) \sum_{\substack{(x_{i})_{i\notin I_{0}} }}\\ \sum_{(x_{i}^{(l)})_{\substack{i\in I_{0}\\ l\in \{1,...,t_{r,a_{i,r}}'\}}}}  e\left( \alpha\Delta^{\tt_{r}}\FFF\left((x_{i})_{i\notin  I_{0}}, (x_{i}^{(l)})_{\substack{l\in \{1,...,t_{r,a_{i,r}}'\} \\ i\in I_{0}}}\right) \right).
\end{multline}

Nous pouvons \`a pr\'esent effectuer les m\^emes op\'erations que pr\'ec\'edemment avec les variables $ (x_{i}^{(l)})_{\substack{i\in J(r-1,k) \\ l\in \{1,...,t_{r,k}'\} }} $ pour les $ k\geqslant 1 $ tels que $ t_{r-1,k}>0 $. On obtient alors la majoration :
\begin{multline*}
|S_{\ee,\kk,\II}(\alpha)|^{2^{D_{r}+D_{r-1}}}\ll\left(\prod_{i=1}^{n+r}T_{i}\right)^{2^{D_{r}+D_{r-1}}}\left(\prod_{i\notin I_{0}}T_{i}\right)^{-1}\left(\prod_{i\in  I_{0}}T_{i}^{-t_{r,a_{i,r}}'t_{r-1,a_{i,r-1}}'}\right) \\ \sum_{\substack{(x_{i})_{i\notin  I_{0} }}} \sum_{(x_{i}^{(l,l')})_{\substack{l\in \{1,...,t_{r,a_{i,r}}'\} \\ l'\in \{1,...,t_{r-1,a_{i,r-1}}'\} \\ i\in  I_{0}}}}  e\left( \alpha\Delta^{\tt_{r-1}}\circ\Delta^{\tt_{r}}\FFF\left((x_{i})_{i\notin I_{0}}, (x_{i}^{(l,l')})_{\substack{l\in \{1,...,t_{r,a_{i,r}}'\} \\ l'\in \{1,...,t_{r-1,a_{i,r-1}}'\} \\ i\in  I_{0}}}\right) \right).
\end{multline*}
En proc\'edant de m\^eme, on obtient, par r\'ecurrence : 
\begin{multline}\label{weyl13}
|S_{\ee,\kk,\II}(\alpha)|^{2^{D_{1}+D_{2}+...+D_{r}}}\ll\left(\prod_{i=1}^{n+r}T_{i}\right)^{2^{D_{1}+...+D_{r}}}\left(\prod_{i\notin I_{0}}T_{i}\right)^{-1}\left(\prod_{i\in  I_{0}}T_{i}^{-\prod_{j=1}^{r}t_{j,a_{i,j}}'}\right) \\ \sum_{\substack{\xx_{0}=(x_{i})_{i\notin  I_{0} }}} \sum_{\widetilde{\XX}=(x_{i}^{\l})_{\substack{(i,\l)\in \tilde{L}}}}  e\left( \alpha\Delta^{\tt_{1}}\circ...\circ\Delta^{\tt_{r}}\FFF\left(\xx_{0}, \widetilde{\XX}\right) \right).
\end{multline}
Rappelons que, d'apr\`es la proposition\;\ref{propdeltat3}, on peut \'ecrire : 

\[ \Delta^{\tt}\FFF(\xx_{0},\widetilde{\XX})=\Gamma(\XX), \]
o\`u $ \Gamma(\XX)\in \ZZ[\XX] $ est de degr\'e $ t_{j,k} $ en $ (x_{i}^{\l})_{(i,\l)\in L(j,k)} $ et lin\'eaire en chaque $ (x_{i}^{\l})_{(i,\l)\in L(j,k,h)} $ pour tout $ (j,k)\in \CCC_{0} $ et $ h\in \{1,...,t_{j,k}\} $. En particulier, on remarque que pour tout $ (j,k)\in \CCC_{0} $, on peut \'ecrire $ \Gamma(\XX) $ sous la forme : \[ e_{i}\sum_{(i,\l)\in E(j,k) }\gamma_{(i,\l)}^{(j,k)}(\widehat{\XX}_{(j,k)})\x_{i}^{\l}, \] o\`u $ \gamma_{(i,\l)}^{(j,k)}(\widehat{\XX}_{(j,k)}) $ est de degr\'e $ t_{j,k}-1 $ en $ (x_{i}^{\l})_{(i,\l)\in L(j,k)} $ et lin\'eaire en chaque $ (x_{i}^{\l})_{(i,\l)\in L(j,k,h)} $ pour tout $ h\in \{1,...,t_{j,k}-1\} $. \\

Notons \`a pr\'esent \begin{equation}
K=\max_{t_{1,k}\neq 0}k.
\end{equation}
Nous pouvons majorer la somme apparaissant dans l'in\'egalit\'e\;\eqref{weyl13} par : \begin{multline}\label{sommeweyl3}
\sum_{\substack{\xx_{0}}} \sum_{(x_{i}^{\l})_{\substack{(i,\l)\in \tilde{L}\setminus E(1,K)}}} \left| \sum_{(x_{i}^{\l})_{(i,\l)\in E(1,K)}} e\left( \alpha\Delta^{\tt}\FFF(\xx_{0},\widehat{\XX}_{(1,K)},(x_{i}^{\l})_{(i,\l)\in E(1,K)}) \right)\right| \\ =\sum_{\substack{\xx_{0}}} \sum_{(x_{i}^{\l})_{\substack{(i,\l)\in \tilde{L}\setminus E(1,K)}}} \left| \sum_{(x_{i}^{\l})_{(i,\l)\in E(1,K)}} e\left( \alpha e_{i}\sum_{(i,\l)\in E(1,K) }\gamma_{(i,\l)}^{(1,K)}(\widehat{\XX}_{(1,K)})x_{i}^{\l} \right)\right|.
\end{multline}
La derni\`ere somme de \eqref{sommeweyl3} peut \^etre r\'e\'ecrite : \begin{multline}\label{prodweil3}
\prod_{\substack{(i,\l)\in E(1,K)}}\left| \sum_{\substack{(x_{i}^{\l})_{(i,\l)\in E(1,K)}}}e\left( \alpha e_{i}\gamma_{(i,\l)}^{(1,K)}(\widehat{\XX}_{(1,K)})\x_{i}^{\l}\right)\right|.
\end{multline}
En notant pour tout r\'eel $ x $ \[ ||x||=\inf_{m\in \ZZ }|x-m|, \] et en consid\'erant la majoration, pour $ P\gg 1 $ \[ \sum_{m\in I(P)\cap\ZZ}e(\beta m) \ll \min\{P,||\beta||^{-1}\} \] pour tout intervalle $ I(P) $ de taille $ O(P) $, on remarque que l'on peut majorer \eqref{prodweil3} par : \begin{multline}\label{facteur3}
\prod_{\substack{(i,\l)\in E(1,K)}} \min\left( T_{i},||\alpha e_{i}\gamma_{(i,\l)}^{(1,K)}(\widehat{\XX}_{(1,K)})||^{-1}\right).
\end{multline}
On obtient ainsi, \`a partir de\;\eqref{weyl13}, la majoration 

\begin{multline*}
|S_{\ee,\kk,\II}(\alpha)|^{2^{D_{1}+D_{2}+...+D_{r}}}\ll\left(\prod_{i=1}^{n+r}T_{i}\right)^{2^{D_{1}+...+D_{r}}}\left(\prod_{i\notin I_{0}}T_{i}\right)^{-1}\left(\prod_{i\in  I_{0}}T_{i}^{-\prod_{j=1}^{r}t_{j,a_{i,j}}'}\right) \\  \sum_{\substack{\xx_{0}}} \sum_{(x_{i}^{\l})_{\substack{(i,\l)\in \tilde{L}\setminus E(1,K)}}}\prod_{\substack{(i,\l)\in E(1,K)}} \min\left( T_{i},||\alpha e_{i}\gamma_{(i,\l)}^{(1,K)}(\widehat{\XX}_{(1,K)})||^{-1}\right)\\  \ll \left(\prod_{i=1}^{n+r}T_{i}\right)^{2^{D_{1}+...+D_{r}}}\left(\prod_{i\in  I_{0}}T_{i}^{-\prod_{j=1}^{r}t_{j,a_{i,j}}}\right)   \sum_{\widehat{\XX}_{(1,K)}=(x_{i}^{\l})_{\substack{(i,\l)\in L\setminus E(1,K)}}}\\ \prod_{\substack{(i,\l)\in E(1,K)}} \min\left( T_{i},||\alpha e_{i}\gamma_{(i,\l)}^{(1,K)}(\widehat{\XX}_{(1,K)})||^{-1}\right)
\end{multline*}
en notant $ t_{j,0}=1 $ pour tout $ j\in \{1,...,r\} $ (puisque le terme\;\eqref{facteur3} est ind\'ependant des $ (x_{i}^{\l})_{\substack{(i,\l)\in \tilde{L}\\ i\in J(j,k), \; l_{j}=t_{j,k}+1}} $ pour tout $ (j,k)\in \CCC_{0} $, on peut remplacer les sommes en ces variables par un \linebreak $ O(\prod_{\substack{(j_{0},k)\in \CCC_{0} }}\prod_{i\in J(j_{0},k)}T_{i}^{t_{j_{0},a_{i,j_{0}}}\prod_{j\neq j_{0}}t_{j,a_{i,j}}'}) $, de telle sorte que les seules variables que nous aurons \`a consid\'erer dor\'enavant sont les variables $ (x_{i}^{\l})_{\substack{(i,\l)\in L}} $).\\

Consid\'erons \`a pr\'esent \[ \rr=(r_{i,\l})_{\substack{(i,\l)\in E(1,K)}}\in \prod_{\substack{(i,\l)\in E(1,K)}}\{0,1,...,T_{i}\} \] 
et choisissons \[ K'=\max_{k\in \PPP(1) \; |\; k<K, \;t_{1,K'}\neq 0} k,  \] o\`u $ \PPP(1)=\{k\in \NN\setminus\{0\} \; |\; \exists i\in I_{0} \; |\; a_{i,1}=k\} $. Pour \[ \widehat{\XX}_{\substack{(1,K) \\(1,K')}}=(x_{i}^{\l})_{\substack{(i,\l)\in\widehat{L}(1,K)\cap \widehat{L}(1,K')}} \] fix\'e, nous noterons $ \mathcal{A}\left( \widehat{\XX}_{\substack{(1,K) \\(1,K')}},\rr\right) $ l'ensemble des vecteurs $ (x_{m}^{\hh})_{\substack{(m,\hh)\in E(1,K')}} $ tels que : \[ \left|x_{m}^{\hh}\right|<T_{m}, \] et \begin{multline*} \forall (i,\l)\in E(1,K), \;  ||\alpha e_{i}\gamma_{(i,\l)}^{(1,K)}(\widehat{\XX}_{(1,K)})||\in [r_{i,\l}T_{i}^{-1},(r_{i,\l}+1)T_{i}^{-1}]
\end{multline*}
et $ A\left(\XX,\rr\right) $ le cardinal de cet ensemble. On a alors, en reprenant les \'equations \eqref{weyl13}, \eqref{prodweil3}, \eqref{sommeweyl3}, l'estimation \begin{multline*}
|S_{\ee,\kk,\II}(\alpha)|^{2^{D_{1}+D_{2}+...+D_{r}}}\ll\left(\prod_{i=1}^{n+r}T_{i}\right)^{2^{D_{1}+...+D_{r}}}\left(\prod_{i\in  I_{0}}T_{i}^{-\prod_{j=1}^{r}t_{j,a_{i,j}}}\right) \\  \sum_{\widehat{\XX}_{\substack{(1,K) \\(1,K')}}}\sum_{\rr}A\left(\widehat{\XX}_{\substack{(1,K) \\(1,K')}},\rr\right) \prod_{\substack{(i,\l)\in E(1,K)  }}\min\left( T_{i},\max\left(\frac{T_{i}}{r_{i,\l}},\frac{T_{i}}{T_{i}-r_{i,\l}-1}\right)\right).
\end{multline*}
Si par ailleurs $ (x_{m}^{\hh})_{\substack{(m,\hh)\in E(1,K')}} $ et $ (y_{m}^{\hh)})_{\substack{(m,\hh)\in E(1,K')}} $ sont des \'el\'ements de $ \mathcal{A}\left(\XX,\rr\right) $, on a alors que  
\[ |x_{m}^{\hh}-y_{m}^{\hh}|\ll T_{m}, \] et d'autre part, pour tous $ (i,\l)\in E(1,K) $ : \begin{multline*}
\gamma_{(i,\l)}^{(1,K)}(\widehat{\XX}_{\substack{(1,K) \\(1,K')}},(x_{m}^{\hh})_{\substack{(m,\hh)\in E(1,K')}}) \\ -\gamma_{(i,\l)}^{(1,K)}(\widehat{\XX}_{\substack{(1,K) \\(1,K')}},(y_{m}^{\hh})_{\substack{(m,\hh)\in E(1,K')}}) \\ = \gamma_{(i,\l)}^{(1,K)}(\widehat{\XX}_{\substack{(1,K) \\(1,K')}},(x_{m}^{\hh}-y_{m}^{\hh})_{\substack{(m,\hh)\in E(1,K')}}).
\end{multline*}
On note alors $ N\left(\widehat{\XX}_{\substack{(1,K) \\(1,K')}}\right) $ le cardinal de l'ensemble des $ (x_{m}^{\hh})_{\substack{(m,\hh)\in E(1,K')}} $ tels que  \[ \left|x_{m}^{\hh}\right|<T_{m}, \] \begin{multline*} \forall (i,\l)\in E(1,K), \;  ||\alpha e_{i}\gamma_{(i,\l)}^{(1,K)}(\widehat{\XX}_{\substack{(1,K) \\(1,K')}},(x_{m}^{\hh})_{\substack{(m,\hh)\in E(1,K')}}) ||<T_{i}^{-1},
\end{multline*}
et on obtient donc  \begin{multline}\label{N(X)3}
\sum_{\rr}A\left(\widehat{\XX}_{\substack{(1,K) \\(1,K')}},\rr\right)\prod_{\substack{(i,\l)\in E(1,K)}}\min\left( T_{i},\max\left(\frac{T_{i}}{r_{i,\l}},\frac{T_{i}}{T_{i}-r_{i,\l}-1}\right)\right) \\ \ll N\left(\widehat{\XX}_{\substack{(1,K) \\(1,K')}}\right) \sum_{\rr}\prod_{\substack{(i,\l)\in E(1,K)}}\min\left( T_{i},\max\left(\frac{T_{i}}{r_{i,\l}},\frac{T_{i}}{T_{i}-r_{i,\l}-1}\right)\right) \\ =N\left(\widehat{\XX}_{\substack{(1,K) \\(1,K')}}\right)\prod_{\substack{(i,\l)\in E(1,K)}}\sum_{r_{i,\l}}\min\left( T_{i},\max\left(\frac{T_{i}}{r_{i,\l}},\frac{T_{i}}{T_{i}-r_{i,\l}-1}\right)\right)\\  \ll N\left(\widehat{\XX}_{\substack{(1,K) \\(1,K')}}\right)\prod_{\substack{(i,\l)\in E(1,K)}}T_{i}\log(T_{i}).
\end{multline}
Notons pour tout $ (j,k)\in \CCC_{0} $ : \begin{equation}
\MM_{(j,k)}\left(\alpha,(A_{m,\hh})_{\substack{(m,\hh)\in L\setminus E(j,k)}}, (B_{i,\l})_{\substack{(i,\l)\in E(j,k)}}\right)
\end{equation} 
 l'ensemble des $ \widehat{\XX}_{(j,k)} $ tels que, pour tout $ (m,\hh)\in L\setminus E(j,k) $ \[ |x_{m}^{(\hh)}|<A_{m,\hh}, \] et pour tout $ (i,\l)\in E(j,k) $  \[ ||\alpha e_{i}\gamma_{(i,\l)}^{(j,k)}(\widehat{\XX}_{\substack{(1,K)}}) ||<B_{i,\l}. \] On note par ailleurs \begin{equation}
M_{(j,k)}\left(\alpha,(A_{m,\hh})_{\substack{(m,\hh)\in L\setminus E(j,k)}}, (B_{i,\l})_{\substack{(i,\l)\in E(j,k)}}\right)
\end{equation} le cardinal de cet ensemble. On obtient (\`a partir de\;\eqref{N(X)3}) : 

\begin{multline}\label{weyl23}
|S_{\ee,\kk,\II}(\alpha)|^{2^{D_{1}+D_{2}+...+D_{r}}} \\ \ll\left(\prod_{i=1}^{n+r}T_{i}\right)^{2^{D_{1}+...+D_{r}}}\left(\prod_{i\in  I_{0}}T_{i}\right)^{-\prod_{j=1}^{r}t_{j,a_{i,j}}}\prod_{\substack{(i,\l)\in E(1,K)}} T_{i}\log(T_{i}) \\ \times M_{(1,K)}\left(\alpha,(T_{m})_{\substack{(m,\hh)\in L\setminus E(1,K)}}, (T_{i}^{-1})_{\substack{(i,\l)\in E(1,K)}}\right).
\end{multline}
En remarquant que \[ \prod_{\substack{(i,\l)\in E(1,K)}} T_{i}\log(T_{i})= \prod_{i\in J(1,K)}(T_{i}\log(T_{i}))^{\prod_{j=2}^{r}t_{j,a_{i,j}}}, \] et en sommant la majoration obtenue dans \eqref{weyl23} sur les $ \kk=(k_{1},...,k_{r}) $ tels que $ k_{j}\leqslant P_{j} $, et sur les $ \II\subset \{1,...,n+r\} $, on obtient le lemme ci-dessous : 
\begin{lemma}\label{dilemme13}
Pour tout $ \varepsilon>0 $ arbitrairement petit on a\begin{multline*}
|S_{\ee}(\alpha)| \ll\left(\prod_{j=1}^{r}P_{j}\right)\left(\prod_{i=1}^{n+r}T_{i}\right)^{(1+\varepsilon)}\left(\prod_{i\in  I_{0}}T_{i}^{-\frac{\prod_{j=1}^{r}t_{j,a_{i,j}}}{2^{D_{1}+D_{2}+...+D_{r}}}}\right)\\ \times \left(\prod_{i\in J(1,K)}T_{i}^{\prod_{j=2}^{r}t_{j,a_{i,j}}}\right)^{2^{-(D_{1}+D_{2}+...+D_{r})}} \\ \times \left( M_{(1,K)}\left(\alpha,(T_{m})_{\substack{(m,\hh)\in L\setminus E(1,K)}}, (T_{i}^{-1})_{\substack{(i,\l)\in E(1,K)}}\right)\right)^{2^{-(D_{1}+D_{2}+...+D_{r})}},
\end{multline*}
\end{lemma}

\subsection{G\'eom\'etrie des r\'eseaux}\label{sectionreseau3}
Dans cette partie, nous allons chercher \`a obtenir une bonne estimation de \begin{equation*}
 M_{(1,K)}\left(\alpha,(T_{m})_{\substack{(m,\hh)\in L\setminus E(1,K)}}, (T_{i}^{-1})_{\substack{(i,\l)\in E(1,K)}}\right).
\end{equation*}
Plus pr\'ecis\'ement si $ \Lambda $ est un r\'eseau de dimension $ d $ et si $ C $ est un corps convexe de $ \Lambda\otimes \RR $ contenant l'origine, on a alors pour tout $ P\geqslant 1 $ et tout r\'eel $ \theta\in [0,1] $ : \begin{equation}\label{inegreseau3} \Card(\Lambda\cap PC)\ll (P/P^{\theta})^{d}\Card(\Lambda\cap P^{\theta}C). \end{equation} C'est le principe g\'en\'eral de la m\'ethode utilis\'ee dans ce qui va suivre : il s'agit de consid\'erer \begin{equation*}
 M_{(1,K)}\left(\alpha,(T_{m})_{\substack{(m,\hh)\in L\setminus E(1,K)}}, (T_{i}^{-1})_{\substack{(i,\l)\in E(1,K)}}\right).
\end{equation*} comme une famille d'intersections de r\'eseaux avec des corps convexes et d'utiliser les lemmes qui vont suivre (qui donnent des versions de l'in\'egalit\'e \eqref{inegreseau3} adapt\'ees aux cas qui nous int\'eressent) afin de r\'eduire les bornes intervenant dans $ M_{(1,K)} $. Nous allons \`a pr\'esent \'etablir des r\'esultats de g\'eom\'etrie des nombres qui nous serons utiles pour la suite de cette section. Il s'agit en fait de g\'en\'eralisations de \cite[Lemme 12.6]{D} et de \cite[Lemme 3.1]{S1}.

\subsubsection{Pr\'eliminaires}

\begin{lemma}\label{geomnomb12}
On consid\`ere deux entiers $ n_{1},n_{2}>0 $, des r\'eels $ (\lambda_{i,j})_{\substack{1\leqslant i\leqslant n_{1} \\1\leqslant i\leqslant n_{2}} } $ et des formes lin\'eaires \[ \forall i \in  \{1,...,n_{1}\}, \; \forall \uu=(u_{1},...,u_{n_{2}}) \; L_{i}(\uu)=\sum_{j=1}^{n_{2}}\lambda_{i,j}u_{j}, \] et \[ \forall j \in  \{1,...,n_{2}\}, \; \forall \uu=(u_{1},...,u_{n_{1}}) \; L^{t}_{j}(\uu)=\sum_{i=1}^{n_{1}}\lambda_{i,j}u_{i}. \] Soient $ a_{1},...,a_{n_{2}},b_{1},...,b_{n_{1}}>1 $ des r\'eels fix\'es. Pour tout $ 0\leqslant Z\leqslant 1 $, on note \begin{multline*} U(Z)=\Card\left\{ (u_{1},...,u_{n_{2}},u_{n_{2}+1},...,u_{n_{2}+n_{1}})\in \ZZ^{n_{1}+n_{2}} \; | \; \forall j \in \{1,...,n_{2}\} \; \right.\\ \left.  |u_{j}|\leqslant a_{j}Z \;\et \; \forall i\in \{1,...,n_{1}\}\; |L_{i}(u_{1},...,u_{n_{2}})-u_{n_{2}+i}|\leqslant b_{i}^{-1}Z \right\}, \end{multline*} \begin{multline*} U^{t}(Z)=\Card\left\{ (u_{1},...,u_{n_{1}},u_{n_{1}+1},...,u_{n_{1}+n_{2}})\in \ZZ^{n_{1}+n_{2}} \; | \; \forall i \in \{1,...,n_{1}\} \; \right.\\ \left. |u_{i}|\leqslant b_{i}Z \;  \et \; \forall j\in \{1,...,n_{2}\}\; |L^{t}_{j}(u_{1},...,u_{n_{1}})-u_{n_{1}+i}|\leqslant a_{j}^{-1}Z \right\}. \end{multline*} Si $ 0<Z_{1}\leqslant Z_{2}\leqslant 1 $, on a alors : \[ U(Z_{2})\ll_{n_{1},n_{2}} \max\left( \left(\frac{Z_{2}}{Z_{1}}\right)^{n_{2}}U(Z_{1}), \frac{Z_{2}^{n_{2}}}{Z_{1}^{n_{1}}}\frac{\prod_{j=1}^{n_{2}}a_{j}}{\prod_{i=1}^{n_{1}}b_{i}}U^{t}(Z_{1})\right). \]
\end{lemma}
\begin{rem}
Le lemme\;3.1 de \cite{S1} pr\'esente uniquement le cas o\`u $ a_{1}=...=a_{n_{2}}=a $ et $ b_{1}=...=b_{n_{1}}=b $. Cette g\'en\'eralisation aux $ a_{i} $ et $ b_{i} $ distincts permet de donner des estimations du nombre de points dans un r\'eseau dont les coordonn\'ees sont born\'ees par des bornes distinctes.
\end{rem}

\begin{proof}[D\'emonstration du lemme\;\ref{geomnomb12}]
On consid\`ere le r\'eseau $ \Lambda $ de $ \RR^{n_{2}+n_{1}} $ d\'efini comme l'ensemble des points \[ (x_{1},...,x_{n_{2}},x_{n_{2}+1},...,x_{n_{2}+n_{1}})\in \RR^{n_{1}+n_{2}} \] tels qu'il existe \[ (u_{1},...,u_{n_{2}},u_{n_{2}+1},...,u_{n_{2}+n_{1}})\in \ZZ^{n_{1}+n_{2}} \] tels que   \[ \begin{array}{rcl} a_{1}x_{1} & = & u_{1}, \\ & \vdots & \\ a_{n_{2}}x_{n_{2}} & = & u_{n_{2}}, \\ b_{1}^{-1}x_{n_{2}+1} & = & L_{1}(u_{1},...,u_{n_{2}})+u_{n_{2}+1}, \\ & \vdots & \\ b_{n_{1}}^{-1}x_{n_{2}+n_{1}} & = & L_{n_{1}}(u_{1},...,u_{n_{2}})+u_{n_{2}+n_{1}}.
\end{array}\]

Ce r\'eseau est d\'efini par la matrice (i.e une base de ce r\'eseau est donn\'ee par les colonnes de la matrice) \[ A=\begin{pmatrix}
a_{1}^{-1} &  & (0) & 0 & \cdots & 0 \\ 
 & \ddots &  & \vdots &  & \vdots \\ 
(0) &  & a_{n_{2}}^{-1} & 0 & \cdots & 0 \\ 
b_{1}\lambda_{1,1} & \cdots & b_{1}\lambda_{1,n_{2}} & b_{1} &  & (0) \\ 
\vdots &  & \vdots &  & \ddots &  \\ 
b_{n_{1}}\lambda_{n_{1},1} & \cdots & b_{n_{1}}\lambda_{n_{1},n_{2}} & (0) &  & b_{n_{1}}
\end{pmatrix}. \] On remarque que $ U(Z) $ est alors le nombre de points $ (x_{1},...,x_{n_{1}+n_{2}}) $ de $ \Lambda $ tels que $ |x_{i}|\leqslant Z $ pour tout $ i\in\{1,...,n_{1}+n_{2}\} $. Par ailleurs,  \[ B=(A^{t})^{-1}=\begin{pmatrix}
a_{1} &  & (0) & -a_{1}\lambda_{1,1} & \cdots & -a_{1}\lambda_{n_{1},1} \\ 
 & \ddots &  & \vdots &  & \vdots \\ 
(0) &  & a_{n_{2}} & -a_{n_{2}}\lambda_{1,n_{2}} & \cdots & -a_{n_{2}}\lambda_{n_{1},n_{2}} \\ 
0& \cdots & 0 & b_{1}^{-1} &  & (0) \\ 
\vdots &  & \vdots &  & \ddots &  \\ 
0 & \cdots & 0 & (0) &  & b_{n_{1}}^{-1}
\end{pmatrix}. \]
d\'efinit un r\'eseau $ \Omega $ ayant les m\^emes minima successifs que le r\'eseau $ \tilde{\Omega} $ d\'efini par la matrice  \[ \tilde{B}=\begin{pmatrix}
b_{1}^{-1} &  & (0) & 0 & \cdots &  \\ 
 & \ddots &  & \vdots &  & \vdots \\ 
(0) &  & b_{n_{1}}^{-1} & 0 & \cdots & 0 \\ 
a_{1}\lambda_{1,1}& \cdots & a_{1}\lambda_{n_{1},1} & a_{1} &  & (0) \\ 
\vdots &  & \vdots &  & \ddots &  \\ 
a_{n_{2}}\lambda_{1,n_{2}} & \cdots & a_{n_{2}}\lambda_{n_{1},n_{2}} & (0) &  & a_{n_{2}}
\end{pmatrix}. \]

On pose $ c=\left(\frac{\prod_{j=1}^{n_{2}}a_{j}}{\prod_{i=1}^{n_{1}}b_{i}}\right)^{\frac{1}{n_{1}+n_{2}}} $ et $ \Lambda^{\nor}=c\Lambda $, $ \Omega^{\nor}=c^{-1}\tilde{\Omega} $ les r\'eseaux normalis\'es (i.e de d\'eterminant $ 1 $) associ\'es \`a $ \Lambda $ et $ \Omega $. Par la d\'emonstration de \cite[Lemme 3.1]{S1}, on a alors \[ U(Z_{2})\ll_{n_{1},n_{2}} \max\left( \left(\frac{Z_{2}}{Z_{1}}\right)^{n_{2}}U(Z_{1}), \frac{Z_{2}^{n_{2}}}{Z_{1}^{n_{1}}}c^{n_{1}+n_{2}}U^{t}(Z_{1})\right). \] d'o\`u le r\'esultat.
\end{proof}

En particulier, lorsque $ n_{1}=n_{2}=n $, $ a_{i}=b_{i} $ pour tout $ i $ et $ \lambda_{i,j}=\lambda_{j,i} $ on obtient le r\'esultat suivant : 

\begin{lemma}\label{geomnomb22}
Soit $ n>0 $ un entier et $ (\lambda_{i,j})_{\substack{1\leqslant i,j\leqslant n} } $ des r\'eels tels que $ \lambda_{i,j}=\lambda_{j,i} $ pour tous $ i,j $, et des formes lin\'eaires \[ \forall i \in  \{1,...,n_{1}\}, \; \forall \uu=(u_{1},...,u_{n})\;  L_{i}(\uu)=\sum_{j=1}^{n}\lambda_{i,j}u_{j}. \] Soient $ a_{1},...,a_{n}>1 $ des r\'eels fix\'es. Pour tout $ 0\leqslant Z\leqslant 1 $, on note \begin{multline*} U(Z)=\Card\left\{ (u_{1},...,u_{n},u_{n+1},...,u_{2n}) \; | \; \forall j \in \{1,...,n\} \; |u_{j}|\leqslant a_{j}Z \; \right.\\ \left. \et \; \forall i\in \{1,...,n\}\; |L_{i}(u_{1},...,u_{n})-u_{n+i}|\leqslant a_{i}^{-1}Z \right\}. \end{multline*} On a alors \[ U(Z_{2})\ll_{n} \left(\frac{Z_{2}}{Z_{1}}\right)^{n}U(Z_{1}). \]
\end{lemma}

\subsubsection{Premi\`ere r\'ecurrence}

Dans cette partie, nous allons appliquer les r\'esultats de g\'eom\'etrie des r\'eseaux uniquement aux variables $ x_{m}^{\hh} $ telles que $ m\in J(j,k) $ pour un $ (j,k) $ fix\'e (en l'occurence $ (j,k)=(1,K) $).\\

Nous allons dans un premier temps majorer, pour $ (x_{m}^{\hh})_{\substack{ (m,\hh)\in L\setminus L(1,K)}}  $ et $ (x_{m}^{\hh})_{\substack{ (m,\hh)\in \widehat{L}(1,K) \\ h_{1}\neq t_{1,K}-1 }} $ fix\'es, le nombre de $ (x_{m}^{\hh})_{\substack{(m,\hh)\in \widehat{L}(1,K) \\ h_{1}=t_{1,K}-1}} $ tels que \[ \widehat{\XX}_{(1,K)}=(x_{m}^{\hh})_{(m,\hh)\in \widehat{L}(1,K) }\in \MM_{(1,K)}\left(\alpha,(T_{m})_{\substack{(m,\hh)\in \widehat{L}(1,K)}}, (T_{i}^{-1})_{\substack{(i,\l)\in E(1,K)}}\right). \]
Nous allons pour cela appliquer les r\'esultats de g\'eom\'etrie des r\'eseaux \'evoqu\'es dans la section\;$ 2.3.2 $. \\

Pour $ \bb=(b_{m})_{m\in J(1,K)} $ fix\'e et tout $ Z\leqslant 1 $ on notera provisoirement \begin{multline*} U(Z)=\Card\{(x_{m}^{\hh})_{\substack{(m,\hh)\in \widehat{L}(1,K) \\ h_{1}=t_{1,K}-1 }} \; | \;   \forall (m,\hh),  \; |x_{m}^{\hh}|\leqslant b_{m}Z \\ \et,\;  \forall (i,\l)\in E(1,K), \;   ||\alpha e_{i}\gamma_{(i,\l)}^{(1,K)}(\widehat{\XX}_{(1,K)})||<b_{i}^{-1}Z \}.
\end{multline*}
En choisissant $ Z_{2}=1 $ et $ b_{m}=T_{m} $ pour tout $ m\in J(1,K) $, on remarque alors que le cardinal que nous souhaitons majorer est $ U(Z_{2}) $. Or d'apr\`es le lemme\;2.3.4, puisque les formes $ \gamma_{(i,\l)}^{(1,K)}(\widehat{\XX}_{(1,K)}) $ sont lin\'eaires en $ (x_{m}^{\hh})_{\substack{(m,\hh)\in \widehat{L}(1,K) \\ h_{1}=t_{1,K}-1}} $ et v\'erifient les conditions de sym\'etrie  \[ \forall (j,k)\in \CCC_{0}, \; e_{i}\sum_{(i,\l)\in E(j,k) }\gamma_{(i,\l)}^{(j,k)}(\widehat{\XX}_{(j,k)})\x_{i}^{\l}=\Gamma(\XX), \] pour tout $ Z_{1}\leqslant Z_{2} $ on a \[ U(Z_{2})\ll \prod_{\substack{(m,\hh)\in \widehat{L}(1,K)\; |\; h_{1}=t_{1,K}-1}}\left(\frac{Z_{2}}{Z_{1}}\right)U(Z_{1}). \] Nous allons choisir $ Z_{1} $ tel que \[ b_{m}Z_{1}=\frac{T_{m}}{P_{1}^{K}}P^{K\theta}, \] pour $ \theta\in [0,1] $ fix\'e quelconque. Ceci est possible en prenant  \[ Z_{1}=\frac{P^{K\theta}}{P_{1}^{K}}. \] On obtient donc \[ U(Z_{2})\ll \prod_{\substack{m\in J(1,K)}}\left(\frac{P_{1}^{K}}{P^{K\theta}}\right)^{\prod_{j=2}^{r}t_{j,a_{m,j}}}U(Z_{1}), \] avec $ K=a_{m,1} $ pour tout $ m\in J(1,K) $. En appliquant le m\^eme proc\'ed\'e avec les variables $ (x_{m}^{\hh})_{\substack{(m,\hh)\in \widehat{L}(1,K)\; |\; h_{1}=t_{1,K}-h}} $ avec $ h $ fix\'e \'egal \`a $ 2,...,t_{1,K}-1 $, nous allons \'etablir le r\'esultat ci-dessous

\begin{lemma}\label{recurrence13}
Pour tout $ \theta\in [0,1] $, on a la majoration ci-dessous : 

\begin{multline*}
M_{(1,K)}\left(\alpha,(T_{m})_{\substack{(m,\hh)\in L\setminus E(1,K)}}, (T_{i}^{-1})_{\substack{(i,\l)\in E(1,K)}}\right) \\ \ll \prod_{\substack{m\in J(1,K)}}\left(\frac{P_{1}^{a_{m,1}}}{P^{a_{m,1}\theta}}\right)^{(t_{1,K}-1)\prod_{j=2}^{r}t_{j,a_{m,j}}} \\ M_{(1,K)}\left(\alpha,\left((T_{m})_{\substack{(m,\hh)\in L\setminus L(1,K)}},(T_{m}P_{1}^{-a_{m,1}}P^{a_{m,1}\theta})_{\substack{(m,\hh)\in \widehat{L}(1,K)}}\right), \right. \\ \left. (P_{1}^{-(t_{1,a_{i,1}}-1)a_{i,1}}P^{(t_{1,a_{i,1}}-1)a_{i,1}\theta}T_{i}^{-1})_{\substack{(i,\l)\in E(1,K)}}\right).
\end{multline*}

\end{lemma}
\begin{proof}
Montrons par r\'ecurrence sur $ h\in \{1,...,t_{1,K}-1\} $ que 
\begin{align}\label{larec3}M_{(1,K)}\left(\alpha,(T_{m})_{\substack{(m,\hh)\in L\setminus E(1,K)}}, (T_{i}^{-1})_{\substack{(i,\l)\in E(1,K)}}\right) &\; \; \; \; \; \; \; \; \; \; \; \; \; \; \;\; \; \; \; \; \; \; \; \; \; \; \; \; \; \;\; \; \; \; \; \; \; \; \; \; \; \; \; \; \;  \end{align}
\begin{align*}\label{proprec13}
 \ll \prod_{\substack{m\in J(1,K)}}\left(\frac{P_{1}^{a_{m,1}}}{P^{a_{m,1}\theta}}\right)^{h\prod_{j=2}^{r}t_{j,a_{m,j}}}   & M_{(1,K)}\left(\alpha,\left((T_{m})_{\substack{(m,\hh)\in L\setminus \bigcup_{l=0}^{h}L(1,K,t_{1,K}-l)}},\right. \right. \\ & \left.(T_{m}P_{1}^{-a_{m,1}}P^{a_{m,1}\theta})_{\substack{(m,\hh)\in  \bigcup_{l=1}^{h}L(1,K,t_{1,K}-l)}}\right), \\ &  \left.(P_{1}^{-ha_{i,1}}P^{ha_{i,1}\theta}T_{i}^{-1})_{\substack{(i,\l)\in E(1,K)}}\right).
\end{align*}
Le r\'esultat a d\'ej\`a \'et\'e d\'emontr\'e pr\'ec\'edemment pour $ h=1 $, en notant que $ K=a_{m,1} $ pour tout $ m\in J(1,K) $. Supposons le r\'esultat vrai au rang $ h $. Fixons des \'el\'ements $ (x_{m}^{\hh})_{(m,\hh)\in L\setminus\bigcup_{l=0}^{h+1}L(1,K,t_{1,K}-l) } $ tels que $ |x_{m}^{\hh}|\leqslant T_{m} $ et $ (x_{m}^{\hh})_{(m,\hh)\in \bigcup_{l=1}^{h}L(1,K,t_{1,K}-l) } $ tels que $ |x_{m}^{\hh}|\leqslant T_{m}P_{1}^{-a_{m,1}}P^{a_{m,1}\theta} $. Posons alors, pour $ Z\geqslant 1 $, $ b_{m}\in \RR $ :  \begin{multline*} U(Z)=\Card\{(x_{m}^{\hh})_{\substack{(m,\hh)\in L(1,K,t_{1,K}-(h+1))}} \; | \;   \forall (m,\hh),  \; |x_{m}^{\hh}|\leqslant b_{m}Z \\ \et,\;  \forall (i,\l)\in E(1,K), \;   ||\alpha e_{i}\gamma_{(i,\l)}^{(1,K)}(\widehat{\XX}_{(1,K)})||<b_{i}^{-1}Z \}.
\end{multline*}
Choisissons alors pour tout $ m\in J(1,K) $ : \[ \begin{array}{l} b_{m}=P_{1}^{\frac{1}{2}hK}P^{-\frac{1}{2}hK\theta}T_{m}  \\ Z_{2}=P_{1}^{-\frac{1}{2}hK}P^{\frac{1}{2}hK\theta} \\ Z_{1}=P_{1}^{-K}P^{K\theta}Z_{2},
\end{array} \] de sorte que : 
 \[ \begin{array}{l} b_{m}Z_{2}=T_{m}  \\ b_{m}Z_{1}=T_{m}P_{1}^{-K}P^{K\theta} \\ b_{m}^{-1}Z_{2}=T_{m}^{-1}P_{1}^{-hK}P^{hK\theta} \\ b_{m}^{-1}Z_{1}=T_{m}^{-1}P_{1}^{-(h+1)K}P^{(h+1)K\theta}.
\end{array} \]
On a alors : 
\begin{multline*} U(Z_{2})=\Card\{(x_{m}^{\hh})_{\substack{(m,\hh)\in L(1,K,t_{1,K}-(h+1))}} \; | \;   \forall (m,\hh),  \; |x_{m}^{\hh}|\leqslant T_{m} \\ \et,\;  \forall (i,\l)\in E(1,K), \;   ||\alpha e_{i}\gamma_{(i,\l)}^{(1,K)}(\widehat{\XX}_{(1,K)})||<T_{m}^{-1}P_{1}^{-hK}P^{hK\theta} \},
\end{multline*}
et d'apr\`es le lemme\;\ref{geomnomb22},  \begin{align*} U(Z_{2}) & \ll \left(\prod_{\substack{(m,\hh)\in \widehat{L}(1,K,t_{1,K}-(h+1))}}\frac{Z_{2}}{Z_{1}}\right)U(Z_{1}) \\ & =\left(\prod_{\substack{(m,\hh)\in \widehat{L}(1,K,t_{1,K}-(h+1))}}\frac{P_{1}^{K}}{P^{K\theta}}\right)U(Z_{1}) \\ & =\left(\prod_{m\in J(1,K)}\left(\frac{P_{1}^{K}}{P^{K\theta}}\right)^{\prod_{j=2}^{r}t_{j,a_{m,j}}}\right)U(Z_{1}), \end{align*}
avec \begin{multline*} U(Z_{1})=\Card\{(x_{m}^{\hh})_{\substack{(m,\hh)\in L(1,K,t_{1,K}-(h+1))}} \; | \;   \forall (m,\hh),  \; |x_{m}^{\hh}|\leqslant T_{m}P_{1}^{-K}P^{K\theta} \\ \et \;  \forall (i,\l)\in E(1,K), \;   ||\alpha e_{i}\gamma_{(i,\l)}^{(1,K)}(\widehat{\XX}_{(1,K)})||<T_{m}^{-1}P_{1}^{-(h+1)K}P^{(h+1)K\theta} \}.
\end{multline*}
En sommant ces majorations sur l'ensemble des  $ (x_{m}^{\hh})_{(m,\hh)\in L\setminus\bigcup_{l=0}^{h+1}L(1,K,t_{1,K}-l) } $ tels que $ |x_{m}^{\hh}|\leqslant T_{m} $ et $ (x_{m}^{\hh})_{(m,\hh)\in \bigcup_{l=1}^{h}L(1,K,t_{1,K}-l) } $ tels que $ |x_{m}^{\hh}|\leqslant T_{m}P_{1}^{-a_{m,1}}P^{a_{m,1}\theta} $, on obtient 

\begin{align*} M_{(1,K)} & \left(\alpha,\left((T_{m})_{\substack{(m,\hh)\in L\setminus \bigcup_{l=0}^{h}L(1,K,t_{1,K}-l)}},\right. \right.  &\; \; \; \; \; \; \; \; \; \; \; \; \; \; \;\; \; \; \; \; \; \; \; \; \; \; \; \; \; \;\; \; \; \; \; \; \; \; \; \; \; \; \; \; \;\\ & \left.(T_{m}P_{1}^{-a_{m,1}}P^{a_{m,1}\theta})_{\substack{(m,\hh)\in  \bigcup_{l=1}^{h}L(1,K,t_{1,K}-l)}}\right), & \\ &  \left.(P_{1}^{-ha_{i,1}}P^{ha_{i,1}\theta}T_{i}^{-1})_{\substack{(i,\l)\in E(1,K)}}\right) & \end{align*}
\begin{align*} \ll \prod_{\substack{m\in J(1,K)}}\left(\frac{P_{1}^{a_{m,1}}}{P^{a_{m,1}\theta}}\right)^{\prod_{j=2}^{r}t_{j,a_{m,j}}}    M_{(1,K)} & \left(\alpha,\left((T_{m})_{\substack{(m,\hh)\in L\setminus \bigcup_{l=0}^{h+1}L(1,K,t_{1,K}-l)}},\right. \right. \\ & \left.(T_{m}P_{1}^{-a_{m,1}}P^{a_{m,1}\theta})_{\substack{(m,\hh)\in  \bigcup_{l=1}^{h+1}L(1,K,t_{1,K}-l)}}\right), \\ &  \left.(P_{1}^{-(h+1)a_{i,1}}P^{(h+1)a_{i,1}\theta}T_{i}^{-1})_{\substack{(i,\l)\in E(1,K)}}\right).
\end{align*}
Puis en utilisant l'hypoth\`ese de r\'ecurrence on obtient : 
\begin{multline*}M_{(1,K)} \left(\alpha,(T_{m})_{\substack{(m,\hh)\in L\setminus E(1,K)}},  (T_{i}^{-1})_{\substack{(i,\l)\in E(1,K)}}\right)   \\  
 \ll \prod_{\substack{m\in J(1,K)}}  \left(\frac{P_{1}^{a_{m,1}}}{P^{a_{m,1}\theta}}\right)^{(h+1)\prod_{j=2}^{r}t_{j,a_{m,j}}} \\ \begin{array}{ll}  M_{(1,K)} & \left(\alpha,\left((T_{m})_{\substack{(m,\hh)\in L\setminus \bigcup_{l=0}^{h+1}L(1,K,t_{1,K}-l)}},\right. \right.  \\   & \left.(T_{m}P_{1}^{-a_{m,1}}P^{a_{m,1}\theta})_{\substack{(m,\hh)\in  \bigcup_{l=1}^{h+1}L(1,K,t_{1,K}-l)}}\right),  \\ &   \left.(P_{1}^{-(h+1)a_{i,1}}P^{(h+1)a_{i,1}\theta}T_{i}^{-1})_{\substack{(i,\l)\in E(1,K)}}\right),  \end{array}
\end{multline*}
ce qui cl\^ot la d\'emonstration de l'assertion\;\eqref{larec3}. Le cas $ h=t_{1,K}-1 $, correspond alors au r\'esultat du lemme. 
\end{proof}

\subsubsection{Seconde r\'ecurrence : changement de poids $ k $}

Dans cette partie nous allons d\'etailler le passage de variables $ (x_{m}^{\hh})_{(m,\hh)\in J(1,k)} $ aux variables $ (x_{m}^{\hh})_{(m,\hh)\in J(1,k')} $ pour un $ k'\neq k $

\begin{lemma}\label{recurrence23}
Pour tout poids $ k\in \PPP(1) $ et tout poids $ k_{1} $ tel que $ k_{1}>k $, on a : 
\begin{multline*}\begin{array}{ll}M_{(1,k_{1})} & \left(\alpha,\left((T_{m})_{\substack{(m,\hh)\in L\setminus \bigcup_{k'>k}L(1,k')}}, \right.\right. \\  & \left.(T_{m}P_{1}^{-a_{m,1}}P^{a_{m,1}\theta})_{\substack{(m,\hh)\in  \bigcup_{k'>k}L(1,k')\setminus E(1,k_{1})}}\right), \\ & \left. (T_{i}^{-1}P_{1}^{k_{1}-\sum_{k'>k}k't_{1,k'}}P^{-(k_{1}-\sum_{k'>k}k't_{1,k'})\theta})_{\substack{(i,\l)\in  E(1,k_{1})}}\right) \end{array}  \\  
 \ll \max\{M_{1},M_{2}\}
\end{multline*}
avec 
\begin{multline*}
  M_{1}=\prod_{\substack{m\in J(1,k)}}  \left(\frac{P_{1}^{a_{m,1}}}{P^{a_{m,1}\theta}}\right)^{\prod_{j=1}^{r}t_{j,a_{m,j}}} \\ \begin{array}{ll}M_{(1,k_{1})} & \left(\alpha,\left((T_{m})_{\substack{(m,\hh)\in L\setminus \bigcup_{k' \geqslant k}L(1,k')}}, \right.\right. \\ & \left. (T_{m}P_{1}^{-a_{m,1}}P^{a_{m,1}\theta})_{\substack{(m,\hh)\in  \bigcup_{k'\geqslant k}L(1,k')\setminus E(1,k_{1})}}\right), \\ & \left. (T_{i}^{-1}P_{1}^{k_{1}-\sum_{k'\geqslant k}k't_{1,k'}}P^{-(k_{1}-\sum_{k'\geqslant k}k't_{1,k'})\theta})_{\substack{(i,\l)\in E(1,k_{1})}}\right), \end{array} 
\end{multline*}
\begin{multline*}
M_{2}=\frac{\prod_{m\in J(1,k)}T_{m}^{\prod_{j=2}^{r}t_{j,a_{m,j}}}}{\prod_{i\in J(1,k_{1})}(P_{1}^{-a_{i,1}}P^{a_{i,1}\theta}T_{i})^{\prod_{j=2}^{r}t_{j,a_{i,j}}}}\prod_{m\in J(1,k)}\left(\frac{P_{1}^{a_{m,1}}}{P^{a_{m,1}\theta}}\right)^{(t_{1,a_{m,1}}-1)\prod_{j=2}^{r}t_{j,a_{m,j}}} \\ \; \; \; \; \; \begin{array}{ll} M_{(1,k)} & \left(\alpha,\left((T_{m})_{\substack{(m,\hh)\in L\setminus \bigcup_{k'\geqslant k}L(1,k')}}, \right.\right. \\ & \left. (T_{m}P_{1}^{-a_{m,1}}P^{a_{m,1}\theta})_{\substack{(m,\hh)\in  \bigcup_{k'\geqslant k}L(1,k')\setminus E(1,k)}}\right), \\ & \left.  (T_{i}^{-1}P_{1}^{k-\sum_{k'\geqslant k}k't_{1,k'}}P^{-(k-\sum_{k'\geqslant k}k't_{1,k'})\theta})_{\substack{(i,\l)\in  E(1,k)}}\right).\end{array}
\end{multline*}

\end{lemma}
\begin{proof}
Fixons des \'el\'ements $ (x_{m}^{\hh})_{\substack{( m,\hh)\in L\setminus \bigcup_{k'> k}L(1,k')}} $, $ (x_{m}^{\hh})_{\substack{( m,\hh)\in L(1,k)\setminus E(1,k)}}  $ tels que $ |x_{m}^{\hh}|\leqslant T_{m} $ et $ (x_{m}^{\hh})_{\substack{( m,\hh)\in \bigcup_{k'> k}L(1,k')}} $ tels que $  |x_{m}^{\hh}|\leqslant T_{m}P_{1}^{-a_{m,1}}P^{a_{m,1}\theta} $ et posons pour tous $ Z\geqslant 1 $ et $ b_{m}>0 $ pour $ m\in J(1,k) $ et $ c_{i}>0 $ pour $ i\in J(1,k_{1}) $ : \begin{multline*} U(Z)=\Card\{(x_{m}^{\hh})_{\substack{(m,\hh)\in E(1,k) }} \; | \;  \forall (m,\hh)\in E(1,k), \; |x_{m}^{\hh}|\leqslant b_{m}Z \\ \et,\;  \forall (i,\l)\in E(1,k), \;  ||\alpha e_{i}\gamma_{(i,\l)}^{(1,k_{1})}(\widehat{\XX}_{(1,k_{1})})||<c_{i}^{-1}Z \}.
\end{multline*}
On choisit \`a pr\'esent  $ \bb=(b_{m})_{m\in J(1,k)} $, $ \cc=(c_{i})_{i\in J(1,k_{1})} $, $ Z_{2}\leqslant 1 $ tels que \[ b_{m}Z_{2}=T_{m}, \; \;  c_{i}^{-1}Z_{2}=T_{i}^{-1}P_{1}^{k_{1}-\sum_{k'>k}k't_{1,k'}}P^{-(k_{1}-\sum_{k'>k}k't_{1,k'})\theta}, \] et le nombre de points que l'on souhaite \'evaluer est alors $ U(Z_{2}) $. Choisissons d'autre part un r\'eel $  Z_{1}\leqslant 1  $ tel que \[ \begin{array}{l}
\forall m\in J(1,k), \; \; b_{m}Z_{1}=P_{1}^{-k}P^{k\theta}T_{m}, \\ \forall i\in J(1,k_{1}), \; \;c_{i}^{-1}Z_{1}=T_{i}^{-1}P_{1}^{k_{1}-k-\sum_{k'> k}k't_{1,k'}}P^{-(k_{1}-k-\sum_{k'> k}k't_{1,k'})\theta}, \\ \forall i\in J(1,k_{1}), \; \;c_{i}Z_{1}=T_{i}P_{1}^{-k_{1}}P^{k_{1}\theta}.
\end{array}   \]  
On v\'erifie que ces conditions sur $ \bb,\cc,Z_{2},Z_{1} $ sont satisfaites si et seulement si \[ \begin{array}{l}
Z_{1}=P_{1}^{-\frac{1}{2}(k+\sum_{k'>k}k't_{1,k'})}P^{\frac{1}{2}(k+\sum_{k'>k}k't_{1,k'})\theta}, \\ Z_{2}=Z_{1}P_{1}^{k}P^{-k\theta} \\ \forall m\in J(1,k), \; \; b_{m}=T_{m}P_{1}^{-k+\frac{1}{2}(k+\sum_{k'>k}k't_{1,k'})}P^{k\theta-\frac{1}{2}(k+\sum_{k'>k}k't_{1,k'})\theta}, \\ \forall i\in J(1,k_{1}), \; \; c_{i}=T_{i}P_{1}^{-k_{1}+\frac{1}{2}(k+\sum_{k'>k}k't_{1,k'})}P^{k_{1}\theta-\frac{1}{2}(k+\sum_{k'>k}k't_{1,k'})\theta},
\end{array}   \]  
et on a alors $ b_{m}^{-1}Z_{1}=P_{1}^{-\sum_{k'>k}k't_{1,k'}}P^{(\sum_{k'>k}k't_{1,k'})\theta}T_{m}^{-1} $. Par ailleurs, d'apr\`es le\;lemme\;\ref{geomnomb12} on a \begin{equation}\label{geomnomb13} U(Z_{2})\ll \max\left(\underbrace{\left(\prod_{\substack{(m,\hh)\in E(1,k)}} \left(\frac{Z_{2}}{Z_{1}}\right)\right)U(Z_{1})}_{(a)}, \underbrace{\frac{\prod_{\substack{(m,\hh)\in E(1,k)}}b_{m}Z_{2}}{\prod_{\substack{(i,\l)\in E(1,k_{1})}}c_{i}Z_{1}}U^{t}(Z_{1})}_{(b)}\right), \end{equation} o\`u \begin{multline*} U^{t}(Z)=\Card\{(x_{i}^{\l})_{\substack{(i,\l)\in E(1,k_{1}) }} \; | \;  \forall (i,\l)\in E(1,k_{1}), \; |x_{i}^{\hh}|\leqslant c_{i}Z \\ \et,\;  \forall (m,\hh)\in L(1,k),\;  ||\alpha e_{i}(\underbrace{\gamma^{(1,k_{1})})^{t}_{(m,\hh)}}_{=\gamma^{(1,k)}_{(m,\hh)}}(\widehat{\XX}_{(1,k)})||<b_{m}^{-1}Z \}.  \end{multline*}

On remarque que \[ \prod_{\substack{(m,\hh)\in E(1,k)}} \left(\frac{Z_{2}}{Z_{1}}\right)=\prod_{\substack{m\in J(1,k)}} \left(\frac{P_{1}^{a_{m,1}}}{P^{a_{m,1}\theta}}\right)^{\prod_{j=2}^{r}t_{j,a_{m,j}}}, \] et de m\^eme 
\[ \frac{\prod_{\substack{(m,\hh)\in E(1,k)}}b_{m}Z_{2}}{\prod_{\substack{(i,\l)\in E(1,k_{1})}}c_{i}Z_{1}}=\frac{\prod_{\substack{m\in J(1,k)}}T_{m}^{\prod_{j=2}^{r}t_{j,a_{m,j}}}}{\prod_{\substack{i\in J(1,k_{1})}}(P_{1}^{-a_{i,1}}P^{a_{i,1}\theta}T_{i})^{\prod_{j=2}^{r}t_{j,a_{i,j}}}}. \]

La formule\;\eqref{geomnomb13} donne alors, par sommation sur les \'el\'ements $ (x_{m}^{\hh})_{\substack{( m,\hh)\in L\setminus \bigcup_{k'> k}L(1,k')}} $, $ (x_{m}^{\hh})_{\substack{( m,\hh)\in L(1,k)\setminus E(1,k)}}  $ tels que $ |x_{m}^{\hh}|\leqslant T_{m} $ et $ (x_{m}^{\hh})_{\substack{( m,\hh)\in \bigcup_{k'> k}L(1,k')}} $ tels que $  |x_{m}^{\hh}|\leqslant T_{m}P_{1}^{-a_{m,1}}P^{a_{m,1}\theta} $ : \begin{multline*}
\begin{array}{ll}M_{(1,k_{1})} & \left(\alpha,\left((T_{m})_{\substack{(m,\hh)\in L\setminus \bigcup_{k'>k}L(1,k')}}, \right.\right. \\  & \left.(T_{m}P_{1}^{-a_{m,1}}P^{a_{m,1}\theta})_{\substack{(m,\hh)\in  \bigcup_{k'>k}L(1,k')\setminus E(1,k_{1})}}\right), \\ & \left. (T_{i}^{-1}P_{1}^{k_{1}-\sum_{k'>k}k't_{1,k'}}P^{-(k_{1}-\sum_{k'>k}k't_{1,k'})\theta})_{\substack{(i,\l)\in  E(1,k_{1})}}\right) \end{array}  \\  
 \ll \max\{ A_{1},A_{2}\}
\end{multline*} 
o\`u \begin{multline*}
  A_{1}= \prod_{\substack{m\in J(1,k_{1})}} \left(\frac{P_{1}^{a_{m,1}}}{P^{a_{m,1}\theta}}\right)^{\prod_{j=2}^{r}t_{j,a_{m,j}}}\\ \; \; \; \; \; \; \; \; \; \; \begin{array}{ll}M_{(1,k_{1})} & \left(\alpha,\left((T_{m})_{\substack{(m,\hh)\in L\setminus \bigcup_{k' > k}L(1,k')}}, \right.\right. \\ & \left. (T_{m}P_{1}^{-a_{m,1}}P^{a_{m,1}\theta})_{\substack{(m,\hh)\in  L(1,k,t_{1,k})\cup\bigcup_{k'> k}L(1,k')\setminus E(1,k_{1})}}\right), \\ & \left. (T_{i}^{-1}P_{1}^{k_{1}-k-\sum_{k'> k}k't_{1,k'}}P^{-(k_{1}-k-\sum_{k'\geqslant k}k't_{1,k'})\theta})_{\substack{(i,\l)\in E(1,k_{1})}}\right), \end{array} 
\end{multline*}
\begin{multline*}
A_{2}= \frac{\prod_{\substack{m\in J(1,k)}}T_{m}^{\prod_{j=2}^{r}t_{j,a_{m,j}}}}{\prod_{\substack{i\in J(1,k_{1})}}(P_{1}^{-a_{i,1}}P^{a_{i,1}\theta}T_{i})^{\prod_{j=2}^{r}t_{j,a_{i,j}}}}\\ \; \; \; \; \; \; \; \; \; \; \begin{array}{ll} M_{(1,k)} & \left(\alpha,\left((T_{m})_{\substack{(m,\hh)\in L\setminus \bigcup_{k'\geqslant k}L(1,k')}}, \right.\right. \\ & \left. (T_{m}P_{1}^{-a_{m,1}}P^{a_{m,1}\theta})_{\substack{(m,\hh)\in  \bigcup_{k'> k}L(1,k')\setminus E(1,k)}}\right), \\ & \left.  (T_{i}^{-1}P_{1}^{-\sum_{k'> k}k't_{1,k'}}P^{\sum_{k'> k}k't_{1,k'})\theta})_{\substack{(i,\l)\in  E(1,k)}}\right).\end{array}
\end{multline*}
Or, par les m\^emes arguments que pour le lemme\;\ref{recurrence13}, on d\'emontre que :
\begin{multline*}\begin{array}{ll} M_{(1,k)} & \left(\alpha,\left((T_{m})_{\substack{(m,\hh)\in L\setminus \bigcup_{k'\geqslant k}L(1,k')}}, \right.\right. \\ & \left. (T_{m}P_{1}^{-a_{m,1}}P^{a_{m,1}\theta})_{\substack{(m,\hh)\in  \bigcup_{k'> k}L(1,k')\setminus E(1,k)}}\right), \\ & \left.  (T_{i}^{-1}P_{1}^{-\sum_{k'> k}k't_{1,k'}}P^{\sum_{k'> k}k't_{1,k'})\theta})_{\substack{(i,\l)\in  E(1,k)}}\right) \end{array} \\ \ll \prod_{m\in J(1,k)}\left(\frac{P_{1}^{a_{m,1}}}{P^{a_{m,1}\theta}}\right)^{(t_{1,a_{m,1}}-1)\prod_{j=2}^{r}t_{j,a_{m,j}}} \\ \; \; \; \; \; \; \; \; \; \;\begin{array}{ll} M_{(1,k)} & \left(\alpha,\left((T_{m})_{\substack{(m,\hh)\in L\setminus \bigcup_{k'\geqslant k}L(1,k')}}, \right.\right. \\ & \left. (T_{m}P_{1}^{-a_{m,1}}P^{a_{m,1}\theta})_{\substack{(m,\hh)\in  \bigcup_{k'\geqslant k}L(1,k')\setminus E(1,k)}}\right), \\ & \left.  (T_{i}^{-1}P_{1}^{k-\sum_{k'\geqslant k}k't_{1,k'}}P^{-(k-\sum_{k'\geqslant k}k't_{1,k'})\theta})_{\substack{(i,\l)\in  E(1,k)}}\right).\end{array} \end{multline*}
On en d\'eduit donc que $ A_{2}\ll M_{2} $.
Il reste \`a r\'eduire les bornes intervenant dans $ A_{1} $. Pour cela on r\'eit\`ere le proc\'ed\'e d\'evelopp\'e au cours de cette d\'emonstration avec les familles de variables $ (x_{m}^{\hh})_{(m,\hh)\in L(1,k_{1},h)} $ pour $ h\in \{1,...,t_{1,k_{1}}-2\} $. On constate qu'\`a terme on obtient : $ A_{1}\ll \max\{M_{1},M_{2}\} $, d'o\`u le r\'esultat. 
\end{proof}
Nous avons donc ici d\'emontr\'e que, pour tout poids $ k\in \PPP(1) $ et tout poids $ k_{1} $ tel que $ k_{1}>k $ : 
\begin{multline*}\left(\prod_{m\in J(1,k_{1})}T_{m}^{\prod_{j=2}^{r}t_{j,a_{m,j}}}\right) \\ \begin{array}{ll}M_{(1,k_{1})} & \left(\alpha,\left((T_{m})_{\substack{(m,\hh)\in L\setminus \bigcup_{k'>k}L(1,k')}}, \right.\right. \\  & \left.(T_{m}P_{1}^{-a_{m,1}}P^{a_{m,1}\theta})_{\substack{(m,\hh)\in  \bigcup_{k'>k}L(1,k')\setminus E(1,k_{1})}}\right), \\ & \left. (T_{i}^{-1}P_{1}^{k_{1}-\sum_{k'>k}k't_{1,k'}}P^{-(k_{1}-\sum_{k'>k}k't_{1,k'})\theta})_{\substack{(i,\l)\in  E(1,k_{1})}}\right) \end{array}  \\  
 \ll \left(\prod_{m\in J(1,k)}(\frac{P_{1}^{a_{m,1}}}{P^{a_{m,1}\theta}})^{\prod_{j=1}^{r}t_{j,a_{m,j}}}\right) \\  \times \max\left\{\prod_{m\in J(1,k_{1})}\left(T_{m}\frac{P^{a_{m,1}\theta}}{P_{1}^{a_{m,1}}}\right)^{\prod_{j=2}^{r}t_{j,a_{m,j}}}\right. \\ \begin{array}{ll}M_{(1,k_{1})} & \left(\alpha,\left((T_{m})_{\substack{(m,\hh)\in L\setminus \bigcup_{k' \geqslant k}L(1,k')}}, \right.\right. \\ & \left. (T_{m}P_{1}^{-a_{m,1}}P^{a_{m,1}\theta})_{\substack{(m,\hh)\in  \bigcup_{k'\geqslant k}L(1,k')\setminus E(1,k_{1})}}\right), \\ & \left. (T_{i}^{-1}P_{1}^{k_{1}-\sum_{k'\geqslant k}k't_{1,k'}}P^{-(k_{1}-\sum_{k'\geqslant k}k't_{1,k'})\theta})_{\substack{(i,\l)\in E(1,k_{1})}}\right), \end{array} \\ \prod_{m\in J(1,k)}\left(T_{m}\frac{P^{a_{m,1}\theta}}{P_{1}^{a_{m,1}}}\right)^{\prod_{j=2}^{r}t_{j,a_{m,j}}} \\ \begin{array}{ll} M_{(1,k)} & \left(\alpha,\left((T_{m})_{\substack{(m,\hh)\in L\setminus \bigcup_{k'\geqslant k}L(1,k')}}, \right.\right. \\ & \left. (T_{m}P_{1}^{-a_{m,1}}P^{a_{m,1}\theta})_{\substack{(m,\hh)\in  \bigcup_{k'\geqslant k}L(1,k')\setminus E(1,k)}}\right), \\ & \left. \left. (T_{i}^{-1}P_{1}^{k-\sum_{k'\geqslant k}k't_{1,k'}}P^{-(k-\sum_{k'\geqslant k}k't_{1,k'})\theta})_{\substack{(i,\l)\in  E(1,k)}}\right)\right\}.\end{array}
\end{multline*}

\begin{cor}\label{recurrence2cor3} Pour tout $ \theta\in [0,1] $, on a la majoration ci-dessous
\begin{multline*}\begin{array}{ll}\left(\prod_{m\in J(1,K)}T_{m}^{\prod_{j=2}^{r}t_{j,a_{m,j}}}\right)M_{(1,K)} & \left(\alpha,(T_{m})_{\substack{(m,\hh)\in  L\setminus E(1,K)}},\right. \\ & \left. (T_{i}^{-1})_{\substack{(i,\l)\in  E(1,K)}}\right) \end{array}  \\  
 \ll \left(\prod_{m\in I_{0}}(\frac{P_{1}^{a_{m,1}}}{P^{a_{m,1}\theta}})^{\prod_{j=1}^{r}t_{j,a_{m,j}}}\right)  \max_{K_{1}\in \PPP(1)}\{\left(\prod_{m\in J(1,K_{1})}\left(T_{m}\frac{P^{a_{m,1}\theta}}{P_{1}^{a_{m,1}}}\right)^{\prod_{j=2}^{r}t_{j,a_{m,j}}}\right) \\  \begin{array}{ll}M_{(1,K_{1})} & \left(\alpha,(T_{m}P_{1}^{-a_{m,1}}P^{a_{m,1}\theta})_{\substack{(m,\hh)\in  L\setminus E(1,K_{1})}}, \right. \\ & \left. (T_{i}^{-1}P_{1}^{K_{1}-d_{1}}P^{(d_{1}-K_{1})\theta})_{\substack{(i,\l)\in E(1,K_{1})}}\right)\}. \end{array}
\end{multline*}
\end{cor}

\subsubsection{Troisi\`eme r\'ecurrence : changement d'indice $ j $}

Dans cette partie nous allons traiter le passage des variables $ (x_{m}^{\hh})_{m\in J(j,k)} $ aux variables $ (x_{m}^{\hh})_{m\in J(j+1,k')} $. Nous allons en effet d\'emontrer le lemme ci-dessous : 
\begin{lemma}\label{recurrence33}
Pour tout $ \theta\in [0,1] $, et pour tout $ l\in \{1,...,r\} $ on a la majoration : 
\begin{multline*} \max_{\substack{(j_{1},k_{j_{1}})\in \CCC_{0}\\ j_{1}\in \{1,...,l\}}}\{\left(\prod_{m\in J(j_{1},k_{j_{1}})}\left(T_{m}\prod_{j=1}^{l}\frac{P^{a_{m,j}\theta}}{P_{j}^{a_{m,j}}}\right)^{\prod_{j\neq j_{1}}t_{j,a_{m,j}}}\right) \\  \begin{array}{ll}M_{(j_{1},k_{j_{1}})} & \left(\alpha,(T_{m}\prod_{j=1}^{l}P_{j}^{-a_{m,j}}P^{a_{m,j}\theta})_{\substack{(m,\hh)\in  L\setminus E(j_{1},k_{j_{1}})}}, \right. \\ & \left. (T_{i}^{-1}\prod_{j=1}^{l}P_{j}^{a_{i,j}-d_{j}}P^{(d_{j}-a_{i,j})\theta})_{\substack{(i,\l)\in E(j_{1},k_{j_{1}})}}\right)\} \end{array} \\ \ll \prod_{i=1}^{n+r}\left(\frac{P_{j}^{a_{m,1}}}{P^{a_{m,1}\theta}}\right)^{\prod_{j=1}^{r}t_{j,a_{m,j}}} \\  \max_{\substack{(j_{1},k_{j_{1}})\in \CCC_{0}\\ j_{1}\in \{1,...,l+1\}}}\{\left(\prod_{m\in J(j_{1},k_{j_{1}})}\left(T_{m}\prod_{j=1}^{l+1}\frac{P^{a_{m,j}\theta}}{P_{j}^{a_{m,j}}}\right)^{\prod_{j\neq j_{1}}t_{j,a_{m,j}}}\right) \\  \begin{array}{ll}M_{(j_{1},k_{j_{1}})} & \left(\alpha,(T_{m}\prod_{j=1}^{l+1}P_{j}^{-a_{m,j}}P^{a_{m,j}\theta})_{\substack{(m,\hh)\in  L\setminus E(j_{1},k_{j_{1}})}}, \right. \\ & \left. (T_{i}^{-1}\prod_{j=1}^{l+1}P_{j}^{a_{i,j}-d_{j}}P^{(d_{j}-a_{i,j})\theta})_{\substack{(i,\l)\in E(j_{1},k_{j_{1}})}}\right)\} \end{array}
\end{multline*}
\end{lemma}
\begin{proof}
Consid\'erons un \'el\'ement $ (j_{1},k_{j_{1}})\in \CCC_{0} $ tel que $ j_{1}\in \{1,...,l\} $. Posons par ailleurs $ k=\max_{i\in I_{0}}a_{i,l+1} $. Fixons dans un premier temps un ensemble de variables $ \yy=\widehat{\XX}_{\substack{(j_{1},k_{j_{1}})\\ (l+1,k)}}=(x_{m}^{\hh})_{(m,\hh)\in L\setminus (E(j_{1},k_{j_{1}})\cup E(l+1,k))} $ v\'erifiant : \[ \forall (m,\hh)\in L\setminus (E(j_{1},k_{j_{1}})\cup E(l+1,k)), \; \; |x_{m}^{\hh}|\leqslant T_{m}\prod_{j=1}^{l}P_{j}^{-a_{m,j}}P^{a_{m,j}\theta}, \] \[ \forall (i,\l)\in E(j_{1},k_{j_{1}})\cap E(l+1,k), \; \; 
||\alpha e_{i}\gamma_{(i,\l)}^{(j_{1},k_{j_{1}})}(\widehat{\XX}_{\substack{(j_{1},k_{j_{1}})\\ (l+1,k)}}) ||<T_{i}^{-1}\prod_{j=1}^{l}P_{j}^{a_{i,j}-d_{j}}P^{(d_{j}-a_{i,j})\theta}. \]

On consid\`ere alors les variables  $(x_{m}^{(\hh)})_{\substack{(m,\hh)\in E(l+1,k)\setminus E(j_{1},k_{j_{1}})}} $. Comme pr\'ec\'edemment, on note pour $ Z\geqslant 1 $ et $ \bb=(b_{m,\hh})_{(m,\hh)\in E(l+1,k)\setminus E(j_{1},k_{j_{1}})} $ et $  \cc=(c_{i,\l})_{(i,\l)\in E(j_{1},k_{j_{1}})\setminus  E(l+1,k)} $ des familles de r\'eels positifs : \begin{multline*} U_{\yy}(Z)=\Card\{ (x_{m}^{(\hh)})_{\substack{(m,\hh)\in E(l+1,k)\setminus E(j_{1},k_{j_{1}})}}\; |\;  \\   \forall (m,\hh)\in E(l+1,k)\setminus E(j_{1},k_{j_{1}}), \;   |x_{m}^{\hh}|\leqslant b_{m,\hh}Z \\ \et \;   \forall (i,\l)\in E(j_{1},k_{j_{1}})\setminus  E(l+1,k), \; \\ ||\alpha e_{i}\gamma_{(i,\l)}^{(j_{1},k_{j_{1}})}(\yy,(x_{m}^{(\hh)})_{\substack{(m,\hh)\in E(l+1,k)\setminus E(j_{1},k_{j_{1}})}}) || <c_{i,\l}^{-1}Z\}. \end{multline*}
Choisissons alors $ Z_{2} $, $ \bb $ et $ \cc $ tels que 
\[ \begin{array}{l}
 b_{m,\hh}Z_{2}=T_{m}\prod_{j=1}^{l}P_{j}^{-a_{m,j}}P^{a_{m,j}\theta}, \\  c_{i,\l}^{-1}Z_{2}=T_{i}^{-1}\prod_{j=1}^{l}P_{j}^{a_{i,j}-d_{j}}P^{(d_{j}-a_{i,j})\theta},
\end{array}  \]
de sorte que \begin{multline}\label{form1rec33}
 \begin{array}{ll}M_{(j_{1},k_{j_{1}})} & \left(\alpha,(T_{m}\prod_{j=1}^{l}P_{j}^{-a_{m,j}}P^{a_{m,j}\theta})_{\substack{(m,\hh)\in  L\setminus E(j_{1},k_{j_{1}})}}, \right. \\ & \left. (T_{i}^{-1}\prod_{j=1}^{l}P_{j}^{a_{i,j}-d_{j}}P^{(d_{j}-a_{i,j})\theta})_{\substack{(i,\l)\in E(j_{1},k_{j_{1}})}}\right)\} \end{array} \\ \ll \sum_{\yy}U_{\yy}(Z_{2}). \end{multline}
On choisit par ailleurs $ Z_{1} $ tel que \[ \begin{array}{l}
b_{m,\hh}Z_{1}=T_{m}P_{l+1}^{-k}P^{k\theta}\prod_{j=1}^{l}P_{j}^{-a_{m,j}}P^{a_{m,j}\theta}=T_{m}\prod_{j=1}^{l+1}P_{j}^{-a_{m,j}}P^{a_{m,j}\theta}, \\ c_{i,\l}^{-1}Z_{1}=T_{i}^{-1}P_{l+1}^{-k}P^{k\theta}\prod_{j=1}^{l}P_{j}^{a_{i,j}-d_{j}}P^{(d_{j}-a_{i,j})\theta}, \\ c_{i,\l}Z_{1}=T_{i}\prod_{j=1}^{l}P_{j}^{-a_{i,j}}P^{a_{i,j}\theta}
\end{array}   \]  
Ces conditions sur $ \bb,\cc,Z_{2},Z_{1} $ sont satisfaites si et seulement si \[ \begin{array}{l}
Z_{1}=P_{l+1}^{-\frac{k}{2}}P^{\frac{k\theta}{2}}\prod_{j=1}^{l}P_{j}^{-\frac{d_{j}}{2}}P^{\frac{d_{j}}{2}\theta}, \\ Z_{2}=Z_{1}\frac{P_{l+1}^{k}}{P^{k\theta}}=P_{l+1}^{\frac{k}{2}}P^{-\frac{k\theta}{2}}\prod_{j=1}^{l}P_{j}^{-\frac{d_{j}}{2}}P^{\frac{d_{j}}{2}\theta} \\ b_{m,\hh}=T_{m}P_{l+1}^{-\frac{k}{2}}P^{\frac{k\theta}{2}}\prod_{j=1}^{l}P_{j}^{\frac{d_{j}}{2}-a_{m,j}}P^{-\frac{d_{j}}{2}\theta+a_{m,j}\theta}, \\ c_{i,\l}=T_{i}P_{l+1}^{\frac{k}{2}}P^{-\frac{k\theta}{2}}\prod_{j=1}^{l}P_{j}^{\frac{d_{j}}{2}-a_{i,j}}P^{-\frac{d_{j}}{2}\theta+a_{i,j}\theta}.
\end{array}   \]  
On a alors $ b_{m,\hh}^{-1}Z_{1}=T_{m}^{-1}\prod_{j=1}^{l}P_{j}^{a_{m,j}-d_{j}}P^{(d_{j}-a_{m,j})\theta} $.\\

Par ailleurs, en appliquant \`a nouveau le lemme\;\ref{geomnomb12},  \begin{multline}\label{geomnomb23} U_{\yy}(Z_{2})\ll \max\left(\prod_{\substack{(m,\hh)\in E(l+1,k)\setminus E(j_{1},k_{j_{1}}) }} \left(\frac{Z_{2}}{Z_{1}}\right)U_{\yy}(Z_{1}), \right. \\ \left. \frac{\prod_{\substack{(m,\hh)\in E(l+1,k)\setminus E(j_{1},k_{j_{1}})}}b_{m,\hh}Z_{2}}{\prod_{\substack{(i,\l)\in E(j_{1},k_{j_{1}})\setminus E(l+1,k)}}c_{i,\l}Z_{1}}U^{t}_{\yy}(Z_{1})\right), \end{multline}
On observe que \begin{multline*}  \prod_{\substack{(m,\hh)\in E(l+1,k)\setminus E(j_{1},k_{j_{1}}) }} \left(\frac{Z_{2}}{Z_{1}}\right)=\prod_{\substack{(m,\hh)\in E(l+1,k)\setminus E(j_{1},k_{j_{1}}) }} \left(\frac{P_{l+1}^{a_{m,l+1}}}{P^{a_{m,l+1}\theta}}\right) \\ =\prod_{\substack{m\in J(l+1,k)\setminus J(j_{1},k_{j_{1}}) }} \left(\frac{P_{l+1}^{a_{m,l+1}}}{P^{a_{m,l+1}\theta}}\right)^{\prod_{j\neq l+1}t_{j,a_{m,j}}} \\ \prod_{\substack{m\in J(l+1,k)\cap J(j_{1},k_{j_{1}})}} \left(\frac{P_{l+1}^{a_{m,l+1}}}{P^{a_{m,l+1}\theta}}\right)^{(t_{j_{1},a_{m,j_{1}}}-1)\prod_{\substack{j\neq l+1\\ j\neq j_{1}}}t_{j,a_{m,j}}} \\ =\prod_{\substack{m\in J(l+1,k)}} \left(\frac{P_{l+1}^{a_{m,l+1}}}{P^{a_{m,l+1}\theta}}\right)^{\prod_{j\neq l+1}t_{j,a_{m,j}}}\prod_{\substack{m\in J(l+1,k)\cap J(j_{1},k_{j_{1}})}} \left(\frac{P_{l+1}^{a_{m,l+1}}}{P^{a_{m,l+1}\theta}}\right)^{-\prod_{\substack{j\neq l+1\\ j\neq j_{1}}}t_{j,a_{m,j}}} \end{multline*}
et de m\^eme 
\begin{multline*} \frac{\prod_{\substack{(m,\hh)\in E(l+1,k)\setminus E(j_{1},k_{j_{1}})}}b_{m,\hh}Z_{2}}{\prod_{\substack{(i,\l)\in E(j_{1},k_{j_{1}})\setminus E(l+1,k)}}c_{i,\l}Z_{1}} =\frac{\prod_{\substack{(m,\hh)\in E(l+1,k)\setminus E(j_{1},k_{j_{1}})}}T_{m}\prod_{j=1}^{l}P_{j}^{-a_{m,j}}P^{a_{m,j}\theta}}{\prod_{\substack{(i,\l)\in E(j_{1},k_{j_{1}})\setminus E(l+1,k)}}T_{i}\prod_{j=1}^{l}P_{j}^{-a_{i,j}}P^{a_{i,j}\theta}}\\ =\frac{\prod_{\substack{m\in J(l+1,k)}}\left(T_{m}\prod_{j=1}^{l}P_{j}^{-a_{m,j}}P^{a_{m,j}\theta}\right)^{\prod_{j\neq l+1}t_{j,a_{m,j}}} }{\prod_{\substack{i\in J(j_{1},k_{j_{1}})}}\left(T_{i}\prod_{j=1}^{l}P_{j}^{-a_{i,j}}P^{a_{i,j}\theta}\right)^{\prod_{j\neq j_{1}}t_{j,a_{i,j}}} } \\ \frac{\prod_{\substack{m\in J(l+1,k)\cap J(j_{1},k_{j_{1}})}} \left(T_{m}\prod_{j=1}^{l}P_{j}^{-a_{m,j}}P^{a_{m,j}\theta}\right)^{-\prod_{\substack{j\neq j_{1}\\ j\neq l+1}}t_{j,a_{m,j}}}}{\prod_{\substack{i\in J(l+1,k)\cap J(j_{1},k_{j_{1}})}} \left(T_{i}\prod_{j=1}^{l}P_{j}^{-a_{i,j}}P^{a_{i,j}\theta}\right)^{-\prod_{\substack{j\neq j_{1}\\ j\neq l+1}}t_{j,a_{i,j}}}} \\ =\frac{\prod_{\substack{m\in J(l+1,k)}}\left(T_{m}\prod_{j=1}^{l}P_{j}^{-a_{m,j}}P^{a_{m,j}\theta}\right)^{\prod_{j\neq l+1}t_{j,a_{m,j}}} }{\prod_{\substack{i\in J(j_{1},k_{j_{1}})}}\left(T_{i}\prod_{j=1}^{l}P_{j}^{-a_{i,j}}P^{a_{i,j}\theta}\right)^{\prod_{j\neq j_{1}}t_{j,a_{i,j}}} } . \end{multline*}
Par ailleurs, on a que 
\begin{multline*} U^{t}_{\yy}(Z_{1})=\Card\{(x_{i}^{\l})_{\substack{(i,\l)\in E(j_{1},k_{j_{1}})\setminus E(l+1,k) }} \; | \;  \forall (i,\l), \; |x_{i}^{\l}|\leqslant T_{i}\prod_{j=1}^{l}P_{j}^{-a_{i,j}}P^{a_{i,j}\theta} \\ \et,\;  \forall (m,\hh)\in E(l+1,k)\setminus E(j_{1},k_{j_{1}}), \; \\  ||\alpha e_{m}(\gamma^{(j_{1},k_{j_{1}})})^{t}_{(m,\hh)}(\yy, (x_{i}^{(\l)})_{(i,\l)\in E(j_{1},k_{j_{1}})\setminus E(l+1,k)})||\leqslant T_{m}^{-1}\prod_{j=1}^{l}P_{j}^{a_{m,j}-d_{j}}P^{(d_{j}-a_{m,j})\theta} \}. \end{multline*} 
On remarque que $ (\gamma^{(j_{1},k_{j_{1}})})^{t}_{(m,\hh)}=\gamma^{(l+1,k)}_{(m,\hh)} $. On a finalement, d'apr\`es les formules\;\eqref{form1rec33} et\;\eqref{geomnomb23} :  \begin{multline*}  \begin{array}{ll}M_{(j_{1},k_{j_{1}})} & \left(\alpha,(T_{m}\prod_{j=1}^{l}P_{j}^{-a_{m,j}}P^{a_{m,j}\theta})_{\substack{(m,\hh)\in  L\setminus E(j_{1},k_{j_{1}})}}, \right. \\ & \left. (T_{i}^{-1}\prod_{j=1}^{l}P_{j}^{a_{i,j}-d_{j}}P^{(d_{j}-a_{i,j})\theta})_{\substack{(i,\l)\in E(j_{1},k_{j_{1}})}}\right)\} \end{array} \\ \ll  \max\left(\prod_{\substack{m\in J(l+1,k)}} \left(\frac{P_{l+1}^{a_{m,l+1}}}{P^{a_{m,l+1}\theta}}\right)^{\prod_{j\neq l+1}t_{j,a_{m,j}}}  \prod_{\substack{m\in J(l+1,k)\\ \cap J(j_{1},k_{j_{1}})}} \left(\frac{P_{l+1}^{a_{m,l+1}}}{P^{a_{m,l+1}\theta}}\right)^{-\prod_{\substack{j\neq l+1\\ j\neq j_{1}}}t_{j,a_{m,j}}}\sum_{\yy}U_{\yy}(Z_{1}), \right. \\ \left. \frac{\prod_{\substack{m\in J(l+1,k)}}\left(T_{m}\prod_{j=1}^{l}P_{j}^{-a_{m,j}}P^{a_{m,j}\theta}\right)^{\prod_{j\neq l+1}t_{j,a_{m,j}}} }{\prod_{\substack{i\in J(j_{1},k_{j_{1}})}}\left(T_{i}\prod_{j=1}^{l}P_{j}^{-a_{i,j}}P^{a_{i,j}\theta}\right)^{\prod_{j\neq j_{1}}t_{j,a_{i,j}}} } \sum_{\yy}U^{t}_{\yy}(Z_{1})\right). \end{multline*}
Or, on observe que \begin{multline*}
\sum_{\yy}U_{\yy}(Z_{1})=M_{(j_{1},k_{j_{1}})}\left(\alpha,(A_{m,\hh})_{\substack{(m,\hh)\in L\setminus E(j_{1},k_{j_{1}})}}, (B_{i,\l})_{\substack{(i,\l)\in E(j_{1},k_{j_{1}})}}\right)
\end{multline*} o\`u  \[ A_{m,\hh}=\left\{\begin{array}{lcr}
T_{m}\prod_{j=1}^{l}P_{j}^{-a_{m,j}}P^{a_{m,j}\theta} & \mbox{si} & (m,\hh)\in L\setminus ( E(l+1,k)\cup E(j_{1},k_{j_{1}})), \\ T_{m}\prod_{j=1}^{l+1}P_{j}^{-a_{m,j}}P^{a_{m,j}\theta} & \mbox{si} & (m,\hh)\in E(l+1,k)\setminus E(j_{1},k_{j_{1}})\end{array}\right., \]
\[ B_{i,\l}=\left\{\begin{array}{lcr}
T_{i}^{-1}P_{l+1}^{-k}P^{k\theta}\prod_{j=1}^{l}P_{j}^{a_{i,j}-d_{j}}P^{(d_{j}-a_{i,j})\theta} & \mbox{si} & (i,\l)\in E(j_{1},k_{j_{1}})\setminus E(l+1,k), \\ T_{i}^{-1}\prod_{j=1}^{l}P_{j}^{a_{i,j}-d_{j}}P^{(d_{j}-a_{i,j})\theta} & \mbox{si} & (m,\hh)\in E(l+1,k)\cap E(j_{1},k_{j_{1}}) \end{array}\right.. \]
et  \begin{multline*}
 \begin{array}{ll}\sum_{\yy}U_{\yy}^{t}(Z_{1})=M_{(l+1,k)} & \left(\alpha,(T_{i}\prod_{j=1}^{l}P_{j}^{-a_{i,j}}P^{a_{i,j}\theta})_{\substack{(i,\l)\in  L\setminus E(l+1,k)}}, \right. \\ & \left.(T_{m}^{-1}\prod_{j=1}^{l}P_{j}^{a_{m,j}-d_{j}}P^{(d_{j}-a_{m,j})\theta})_{(m,\hh)\in E(l+1,k)}\right)\} \end{array}
\end{multline*}
Or en employant les m\^emes arguments que pour le lemme\;\ref{recurrence23}, on montre que :  \begin{multline*}
 \left(\prod_{m\in J(l+1,k)}\left(T_{m}\prod_{j=1}^{l}\frac{P^{a_{m,j}\theta}}{P_{j}^{a_{m,j}}}\right)^{\prod_{j\neq l+1}t_{j,a_{m,j}}}\right) \\ \begin{array}{ll}M_{(l+1,k)} & \left(\alpha,(T_{i}\prod_{j=1}^{l}P_{j}^{-a_{i,j}}P^{a_{i,j}\theta})_{\substack{(i,\l)\in  L\setminus E(l+1,k)}}, \right. \\ & \left.(T_{m}^{-1}\prod_{j=1}^{l}P_{j}^{a_{m,j}-d_{j}}P^{(d_{j}-a_{m,j})\theta})_{(m,\hh)\in E(l+1,k)}\right) \end{array} \\ \ll \left(\prod_{m\in I_{0}}(\frac{P_{l+1}^{a_{m,l+1}}}{P^{a_{m,l+1}\theta}})^{\prod_{j=1}^{r}t_{j,a_{m,j}}}\right)  \max_{K_{1}\in \PPP(l+1)}\{\left(\prod_{m\in J(l+1,K_{1})}\left(T_{m}\prod_{j=1}^{l+1}\frac{P^{a_{m,j}\theta}}{P_{1}^{a_{m,j}}}\right)^{\prod_{j\neq l+1}t_{j,a_{m,j}}}\right) \\  \begin{array}{ll}M_{(l+1,K_{1})} & \left(\alpha,(T_{m}\prod_{j=1}^{l+1}P_{j}^{-a_{m,j}}P^{a_{m,j}\theta})_{\substack{(m,\hh)\in  L\setminus E(l+1,k)}}, \right. \\ & \left.(T_{i}^{-1}\prod_{j=1}^{l+1}P_{j}^{a_{i,j}-d_{j}}P^{(d_{j}-a_{i,j})\theta})_{(i,\l)\in E(l+1,k)}\right)\}. \end{array}
\end{multline*}
En regroupant les r\'esultats obtenus, on obtient donc  \begin{multline*}
\prod_{m\in J(j_{1},k_{j_{1}})}\left(T_{m}\prod_{j=1}^{l}\frac{P^{a_{m,j}\theta}}{P_{j}^{a_{m,j}}}\right)^{\prod_{j\neq j_{1}}t_{j,a_{m,j}}}\\  \begin{array}{ll}M_{(j_{1},k_{j_{1}})} & \left(\alpha,(T_{m}\prod_{j=1}^{l}P_{j}^{-a_{m,j}}P^{a_{m,j}\theta})_{\substack{(m,\hh)\in  L\setminus E(j_{1},k_{j_{1}})}}, \right. \\ & \left. (T_{i}^{-1}\prod_{j=1}^{l}P_{j}^{a_{i,j}-d_{j}}P^{(d_{j}-a_{i,j})\theta})_{\substack{(i,\l)\in E(j_{1},k_{j_{1}})}}\right)\} \end{array} \\ \ll  \max\{A_{1},A_{2}\} \end{multline*}

\begin{multline*}
A_{1}=\prod_{\substack{m\in J(l+1,k)}} \left(\frac{P_{l+1}^{a_{m,l+1}}}{P^{a_{m,l+1}\theta}}\right)^{\prod_{j\neq l+1}t_{j,a_{m,j}}}  \prod_{\substack{m\in J(l+1,k)\\ \cap J(j_{1},k_{j_{1}})}} \left(\frac{P_{l+1}^{a_{m,l+1}}}{P^{a_{m,l+1}\theta}}\right)^{-\prod_{\substack{j\neq l+1\\ j\neq j_{1}}}t_{j,a_{m,j}}} \\  \prod_{\substack{m\in J(j_{1},k_{j_{1}})}}\left(T_{m}\prod_{j=1}^{l+1}\frac{P^{a_{m,j}\theta}}{P_{j}^{a_{m,j}}}\right)^{\prod_{j\neq j_{1}}t_{j,a_{m,j}}}\begin{array}{ll} M_{(j_{1},k_{j_{1}})} & \left(\alpha,(A_{m,\hh})_{\substack{(m,\hh)\in L\setminus E(j_{1},k_{j_{1}})}}, \right. \\ & \left. (B_{i,\l})_{\substack{(i,\l)\in E(j_{1},k_{j_{1}})}}\right),\end{array}\end{multline*}  \begin{multline*} A_{2} =\left(\prod_{m\in I_{0}}(\frac{P_{l+1}^{a_{m,l+1}}}{P^{a_{m,l+1}\theta}})^{\prod_{j=1}^{r}t_{j,a_{m,j}}}\right) \max_{K_{1}\in \PPP(l+1)}\{\prod_{\substack{m\in J(l+1,k)}}\left(T_{m}\prod_{j=1}^{l+1}\frac{P^{a_{m,j}\theta}}{P_{j}^{a_{m,j}}}\right)^{\prod_{j\neq l+1}t_{j,a_{m,j}}} \\  \begin{array}{ll}M_{(l+1,K_{1})} & \left(\alpha,(T_{m}\prod_{j=1}^{l+1}P_{j}^{-a_{m,j}}P^{a_{m,j}\theta})_{\substack{(m,\hh)\in  L\setminus E(l+1,k)}}, \right. \\ & \left.(T_{i}^{-1}\prod_{j=1}^{l+1}P_{j}^{a_{i,j}-d_{j}}P^{(d_{j}-a_{i,j})\theta})_{(i,\l)\in E(l+1,k)}\right)\}. \end{array}. \end{multline*}
Le terme $ A_{2} $ est du type de ceux intervenant \`a droite de l'in\'egalit\'e du lemme. Il nous reste donc \`a r\'eduire les bornes intervenant dans $ A_{1} $. Pour cela, il suffit d'appliquer exactement le m\^eme proc\'ed\'e que celui que nous avons d\'evelopp\'e pour toutes les familles de variables $ (x_{m})_{(m,\hh)\in L(j_{1},k_{j_{1}},h)} $ pour tous $ h\in \{1,...,t_{j_{1},k_{j_{1}}}-2\} $ puis pour les familles $ (x_{m})_{(m,\hh)\in L(j_{1},k',h)} $ pour tout poids $ k'\in \PPP(j_{1})\setminus \{k_{j_{1}}\} $. \`A terme, on obtient bien l'in\'egalit\'e souhait\'ee.  
\end{proof}

% ajouter complément

En utilisant les lemmes\;\ref{recurrence2cor3} et\;\ref{recurrence33}, on obtient directement, par r\'ecur\-rence, la majoration : 

\begin{multline*}
 \left(\prod_{\substack{i\in J(1,K)}}T_{i}^{\prod_{j=2}^{r}t_{j,a_{i,j}}}\right)M_{(1,K)}\left(\alpha,(T_{m})_{\substack{(m,\hh)\in L\setminus E(1,K) }}, (T_{i}^{-1})_{\substack{(i,\l)\in E(1,K)}}\right) \\ \ll \prod_{i\in I_{0}}\left(\frac{\prod_{j=1}^{r}P_{j}^{a_{i,j}}}{\prod_{j=1}^{r}P^{a_{i,j}\theta}}\right)^{\prod_{j=1}^{r}t_{j,a_{i,j}}}\max_{\substack{ (j_{0},K_{j_{0}})\in \CCC_{0}}}\frac{\prod_{\substack{m\in J(j_{0},K_{j_{0}})}}T_{m}^{\prod_{j\neq j_{0}}t_{j,a_{m,j}}}}{\prod_{\substack{m\in J(j_{0},K_{j_{0}})}}(\prod_{j=1}^{r}P_{j}^{a_{m,j}}P^{-a_{m,j}\theta})^{\prod_{j\neq j_{0}}t_{j,a_{m,j}}}} \\  M_{(j_{0},K_{j_{0}})}\left( \alpha, (T_{m}\prod_{j=1}^{r}P_{j}^{-a_{m,j}}P^{a_{m,j}\theta})_{\substack{(m,\hh)\in L\setminus E(j_{0},K_{j_{0}})}}, \right. \\ \left. (T_{i}^{-1}\prod_{j=1}^{r}P_{j}^{-d_{j}+a_{i,j}}P^{(d_{j}-a_{i,j})\theta})_{\substack{(i,\l)\in E(j_{0},K_{j_{0}})}}\right).
\end{multline*}
En rappelant que $ T_{i}=\frac{1}{e_{i}}\prod_{j=1}^{r}P_{j}^{a_{i,j}} $, on remarque que le membre de droite de la majoration ci-dessus peut se r\'e\'ecrire : \begin{multline*}
 \prod_{i\in I_{0}}\left(\frac{e_{i}T_{i}}{\prod_{j=1}^{r}P^{a_{i,j}\theta}}\right)^{\prod_{j=1}^{r}t_{j,a_{i,j}}}\max_{\substack{ (j_{0},K_{j_{0}})\in \CCC_{0}}}\prod_{\substack{m\in J(j_{0},K_{j_{0}})}}(e_{m}^{-1}P^{(\sum_{j=1}^{r}a_{m,j})\theta})^{\prod_{j\neq j_{0}}t_{j,a_{m,j}}} \\  M_{(j_{0},K_{j_{0}})}\left( \alpha, (\frac{1}{e_{m}}P^{\sum_{j=1}^{r}a_{m,j}\theta})_{\substack{(m,\hh)\in L\setminus E(j_{0},K_{j_{0}})}}, \right. \\ \left. (e_{i}P^{\sum_{j=1}^{r}(d_{j}-a_{i,j})\theta}\prod_{j=1}^{r}P_{j}^{-d_{j}})_{\substack{(i,\l)\in E(j_{0},K_{j_{0}})}}\right).
\end{multline*}
Nous en d\'eduisons alors une nouvelle version du lemme \ref{dilemme13} :

\begin{lemma}\label{dilemme13}
Pour tout $ \varepsilon>0 $ arbitrairement petit on a\begin{multline*}
|S_{\ee}(\alpha)| \ll\left(\prod_{j=1}^{r}P_{j}\right)\left(\prod_{i=1}^{n+r}T_{i}\right)^{(1+\varepsilon)}\left(\prod_{i\in  I_{0}}e_{i}^{\frac{\prod_{j=1}^{r}t_{j,a_{i,j}}}{2^{D_{1}+D_{2}+...+D_{r}}}}\right)\\ \times \prod_{i\in I_{0}}\left(P^{\sum_{j=1}^{r}a_{i,j}\theta}\right)^{-\frac{\prod_{j=1}^{r}t_{j,a_{i,j}}}{2^{D_{1}+D_{2}+...+D_{r}}}}\max_{\substack{ (j_{0},K_{j_{0}})\in \CCC_{0}}}\prod_{\substack{m\in J(j_{0},K_{j_{0}})}}(e_{m}^{-1}P^{(\sum_{j=1}^{r}a_{m,j})\theta})^{\frac{\prod_{j\neq j_{0}}t_{j,a_{m,j}}}{2^{D_{1}+D_{2}+...+D_{r}}}} \\  M_{(j_{0},K_{j_{0}})}\left( \alpha, (\frac{1}{e_{m}}P^{\sum_{j=1}^{r}a_{m,j}\theta})_{\substack{(m,\hh)\in L\setminus E(j_{0},K_{j_{0}})}}, \right. \\ \left. (e_{i}P^{\sum_{j=1}^{r}(d_{j}-a_{i,j})\theta}\prod_{j=1}^{r}P_{j}^{-d_{j}})_{\substack{(i,\l)\in E(j_{0},K_{j_{0}})}}\right)^{2^{-(D_{1}+D_{2}+...+D_{r})}} .,
\end{multline*}
\end{lemma}

On d\'eduit de ce lemme le r\'esultat ci-dessous

\begin{lemma}\label{dilemme23}
Pour tout $ \varepsilon>0 $ arbitrairement petit et tous $ \kappa>0, \; P\geqslant 1 $, l'une au moins des assertions suivantes est vraie : \begin{enumerate}
\item  \[|S_{\ee}(\alpha)|\ll \left(\prod_{i\in I_{0}}e_{i}^{\frac{\prod_{j=1}^{r}t_{j,a_{i,j}}}{2^{\sum_{j=1}^{r}D_{j}}}}\right)\left(\prod_{j=1}^{r}P_{j}\right)\left(\prod_{i=1}^{n+r}T_{i}\right)^{(1+\varepsilon)}P^{-\kappa},\]
\item  Il existe un certain  $ j_{0}\in \{1,...,r\} $ et $ K_{j_{0}}\in \{1,...,d_{j_{0}}\} $ tel que $ J(j_{0},K_{j_{0}})\neq \emptyset $\begin{multline*} M_{(j_{0},K_{j_{0}})}\left( \alpha, (\frac{1}{e_{m}}P^{\sum_{j=1}^{r}a_{m,j}\theta})_{\substack{(m,\hh)\in L\setminus E(j_{0},K_{j_{0}})}}, \right. \\ \left. (e_{i}P^{\sum_{j=1}^{r}(d_{j}-a_{i,j})\theta}\prod_{j=1}^{r}P_{j}^{-d_{j}})_{\substack{(i,\l)\in E(j_{0},K_{j_{0}})}}\right)
 \\ \gg \left(\prod_{i\in I_{0}}(P^{(\sum_{j=1}^{r}a_{i,j})\theta})^{\prod_{j=1}^{r}t_{j,a_{i,j}}}\right)\left(\prod_{i\in J(j_{0},K_{j_{0}})}(P^{\sum_{j=1}^{r}a_{i,j}\theta})^{\prod_{j\neq j_{0}}t_{j,a_{i,j}}}\right)^{-1}P^{2^{-\sum_{j=1}^{r}D_{j}}\kappa}.
\end{multline*}
\end{enumerate}
\end{lemma}
\begin{rem}
Si $ \kappa $ est petit, la condition $ 1 $ donne une majoration de $ |S_{\ee}(\alpha)| $ plus grande que la majoration triviale, \[ |S_{\ee}(\alpha)| \ll \left(\prod_{i=1}^{n+r}e_{i}\right)^{-1}\prod_{j=1}^{r}P_{j}^{n_{j}}, \] c'est pourquoi nous utiliserons uniquement cette majoration pour $ P^{\kappa}>\prod_{j=1}^{r}P_{j}^{d_{j}} $. \end{rem}

Supposons \`a pr\'esent qu'il existe $ (j_{0},K_{j_{0}})\in \CCC_{0} $ et \begin{multline*} (x_{m}^{(\hh)})_{\substack {(m,\hh)\in L\setminus E((j_{0},K_{j_{0}})}}\in \MM_{(j_{0},K_{j_{0}})}\left( \alpha, (\frac{1}{e_{m}}P^{\sum_{j=1}^{r}a_{m,j}\theta})_{\substack{(m,\hh)\in L\setminus E(j_{0},K_{j_{0}})}}, \right. \\ \left. (e_{i}P^{\sum_{j=1}^{r}(d_{j}-a_{i,j})\theta}\prod_{j=1}^{r}P_{j}^{-d_{j}})_{\substack{(i,\l)\in E(j_{0},K_{j_{0}})}}\right) \end{multline*} tel qu'il existe $ (i_{0},\l_{0})\in E(j_{0},K_{j_{0}}) $ tel que \[ \gamma_{(i,\l_{0})}^{(j_{0},K_{j_{0}})}(x_{m}^{\hh})_{(m,\hh)\in L\setminus E(j_{0},K_{j_{0}})}\neq 0.  \]

On pose $ q=|\gamma_{(i,\l_{0})}^{(j_{0},K_{j_{0}})}((x_{m}^{\hh})_{(m,\hh)\in L\setminus E(j_{0},K_{j_{0}})})| $, on a alors \[ ||\alpha e_{i_{0}}q||<e_{i_{0}}P^{\sum_{j=1}^{r}(d_{j}-a_{i_{0},j})\theta}\prod_{j=1}^{r}P_{j}^{-d_{j}} \] et de plus, \[  q\ll  P^{\sum_{j=1}^{r}(d_{j}-a_{i_{0},j})\theta},  \] (quitte \`a modifier $ \theta $, pour $ P $ grand on pourra supposer $ q\leqslant  P^{\sum_{j=1}^{r}(d_{j}-a_{i_{0},j})\theta} $). On en d\'eduit : 
\begin{lemma}\label{dilemme33}
Pour tout $ \varepsilon>0 $ arbitrairement petit et tous $ \kappa>0, \; P\geqslant 1 $, l'une au moins des assertions suivantes est vraie : \begin{enumerate}
\item  \[|S_{\ee}(\alpha)|\ll \left(\prod_{i\in I_{0}}e_{i}^{\frac{\prod_{j=1}^{r}t_{j,a_{i,j}}}{2^{\sum_{j=1}^{r}D_{j}}}}\right)\left(\prod_{j=1}^{r}P_{j}\right)\left(\prod_{i=1}^{n+r}T_{i}\right)^{(1+\varepsilon)}P^{-\kappa},\]
\item Il existe $ i\in I_{0} $ et des entiers $ a,q $ tels que $ q\leqslant P^{\sum_{j=1}^{r}(d_{j}-a_{i,j})\theta} $, $ a<e_{i}q $ et \[ |\alpha e_{i}q-a|\leqslant e_{i}\prod_{j=1}^{r}P_{j}^{-d_{j}}P^{\sum_{j=1}^{r}(d_{j}-a_{i,j})\theta}. \]
\item  Il existe un certain  $ (j_{0},K_{j_{0}})\in \CCC_{0}\} $ tel que \begin{multline*} \Card\{ (x_{m}^{(\hh)})_{\substack {(m,\hh)\in L\setminus E(j_{0},K_{j_{0}})}} \; | \; |x_{m}^{(\hh)}|\leqslant \frac{1}{e_{m}}P^{\sum_{j=1}^{r}a_{m,j}\theta}, \; \\ \forall (i,\l)\in E(j_{0},K_{j_{0}}), \; \gamma_{(i,\l)}^{(j_{0},K_{j_{0}})}(x_{m}^{(\hh)})=0  \} \\ \gg \left(\prod_{i\in I_{0}}(P^{(\sum_{j=1}^{r}a_{i,j})\theta})^{\prod_{j=1}^{r}t_{j,a_{i,j}}}\right)\left(\prod_{i\in J(j_{0},K_{j_{0}})}(P^{\sum_{j=1}^{r}a_{i,j}\theta})^{\prod_{j\neq j_{0}}t_{j,a_{i,j}}}\right)^{-1}P^{2^{-\sum_{j=1}^{r}D_{j}}\kappa}.
\end{multline*}
\end{enumerate}
\end{lemma}
Avant d'aller plus loin, nous introduisons les lemmes ci-dessous qui seront utiles \`a plusieurs reprise par la suite :

 \begin{lemma}\label{lemmedebile2}
 On consid\`ere $ p,q,r\in \NN $ et $ (L_{i})_{i\in\{1,...,r\} } $ des formes lin\'eaires \`a $ p+q $ variables. Pour des constantes $ A $, $ B $ et $ (C_{i})_{i\in I} $fix\'ees on note \begin{multline*} M\left(A,B,(C_{i})_{i\in \{1,...,r\}}\right)=\card\left\{ (\xx,\yy)\in \ZZ^{p}\times \ZZ^{q} \; | \; |\xx|\leqslant A ,\;  |\yy|\leqslant B, \; \right. \\  \left. \forall i  \in \{1,...,r\}, \;  ||L_{i}(\xx,\yy)||< C_{i} \right\}. \end{multline*}
 
 On a alors pour tout $ \xi\geqslant 1 $ : \[ M\left(A,B,(C_{i})_{i\in \{1,...,r\}}\right)\leqslant (2\xi)^{q} M\left(2A,\frac{B}{\xi},(2C_{i})_{i\in \{1,...,r\}}\right).\]

 \end{lemma}
 \begin{proof}
 On subdivise le cube $ [-B,B]^{q} $ en $ (2\xi)^{q}  $ cubes de taille $ B/\xi $. Prenons un tel cube $ \mathcal{C} $ et consid\'erons \begin{multline*} E(\mathcal{C})=\card\left\{ (\xx,\yy)\in \ZZ^{p}\times \ZZ^{q} \; | \; |\xx|\leqslant A ,\; \yy \in \mathcal{C} , \; \right. \\  \left. \forall i  \in \{1,...,r\}, \; ||L_{i}(\xx,\yy)||\leqslant C_{i}  \right\}. \end{multline*} Si $ (\xx,\yy), (\xx',\yy') $ sont deux points de $ E(\mathcal{C}) $, on a alors que \[ |\xx-\xx'|\leqslant 2A, \; \; |\yy-\yy'|\leqslant B/\xi \] et pour tout $ i \in \{1,...,r\} $ : \[ |L_{i}(\xx-\xx',\yy-\yy') |\leqslant 2C_{i}. \] On a donc : \[ E(\mathcal{C})\leqslant  M\left(2A,\frac{B}{\xi},(2C_{i})_{i\in \{1,...,r\}}\right) \] pour tout cube $ \mathcal{C} $. D'o\`u le r\'esultat. 
 \end{proof}
De la m\^eme mani\`ere, on \'etablit :
  \begin{lemma}\label{lemmedebilecas02}
 On consid\`ere $ p,q,r\in \NN $ et $ (L_{i})_{i\in\{1,...,r\} } $ des formes lin\'eaires \`a $ p+q $ variables. Pour des constantes $ A $, $ B $ on note \begin{multline*} M\left(A,B\right)=\card\left\{ (\xx,\yy)\in \ZZ^{p}\times \ZZ^{q} \; | \; |\xx|\leqslant A ,\;  |\yy|\leqslant B, \; \right. \\  \left. \forall i  \in \{1,...,r\}, \;  L_{i}(\xx,\yy)=0 \right\}. \end{multline*}
 
 On a alors pour tout $ \xi\geqslant 1 $ : \[ M\left(A,B\right)\leqslant (2\xi)^{q} M\left(2A,\frac{B}{\xi}\right).\]
 \end{lemma}

Supposons \`a pr\'esent que nous sommes dans le cas $ 3 $ du lemme\;\ref{dilemme33}, et remarquons que le cardinal consid\'er\'e peut \^etre major\'e trivialement par \begin{multline*} \Card\{ (x_{m}^{(\hh)})_{\substack {(m,\hh)\in L\setminus E(j_{0},K_{j_{0}})  }} \; | \; |x_{m}^{(\hh)}|\leqslant P^{\sum_{j=1}^{r}a_{m,j}\theta}, \; \\ \forall (i,\l)\in E(j_{0},K_{j_{0}}), \; \gamma_{(i,\l)}^{(j_{0},K_{j_{0}})}(x_{m}^{(\hh)})=0  \},
\end{multline*}
qui peut \^etre major\'e, en appliquant le lemme\;\ref{lemmedebilecas02} par : \begin{multline*} A_{(j_{0},K_{j_{0}})}(P^{\theta}) \prod_{\substack{(m,\hh)\in L\setminus E(j_{0},K_{j_{0}}) }}P^{((\sum_{j=1}^{r}a_{m,j})-1)\theta} \\ =A_{(j_{0},K_{j_{0}})}(P^{\theta})\left(\prod_{i\in I_{0}}(P^{((\sum_{j=1}^{r}a_{i,j})-1)\theta})^{\prod_{j=1}^{r}t_{j,a_{i,j}}}\right) \\ \left(\prod_{i\in J(j_{0},K_{j_{0}})}(P^{((\sum_{j=1}^{r}a_{i,j})-1)\theta})^{\prod_{j\neq j_{0}}t_{j,a_{i,j}}}\right)^{-1}
\end{multline*}
o\`u l'on a pos\'e
\begin{multline*} A_{(j_{0},K_{j_{0}})}(P^{\theta})=\Card\left\{ (x_{m}^{(\hh)})_{\substack {(m,\hh)\in L\setminus E(j_{0},K_{j_{0}})}} \; | \; |x_{m}^{(\hh)}|\leqslant P^{\theta}, \; \right.\\ \left.  \forall (i,\l)\in E(j_{0},K_{j_{0}}), \;\gamma_{(i,\l)}^{(j_{0},K_{j_{0}})}(\frac{1}{e_{m}}x_{m}^{(\hh)})=0  \right\}.
\end{multline*}

Ainsi, la condition $ 3 $ implique \begin{equation}\label{cardA3}
A_{(j_{0},K_{j_{0}})}(P^{\theta}) \gg P^{\theta(\sum_{i\in I_{0}}\prod_{j=1}^{r}t_{j,a_{i,j}}-\sum_{i\in J(j_{0},K_{j_{0}})}\prod_{j\neq j_{0}}t_{j,a_{i,j}})-2^{\sum_{j=1}^{r}D_{j}}\kappa}.
\end{equation}   
On consid\`ere $ \mathcal{L} $ le sous-espace de $ \AA_{\CC}^{n(j_{0},K_{j_{0}})} $ o\`u \[ n(j_{0},K_{j_{0}})=\Card (L\setminus E(j_{0},K_{j_{0}}))=\sum_{i\in I_{0}}\prod_{j=1}^{r}t_{j,a_{i,j}}-\sum_{i\in J(j_{0},K_{j_{0}})}\prod_{j\neq j_{0}}t_{j,a_{i,j}} \] d\'efini par les $ \Card (E(j_{0},K_{j_{0}}))=\sum_{i\in J(j_{0},K_{j_{0}})}\prod_{j\neq j_{0}}t_{j,a_{i,j}} $ \'equations \[ \gamma_{(i,\l)}^{(j_{0},K_{j_{0}})}((\frac{1}{e_{m}}x_{m}^{(\hh)})_{\substack{(m,\hh)\in L\setminus E(j_{0},K_{j_{0}}) }})=0 .\] D'apr\`es la d\'emonstation de \cite[Th\'eor\`eme 3.1]{Br}, la majoration \eqref{cardA3} implique (en posant $ \kappa=K\theta $) : \begin{equation}\label{dimL3} \dim \mathcal{L}\geqslant n(j_{0},K_{j_{0}})-2^{\sum_{j=1}^{r}D_{j}}K. \end{equation} Consid\'erons \`a pr\'esent la sous-vari\'et\'e $ \mathcal{D} $ de $ \AA_{\CC}^{n(j_{0},K_{j_{0}})} $ d\'efinie par les $ \sum_{i\in I_{0}}(\prod_{j=1}^{r}t_{j,a_{i,j}}-1)-\sum_{i\in J(j_{0},K_{j_{0}})}\prod_{j\neq j_{0}}t_{j,a_{i,j}} $ \'equations \[ \forall (m,\hh)\in L\setminus E(j_{0},K_{j_{0}}) \; \;  x_{m}^{(\hh)}=x_{m}^{(1,...,1)}. \] Nous avons alors, d'apr\`es\;\eqref{dimL3} : \begin{align*} \dim \mathcal{D}\cap \mathcal{L} & \geqslant \dim \mathcal{L}-\sum_{i\in I_{0}}(\prod_{j=1}^{r}t_{j,a_{i,j}}-1)+\sum_{i\in J(j_{0},K_{j_{0}})}\prod_{j\neq j_{0}}t_{j,a_{i,j}} \\ & \geqslant \Card I_{0}-2^{\sum_{j=1}^{r}D_{j}}K. \end{align*}
On remarque par ailleurs que $ \mathcal{D}\cap \mathcal{L} $ est isomorphe \`a la vari\'et\'e : 
\[ \{ (x_{m})_{m\in I_{0}} \; | \; \forall i\in J(j_{0},K_{j_{0}}), \; \frac{\partial F_{\tt}}{\partial x_{i}}((x_{m})_{m\in I_{0}})=0 \} \] (en effet, par construction, si $ x_{m}^{(\hh)}=x_{m}^{(1,...,1)}=x_{m} $ pour tous $ \hh,m $, on a pour tout $ i,\l $, $ \gamma_{i,\l}^{(j_{0},K_{j_{0}})}(\frac{1}{e_{m}}x_{m}^{(\hh)})=N \frac{\partial F_{\tt}}{\partial x_{i}}(\xx) $ pour un certain entier $ N $).
En notant \[ V_{\tt,(j_{0},K_{j_{0}})}^{\ast}=\{ \xx\in \AA_{\CC}^{n+r}\; | \; \forall i\in J(j_{0},K_{j_{0}}), \; \frac{\partial F_{\tt}}{\partial x_{i}}(\xx)=0 \}, \] l'in\'egalit\'e ci-dessus implique alors : \[  \dim V_{\tt,(j_{0},K_{j_{0}})}^{\ast} \geqslant n+r-2^{\sum_{j=1}^{r}D_{j}}K. \] Par cons\'equent, nous fixerons dor\'enavant \begin{equation} K=(n+r-\dim V_{\tt,(j_{0},K_{j_{0}})}^{\ast}+\varepsilon)/2^{\sum_{j=1}^{r}D_{j}}, \end{equation} de sorte que la condition $ 3 $ n'est plus possible. \\
 
Dans tout ce qui va suivre nous choisirons $ P=\prod_{j=1}^{r}P_{j}^{d_{j}} $ et nous noterons $ \tilde{d}=(\sum_{j=1}^{r}d_{j})-1 $. Remarquons que pour tout $ i $, $ \sum_{j=1}^{r}(d_{j}-a_{i,j})\leqslant \tilde{d} $. D\'efinissons, pour tout $ i\in I_{0} $, la famille d'arcs majeurs\begin{equation}
\mathfrak{M}_{a,q}^{(i)}(\theta)=\{\alpha\in [0,1[ \; | \; 2|\alpha e_{i}q-a|<e_{i}P^{-1+\tilde{d}\theta}\}
\end{equation}
et posons \begin{equation}
\mathfrak{M}^{(i)}(\theta)=\bigcup_{0<q\leqslant P^{\tilde{d}\theta}}\bigcup_{0\leqslant a<e_{i}q}\mathfrak{M}_{a,q}^{(i)}(\theta),
\end{equation}
\begin{equation}
\mathfrak{M}(\theta)=\bigcup_{i\in I_{0}}\mathfrak{M}^{(i)}(\theta), \; \; \; \mathfrak{m}(\theta)=[0,1[\setminus\mathfrak{M}(\theta).
\end{equation}
Avec ces notations, \`a partir du lemme\;\ref{dilemme33} et en remarquant que pour tout $ i\in I_{0} $, $ \frac{\prod_{j=1}^{r}t_{j,a_{i,j}}}{2^{\sum_{j=1}^{r}D_{j}}}\leqslant \frac{1}{2^{r}} $ et que \begin{align*} \prod_{i=1}^{n+r}T_{i}& =\prod_{i=1}^{n+r}\left(\frac{1}{e_{i}}\prod_{j=1}^{r}P_{j}^{a_{i,j}}\right)\\ & =\left(\prod_{i=1}^{n+r}e_{i}\right)^{-1}\left(\prod_{j=1}^{r}P_{j}^{\sum_{i=1}^{n+r}a_{i,j}}\right)\\ &=
\left(\prod_{i=1}^{n+r}e_{i}\right)^{-1}\left(\prod_{j=1}^{r}P_{j}^{n_{j}}\right),\end{align*} nous obtenons la proposition\;\ref{dilemme43}.

\subsection{M\'ethode du cercle}
\subsubsection{Les arcs mineurs }
Posons, pour tout $ j\in \{1,...,r\}  $, $ P_{j}=P_{r}^{b_{j}} $. On consid\`ere un r\'eel $ \delta>0 $ arbitrairement petit, et on suppose que $ \theta\in [0,1] $ et $ K $ v\'erifient :  \begin{equation}\label{thetacond13}
K-2\tilde{d}>\left( 2\delta+\frac{\sum_{j=1}^{r}b_{j}}{\sum_{j=1}^{r}b_{j}d_{j}}\right)\theta^{-1},
\end{equation}
\begin{equation}\label{thetacond23}
K>(2\delta+1)(\sum_{j=1}^{r}b_{j}d_{j}),
\end{equation}
\begin{equation}\label{thetacond33}
1>(\sum_{j=1}^{r}b_{j}d_{j})(5\tilde{d}\theta+\delta).
\end{equation}
\begin{rem}
Les conditions \eqref{thetacond13} et \eqref{thetacond33} impliquent en particulier que $ K>(5\sum_{j=1}^{r}b_{j}+2)\tilde{d} $, ce que nous supposerons dor\'enavant. 
\end{rem}
\begin{lemma}\label{arcmin3}
On a la majoration : 
\[ \int_{\mathfrak{m}(\theta)}|S_{\ee}(\alpha)|d\alpha\ll  e_{0}\left(\prod_{i=1}^{n+r}e_{i}\right)^{-1+\frac{1}{2^{r}}}\left(\prod_{j=1}^{r}P_{j}^{n_{j}}\right)P^{-1-\delta}. \]
\end{lemma}
\begin{proof}
On consid\`ere une suite $ (\theta_{i})_{i\in\{0,...,T\}} $ telle que \[ \theta=\theta_{0}<\theta_{1}<...<\theta_{T}, \] \begin{equation}\label{thetaT3}
\theta_{T}\leqslant \frac{1}{\sum_{j=1}^{r}b_{j}d_{j}}, \; \; K\theta_{T}>2\delta+1+\frac{\sum_{j=1}^{r}b_{j}}{\sum_{j=1}^{r}b_{j}d_{j}}, \end{equation} \begin{equation} \forall i\in \{0,...,T-1\}, \; \; 2\tilde{d}(\theta_{i+1}-\theta_{i})<\delta/2. \end{equation} On suppose de plus que $ T $ est tel que $ T\ll P^{\frac{\delta}{2}} $. D'apr\`es la proposition\;\ref{dilemme43}, 
\begin{align*}
\int_{\alpha\notin\mathfrak{M}(\theta_{T})}|S_{\ee}(\alpha)|d\alpha & \ll \left(\prod_{i=1}^{n+r}e_{i}\right)^{-1+\frac{1}{2^{r}}} \left(\prod_{j=1}^{r}P_{j}^{n_{j}}\right)P^{\frac{\sum_{j=1}^{r}b_{j}}{\sum_{j=1}^{r}b_{j}d_{j}}-K\theta_{T}+\varepsilon} \\ & \ll  \left(\prod_{i=1}^{n+r}e_{i}\right)^{-1+\frac{1}{2^{r}}}\left(\prod_{j=1}^{r}P_{j}^{n_{j}}\right)P^{-1-\delta}
\end{align*}
(d'apr\`es la condition \eqref{thetaT3}). Remarquons par ailleurs que pour tout $ \theta $ et tout $ i\in I_{0} $ : \begin{align*}
\Vol(\mathfrak{M}^{(i)}(\theta))&\ll \sum_{0<q\leqslant P^{\tilde{d}\theta}}\sum_{0\leqslant a <e_{i}q}\frac{1}{q}P^{-1+\tilde{d}\theta} \\ & \ll e_{0}P^{-1+2\tilde{d}\theta}. 
\end{align*}
On a donc, pour tout $ t\in \{0,...,T-1\} $, en utilisant la condition\;\eqref{thetacond13} : \begin{multline*} \int_{\alpha in\mathfrak{M}(\theta_{t+1})\setminus\mathfrak{M}(\theta_{t})}|S_{\ee}(\alpha)|d\alpha \\ \ll e_{0}\left(\prod_{i=1}^{n+r}e_{i}\right)^{-1+\frac{1}{2^{r}}}\left(\prod_{j=1}^{r}P_{j}^{n_{j}}\right)P^{\frac{\sum_{j=1}^{r}b_{j}}{\sum_{j=1}^{r}b_{j}d_{j}}-K\theta_{t}+2\tilde{d}\theta_{t+1}-1+\varepsilon} \\  \ll e_{0}\left(\prod_{i=1}^{n+r}e_{i}\right)^{-1+\frac{1}{2^{r}}}  \left(\prod_{j=1}^{r}P_{j}^{n_{j}}\right)P^{\frac{\sum_{j=1}^{r}b_{j}}{\sum_{j=1}^{r}b_{j}d_{j}}-(K-2\tilde{d})\theta_{t}-1+2\tilde{d}(\theta_{t+1}-\theta_{t})+\varepsilon}\\  \ll e_{0}\left(\prod_{i=1}^{n+r}e_{i}\right)^{-1+\frac{1}{2^{r}}}     \left(\prod_{j=1}^{r}P_{j}^{n_{j}}\right)P^{-1-\frac{3}{2}\delta}, \end{multline*}
et on obtient le r\'esultat souhait\'e en sommant sur tous les $ t\in \{0,...,T-1\} $.
\end{proof}

\subsubsection{Les arcs majeurs}

Posons \begin{equation}
e_{0}=\max_{i\in I_{0}} e_{i},
\end{equation}

et introduisons \`a pr\'esent la nouvelle famille d'arcs majeurs : 
\begin{equation}
\mathfrak{M}_{a,q}'(\theta)=\{\alpha\in [0,1[ \; | \; 2|\alpha q-a|<qP^{-1+\tilde{d}\theta}\}
\end{equation}
et posons \begin{equation}
\mathfrak{M}'(\theta)=\bigcup_{0<q\leqslant e_{0}P^{\tilde{d}\theta}}\bigcup_{\substack{0\leqslant a<q\\ \PGCD(a,q)=1}}\mathfrak{M}_{a,q}'(\theta),
\end{equation}
On remarque que, pour tout $ i\in I_{0} $, \[  \mathfrak{M}^{(i)}(\theta)\subset \mathfrak{M}'(\theta). \]
En effet, si $ \alpha\in \mathfrak{M}_{a,q}^{(i)}(\theta) $, alors \[ 2|\alpha q e_{i}-a|\leqslant e_{i}P^{-1+\tilde{d}\theta} \]  et donc en posant $ q'=(qe_{i})/\PGCD(qe_{i},a) $ et $ a'=a/\PGCD(qe_{i},a) $, on trouve \[ 2|\alpha q' -a'|\leqslant \frac{e_{i}}{\PGCD(qe_{i},a) }P^{-1+\tilde{d}\theta}\leqslant q'P^{-1+\tilde{d}\theta} \] et on a de plus $ q'\leqslant  e_{0}P^{\tilde{d}\theta} $, $ 0\leqslant a'<q' $ et $ \PGCD(a',q')=1 $, donc $ \alpha\in \mathfrak{M}_{a',q'}'(\theta) $. Les arcs majeurs $ \mathfrak{M}_{a,q}'(\theta) $ v\'erifient par ailleurs le lemme ci-dessous : 
\begin{lemma}\label{disjoints3}
Pour tout $ \ee\in \NN^{n+r} $ tel que $ e_{0}^{2}<P^{1-3\tilde{d}\theta} $, les intervalles $ \mathfrak{M}_{a,q}'(\theta) $ sont disjoints deux \`a deux. 
\end{lemma}
\begin{proof}
Supposons qu'il existe $ \alpha\in \mathfrak{M}_{a,q}'(\theta)\cap\mathfrak{M}_{a',q'}'(\theta) $ avec $ q,q'\leqslant e_{0}P^{\tilde{d}\theta} $, $ 0\leqslant a<q $, $ 0\leqslant a'<q' $, $ \PGCD(a,q)=1 $, $ \PGCD(a',q')=1 $ et $ (a,q)\neq(a',q') $. On a alors \[ \frac{1}{qq'}\leqslant \frac{|aq'-a'q|}{qq'}=\left|\frac{a}{q}-\frac{a'}{q'}\right|\leqslant \left|\alpha-\frac{a}{q}\right|+\left|\alpha-\frac{a'}{q'}\right|\leqslant P^{-1+\tilde{d}\theta} \] et donc \[ 1\leqslant qq'P^{-1+\tilde{d}\theta}<e_{0}^{2}P^{-1+3\tilde{d}\theta} \] d'o\`u le r\'esultat.  
\end{proof}
En combinant les r\'esultats des lemmes\;\ref{arcmin3} et\;\ref{disjoints3}, on obtient l'estimation : \begin{multline}\label{formulesomme3}
N_{\ee}(P_{1},...,P_{r})=\sum_{1\leqslant q \leqslant e_{0}P^{\tilde{d}\theta}}\sum_{\substack{0\leqslant a<q\\ \PGCD(a,q)=1}}\int_{\mathfrak{M}_{a,q}'(\theta)}S_{\ee}(\alpha)d\alpha \\ +O\left( e_{0}\left(\prod_{i=1}^{n+r}e_{i}\right)^{-1+\frac{1}{2^{r}}}\left(\prod_{j=1}^{r}P_{j}^{n_{j}-d_{j}-\delta}\right)\right)
\end{multline}
\'Etant donn\'e un \'el\'ement $ \alpha\in \mathfrak{M}_{a,q}'(\theta) $, nous poserons $ \alpha=\frac{a}{q}+\beta $ avec $ |\beta|\leqslant \frac{1}{2}P^{-1+\tilde{d}\theta} $, et nous noterons : \begin{equation}
S_{a,q,\ee}=\sum_{\bb\in (\ZZ/q\ZZ)^{n+r}}e\left(\frac{a}{q}F(\ee.\bb)\right),
\end{equation}
\begin{equation}
I(\beta)=\int_{\substack{\uu\in\RR^{n+r} \\\forall i, \; |u_{i}|\leqslant |\uu^{E(i)}| \\ \forall j\in \{1,...,r\}, \; |\uu^{E(n+j)}|\leqslant 1}}e\left(\beta F(\uu)\right)d\uu. 
\end{equation}
On \'etablit alors le r\'esultat suivant : 

\begin{lemma}\label{separation3} Si $ \alpha $ appartient \`a $ \mathfrak{M}_{a,q}'(\theta) $, on a alors : \begin{multline*} S_{\ee}(\alpha)=\left(\prod_{i=1}^{n+r}e_{i}\right)^{-1}q^{-(n+r)}S_{a,q,\ee}I(P\beta)\left(\prod_{j=1}^{r}P_{j}^{n_{j}}\right)
\\ +O\left(e_{0}\left(\prod_{i=1}^{n+r}e_{i}\right)^{-1}\left(\prod_{j=1}^{r}P_{j}^{n_{j}}\right)qP_{r}^{-1}P^{\tilde{d}\theta}\right).
\end{multline*}
\end{lemma}
\begin{proof}
Remarquons avant tout que lorsque $ e_{0}q>P_{r} $, l'\'egalit\'e du lemme est triviale car le terme d'erreur est alors dominant. En effet on a dans ce cas :
\begin{align*} |S_{\ee}(\alpha)| & \ll \prod_{i=1}^{n+r}\frac{1}{e_{i}}\prod_{j=1}^{r}P_{j}^{a_{i,j}}\\ & =\left(\prod_{i=1}^{n+r}e_{i}\right)^{-1}\left(\prod_{j=1}^{r}P_{j}^{n_{j}}\right) \\ & \ll \left(\prod_{i=1}^{n+r}e_{i}\right)^{-1}\left(\prod_{j=1}^{r}P_{j}^{n_{j}}\right)P_{r}^{-1}e_{0}P^{\tilde{d}\theta }q
\end{align*}
et en utilisant les estimations triviales $ |S_{a,q,\ee}|\ll q^{n+r} $ et $ |I(P\beta)|\ll 1 $ :\begin{align*}
\left(\prod_{i=1}^{n+r}e_{i}\right)^{-1}q^{-(n+r)}|S_{a,q,\ee}||I(P\beta)|\left(\prod_{j=1}^{r}P_{j}^{n_{j}}\right) & \ll \left(\prod_{i=1}^{n+r}e_{i}\right)^{-1}\left(\prod_{j=1}^{r}P_{j}^{n_{j}}\right)  \\ & \ll \left(\prod_{i=1}^{n+r}e_{i}\right)^{-1}\left(\prod_{j=1}^{r}P_{j}^{n_{j}}\right)P_{r}^{-1}qe_{0}P^{\tilde{d}\theta}.
\end{align*}
d'o\`u le r\'esultat. Nous supposerons donc $ e_{0}q\leqslant P_{r} $.\\

Remarquons \`a pr\'esent que, d'apr\`es\;\eqref{Salpha3}: \begin{multline}\label{formulesep3} S_{\ee}(\alpha)=\sum_{\bb\in (\ZZ/q\ZZ)^{n+r}}e\left(\frac{a}{q}F(\ee.\bb)\right)\sum_{\substack{\xx\equiv \bb (q) \\ \left\lfloor|(\ee.\xx)^{E(n+j)}|\right\rfloor\leqslant P_{j} \\ |x_{i}|\leqslant \frac{1}{e_{i}}\prod_{j=1}^{r}(\left\lfloor|(\ee.\xx)^{E(n+j)}|\right\rfloor+1)^{a_{i,j}} }}e(\beta F(\ee.\xx))\\ =\sum_{\bb\in (\ZZ/q\ZZ)^{n+r}}e\left(\frac{a}{q}F(\ee.\bb)\right)\tilde{S}(\bb)+O\left(e_{0}\left(\prod_{i=1}^{n+r}e_{i}\right)^{-1}\left( \prod_{j=1}^{r}P_{j}^{n_{j}}\right)P_{r}^{-1}\right) \end{multline} o\`u \[ \tilde{S}(\bb)=\sum_{\substack{\xx\equiv \bb (q) \\ |(\ee.\xx)^{E(n+j)}|\leqslant P_{j} \\ |x_{i}|\leqslant \frac{1}{e_{i}}\prod_{j=1}^{r}(|(\ee.\xx)^{E(n+j)}|+1)^{a_{i,j}} }}e(\beta F(\ee.\xx)). \]
Soient $ \xx' $ et $ \xx'' $ deux \'el\'ement de $ \RR^{n+r} $ tels que \[|\xx'-\xx''|\leqslant 2\] et, pour tout $ i $, \[  \begin{array}{l}
|qx_{i}'+b_{i}|\leqslant \frac{1}{e_{i}}\prod_{j=1}^{r}(|(\ee.(q\xx'+\bb))^{E(n+j)}|+1)^{a_{i,j}}, \\ |qx_{i}''+b_{i}|\leqslant \frac{1}{e_{i}}\prod_{j=1}^{r}(|(\ee.(q\xx''+\bb))^{E(n+j)}|+1)^{a_{i,j}},
\end{array} \] \[ \forall j \in \{1,...,r\},\; \; \;|(\ee.(q\xx'+\bb))^{E(n+j)}|\leqslant P_{j}, \; \; \; |(\ee.(q\xx''+\bb))^{E(n+j)}|\leqslant P_{j}. \]
On observe alors que \[\left| F(\ee.(q\xx'+\bb))-F(\ee.(q\xx''+\bb))\right|\ll qe_{0}P_{1}^{d_{1}}...P_{r-1}^{d_{r-1}}P_{r}^{d_{r}-1}. \] Par cons\'equent,  \begin{multline*}
\tilde{S}(\bb)=\int_{\substack{q\tilde{\uu}\in \RR^{n+r} \\ |qe_{i}\tilde{u}_{i}|\leqslant |\prod_{j=1}^{r}(\ee.(q\tilde{\uu}))^{a_{i,j}E(n+j)}| \\ |(\ee.(q\tilde{\uu}))^{E(n+j)}|\leqslant P_{j}}}e(\beta F(\ee.(q\tilde{\uu})))d\tilde{\uu} \\ + O\left(\underbrace{|\beta|}_{\leqslant P^{-1+\tilde{d}\theta}}\left(\prod_{i=1}^{n+r}e_{i}\right)^{-1}q^{-(n+r)}\left(\prod_{j=1}^{r}P_{j}^{n_{j}}\right)\left(\prod_{j=1}^{r}P_{j}^{d_{j}}\right)P_{r}^{-1}qe_{0}\right) \\ + O\left(e_{0}\left(\prod_{i=1}^{n+r}e_{i}\right)^{-1}\left(\prod_{j=1}^{r}P_{j}^{n_{j}}\right)q^{-(n+r)+1}P_{r}^{-1}\right)
\end{multline*}
(le deuxi\`eme terme d'erreur correspondant aux points rencontrant le bord du domaine de sommation).\\

En effectuant le changement de variables \[ \forall i\in \{1,...,n+r\}, \; \; q\tilde{u}_{i}=e_{i}^{-1}\left(\prod_{j=1}^{r}P_{j}^{a_{i,j}}\right)u_{i}, \] on trouve alors  \begin{multline*}
\tilde{S}(\bb)=q^{-(n+r)}\left(\prod_{i=1}^{n+r}e_{i}\right)^{-1}\left(\prod_{j=1}^{r}P_{j}^{n_{j}}\right)\int_{\substack{\uu\in \RR^{n+r} \\ |\tilde{u}_{i}|\leqslant |\uu^{E(i)}| \\ |\uu^{E(n+j)}|\leqslant 1}}e(\beta F(\uu))d\tilde{\uu} \\ + O\left(e_{0}q^{-(n+r)}\left(\prod_{i=1}^{n+r}e_{i}\right)^{-1}\left(\prod_{j=1}^{r}P_{j}^{n_{j}}\right)qP_{r}^{-1}P^{\tilde{d}\theta}\right) \\ + O\left(e_{0}\left(\prod_{i=1}^{n+r}e_{i}\right)^{-1}\left(\prod_{j=1}^{r}P_{j}^{n_{j}}\right)q^{-(n+r)+1}P_{r}^{-1}\right).
\end{multline*}
En rempla\c{c}ant $ \tilde{S}(\bb) $ par cette expression dans \eqref{formulesep3}, nous obtenons l'\'egalit\'e du lemme. 
\end{proof}
Posons \`a pr\'esent \begin{equation}
\mathfrak{S}_{\ee}(Q)=\sum_{1\leqslant q\leqslant Q}q^{-(n+r)}\sum_{\substack{0\leqslant a<q\\ \PGCD(a,q)=1}}S_{a,q,\ee},
\end{equation}
\begin{equation}
J_{\sigma}(\phi)=\int_{|\beta|\leqslant \phi}I(\beta)d\beta. 
\end{equation}
o\`u $ \sigma $ d\'esigne le c\^one maximal consid\'er\'e auquel la fonction $ h_{\ee,V} $ est associ\'ee (cf. section $ 2.2 $). Avec ces notations, en utilisant le lemme pr\'ec\'edent dans la formule \eqref{formulesomme3}, et en remarquant que : \[ \int_{|\beta|\leqslant \frac{1}{2}P^{-1+\tilde{d}\theta}}I(P\beta)d\beta=P^{-1}\int_{|\beta|\leqslant \frac{1}{2}P^{\tilde{d}\theta}}I(\beta)d\beta=P^{-1}J\left(\frac{1}{2}P^{\tilde{d}\theta}\right), \] on trouve : \begin{multline*}
N_{\ee}(P_{1},...,P_{r})=\mathfrak{S}_{\ee}(e_{0}P^{\tilde{d}\theta})J_{\sigma}(\frac{1}{2}P^{\tilde{d}\theta})\left(\prod_{i=1}^{n+r}e_{i}\right)^{-1}\left(\prod_{j=1}^{r}P_{j}^{n_{j}-d_{j}}\right)
\\ +O\left(e_{0}^{2}\left(\prod_{i=1}^{n+r}e_{i}\right)^{-1}\left(\prod_{j=1}^{r}P_{j}^{n_{j}}\right)P_{r}^{-1}P^{2\tilde{d}\theta}\Vol(\mathfrak{M}'(\theta))\right)
\\ + O\left(e_{0}\left(\prod_{i=1}^{n+r}e_{i}\right)^{-1+\frac{1}{2^{r}}}\left(\prod_{j=1}^{r}P_{j}^{n_{j}-d_{j}-\delta}\right)\right).
\end{multline*}
Or, \[ \Vol(\mathfrak{M}'(\theta))\ll \sum_{1\leqslant q \leqslant e_{0}P^{\tilde{d}\theta}}\sum_{\substack{0\leqslant a<q\\ \PGCD(a,q)=1}}P^{-1+\tilde{d}\theta} \ll e_{0}^{2}P^{-1+3\tilde{d}\theta}. \] On a donc : \begin{multline*}
e_{0}^{2}\left(\prod_{i=1}^{n+r}e_{i}\right)^{-1}\left(\prod_{j=1}^{r}P_{j}^{n_{j}}\right)P_{r}^{-1}P^{2\tilde{d}\theta}\Vol(\mathfrak{M}'(\theta)) \\ \ll e_{0}^{4}\left(\prod_{i=1}^{n+r}e_{i}\right)^{-1}\left(\prod_{j=1}^{r}P_{j}^{n_{j}-d_{j}}\right)P_{r}^{-1}P^{5\tilde{d}\theta}.
\end{multline*}
En remarquant que $ P_{r}^{-1}P^{5\tilde{d}\theta}\ll P^{-\delta} $ d'apr\`es \eqref{thetacond33}, on conclut que : \begin{multline}\label{presquefin3}
N_{\ee}(P_{1},...,P_{r})=\left(\prod_{i=1}^{n+r}e_{i}\right)^{-1}\mathfrak{S}_{\ee}(e_{0}P^{\tilde{d}\theta})J_{\sigma}(\frac{1}{
2}P^{\tilde{d}\theta})\left(\prod_{j=1}^{r}P_{j}^{n_{j}-d_{j}}\right)
\\ +O\left(e_{0}^{4}\left(\prod_{i=1}^{n+r}e_{i}\right)^{-1}\left(\prod_{j=1}^{r}P_{j}^{n_{j}-d_{j}-\delta}\right)\right).
\end{multline}
Notons \`a pr\'esent : 
\begin{equation}
\mathfrak{S}_{\ee}=\sum_{q=1}^{\infty}q^{-(n+r)}\sum_{\substack{0\leqslant a<q\\ \PGCD(a,q)=1}}S_{a,q,\ee},
\end{equation}
\begin{equation}
J_{\sigma}=\int_{\RR}I(\beta)d\beta. 
\end{equation}
Nous allons chercher \`a remplacer $ \mathfrak{S}(e_{0}P^{\tilde{d}\theta}) $ par $ \mathfrak{S}_{\ee} $ et $ J_{\sigma}(\frac{1}{
2}P^{\tilde{d}\theta}) $ par $ J_{\sigma} $ dans \eqref{presquefin3}. Pour cela nous utiliserons le lemme ci-dessous : \begin{lemma}
Pour tous $ a,q,\ee $, on a l'estimation suivante : \[ |S_{a,q,\ee}|\ll \max\left\{e_{0}^{4+\delta}, e_{0}^{2}\left(\prod_{i=1}^{n+r}e_{i}\right)^{\frac{1}{2^{r}}}\right\}q^{n+r-2-\delta}. \]
\end{lemma}
\begin{proof}
Consid\'erons $ \alpha=\frac{a}{q} $  pour un certain et soit $ \theta_{0}\in [0,1] $ v\'erifiant les conditions \eqref{thetacond13}, \eqref{thetacond23}, \eqref{thetacond33}, et $ \PGCD(a,q)=1 $. On choisit par ailleurs $ P $ tel que $ q=e_{0}P^{\tilde{d}\theta_{0}} $. On suppose de plus que $ P $ et $ \theta_{0} $ v\'erifient l'hypoth\`ese $ e_{0}^{2}<P^{1-2\tilde{d}\theta_{0}} $ (si ce n'est pas le cas la majoration du lemme devient triviale car alors $ e_{0}^{2}>q $ et donc $ e_{0}^{4+\delta}q^{n+r-2-\delta}>q^{n+r} $). On pose par ailleurs $ \theta_{0}'=\theta_{0}-\nu $, pour $ \nu>0 $ fix\'e arbitrairement petit. On a alors que $ \alpha\notin\mathfrak{M}(\theta_{0}') $. En effet supposons qu'il existe $ (a',q') $ tels que $ q'\leqslant P^{\tilde{d}\theta_{0}'}<q $, $ 0\leqslant a'<q' $ et $ \alpha\in \mathfrak{M}_{a',q'}^{(i)}(\theta_{0}')  $. Remarquons que si l'on a $ aq'e_{i}=qa' $, alors, puisque $ \PGCD(a,q)=1 $, $ a'=au $ avec $ a\neq 0 $ (sinon on a $ q=1 $ ce qui contredit $ 0<q'<q $ ) et donc $ 0<q'e_{i}=qu $ donc $ q'e_{i}\geqslant q $, ce qui est absurde puisque $ e_{i}q'<q $. On a donc $ aq'e_{i}\neq qa' $ et ainsi : \[ 1\leqslant |aq'e_{i}-qa'|=q|\alpha e_{i}q'-a'|<qe_{i}P^{-1+\tilde{d}\theta_{0}}\leqslant e_{0}^{2}P^{-1+2\tilde{d}\theta_{0}}<1
\] ce qui est absurde. \\

Par cons\'equent, d'apr\`es la proposition\;\ref{dilemme43} : \[|S_{\ee}(\alpha)|\ll \left(\prod_{i=1}^{n+r}e_{i}\right)^{-1+\frac{1}{2^{r}}}\left(\prod_{j=1}^{r}P_{j}^{n_{j}}\right)P^{\frac{\sum_{j=1}^{r}b_{j}}{\sum_{j=1}^{r}b_{j}d_{j}}-K\theta_{0}+\varepsilon}.\]
Le lemme\;\ref{separation3} donne donc : \begin{multline*}\left(\prod_{i=1}^{n+r}e_{i}\right)^{-1}\left(\prod_{j=1}^{r}P_{j}^{n_{j}}\right)(q^{-(n+r)}S_{a,q,\ee}I(0)+O(e_{0}qP_{r}^{-1}P^{\tilde{d}\theta_{0}})) 
\\ \ll \left(\prod_{i=1}^{n+r}e_{i}\right)^{-1+\frac{1}{2^{r}}}\left(\prod_{j=1}^{r}P_{j}^{n_{j}}\right)P^{\frac{\sum_{j=1}^{r}b_{j}}{\sum_{j=1}^{r}b_{j}d_{j}}-K\theta_{0}+\varepsilon} \end{multline*}
donc (puisque $ I(0)\asymp 1 $) : \[|S_{a,q,\ee}|\ll e_{0}q^{n+r+1}P_{r}^{-1}P^{\tilde{d}\theta_{0}}+  \left(\prod_{i=1}^{n+r}e_{i}\right)^{\frac{1}{2^{r}}}q^{n+r}P^{\frac{\sum_{j=1}^{r}b_{j}}{\sum_{j=1}^{r}b_{j}d_{j}}-K\theta_{0}+\varepsilon}. \] Or on a d'une part \[e_{0}q^{n+r+1}P_{r}^{-1}P^{\tilde{d}\theta_{0}} \ll e_{0}^{4+\delta}q^{n+r+1}q^{-(3+\delta)} \] (car $ P_{r}^{-1}<P^{-4\tilde{d}\theta_{0}} $, par la condition\;\eqref{thetacond33}), et d'autre part \[ K\theta_{0}\geqslant \frac{\sum_{j=1}^{r}b_{j}}{\sum_{j=1}^{r}b_{j}d_{j}}+2\tilde{\theta_{0}}+\delta \] (d'apr\`es la condition \eqref{thetacond13}), et on obtient donc \[  |S_{a,q,\ee}|\ll \max\left\{e_{0}^{4+\delta}, e_{0}^{2}\left(\prod_{i=1}^{n+r}e_{i}\right)^{\frac{1}{2^{r}}}\right\}q^{n+r-2-\delta}. \]

\end{proof}
Nous pouvons \`a pr\'esent d\'emontrer le lemme suivant : 

\begin{lemma}\label{S3}
La s\'erie $ \mathfrak{S}_{\ee} $ est absolument convergente et on a de plus \[ |\mathfrak{S}_{\ee}(Q)-\mathfrak{S}_{\ee}|\ll \max\left\{e_{0}^{4+\delta}, e_{0}^{2}\left(\prod_{i=1}^{n+r}e_{i}\right)^{\frac{1}{2^{r}}}\right\}Q^{-\delta} \] pour tout $ Q>e_{0} $.
\end{lemma}
\begin{proof}
Nous avons montr\'e pr\'ec\'edemment que \[ |S_{a,q,\ee}|\ll \max\left\{e_{0}^{4+\delta}, e_{0}^{2}\left(\prod_{i=1}^{n+r}e_{i}\right)^{\frac{1}{2^{r}}}\right\}q^{n+r-2-\delta}. \] Par cons\'equent, \begin{align*} |\mathfrak{S}_{\ee}(Q)-\mathfrak{S}_{\ee}| & \ll \sum_{q>Q}q^{-(n+r)}\sum_{\substack{0\leqslant a<q\\ \PGCD(a,q)=1}}|S_{a,q,\ee}| \\ & \ll \max\left\{e_{0}^{4+\delta}, e_{0}^{2}\left(\prod_{i=1}^{n+r}e_{i}\right)^{\frac{1}{2^{r}}}\right\}\sum_{q>Q}\sum_{\substack{0\leqslant a<q\\ \PGCD(a,q)=1}}q^{-2-\delta} \\ & \ll \max\left\{e_{0}^{4+\delta}, e_{0}^{2}\left(\prod_{i=1}^{n+r}e_{i}\right)^{\frac{1}{2^{r}}}\right\}Q^{-\delta}. \end{align*}

\end{proof}

\begin{lemma}\label{J3}
L'int\'egrale $ J_{\sigma} $ est absolument convergente et on a de plus \[ |J_{\sigma}(\phi)-J_{\sigma}|\ll \phi^{-1} \] pour tout $ \phi $ assez grand.
\end{lemma}
\begin{proof}
On consid\`ere $ \beta\in \RR $ fix\'e quelconque et on choisit $ \theta'\in [0,1] $ v\'erifiant les conditions \eqref{thetacond13}, \eqref{thetacond23}, \eqref{thetacond33}. On prend $ P $ tel que $ 2|\beta|=P^{\tilde{d}\theta'} $. On choisit par ailleurs $ e_{1}=e_{2}=...=e_{n+r}=1 $ (et donc $ e_{0}=1 $). On a alors que $ P^{-1}\beta\in \mathfrak{M}_{0,1}(\theta') $ et donc d'apr\`es le lemme\;\ref{separation3} appliqu\'e \`a $ a=0 $ et $ q=1 $ : \[ S_{\ee}(P^{-1}\beta)=\left(\prod_{j=1}^{r}P_{j}^{n_{j}}\right)I(\beta)+O\left(\left(\prod_{j=1}^{r}P_{j}^{n_{j}}\right)P_{r}^{-1}P^{2\tilde{d}\theta'}\right).
\]

D'autre part, puisque les $ \mathfrak{M}_{a,q}(\theta') $ sont disjoints, d'apr\`es la proposition\;\ref{dilemme43}, on a : \[ |S_{\ee}(P^{-1}\beta)| \ll \left(\prod_{j=1}^{r}P_{j}^{n_{j}}\right)P^{\frac{\sum_{j=1}^{r}b_{j}}{\sum_{j=1}^{r}b_{j}d_{j}}-K\theta'+\varepsilon}. \] On trouve donc (en utilisant les conditions \eqref{thetacond13} et \eqref{thetacond33}) : \begin{align*}
I(\beta) & \ll P^{2\tilde{d}\theta'}P_{r}^{-1}+P^{\frac{\sum_{j=1}^{r}b_{j}}{\sum_{j=1}^{r}b_{j}d_{j}}-K\theta'+\varepsilon} \\ & \ll P^{2\tilde{d}\theta'-\frac{1}{\sum_{j=1}^{r}b_{j}d_{j}}} +P^{\frac{\sum_{j=1}^{r}b_{j}}{\sum_{j=1}^{r}b_{j}d_{j}}-K\theta'+\varepsilon} \\ & \ll P^{-3\tilde{d}\theta'-\delta} +P^{-2\tilde{d}\theta'} \ll |\beta|^{-2}.
\end{align*}
Par cons\'equent, \[ |J_{\sigma}(\phi)-J_{\sigma}|\ll \int_{|\beta|>\phi}|I(\beta)|d\beta \ll \phi^{-1}. \] 
\end{proof}
Ces deux derniers lemmes permettent finalement d'\'etablir la formule ci-dessous valable lorsque $ e_{0}^{2}<P^{1-3\tilde{d}\theta} $ et $ K>\max\{(5\sum_{j=1}^{r}b_{j}+2)\tilde{d},(2\delta+1)\sum_{j=1}^{r}b_{j}d_{j}\} $ : 
\begin{multline}\label{formulefin3} N_{\ee}(P_{1},...,P_{r})=C_{\sigma,\ee}\left(\prod_{j=1}^{r}P_{j}^{n_{j}-d_{j}}\right)
\\ +O\left(\max\left\{e_{0}^{4+\delta}, e_{0}^{2}\left(\prod_{i=1}^{n+r}e_{i}\right)^{\frac{1}{2^{r}}}\right\}\left(\prod_{i=1}^{n+r}e_{i}\right)^{-1}\left(\prod_{j=1}^{r}P_{j}^{n_{j}-d_{j}-\delta}\right)\right). \end{multline}
o\`u l'on a not\'e \begin{equation} C_{\sigma,\ee}=\left(\prod_{i=1}^{n+r}e_{i}\right)^{-1}\mathfrak{S}_{\ee}J_{\sigma}\end{equation}
\begin{rem}\label{remsigmae3}
D'apr\`es le lemme\;\ref{S3}, on a \[ |\mathfrak{S}_{\ee}(e_{0})-\mathfrak{S}_{\ee}|\ll \max\left\{e_{0}^{4+\delta}, e_{0}^{2}\left(\prod_{i=1}^{n+r}e_{i}\right)^{\frac{1}{2^{r}}}\right\}  \]  donc en utilisant la majoration triviale $  \mathfrak{S}_{\ee}(e_{0})\ll e_{0}^{2} $, on en d\'eduit $ \mathfrak{S}_{\ee}\ll \max\left\{e_{0}^{4+\delta}, e_{0}^{2}\left(\prod_{i=1}^{n+r}e_{i}\right)^{\frac{1}{2^{r}}}\right\} $. 
On a donc \[ |C_{\sigma,\ee}|\ll \max\left\{e_{0}^{4+\delta}, e_{0}^{2}\left(\prod_{i=1}^{n+r}e_{i}\right)^{\frac{1}{2^{r}}}\right\}\left(\prod_{i=1}^{n+r}e_{i}\right)^{-1}. \]
\end{rem}

Remarquons \`a pr\'esent que, pour un choix appropri\'e de $ \theta $, on peut retirer la condition $ e_{0}^{2}<P^{1-3\tilde{d}\theta} $. En effet, supposons dans un premier temps que $ e_{0}>P^{\frac{1}{4}} $. On a l'estimation triviale : \begin{align*}
 N_{\ee}(P_{1},...,P_{r}) & \leqslant \prod_{i=1}^{n+r}\frac{1}{e_{i}}\prod_{j=1}^{r}P_{j}^{a_{i,j}} \\ & =\left(\prod_{i=1}^{n+r}e_{i}\right)^{-1}\left(\prod_{j=1}^{r}P_{j}^{n_{j}}\right)\\ & \ll e_{0}^{4+\delta}\left(\prod_{i=1}^{n+r}e_{i}\right)^{-1}\left(\prod_{j=1}^{r}P_{j}^{n_{j}-d_{j}-\delta}\right),
\end{align*} 
et par ailleurs, par la remarque ci-dessus \[ C_{\sigma,\ee}\left(\prod_{j=1}^{r}P_{j}^{n_{j}-d_{j}}\right)\ll\max\left\{e_{0}^{4+\delta}, e_{0}^{2}\left(\prod_{i=1}^{n+r}e_{i}\right)^{\frac{1}{2^{r}}}\right\} \left(\prod_{i=1}^{n+r}e_{i}\right)^{-1}\left(\prod_{j=1}^{r}P_{j}^{n_{j}-d_{j}-\delta}\right), \] (quitte \`a prendre $ \delta>0 $ plus petit). La formule\;\eqref{formulefin3} est donc trivialement v\'erifi\'ee lorsque $ e_{0}>P^{\frac{1}{4}} $. Si l'on suppose que $ e_{0}\leqslant P^{\frac{1}{4}} $ la condition $ e_{0}^{2}<P^{1-3\tilde{d}\theta} $ est v\'erifi\'ee lorsque $ 1<P^{\frac{1}{2}-3\tilde{d}\theta} $ et donc lorsque $ \theta<\frac{1}{6\tilde{d}} $, ce que l'on peut supposer. En observant que, puisque $ r\geqslant 2 $, \[ \max\left\{e_{0}^{4+\delta}, e_{0}^{2}\left(\prod_{i=1}^{n+r}e_{i}\right)^{\frac{1}{2^{r}}}\right\}\ll e_{0}^{4+\delta}\left(\prod_{i=1}^{n+r}e_{i}\right)^{-1}+\left(\prod_{i=1}^{n+r}e_{i}\right)^{-\frac{1}{2}}, \] on obtient finalement le th\'eor\`eme suivant :

\begin{thm}\label{propfin3}
Si l'on suppose $ P_{1}\geqslant P_{2}\geqslant ...\geqslant P_{r} $, si $ P_{j}=P_{r}^{b_{j}} $ pour tout $ j\in \{1,...,r\} $ et si de plus $ K>\max\{(5\sum_{j=1}^{r}b_{j}+2)\tilde{d}, (2\delta+1)\sum_{j=1}^{r}b_{j}d_{j}\} $, alors \begin{multline*} N_{\ee}(P_{1},...,P_{r})=C_{\sigma,\ee}\left(\prod_{j=1}^{r}P_{j}^{n_{j}-d_{j}}\right)
\\ +O\left(\left(e_{0}^{4+\delta}\left(\prod_{i=1}^{n+r}e_{i}\right)^{-1}+\left(\prod_{i=1}^{n+r}e_{i}\right)^{-\frac{1}{2}}\right)\left(\prod_{j=1}^{r}P_{j}^{n_{j}-d_{j}-\delta}\right)\right). \end{multline*}

\end{thm}

\section{Deuxi\`eme \'etape}

Dans cette partie, nous supposerons encore que $ P_{1}\geqslant P_{2}\geqslant ...\geqslant P_{r} $. L'objectif de cette section est de donner, pour $ m\in \{1,...,r-1\} $ et $ \kk=(k_{m+1},...,k_{r}) $ fix\'es une formule asymptotique pour   \begin{multline*} N_{m,\kk,\ee}(P_{1},...,P_{m})=\sum_{\forall j\in\{1,...,m\}, \; k_{j}\leqslant P_{j}}\overline{h}_{\ee}(k_{1},...,k_{r}) \\ = \Card\left\{ 
 \xx\in (\ZZ\setminus\{0\})^{n+r} \; | \; F(\ee.\xx)=0,  \; \forall j\in \{m+1,...,r\}, \;  \left\lfloor|(\ee.\xx)^{E(n+j)}|\right\rfloor= k_{j}\;  \right. \\ \left. \forall j\in \{1,...,m\}, \;  \left\lfloor|(\ee.\xx)^{E(n+j)}|\right\rfloor\leqslant P_{j}, \; \forall i\in \{1,...,n+r\},  \; \right. \\ \left. |x_{i}|\leqslant \frac{1}{e_{i}}\left(\prod_{j=m+1}^{r}(k_{j}+1)^{a_{i,j}}\right)\prod_{j=1}^{m}(\left\lfloor|(\ee.\xx)^{E(n+j)}|\right\rfloor+1)^{a_{i,j}} \right\}. \end{multline*} 
Nous allons dans un premier temps fixer tous les $ x_{i} $ tels que $ a_{i,j}=0 $ pour tout $ j\in \{1,...,m\} $. Notons \[ I_{m}=\{i\in \{1,...,n+r\} \; |\; \forall j\in \{1,...,m\}, \; a_{i,j}=0 \}. \]
Fixons un \'el\'ement $ (x_{i})_{i\in I_{m}} $ tel que si $ \xx^{E(n+j)} $ est un mon\^ome en $ (x_{i})_{i\in I_{m}} $ alors $ \left\lfloor|(\ee.\xx)^{E(n+j)}|\right\rfloor= k_{j} $ et pour tout $ i\in I_{m} $, $ |x_{i}|\leqslant \frac{1}{e_{i}}\left(\prod_{j=m+1}^{r}(k_{j}+1)^{a_{i,j}}\right) $. On notera $ s=\Card I_{m} $ et : \begin{multline}
N_{(x_{i})_{i\in I_{m}},\ee}(P_{1},...,P_{m})=\Card\left\{ 
 (x_{i})_{i\notin I_{m}}\in (\ZZ\setminus\{0\})^{n+r-s} \; | \; F(\ee.\xx)=0,  \; \forall j\in \{m+1,...,r\}, \;   \right.\\ \left.\left\lfloor|(\ee.\xx)^{E(n+j)}|\right\rfloor= k_{j},\; \forall j\in \{1,...,m\}, \;  \left\lfloor|(\ee.\xx)^{E(n+j)}|\right\rfloor\leqslant P_{j}, \; \right. \\ \left. \forall i\notin I_{m},  \; |x_{i}|\leqslant \frac{1}{e_{i}}\left(\prod_{j=m+1}^{r}(k_{j}+1)^{a_{i,j}}\right)\prod_{j=1}^{m}(\left\lfloor|(\ee.\xx)^{E(n+j)}|\right\rfloor+1)^{a_{i,j}} \right\}.
\end{multline}
On \'ecrit alors : \[ N_{(x_{i})_{i\in I_{m}},\ee}(P_{1},...,P_{m})=\int_{0}^{1}S_{(x_{i})_{i\in I_{m}},\ee}(\alpha)d\alpha \] o\`u \begin{equation}
S_{(x_{i})_{i\in I_{m}},\ee}(\alpha)=\sum_{\substack{(x_{i})_{i\notin I_{m}}\in (\ZZ\setminus \{0\})^{n+r-s} \\ \forall j\in \{m+1,...,r\}, \left\lfloor|(\ee.\xx)^{E(n+j)}|\right\rfloor= k_{j} \\ \forall j \in \{1,...m\},\; \left\lfloor|(\ee.\xx)^{E(n+j)}|\right\rfloor\leqslant P_{j} \\ |x_{i}|\leqslant \frac{1}{e_{i}}\left(\prod_{j=m+1}^{r}(k_{j}+1)^{a_{i,j}}\right)\prod_{j=1}^{m}(\left\lfloor|(\ee.\xx)^{E(n+j)}|\right\rfloor+1)^{a_{i,j}} }}e\left(\alpha F(\ee.\xx)\right).
\end{equation} 

Dans toute cette section le symbole $ \ll $ d\'esignera une majoration \`a une constante multiplicative ind\'ependante de $ (x_{i})_{i\in I_{m}} $ pr\`es.

\subsection{In\'egalit\'e de Weyl}

 Fixons un ensemble de degr\'es $ (t_{j,k}^{(m)})_{\substack{j\in \{1,...,m\} \\ k\in \{1,...,d_{j}\} }} $ et on note \[ I_{0,m}=\{i\notin I_{m}\; |\; \forall j\in \{1,...,m\}, \;t^{(m)}_{j,a_{i,j}}\neq 0\}, \]\[ \CCC_{0,m}=\{(j,k)\; |\; j\in \{1,...,m\}, \; \exists i \in I_{0,m}, \; a_{i,j}=k\}, \]\[ \tt^{(m)}=(t^{(m)}_{j,k})_{(j,k)\in \CCC_{0,m}}, \] \[ \forall (j,k)\in \CCC_{0,m}, \; J_{m}(j,k)=\{i \in I_{0,m} \; |\; a_{i,j}=k \}. \] On notera d'autre part, comme dans la section pr\'ec\'edente :
 \begin{itemize}
\item $ \tilde{L}_{m}=\{(i,\l)\; |\; i\in I_{0,m}, \; \l=(l_{1},...,l_{m})\in\prod_{j=1}^{m}\{1,...,t_{j,a_{i,j}}+1\}\}  $,
\item $ L_{m}=\{(i,\l)\; |\; i\in I_{0,m}, \; \l=(l_{1},...,l_{m})\in\prod_{j=1}^{m}\{1,...,t_{j,a_{i,j}}\}\}  $,
\item $ \forall (j,k)\in \CCC_{0,m} $, $ L_{m}(j,k)=\{(i,\l)\in L \; |\; i\in J_{m}(j,k)\}  $,
\item  $ \forall (j,k)\in \CCC_{0,m} $, $ \widehat{L}_{m}(j,k)=\{(i,\l)\in L_{m}(j,k)\; |\; l_{j}\neq t_{j,k} \} $,
\item  $ \forall (j,k)\in \CCC_{0,m} $, $ E_{m}(j,k)=L_{m}(j,k)\setminus \widehat{L}_{m}(j,k)=\{(i,\l)\in L_{m}(j,k)\; |\; l_{j}=t_{j,k} \} $,
\end{itemize}
Comme dans la section pr\'ec\'edente, nous remarquons que l'on peut \'ecrire le polyn\^ome $ F $ sous la forme 

 \[ F(\xx)=\sum_{\substack{\dd=(d_{j,k})_{ (j,k)\in \CCC_{0,m}} \\ \forall j\in \{1,...,m\}, \; \sum_{k\geqslant 1}kd_{j,k}=d_{j}}}F_{\dd}(\xx), \] o\`u $ F_{\dd}(\xx) $ est un polyn\^ome homog\`ene de degr\'e $ d_{j,k} $ en les variables $ (x_{i})_{i\in I_{0,m}} $ telles que $ a_{i,j}=k $ pour tout $ j\in \{1,...,m\} $. \\

En effectuant les m\^emes op\'erations que pour les sections\;\ref{sectionweyl3} et\;\ref{sectionreseau3} (en ne consid\'erant cette fois-ci que les variables $ (x_{i})_{i\in I_{0,m}} $) en posant $ D^{(m)}_{j}=\sum_{k\geqslant 1}t^{(m)}_{j,k} $, $ \Delta^{\tt^{(m)}}=\Delta^{\tt_{1}^{(m)}}\circ \Delta^{\tt_{2}^{(m)}}\circ...\circ \Delta^{\tt_{m}^{(m)}} $, et pour tout $ (j,k)\in \CCC_{0,m}  $ \begin{multline*} \Delta^{\tt^{(m)}}F\left((x_{i})_{i\in I_{m}}(x_{i}^{\l})_{\substack{(i,\l)\in L_{m}}}\right)\\ =\sum_{(i,\l)\in E_{m}(j,k)}\gamma_{(i,\l,m)}^{(j,k)}\left((x_{i})_{i\in I_{m}}(x_{i}^{\hh})_{\substack{(i,\hh)\in L_{m}\setminus E_{m}(j,k)}}\right)e_{i}x_{i}^{\l}  \end{multline*}
et en notant alors pour tous $ (A_{i,\hh})_{(i,\hh)\in L_{m}\setminus E_{m}(i_{0},K_{j_{0}})} $, $(B_{i,\l})_{E_{m}(i_{0},K_{j_{0}})} $ :

\begin{multline*} M^{(m)}_{(j_{0},K_{j_{0}})}\left( \alpha, (A_{i,\hh})_{(i,\hh)\in L_{m}\setminus E_{m}(i_{0},K_{j_{0}})}, (B_{i,\l})_{E_{m}(i_{0},K_{j_{0}})}\right) \\ =\Card \left\{ (x_{i}^{\hh})_{(m,\hh)\in L_{m}\setminus E_{m}(i_{0},K_{j_{0}})}\; |\; \forall (m,\hh), \; |x_{i}^{\hh}|\leqslant A_{m,\hh} \; \et \;\right.\\ \left. \forall (i,\l)\in E_{m}(i_{0},K_{j_{0}}), \; \left|\left|\alpha e_{i}\gamma_{(i,\l,m)}^{(j_{0},K_{j_{0}})}\left(((x_{i})_{i\in I_{m}}(x_{i}^{\hh})_{\substack{(i,\hh)\in L_{m}\setminus E_{m}(j_{0},K_{j_{0}})}}\right)\right|\right|<B_{i,\l}\right\}, \end{multline*}
on obtient le r\'esultat suivant qui est un analogue du lemme\;\ref{dilemme33} :
\begin{lemma}
Pour tout $ \varepsilon>0 $ arbitrairement petit, $ K_{m}>0 $ et $ P\geqslant 1 $, l'une au moins des assertions suivantes est vraie : \begin{enumerate}
\item  \begin{multline*}|S_{(x_{i})_{i\in I_{m}},\ee}(\alpha)|\ll \left(\prod_{i\in I_{0,m}}e_{i}^{1+\varepsilon-\frac{\prod_{j=1}^{m}t^{(m)}_{j,a_{i,j}}}{2^{\sum_{j=1}^{m}D^{(m)}_{j}}}}\right)^{-1} \\ \left(\prod_{j=m+1}^{r}k_{j}^{\sum_{i\notin I_{m}}a_{i,j}+\varepsilon}\right)\left(\prod_{j=1}^{m}P_{j}^{n_{j}+1+\varepsilon}\right)P^{-K_{m}\theta},\end{multline*}
\item   Il existe un certain  $ (j_{0},K_{j_{0}})\in \CCC_{0,m} $ tel que \begin{multline*}
M^{(m)}_{(j_{0},K_{j_{0}})}\left( \alpha, \left(\frac{1}{e_{i}}\left(\prod_{j=m+1}^{r}k_{j}^{a_{i,j}}\right)P^{\sum_{j=1}^{m}a_{i,j}\theta}\right)_{\substack{(i,\l)\in \widehat{L}_{m}(j_{0},K_{j_{0}}) }}, \right. \\ \left. (e_{i}P^{\sum_{j=1}^{m}(d_{j}-a_{i,j})\theta}\prod_{j=1}^{m}P_{j}^{-d_{j}})_{\substack{(i,\l)\in E_{m}(j_{0},K_{j_{0}})}}\right) \\ \gg \left(\prod_{i\in I_{0,m}}\left(\prod_{j=m+1}^{r}k_{j}^{a_{i,j}}\right)^{\prod_{j=1}^{m}t_{j,a_{i,j}}}(P^{(\sum_{j=1}^{m}a_{i,j})\theta})^{\prod_{j=1}^{m}t_{j,a_{i,j}}}\right) \\ \left(\prod_{i\in J_{m}(j_{0},K_{j_{0}})}\left(\prod_{j=m+1}^{r}k_{j}^{a_{i,j}}\right)^{\prod_{j\neq j_{0}}t_{j,a_{i,j}}}(P^{\sum_{j=1}^{m}a_{i,j}\theta})^{\prod_{j\neq j_{0}}t_{j,a_{i,j}}}\right)^{-1}P^{-2^{\sum_{j=1}^{m}D^{(m)}_{j}}K_{m}\theta}.
\end{multline*}
\end{enumerate}
\end{lemma}
Remarquons que le cardinal consid\'er\'e dans le cas $ 2 $ peut \^etre major\'e trivialement par \begin{multline*} \Card\{ (x_{i}^{\hh})_{\substack {(i,\hh)\in  \widehat{L}_{m}(j_{0},K_{j_{0}}) }} \; | \; |x_{i}^{\hh}|\leqslant \left(\prod_{j=m+1}^{r}k_{j}^{a_{i,j}}\right)P^{\sum_{j=1}^{r}a_{m,j}\theta}, \; \\ \forall (i,\l)\in E_{m}(j_{0},K_{j_{0}}), \; ||\alpha e_{i}\gamma_{i,\l,m}^{(j_{0},K_{j_{0}})}(x_{i}^{\hh})||<e_{i}P^{\sum_{j=1}^{m}(d_{j}-a_{i,j})\theta}\prod_{j=1}^{m}P_{j}^{-d_{j}}  \},
\end{multline*}
que, d'apr\`es le lemme\;\ref{lemmedebilecas02}, nous pouvons majorer par : \begin{multline*}\prod_{\substack{(i,\hh)\in  \widehat{L}_{m}(j_{0},K_{j_{0}})}}\left(\prod_{j=m+1}^{r}k_{j}^{a_{i,j}}\right)P^{((\sum_{j=1}^{r}a_{i,j})-1)\theta} \\ \Card\left\{ (x_{i}^{\hh})_{\substack { (i,\hh)\in\widehat{L}_{m}(j_{0},K_{j_{0}}) }} \; | \; |x_{i}^{(\hh)}|\leqslant P^{\theta}, \; \forall (i,\l)\in E_{m}(j_{0},K_{j_{0}}), \;\right.\\ \left.  ||\alpha e_{i}\gamma_{i,\l,m}^{(j_{0},K_{j_{0}})}(x_{i}^{\hh})||<e_{i}P^{\sum_{j=1}^{m}(d_{j}-a_{i,j})\theta}\prod_{j=1}^{m}P_{j}^{-d_{j}}\right\}.
\end{multline*}
Si l'un des \'el\'ements $ (x_{i}^{\hh}) $ compt\'es par le cardinal ci-dessus est tel qu'il existe $ (i,\l)\in  E_{m}(j_{0},K_{j_{0}}) $ tels que $ q=\gamma_{i,\l,m}^{(j_{0},K_{j_{0}})}(x_{i}^{\hh})\neq 0 $, on a alors (quitte \`a changer $ \theta $) \[ 2||\alpha q e_{i}||\leqslant e_{i}P^{((\sum_{j=1}^{m}d_{j})-1)\theta}\prod_{j=1}^{m}P_{j}^{-d_{j}} \] et d'autre part \[ |q|\leqslant \left(\prod_{j=m+1}^{r}k_{j}^{d_{j}}\right)P^{((\sum_{j=1}^{m}d_{j})-1)\theta}. \] Par cons\'equent si l'on pose $ \tilde{d}_{m}=(\sum_{j=1}^{m}d_{j})-1 $ et 
\begin{equation}
\mathfrak{M}_{a,q}^{(\kk,i)}(\theta)=\{\alpha\in [0,1[ \; | \; |2 \alpha e_{i}q-a|<e_{i}P^{-1+\tilde{d}_{m}\theta}\}
\end{equation}
 \begin{equation}
\mathfrak{M}^{(\kk,i)}(\theta)=\bigcup_{0<q\leqslant \left(\prod_{j=m+1}^{r}k_{j}^{d_{j}}\right)P^{\tilde{d}_{m}\theta}}\bigcup_{0\leqslant a<e_{i}q}\mathfrak{M}_{a,q}^{(\kk,i)}(\theta),
\end{equation}
\begin{equation}
\mathfrak{M}^{(\kk)}(\theta)=\bigcup_{i\in I_{0}}\mathfrak{M}^{(\kk,i)}(\theta), \; \; \; \mathfrak{m}^{(\kk)}(\theta)=[0,1[\setminus\mathfrak{M}^{(\kk)}(\theta).
\end{equation}
et par ailleurs \begin{multline*} A^{(x_{i})_{i\in I_{m}}}_{(j_{0},K_{j_{0}})}(P^{\theta})=\Card\left\{ (x_{i}^{\hh})_{\substack {(i,\hh)\in  \widehat{L}_{m}(j_{0},K_{j_{0}})}} \; | \; |x_{i}^{\hh}|\leqslant P^{\theta}, \;  \right.\\ \left.\forall (i,\l)\in E_{m}(j_{0},K_{j_{0}}), \; \gamma_{i,\l,m}^{(j_{0},K_{j_{0}})}(x_{i}^{\hh})=0  \right\}.
\end{multline*}

\begin{lemma}\label{dilemme3x3}
Pour tout $ \varepsilon>0 $ arbitrairement petit et tous $ \kappa>0, \; P\geqslant 1 $ , l'une au moins des assertions suivantes est vraie : \begin{enumerate}
\item  \begin{multline*}|S_{(x_{i})_{i\in I_{m}},\ee}(\alpha)|\ll \left(\prod_{i\in I_{0,m}}e_{i}^{1+\varepsilon+\frac{\prod_{j=1}^{m}t_{j,a_{i,j}}}{2^{\sum_{j=1}^{m}D^{(m)}_{j}}}}\right)^{-1} \\ \left(\prod_{j=m+1}^{r}k_{j}^{\sum_{i\notin I_{m}}a_{i,j}+\varepsilon}\right)\left(\prod_{j=1}^{m}P_{j}^{n_{j}+1+\varepsilon}\right)P^{-K_{m}\theta},\end{multline*}
\item le r\'eel $ \alpha $ appartient \`a $ \mathfrak{M}^{(\kk)}(\theta) $,
\item  Il existe un certain  $ (j_{0},K_{j_{0}})\in \CCC_{0,m} $ tel que \begin{multline*}  A^{(x_{i})_{i\in I_{m}}}_{(j_{0},K_{j_{0}})}(P^{\theta}) \gg \left(\prod_{i\in I_{0,m}}(P^{\theta})^{\prod_{j=1}^{r}t_{j,a_{i,j}}}\right) \\ \left(\prod_{i\in J_{m}(j_{0},K_{j_{0}})}(P^{\theta})^{\prod_{j\neq j_{0}}^{r}t_{j,a_{i,j}}}\right)^{-1}P^{-2^{\sum_{j=1}^{r}D^{(m)}_{j}}K_{m}\theta}.
\end{multline*}
\end{enumerate}
\end{lemma}
Si l'on se place dans le cas $ 3 $, consid\'erons $ \mathcal{L}_{(x_{i})_{i\in I_{m}}} $ le sous-espace affine de $ \AA_{\CC}^{n(j_{0},K_{j_{0}})} $ (o\`u $ n(j_{0},K_{j_{0}})=\sum_{i\in I_{0,m}}\prod_{j=1}^{m}t_{j,a_{i,j}}-\sum_{i\in J_{m}(j_{0},K_{j_{0}})}\prod_{j\neq j_{0}}t_{j,a_{i,j}} $ d\'efini par les $ \sum_{i\in J_{m}(j_{0},K_{j_{0}})}\prod_{j\neq j_{0}}t_{j,a_{i,j}} $ \'equations $ \gamma_{i,\l,m}^{(j_{0},K_{j_{0}})}(x_{i}^{\hh})=0 $). La condition $ 3 $ implique alors : \[ \dim \mathcal{L}_{(x_{i})_{i\in I_{m}}}\geqslant n(j_{0},K_{j_{0}})-2^{\sum_{j=1}^{m}D^{(m)}_{j}}K. \] Soit $ \mathcal{D} $ la sous-vari\'et\'e de $ \AA_{\CC}^{n(j_{0},K_{j_{0}})} $ d\'efinie par les $ \sum_{i\in I_{0,m}}(\prod_{j=1}^{m}t_{j,a_{i,j}}-1)-\sum_{i\in J_{m}(j_{0},K_{j_{0}})}\prod_{j\neq j_{0}}t_{j,a_{i,j}} $ \'equations \[ \forall i\in I_{0,m},\hh \; |\;  (i,h^{(j_{0})})\neq (K_{j_{0}},t_{j_{0},K_{j_{0}}}), \; \;  x_{i}^{(\hh)}=x_{i}^{(1,...,1)}. \] On observe alors que : \begin{multline}\label{inegalitedetl3} \dim \mathcal{D}\cap \mathcal{L}_{(x_{i})_{i\in I_{m}}}  \geqslant \dim \mathcal{L}_{(x_{i})_{i\in I_{m}}}-\sum_{i\in I_{0,m}}(\prod_{j=1}^{m}t_{j,a_{i,j}}-1)+\sum_{i\in J_{m}(j_{0},K_{j_{0}})}\prod_{j\neq j_{0}}t_{j,a_{i,j}} \\  \geqslant \Card I_{0,m}-2^{\sum_{j=1}^{m}D^{(m)}_{j}}K. \end{multline}
Or,  $ \mathcal{D}\cap \mathcal{L}_{(x_{i})_{i\in I_{m}}} $ est isomorphe \`a la vari\'et\'e : 
\[ \{ (x_{i})_{i\in I_{0,m}} \; | \; \forall i\in J_{m}(j_{0},K_{j_{0}}), \; \frac{\partial F_{\tt^{(m)}}}{\partial x_{i}}((x_{i})_{i\in I_{0,m}})=0 \}. \]
En notant \begin{multline*} V_{m,(x_{i})_{i\in I_{m}},\tt,(j_{0},K_{j_{0}})}^{\ast}=\{ (x_{i})_{i\notin I_{m}}\in \AA_{\CC}^{n+r-s}\; | \; \\ \forall i\in J_{m}(j_{0},K_{j_{0}}), \; \frac{\partial F_{\tt^{(m)}}}{\partial x_{i}}((x_{i})_{i\in I_{m}},(x_{i})_{i\in I_{0,m}})=0 \}, \end{multline*} l'in\'egalit\'e\;\eqref{inegalitedetl3} implique alors : \[  \dim V_{m,(e_{i}x_{i})_{i\in I_{m}},\tt^{(m)},(j_{0},K_{j_{0}})}^{\ast} \geqslant n+r-s-2^{\sum_{j=1}^{m}D^{(m)}_{j}}K. \] 

Nous introduisons \`a pr\'esent un param\`etre $ \lambda\in \ZZ $ (que nous pr\'eciserons ult\'erieurement) et nous d\'efinissons \begin{equation}
\mathcal{A}_{m}^{\lambda}=\left\{ (x_{i})_{i\in I_{m}}\; | \; \forall (j_{0},K_{j_{0}}),\;  \dim V_{m,(e_{i},x_{i})_{i\in I_{m}},\tt^{(m)},(j_{0},K_{j_{0}})}^{\ast}<\dim V_{m,\tt^{(m)},(j_{0},K_{j_{0}})}^{\ast}-s+\lambda \right\}
\end{equation}
o\`u l'on a pos\'e  \[ V_{m,\tt^{(m)},(j_{0},K_{j_{0}})}^{\ast}=\{ \xx\in \AA_{\CC}^{n+r}\; | \; \forall i\in J_{m}(j_{0},K_{j_{0}}), \; \frac{\partial F_{\tt^{(m)}}}{\partial x_{i}}(\xx)=0 \}. \]
Nous fixerons dor\'enavant \begin{equation} K_{m}=(n+r-\max_{(j_{0},K_{j_{0}})}\dim V_{m,\tt^{(m)},(j_{0},K_{j_{0}})}^{\ast}-\lambda+\varepsilon)/2^{\sum_{j=1}^{m}D^{(m)}_{j}}, \end{equation} de sorte que pour tout $ (x_{i})_{i\in I_{m}}\in \mathcal{A}_{m}^{\lambda} $ la condition $ 3 $ n'est plus possible. On choisit \begin{equation}\label{valeurPcasm3} P=\prod_{j=1}^{m}P_{j}^{d_{j}} \end{equation} et on obtient donc finalement un analogue de la proposition\;\ref{dilemme43} pour $ (x_{i})_{i\in I_{m}} $ fix\'e : \begin{lemma}\label{dilemme4x3}
Pour tout $ \varepsilon>0 $ arbitrairement petit et pour tout $ (x_{i})_{i\in I_{m}} $ tel que $ (e_{i}x_{i})_{i\in I_{m}}\in \mathcal{A}_{m}^{\lambda} $, l'une au moins des assertions suivantes est vraie : \begin{enumerate}
\item  on a la majoration : \begin{multline*}|S_{(x_{i})_{i\in I_{m}},\ee}(\alpha)|\ll \left(\prod_{i\notin I_{m}}e_{i}\right)^{-1+\frac{1}{2^{m}}} \left(\prod_{j=m+1}^{r}k_{j}^{\sum_{i\notin I_{m}}a_{i,j}+\varepsilon}\right) \left(\prod_{j=1}^{m}P_{j}^{n_{j}+1}\right)P^{-K_{m}\theta+\varepsilon},\end{multline*}
\item le r\'eel $ \alpha $ appartient \`a $ \mathfrak{M}^{(\kk)}(\theta) $.
\end{enumerate}
\end{lemma}

\subsection{M\'ethode du cercle}

On fixe \`a pr\'esent un r\'eel $ \theta\in [0,1] $ et on suppose de plus que \[ K_{m}>2\tilde{d}_{m}\geqslant 1.\]
On notera par ailleurs \begin{equation}
e_{0,m}=\max_{i\notin I_{m}}e_{i},
\end{equation}
\begin{equation}
\phi(\ee,\kk,\theta)=e_{0,m}\left(\prod_{j=m+1}^{r}k_{j}^{d_{j}}\right)P^{\tilde{d}_{m}\theta},
\end{equation}
\begin{equation}
\Delta(\theta,K_{m})=\theta(K_{m}-2\tilde{d}_{m}).
\end{equation}

\subsubsection{Les arcs mineurs}

Nous commen\c{c}ons par donner une estimation de la contribution des arcs mineurs :

\begin{lemma}\label{arcminx3}
On a la majoration : 
\begin{multline*} \int_{\mathfrak{m}^{(\kk)}(\theta)}|S_{(x_{i})_{i\in I_{m}},\ee}(\alpha)|d\alpha\ll  e_{0,m}\left(\prod_{i\notin I_{m}}e_{i}\right)^{-1+\frac{1}{2^{m}}}\\ \left(\prod_{j=m+1}^{r}k_{j}^{\sum_{i\notin I_{m}}a_{i,j}+d_{j}+\varepsilon}\right)\left(\prod_{j=1}^{m}P_{j}^{n_{j}-d_{j}+1}\right)P^{2\varepsilon-\Delta(\theta,K_{m})}. 
\end{multline*}
\end{lemma}

\begin{proof}
Consid\'erons une suite $ (\theta_{i})_{i\in\{0,...,T\}} $ (avec $ T $ tel que $ T\ll P^{\frac{\varepsilon}{2}} $) v\'erifiant :  \[ \theta=\theta_{0}<\theta_{1}<...<\theta_{T}, \] et 
\begin{equation} \forall i\in \{0,...,T-1\}, \; \; 2\tilde{d}_{m}(\theta_{i+1}-\theta_{i})<\varepsilon/2. \end{equation} D'apr\`es le lemme\;\ref{dilemme4x3}, 
\begin{multline*}
\int_{\alpha\notin\mathfrak{M}^{(\kk)}(\theta_{T})}|S_{(x_{i})_{i\in I_{m}},\ee}(\alpha)|d\alpha  \\  \ll  \left(\prod_{i\notin I_{m}}e_{i}\right)^{-1+\frac{1}{2^{m}}}\left(\prod_{j=m+1}^{r}k_{j}^{\sum_{i\notin I_{m}}a_{i,j}+\varepsilon}\right) \left(\prod_{j=1}^{m}P_{j}^{n_{j}+1}\right)P^{-K_{m}\theta_{T}+\varepsilon}  \\  \ll \left(\prod_{i\notin I_{m}}e_{i}\right)^{-1+\frac{1}{2^{m}}}  \left(\prod_{j=m+1}^{r}k_{j}^{\sum_{i\notin I_{m}}a_{i,j}+\varepsilon}\right)\left(\prod_{j=1}^{m}P_{j}^{n_{j}-d_{j}+1}\right)\underbrace{P^{\varepsilon+1-K_{m}\theta_{T}}}_{\leqslant P^{-\Delta(K_{m},\theta)}}.
\end{multline*}
Par ailleurs, pour tout $ \theta $ et tout $ i\in I_{0,m} $ : \begin{align*}
\Vol(\mathfrak{M}^{(i)}(\theta))&\ll \sum_{0<q\leqslant \left(\prod_{j=m+1}^{r}k_{j}^{d_{j}}\right)P^{\tilde{d}_{m}\theta}}\sum_{0\leqslant a <e_{i}q}\frac{1}{q}P^{-1+\tilde{d}_{m}\theta} \\ & \ll e_{0,m}\left(\prod_{j=m+1}^{r}k_{j}^{d_{j}}\right)P^{-1+2\tilde{d}_{m}\theta}. 
\end{align*}
On a donc, pour tout $ t\in \{0,...,T-1\} $ : \begin{multline*} \int_{\alpha \in\mathfrak{M}^{(\kk)}(\theta_{t+1})\setminus\mathfrak{M}^{(\kk)}(\theta_{t})}|S_{(x_{i})_{i\in I_{m}},\ee}(\alpha)|d\alpha \\ \ll e_{0,m}\left(\prod_{i\notin I_{m}}e_{i}\right)^{-1+\frac{1}{2^{m}}}\left(\prod_{j=m+1}^{r}k_{j}^{\sum_{i\notin I_{m}}a_{i,j}+d_{j}+\varepsilon}\right)\left(\prod_{j=1}^{m}P_{j}^{n_{j}-d_{j}+1}\right)P^{-K_{m}\theta_{t}+2\tilde{d}_{m}\theta_{t+1}+\varepsilon}.\end{multline*}
Or on observe que \begin{align*}
P^{-K_{m}\theta_{t}+2\tilde{d}_{m}\theta_{t+1}+\varepsilon} & =P^{-(K_{m}-2\tilde{d}_{m})\theta_{t}+2\tilde{d}_{m}(\theta_{t+1}-\theta_{t})+\varepsilon} \\ & =P^{-\Delta(K_{m},\theta_{t})+3\varepsilon/2} \\ & \leqslant P^{-\Delta(K_{m},\theta)+3\varepsilon/2}.
\end{align*}
On obtient ainsi le r\'esultat en sommant sur tous les $ t\in \{0,...,T-1\} $.
\end{proof}

 \subsubsection{Les arcs majeurs}

D\'efinissons \`a pr\'esent une nouvelle famille d'arcs majeurs  : 

\begin{equation}
\mathfrak{M}_{a,q}^{(\kk)'}(\theta)=\{\alpha\in [0,1[ \; | \; 2|\alpha q-a|<qP^{-1+\tilde{d}_{m}\theta}\}
\end{equation}
et posons \begin{equation}
\mathfrak{M}^{(\kk)'}(\theta)=\bigcup_{0<q\leqslant \phi(\ee,\kk,\theta)}\bigcup_{\substack{0\leqslant a<q\\ \PGCD(a,q)=1}}\mathfrak{M}_{a,q}^{(\kk)'}(\theta),
\end{equation}
On remarque que comme dans la section pr\'ec\'edente : pour tout $ i\in I_{0,m} $, \[  \mathfrak{M}^{(\kk,i)}(\theta)\subset \mathfrak{M}^{(\kk)'}(\theta). \]
De la m\^eme mani\`ere que nous avons \'etabli le lemme\;\ref{disjoints3}, on d\'emontre le r\'esultat ci-dessous : 
\begin{lemma}\label{disjointsx3}
Pour tout $ \ee\in \NN^{n+r} $ tel que $ \left(e_{0,m}\prod_{j=m+1}^{r}k_{j}^{d_{j}}\right)^{2}<P^{1-3\tilde{d}_{m}\theta} $, les intervalles $ \mathfrak{M}_{a,q}^{(\kk)'}(\theta) $ sont disjoints deux \`a deux. 
\end{lemma}
D'apr\`es les r\'esultats obtenus sur les arcs mineurs, si $ \left(e_{0,m}\prod_{j=m+1}^{r}k_{j}^{d_{j}}\right)^{2}<P^{1-3\tilde{d}_{m}\theta} $ : 

\begin{multline}\label{formulesommex3}
N_{(x_{i})_{i\in I_{m}},\ee}(P_{1},...,P_{m})=\sum_{1\leqslant q \leqslant \phi(\ee,\kk,\theta)}\sum_{\substack{0\leqslant a<q\\ \PGCD(a,q)=1}}\int_{\mathfrak{M}_{a,q}^{(\kk)'}(\theta)}S_{(x_{i})_{i\in I_{m}},\ee}(\alpha)d\alpha \\ +O\left(e_{0,m}\left(\prod_{i\notin I_{m}}e_{i}\right)^{-1+\frac{1}{2^{m}}} \left(\prod_{j=m+1}^{r}k_{j}^{\sum_{i\notin I_{m}}a_{i,j}+d_{j}+\varepsilon}\right)\left(\prod_{j=1}^{m}P_{j}^{n_{j}-d_{j}+1}\right)P^{2\varepsilon-\Delta(\theta,K_{m})} \right)
\end{multline}
\'Etablissons \`a pr\'esent le lemme suivant : 

\begin{lemma}\label{separationx3} Si $ \alpha $ appartient \`a $ \mathfrak{M}_{a,q}^{(\kk)'}(\theta) $ et si l'on pose $ \alpha=\frac{a}{q}+\beta $ avec $ |\beta|\leqslant \frac{1}{2}P^{-1+\tilde{d}_{m}\theta} $, alors : \begin{multline*} S_{(x_{i})_{i\in I_{m}},\ee}(\alpha)=\left(\prod_{i\notin I_{m}}e_{i}\right)^{-1}\left(\prod_{j=m+1}^{r}k_{j}^{\sum_{i\notin I_{m}}a_{i,j}}\right)\left(\prod_{j=1}^{m}P_{j}^{n_{j}}\right) \\ q^{-(n+r-s)}S_{a,q,\ee}((x_{i})_{i\in I_{m}})I_{\kk,(x_{i},e_{i})_{i\in I_{m}}}(P\beta)
\\ +O\left(e_{0,m}\left(\prod_{i\notin I_{m}}e_{i}\right)^{-1}\left(\prod_{j=m+1}^{r}k_{j}^{\sum_{i\notin I_{m}}a_{i,j}}\right)\left(\prod_{j=1}^{m}P_{j}^{n_{j}}\right)qP_{m}^{-1}P^{\tilde{d}_{m}\theta}\right),
\end{multline*}
o\`u l'on a not\'e 
\begin{equation}
S_{a,q,\ee}((x_{i})_{i\in I_{m}})=\sum_{(b_{i})_{i\notin I_{m}}\in (\ZZ/q\ZZ)^{n+r-s}}e\left(\frac{a}{q}F((e_{i}x_{i})_{i\in I_{m}},(e_{i}b_{i})_{i\notin I_{m}})\right),
\end{equation}
\begin{multline*}
I_{\kk,(x_{i},e_{i})_{i\in I_{m}}}(\beta) \\ =\int_{\substack{(u_{i})_{i\notin I_{m}}\in \Xi(\kk,(x_{i})_{i\in I_{m}})}}e\left(\beta F((e_{i}x_{i})_{i\in I_{m}},(\left(\prod_{j=m+1}^{r}k_{j}^{a_{i,j}}\right)u_{i})_{i\notin I_{m}})\right)d(u_{i})_{i\notin I_{m}}. 
\end{multline*}
o\`u  \begin{multline}
 \Xi(\kk,(x_{i})_{i\in I_{m}})=\{(u_{i})_{i\notin I_{m}}\in \RR^{n+r-s} \; |\; \\  \forall i\notin I_{m}, \; \left|\left(\prod_{j=m+1}^{r}k_{j}^{a_{i,j}}\right)u_{i}\right|\leqslant \left|((e_{l}x_{l})_{l\in I_{m}},(\left(\prod_{j=m+1}^{r}k_{j}^{a_{l,j}}\right)u_{l})_{l\notin I_{m}})^{E(i)}\right| \\ \forall j\in \{m+1,...,r\}, k_{j}\leqslant \left((e_{l}x_{l})_{l\in I_{m}},(\left(\prod_{j=m+1}^{r}k_{j}^{a_{l,j}}\right)u_{l})_{l\notin I_{m}}\right)^{E(n+j)}<k_{j}+1 \\ \forall j\in \{1,...,m\}, \; \left|((e_{l}x_{l})_{l\in I_{m}},(\left(\prod_{j=m+1}^{r}k_{j}^{a_{l,j}}\right)u_{l})_{l\notin I_{m}})^{E(n+j)}\right|\leqslant 1\}
\end{multline}
\end{lemma}
\begin{proof}
Commen\c{c}ons par montrer que lorsque $ e_{0,m}q>P_{m} $, l'\'egalit\'e du lemme est triviale. On a dans ce cas :
\begin{align*} |S_{(x_{i})_{i\in I_{m}},\ee}(\alpha)| & \ll \prod_{i\notin I_{m}}\frac{1}{e_{i}}\left(\prod_{j=m+1}^{r}k_{j}^{a_{i,j}}\right)\left(\prod_{j=1}^{m}P_{j}^{a_{i,j}}\right)\\ & =\left(\prod_{i\notin I_{m}}e_{i}\right)^{-1}\left(\prod_{j=1}^{m}P_{j}^{n_{j}}\right)\left(\prod_{j=m+1}^{r}k_{j}^{\sum_{i\notin I_{m}}a_{i,j}}\right) \\ & \ll \left(\prod_{i\notin I_{m}}e_{i}\right)^{-1}\left(\prod_{j=m+1}^{r}k_{j}^{\sum_{i\notin I_{m}}a_{i,j}}\right) \left(\prod_{j=1}^{r}P_{j}^{n_{j}}\right)P_{m}^{-1}e_{0,m}P^{\tilde{d}_{m}\theta}
\end{align*}
et en utilisant les estimations triviales \[ |S_{a,q,\ee}((x_{i})_{i\in I_{m}})|\ll q^{n+r-s} \;\; \et \; \;  |I_{\kk,(x_{i},e_{i})_{i\in I_{m}}}(P\beta)|\ll 1, \] on a alors :\begin{multline*}
\left(\prod_{i\notin I_{m}}e_{i}\right)^{-1}\left(\prod_{j=m+1}^{r}k_{j}^{\sum_{i\notin I_{m}}a_{i,j}}\right)\left(\prod_{j=1}^{m}P_{j}^{n_{j}}\right) \\ q^{-(n+r-s)}|S_{a,q,\ee}((x_{i})_{i\in I_{m}})||I_{\kk,(x_{i})_{i\in I_{m}}}(P\beta)| \\ \ll \left(\prod_{i\notin I_{m}}e_{i}\right)^{-1}\left(\prod_{j=m+1}^{r}k_{j}^{\sum_{i\notin I_{m}}a_{i,j}}\right)\left(\prod_{j=1}^{m}P_{j}^{n_{j}}\right)  \\  \ll \left(\prod_{i\notin I_{m}}e_{i}\right)^{-1}\left(\prod_{j=m+1}^{r}k_{j}^{\sum_{i\notin I_{m}}a_{i,j}}\right)\left(\prod_{j=1}^{m}P_{j}^{n_{j}}\right)P_{m}^{-1}qe_{0,m}P^{\tilde{d}_{m}\theta}.
\end{multline*}
d'o\`u le r\'esultat. Nous supposerons donc $ e_{0,m}q\leqslant P_{m} $.\\

Comme dans la d\'emonstration du lemme\;\ref{separation3}, on peut \'ecrire :  
 \begin{equation}\label{formulesepx3} S_{(x_{i})_{i\in I_{m}},\ee}(\alpha)=\sum_{\bb=(b_{i})_{i\notin I_{m}}\in (\ZZ/q\ZZ)^{n+r-s}}e\left(\frac{a}{q}F((e_{i}x_{i})_{i\in I_{m}},(e_{i}b_{i})_{i\notin I_{m}})\right)\tilde{S}_{m}(\bb) \end{equation} o\`u \[ \tilde{S}_{m}(\bb)=\sum_{\substack{ (x_{i})_{i\notin I_{m}}\in I(q,\kk)}}e(\beta F(\ee.\xx)) \]
avec \begin{multline*}
 I(q,\kk)=\{(x_{i})_{i\notin I_{m}}\; |\; \forall i\notin I_{m}, \;x_{i}\equiv b_{i} (q) \\ \forall j\in \{1,...,m\},\; \left\lfloor|(\ee.\xx)^{E(n+j)}|\right\rfloor\leqslant P_{j} \\ \forall j\in \{m+1,...,r\},\; k_{j}\leqslant|(\ee.\xx)^{E(n+j)}|<k_{j}+1 \\ |x_{i}|\leqslant \frac{1}{e_{i}}\prod_{j=1}^{r}(\left\lfloor|(\ee.\xx)^{E(n+j)}|\right\rfloor+1)^{a_{i,j}} \}
\end{multline*}
On consid\`ere  deux \'el\'ements $ (x_{i}')_{i\notin I_{m}} $ et $ (x_{i}'')_{i\notin I_{m}} $ deux \'el\'ement de $ \RR^{n+r-s} $ tels que \[|\xx'-\xx''|\leqslant 2\]  et tels que $ (qx_{i}'+b_{i})_{i\notin I_{m}} $, $ (qx_{i}''+b_{i})_{i\notin I_{m}} $ soient deux \'el\'ements de $ I(q,\kk) $. On a alors\[\left| F(\ee.(q\xx'+\bb))-F(\ee.(q\xx''+\bb))\right|\ll e_{0,m}q\left(\prod_{j=m+1}^{r}k_{j}^{d_{j}}\right)\left(\prod_{j=1}^{m}P_{j}^{d_{j}}\right)P_{m}^{-1}. \] Par cons\'equent,
  \begin{multline*}
\tilde{S}_{m}(\bb)=\int_{\substack{(\tilde{u}_{i})_{i\notin I_{m}}\in \tilde{\Xi}(q,\kk,(x_{i})_{i\in I_{m}}) }}e(\beta F((e_{i}u_{i})_{i\in I_{m}},(e_{i}q\tilde{u}_{i})_{i\notin I_{m}}))(d\tilde{u}_{i})_{i\notin I_{m}} \\ + O\left(\underbrace{|\beta|}_{\leqslant P^{-1+\tilde{d}_{m}\theta}}\left(\prod_{i\notin I_{m}}e_{i}\right)^{-1}q^{-(n+r-s)}\left(\prod_{j=m+1}^{r}k_{j}^{\sum_{i\notin I_{m}}a_{i,j}}\right)\right. \\ \left.\left(\prod_{j=1}^{m}P_{j}^{n_{j}}\right)\left(\prod_{j=m+1}^{r}k_{j}^{d_{j}}\right)\left(\prod_{j=1}^{m}P_{j}^{d_{j}}\right)P_{m}^{-1}qe_{0,m}\right) \\ + O\left(e_{0,m}\left(\prod_{i\notin I_{m}}e_{i}\right)^{-1}\left(\prod_{j=m+1}^{r}k_{j}^{\sum_{i\notin I_{m}}a_{i,j}}\right)\left(\prod_{j=1}^{m}P_{j}^{n_{j}}\right)q^{-(n+r-s)+1}P_{m}^{-1}\right)
\end{multline*}
o\`u \begin{multline*}
 \tilde{\Xi}(q,\kk,(x_{i})_{i\in I_{m}})=\{(\tilde{u}_{i})_{i\notin I_{m}}\in \RR^{n+r-s} \; |\; \\ \forall i\notin I_{m}, \; |q\tilde{u}_{i}|\leqslant |((e_{l}x_{l})_{l\in I_{m}},(e_{l}q\tilde{u}_{l})_{l\notin I_{m}})^{E(i)}| \\ \forall j\in \{m+1,...,r\}, k_{j}\leqslant | ((e_{l}x_{l})_{l\in I_{m}},(e_{l}q\tilde{u}_{l})_{l\notin I_{m}})^{E(n+j)}|<k_{j}+1 \\ \forall j\in \{1,...,m\}, \; |((e_{l}x_{l})_{l\in I_{m}},(e_{l}q\tilde{u}_{l})_{l\notin I_{m}})^{E(n+j)}|\leqslant P_{j}\}.
\end{multline*}
En effectuant le changement de variables \[ \forall i\notin I_{m}, \; \; q\tilde{u_{i}}=\frac{1}{e_{i}}\left(\prod_{j=m+1}^{r}k_{j}^{a_{i,j}}\right)\left(\prod_{j=1}^{m}P_{j}^{a_{i,j}}\right)u_{i}, \]
on obtient : \begin{multline*}
\tilde{S}_{m}(\bb)=\left(\prod_{i\notin I_{m}}e_{i}\right)^{-1}\left(\prod_{j=m+1}^{r}k_{j}^{\sum_{i\notin I_{m}}a_{i,j}}\right)\left(\prod_{j=1}^{m}P_{j}^{n_{j}}\right)q^{-(n+r-s)} \\ \int_{\substack{(u_{i})_{i\notin I_{m}}\in \Xi(\kk,(x_{i})_{i\in I_{m}}) }}e(\beta PF((e_{i}x_{i})_{i\in I_{m}},(u_{i})_{i\notin I_{m}}))(du_{i})_{i\notin I_{m}} \\ + O\left( \left(\prod_{i\notin I_{m}}e_{i}\right)^{-1}q^{-(n+r-s)}\left(\prod_{j=m+1}^{r}k_{j}^{\sum_{i\notin I_{m}}a_{i,j}}\right)\right. \\ \left.\left(\prod_{j=1}^{m}P_{j}^{n_{j}}\right)P_{m}^{-1}qe_{0,m}P^{\tilde{d}_{m}\theta}\right) 
\end{multline*}
Il suffit ensuite de remplacer $ \tilde{S}(\bb) $ par cette expression dans \eqref{formulesepx3} pour trouver l'\'egalit\'e du lemme.

\end{proof}
En observant que : 
\[ \Vol(\mathfrak{M}^{(\kk)'}(\theta))\ll \sum_{1\leqslant q \leqslant \phi(\ee,\kk,\theta)}\sum_{\substack{0\leqslant a<q\\ \PGCD(a,q)=1}}P^{-1+\tilde{d}_{m}\theta} \ll \left(e_{0,m}\prod_{j=m+1}^{r}k_{j}^{d_{j}}\right)^{2}P^{-1+3\tilde{d}_{m}\theta}. \] et en utilisant le lemme\;\ref{separationx3} dans la formule \eqref{formulesommex3}, on obtient pour tout $ (x_{i})_{i\in I_{m}}\in \mathcal{A}_{m}^{\lambda}\cap \ZZ^{s} $ \begin{multline}\label{presquefinx3}
N_{(x_{i})_{i\in I_{m}},\ee}(P_{1},...,P_{m})=\mathfrak{S}_{(x_{i})_{i\in I_{m}},\ee}(\phi(\ee,\kk,\theta))J_{(x_{i},e_{i})_{i\in I_{m}},\kk}(\frac{1}{2}P^{\tilde{d}_{m}\theta}) \\ \left(\prod_{i\notin I_{m}}e_{i}\right)^{-1}\left(\prod_{j=m+1}^{r}k_{j}^{\sum_{i\notin I_{m}}a_{i,j}}\right)\left(\prod_{j=1}^{m}P_{j}^{n_{j}-d_{j}}\right)
\\ +O\left(e_{0,m}^{4}\left(\prod_{i\notin I_{m}}e_{i}\right)^{-1}\left(\prod_{j=m+1}^{r}k_{j}^{3d_{j}+\sum_{i\notin I_{m}}a_{i,j}}\right)\left(\prod_{j=1}^{m}P_{j}^{n_{j}-d_{j}}\right)P_{m}^{-1}P^{5\tilde{d}_{m}\theta}\right).
\\ + O\left(e_{0,m}\left(\prod_{i\notin I_{m}}e_{i}\right)^{-1+\frac{1}{2^{m}}}\left(\prod_{j=m+1}^{r}k_{j}^{d_{j}+\varepsilon+\sum_{i\notin I_{m}}a_{i,j}}\right)\right. \\ \left.\left(\prod_{j=1}^{m}P_{j}^{n_{j}-d_{j}+1}\right)P^{2\varepsilon-\Delta(\theta,K_{m})}\right),
\end{multline}
avec \begin{equation}
\mathfrak{S}_{(x_{i})_{i\in I_{m}},\ee}(Q)=\sum_{1\leqslant q\leqslant Q}q^{-(n+r-s)}\sum_{\substack{0\leqslant a<q\\ \PGCD(a,q)=1}}S_{a,q,\ee}((x_{i})_{i\in I_{m}}),
\end{equation}
\begin{equation}
J_{(x_{i},e_{i})_{i\in I_{m}},\kk}(\phi)=\int_{|\beta|\leqslant \phi}I_{\kk,(x_{i},e_{i})_{i\in I_{m}}}(\beta)d\beta. 
\end{equation}
Posons \`a pr\'esent : 
\begin{equation}
\mathfrak{S}_{(x_{i})_{i\in I_{m}},\ee}=\sum_{q=1}^{\infty}q^{-(n+r-s)}\sum_{\substack{0\leqslant a<q\\ \PGCD(a,q)=1}}S_{a,q,\ee}((x_{i})_{i\in I_{m}}),
\end{equation}
\begin{equation}
J_{(x_{i},e_{i})_{i\in I_{m}},\kk}=\int_{\beta\in \RR}I_{\kk,(x_{i},e_{i})_{i\in I_{m}}}(\beta)d\beta. 
\end{equation}
Nous allons \`a pr\'esent d\'emontrer des analogues des lemmes\;\ref{S3} et\;\ref{J3} : 

\begin{lemma}\label{Jx3}
Soit $ (x_{i})_{i\in I_{m}}\in \mathcal{A}_{m}^{\lambda}\cap \ZZ^{s} $. Si l'on suppose \[ K_{m}b_{m}>\left(b_{m}+3\sum_{j=1}^{m}b_{j}\right)\tilde{d}_{m}+2, \] l'int\'egrale $ J_{(x_{i},e_{i})_{i\in I_{m}},\kk} $ est absolument convergente et on a de plus \begin{multline*} |J_{(x_{i},e_{i})_{i\in I_{m}},\kk}(\frac{1}{2}P^{\tilde{d}_{m}\theta})-J_{(x_{i},e_{i})_{i\in I_{m}},\kk}|\\ \ll \left(\prod_{j=m+1}^{r}k_{j}^{d_{j}}\right)^{\frac{2\sum_{j=1}^{m}b_{j}}{b_{m}+\sum_{j=1}^{m}b_{j}}+\varepsilon}P^{\left( (1+\frac{2\sum_{j=1}^{m}b_{j}}{b_{m}+\sum_{j=1}^{m}b_{j}})\tilde{d}_{m}-\frac{b_{m}K_{m}}{b_{m}+\sum_{j=1}^{m}b_{j}}\right)\theta+\varepsilon} . \end{multline*} Par ailleurs \[ | J_{(x_{i},e_{i})_{i\in I_{m}},\kk}|\ll \left(\prod_{j=m+1}^{r}k_{j}^{d_{j}}\right)^{\frac{2\sum_{j=1}^{m}b_{j}}{b_{m}+\sum_{j=1}^{m}b_{j}}+\varepsilon}. \]
\end{lemma}
\begin{proof}
Consid\'erons $ \beta\in \RR $ et choisissons des param\`etres $ P_{1},...,P_{m} $, $ P $ et $ \theta' $ tels que \begin{equation}\label{eqJx13} |\beta|=\frac{1}{2}P^{\tilde{d}_{m}\theta'}, \end{equation}
\begin{equation}\label{eqJx23}
 \left(\prod_{j=1}^{m}P_{j}\right)P^{-K_{m}\theta'}=\left(e_{0,m}\prod_{j=m+1}^{r}k_{j}^{d_{j}}\right)^{2}P^{2\tilde{d}_{m}\theta'}P_{m}^{-1}. \end{equation} 
Ces deux \'egalit\'es impliquent alors \[ \left(\prod_{j=1}^{m}P_{j}\right)P^{-K_{m}\theta'}=\left(e_{0,m}\prod_{j=m+1}^{r}k_{j}^{d_{j}}\right)^{2}(2|\beta|)^{2}P_{m}^{-1} \] et donc \begin{equation}\label{eqJx33}
\theta'=\frac{\frac{1}{\tilde{d}_{m}}\log(2|\beta|)\frac{b_{m}+\sum_{j=1}^{m}b_{j}}{\sum_{j=1}^{m}b_{j}d_{j}}}{2\log(e_{0,m}\prod_{j=m+1}^{r}k_{j}^{d_{j}})+\left(\frac{K_{m}}{\tilde{d}_{m}}+2\right)\log(2|\beta|)}.
\end{equation}  
En particulier \[ \theta'\gg \min\{1,\frac{\log(2|\beta|)}{\log(e_{0,m}\prod_{j=m+1}^{r}k_{j}^{d_{j}})}\}. \] 
Par ailleurs l'\'egalit\'e\;\eqref{eqJx23} implique \[ \left(e_{0,m}\prod_{j=m+1}^{r}k_{j}^{d_{j}}\right)^{2}P^{-1+3\theta'\tilde{d}_{m}}=\underbrace{\left(\prod_{j=1}^{m}P_{j}^{d_{j}}\right)^{-1}\left(\prod_{j=1}^{m}P_{j}\right)P_{m}}_{\leqslant 1}\underbrace{P^{(\tilde{d}_{m}-K_{m})\theta'}}_{<1}<1, \] et donc d'apr\`es le lemme\;\ref{disjointsx3}, les arcs majeurs $ \mathfrak{M}_{a,q}^{(\kk)'}(\theta') $ sont disjoints. Dans tout ce qui va suivre, nous fixerons $ e_{i}=1 $ pour tout $ i\notin I_{m} $, et donc $ e_{0,m}=1 $. On remarque que le r\'eel $ P^{-1}\beta $ appartient au bord de $ \mathfrak{M}_{0,1}^{(\kk)}(\theta') $. Par cons\'equent, en utilisant le lemme\;\ref{dilemme4x3}, on a : \begin{equation}\label{ineqJx13}
|S_{(x_{i})_{i\in I_{m}},\ee}(P^{-1}\beta)| \ll  \left(\prod_{j=m+1}^{r}k_{j}^{\sum_{i\notin I_{m}}a_{i,j}+\varepsilon}\right) \left(\prod_{j=1}^{m}P_{j}^{n_{j}+1}\right)P^{-K_{m}\theta'+\varepsilon},
\end{equation}
D'autre part, d'apr\`es le lemme\;\ref{separationx3} : 
 \begin{multline}\label{inegJx23} S_{(x_{i})_{i\in I_{m}},\ee}(P^{-1}\beta)=\left(\prod_{j=m+1}^{r}k_{j}^{\sum_{i\notin I_{m}}a_{i,j}}\right)\left(\prod_{j=1}^{m}P_{j}^{n_{j}}\right) I_{\kk,(x_{i},e_{i})_{i\in I_{m}}}(\beta)
\\ +O\left(\left(\prod_{j=m+1}^{r}k_{j}^{\sum_{i\notin I_{m}}a_{i,j}}\right)\left(\prod_{j=1}^{m}P_{j}^{n_{j}}\right)qP_{m}^{-1}P^{\tilde{d}_{m}\theta'}\right).
\end{multline}
Nous d\'eduisons de\;\eqref{ineqJx13} et\;\eqref{inegJx23} que : \begin{align*}
|I_{\kk,(x_{i},e_{i})_{i\in I_{m}}}(\beta)| & \ll \left(\prod_{j=m+1}^{r}k_{j}^{\varepsilon}\right)\left(\prod_{j=1}^{m}P_{j}\right)P^{-K_{m}\theta'+\varepsilon} + \left(\prod_{j=m+1}^{r}k_{j}^{d_{j}}\right)P^{2\tilde{d}_{m}\theta'}P_{m}^{-1} \\ & \ll \left(\prod_{j=m+1}^{r}k_{j}\right)^{\varepsilon}\left(\prod_{j=1}^{m}P_{j}\right)P^{-K_{m}\theta'+\varepsilon} \\ & =\left(\prod_{j=m+1}^{r}k_{j}\right)^{\varepsilon}P^{\left(\frac{\sum_{j=1}^{m}b_{j}}{\sum_{j=1}^{m}b_{j}d_{j}}\right)-K_{m}\theta'+\varepsilon} \\ &\ll \left(\prod_{j=m+1}^{r}k_{j}\right)^{\varepsilon}|\beta|^{\frac{1}{\tilde{d}_{m}\theta'}\left(\frac{\sum_{j=1}^{m}b_{j}}{\sum_{j=1}^{m}b_{j}d_{j}}+\varepsilon\right)-\frac{K_{m}}{\tilde{d_{m}}}}.
\end{align*} pour tout $ \beta\in \RR $. \\

On a donc que \begin{multline*}
 |J_{(x_{i},e_{i})_{i\in I_{m}},\kk}(\frac{1}{2}P^{\tilde{d}_{m}\theta})-J_{(x_{i},e_{i})_{i\in I_{m}},\kk}| \\ \ll \left(\prod_{j=m+1}^{r}k_{j}\right)^{\varepsilon}\int_{|\beta|\geqslant \frac{1}{2}P^{\tilde{d}_{m}\theta} }|\beta|^{\frac{1}{\tilde{d}_{m}\theta'}\left(\frac{\sum_{j=1}^{m}b_{j}}{\sum_{j=1}^{m}b_{j}d_{j}}+\varepsilon\right)-\frac{K_{m}}{\tilde{d_{m}}}}d\beta.
\end{multline*}
Or, d'apr\`es l'\'egalit\'e\;\eqref{eqJx33} : \[ \frac{1}{\tilde{d}_{m}\theta'}\frac{b_{m}+\sum_{j=1}^{m}b_{j}}{\sum_{j=1}^{m}b_{j}d_{j}} \leqslant \frac{2\log\left(\prod_{j=m+1}^{r}k_{j}^{d_{j}}\right)}{\log(2|\beta|)}+\left(\frac{K_{m}}{\tilde{d}_{m}}+2\right), \] donc \[  \frac{1}{\tilde{d}_{m}\theta'}\frac{\sum_{j=1}^{m}b_{j}}{\sum_{j=1}^{m}b_{j}d_{j}} \leqslant \left(\frac{\sum_{j=1}^{m}b_{j}}{b_{m}+\sum_{j=1}^{m}b_{j}}\right)\left(\frac{2\log\left(\prod_{j=m+1}^{r}k_{j}^{d_{j}}\right)}{\log(2|\beta|)}+\left(\frac{K_{m}}{\tilde{d}_{m}}+2\right)\right), \]
et ainsi : \[ |\beta|^{\frac{1}{\tilde{d}_{m}\theta'}\left(\frac{\sum_{j=1}^{m}b_{j}}{\sum_{j=1}^{m}b_{j}d_{j}}+\varepsilon\right)} \ll \left(\prod_{j=m+1}^{r}k_{j}^{d_{j}+\varepsilon}\right)^{\frac{2\sum_{j=1}^{m}b_{j}}{b_{m}+\sum_{j=1}^{m}b_{j}}}|\beta|^{\frac{\sum_{j=1}^{m}b_{j}}{b_{m}+\sum_{j=1}^{m}b_{j}}\left(\frac{K_{m}}{\tilde{d}_{m}}+2\right)+\varepsilon}. \]
On obtient alors :  \begin{multline*}
 |J_{(x_{i},e_{i})_{i\in I_{m}},\kk}(\frac{1}{2}P^{\tilde{d}_{m}\theta})-J_{(x_{i},e_{i})_{i\in I_{m}},\kk}| \\ \ll \left(\prod_{j=m+1}^{r}k_{j}^{d_{j}+\varepsilon}\right)^{\frac{2\sum_{j=1}^{m}b_{j}}{b_{m}+\sum_{j=1}^{m}b_{j}}}\int_{|\beta|\geqslant \frac{1}{2}P^{\tilde{d}_{m}\theta} }|\beta|^{-\frac{b_{m}}{b_{m}+\sum_{j=1}^{m}b_{j}}\frac{K_{m}}{\tilde{d_{m}}}+\frac{2\sum_{j=1}^{m}b_{j}}{b_{m}+\sum_{j=1}^{m}b_{j}}+\varepsilon}d\beta \\ \ll  \left(\prod_{j=m+1}^{r}k_{j}^{d_{j}}\right)^{\frac{2\sum_{j=1}^{m}b_{j}}{b_{m}+\sum_{j=1}^{m}b_{j}}+\varepsilon}P^{\left(\left(1+\frac{2\sum_{j=1}^{m}b_{j}}{b_{m}+\sum_{j=1}^{m}b_{j}}\right)\tilde{d}_{m}-\frac{b_{m}K_{m}}{b_{m}+\sum_{j=1}^{m}b_{j}}\right)\theta+\varepsilon},
\end{multline*} car on a suppos\'e $ K_{m}b_{m}>\left(b_{m}+3\sum_{j=1}^{m}b_{j}\right)\tilde{d}_{m}+2 $, ce qui implique la convergence. En particulier, si l'on choisit $ P\ll 1 $, cette majoration donne : \[
 |\underbrace{J_{(x_{i},e_{i})_{i\in I_{m}},\kk}(\frac{1}{2}P^{\tilde{d}_{m}\theta})}_{\ll 1}-J_{(x_{i},e_{i})_{i\in I_{m}},\kk}|  \ll   \left(\prod_{j=m+1}^{r}k_{j}^{d_{j}}\right)^{\frac{2\sum_{j=1}^{m}b_{j}}{b_{m}+\sum_{j=1}^{m}b_{j}}+\varepsilon},
\] et donc \[ | J_{(x_{i},e_{i})_{i\in I_{m}},\kk}|\ll \left(\prod_{j=m+1}^{r}k_{j}^{d_{j}}\right)^{\frac{2\sum_{j=1}^{m}b_{j}}{b_{m}+\sum_{j=1}^{m}b_{j}}+\varepsilon}, \] et le lemme est d\'emontr\'e.

\end{proof}
Nous allons \`a pr\'esent comparer $ \mathfrak{S}_{(x_{i})_{i\in I_{m}},\ee} $ et $ \mathfrak{S}_{(x_{i})_{i\in I_{m}},\ee}(\phi(\ee,\kk,\theta)) $. Pour cela, introduisons la s\'erie g\'en\'eratrice (pour $ T\geqslant 1 $ fix\'e) : \begin{equation}
\tilde{S}_{(x_{i})_{i\in I_{m}},\ee}(\alpha)=\sum_{\substack{(x_{i})_{i\notin I_{m}}\\ |x_{i}|\leqslant T}}e(\alpha F(\ee.\xx)).
\end{equation}
Comme nous l'avons fait pour le lemme\;\ref{dilemme4x3}, nous pouvons \'etablir : 
\begin{lemma}
Pour tout r\'eel $ T\geqslant 1 $, l'une au moins des assertions suivantes est vraie : 
 \begin{enumerate}
\item  on a la majoration : \[|\tilde{S}_{(x_{i})_{i\in I_{m}},\ee}(\alpha)|\ll T^{n+r-s+\varepsilon-K_{m}\theta}\]
\item le r\'eel $ \alpha $ appartient \`a $ \tilde{\mathfrak{M}}^{(\kk)}(\theta) $,
\end{enumerate}
o\`u \begin{equation}
\tilde{\mathfrak{M}}^{(\kk)}(\theta)=\bigcup_{i\in I_{0,m}}\bigcup_{0<q\leqslant e_{0,m} \left(\prod_{j=m+1}^{r}k_{j}^{d_{j}}\right)T^{\tilde{d}_{m}\theta}}\bigcup_{\substack{0\leqslant a<q \\ \PGCD(a,q)=1}}\tilde{\mathfrak{M}}_{a,q}^{(i,\kk)}(\theta),
\end{equation}
\begin{equation}
\tilde{\mathfrak{M}}_{a,q}^{(i,\kk)}(\theta)=\{\alpha\in [0,1[ \; | \; |\alpha q-a|<e_{i}T^{-\max\{\sum_{k\geqslant 1}t_{j,k}\}+\tilde{d}_{m}\theta}\}.
\end{equation}
\end{lemma}
Nous pouvons alors d\'emontrer : 
\begin{lemma}\label{Sx3}
Si l'on suppose que $ \tt $ est tel que, il existe $ j\in \{1,...,r\} $ tel que $ \sum_{k\geqslant 1}t_{j,k}\geqslant 2 $, la s\'erie $ \mathfrak{S}_{(x_{i})_{i\in I_{m}},\ee} $ est absolument convergente et on a de plus \[ |\mathfrak{S}_{(x_{i})_{i\in I_{m}},\ee}(\phi(\ee,\kk,\theta))-\mathfrak{S}_{(x_{i})_{i\in I_{m}},\ee}|\ll \left(e_{0,m}\prod_{j=m+1}^{r}k_{j}^{d_{j}}\right)^{2+\delta}P^{\theta(2\tilde{d}_{m}-K_{m})+\varepsilon}, \] et \[|\mathfrak{S}_{(x_{i})_{i\in I_{m}},\ee}|\ll \left(e_{0,m}\prod_{j=m+1}^{r}k_{j}^{d_{j}}\right)^{2+\delta}.\]
\end{lemma}
\begin{proof}
On consid\`ere $ q>\phi(\ee,\kk,\theta) $ et $ \alpha=\frac{a}{q} $ avec $ 0\leqslant a<q $ et $ \PGCD(a,q)=1 $. On observe que \[ S_{a,q,\ee}((x_{i})_{i\in I_{m}})=\tilde{S}_{(x_{i})_{i\in I_{m}},\ee}(\alpha), \] avec $ T=q $. On consid\`ere $ \theta' $ tel que $ q=\left(e_{0,m}\prod_{j=m+1}^{r}k_{j}^{d_{j}}\right)q^{\tilde{d}_{m}\theta'} $. Posons par ailleurs $ \theta''=\theta'-\nu $ avec $ \nu>0 $ arbitrairement petit. Supposons qu'il existe $ a',q'\in \ZZ $ tels que $ q'\leqslant\left(e_{0,m}\prod_{j=m+1}^{r}k_{j}^{d_{j}}\right)q^{\tilde{d}_{m}\theta''}<q  $, $ 0\leqslant a'<q' $, $ \PGCD(a',q')=1 $ et $ \alpha\in \tilde{\mathfrak{M}}_{a,q}^{(i,\kk)}(\theta'') $. On a alors \[ 1\leqslant |aq'-a'q|\leqslant qe_{i}q^{\tilde{d}_{m}\theta''-\max_{j}\{\sum_{k\geqslant 1}t_{j,k}\}} < q^{2-\max_{j}\{\sum_{k\geqslant 1}t_{j,k}\}}\leqslant 1, \] 
ce qui est absurde. Donc $ \alpha\notin \tilde{\mathfrak{M}}^{(\kk)}(\theta'') $. D'apr\`es le lemme pr\'ec\'edent, on a donc : 
\[ |S_{a,q,\ee}((x_{i})_{i\in I_{m}})|=|\tilde{S}_{(x_{i})_{i\in I_{m}},\ee}(\alpha)|\ll q^{n+r-s+\varepsilon-K_{m}\theta'}. \]
Par cons\'equent, 
\begin{multline*}
|\mathfrak{S}_{(x_{i})_{i\in I_{m}},\ee}(\phi(\ee,\kk,\theta))-\mathfrak{S}_{(x_{i})_{i\in I_{m}},\ee}| \\ \ll\sum_{ q>\phi(\ee,\kk,\theta) }q^{-(n+r-s)}\sum_{\substack{0\leqslant a<q\\ \PGCD(a,q)=1}}|S_{a,q,\ee}((x_{i})_{i\in I_{m}})| \\  \ll \sum_{ q>\phi(\ee,\kk,\theta) }q^{1+\varepsilon-K_{m}\theta'} \\  \ll \sum_{ q>\phi(\ee,\kk,\theta) }q^{-\frac{K_{m}}{\tilde{d}_{m}}+1+\varepsilon}\left(e_{0,m}\prod_{j=m+1}^{r}k_{j}^{d_{j}}\right)^{\frac{K_{m}}{\tilde{d}_{m}}} \\  \ll \left(e_{0,m}\prod_{j=m+1}^{r}k_{j}^{d_{j}}\right)^{2+\delta}P^{\theta(2\tilde{d}_{m}-K_{m})+\varepsilon}.
\end{multline*}
En choisissant par ailleurs $ P\ll 1 $, cette majoration implique (puisque $ \mathfrak{S}_{(x_{i})_{i\in I_{m}},\ee}(\phi(\ee,\kk,\theta)) \ll  \left(e_{0,m}\prod_{j=m+1}^{r}k_{j}^{d_{j}}\right)^{2} $) : \[|\mathfrak{S}_{(x_{i})_{i\in I_{m}},\ee}|\ll \left(e_{0,m}\prod_{j=m+1}^{r}k_{j}^{d_{j}}\right)^{2+\delta}.\]
\end{proof}
 Nous pouvons alors \'etablir le r\'esultat suivant : 
\begin{lemma}\label{propfinx3}
Soit $ (x_{i})_{i\in I_{m}}\in \mathcal{A}_{m}^{\lambda}\cap \ZZ^{s} $. On suppose $ \theta\in [0,1] $ fix\'e, et $ P\geqslant 1 $ tels que $ \left(e_{0,m}\prod_{j=m+1}^{r}k_{j}^{d_{j}}\right)^{2}<P^{1-3\tilde{d}_{m}\theta} $. On a alors \begin{multline*} N_{(x_{i})_{i\in I_{m}},\ee}(P_{1},...,P_{m})=\mathfrak{S}_{(x_{i})_{i\in I_{m}},\ee}J_{(x_{i},e_{i})_{i\in I_{m}},\kk} \\ \left(\prod_{i\notin I_{m}}e_{i}\right)^{-1}\left(\prod_{j=m+1}^{r}k_{j}^{\sum_{i\notin I_{m}}a_{i,j}}\right)\left(\prod_{j=1}^{m}P_{j}^{n_{j}-d_{j}}\right)+ O(E_{1})+O(E_{2})+O(E_{3}) \end{multline*}
o\`u \[ E_{1}= e_{0,m}\left(\prod_{i\notin I_{m}}e_{i}\right)^{-1+\frac{1}{2^{m}}}\left(\prod_{j=m+1}^{r}k_{j}^{d_{j}+\varepsilon+\sum_{i\notin I_{m}}a_{i,j}}\right) \\ \left(\prod_{j=1}^{m}P_{j}^{n_{j}-d_{j}+1}\right)P^{2\varepsilon-\Delta(\theta,K_{m})}, \]\[E_{2}=e_{0,m}^{4}\left(\prod_{i\notin I_{m}}e_{i}\right)^{-1}\left(\prod_{j=m+1}^{r}k_{j}^{4d_{j}+\sum_{i\notin I_{m}}a_{i,j}}\right)\left(\prod_{j=1}^{m}P_{j}^{n_{j}-d_{j}}\right)P_{m}^{-1}P^{5\tilde{d}_{m}\theta},\] \begin{multline*} 
E_{3}=e_{0,m}^{2+\varepsilon}\left(\prod_{i\notin I_{m}}e_{i}\right)^{-1}\left(\prod_{j=m+1}^{r}k_{j}^{2d_{j}\left(1+\frac{\sum_{j=1}^{m}b_{j}}{b_{m}+\sum_{j=1}^{m}b_{j}}+\varepsilon\right)+\sum_{i\notin I_{m}}a_{i,j}}\right) \\ \left(\prod_{j=1}^{m}P_{j}^{n_{j}-d_{j}}\right)P^{\left( (1+\frac{2\sum_{j=1}^{m}b_{j}}{b_{m}+\sum_{j=1}^{m}b_{j}})\tilde{d}_{m}-\frac{b_{m}K_{m}}{b_{m}+\sum_{j=1}^{m}b_{j}}\right)\theta+\varepsilon} \end{multline*}
\end{lemma}
\begin{proof}
Supposons dans un premier temps qu'il existe $ j\in \{1,...,r\} $ tel que $ \sum_{k\geqslant 1}t_{j,k}\geqslant 2 $. D'apr\`es la formule\;\eqref{presquefinx3}, on a \begin{multline*} N_{(x_{i})_{i\in I_{m}},\ee}(P_{1},...,P_{m})=\mathfrak{S}_{(x_{i})_{i\in I_{m}},\ee}(\phi(\ee,\kk,\theta))J_{(x_{i},e_{i})_{i\in I_{m}},\kk}(\frac{1}{2}P^{\tilde{d}_{m}\theta}) \\ \left(\prod_{i\notin I_{m}}e_{i}\right)^{-1}\left(\prod_{j=m+1}^{r}k_{j}^{\sum_{i\notin I_{m}}a_{i,j}}\right)\left(\prod_{j=1}^{m}P_{j}^{n_{j}-d_{j}}\right)+ O(E_{2})+O(E_{1}), \end{multline*}
Par ailleurs, d'apr\`es les lemmes\;\ref{Jx3} et\;\ref{Sx3} : \begin{multline*}
|\mathfrak{S}_{(x_{i})_{i\in I_{m}},\ee}(\phi(\ee,\kk,\theta))J_{(x_{i},e_{i})_{i\in I_{m}},\kk}(\frac{1}{2}P^{\tilde{d}_{m}\theta})-\mathfrak{S}_{(x_{i})_{i\in I_{m}},\ee}J_{(x_{i},e_{i})_{i\in I_{m}},\kk}| \\ \ll |\mathfrak{S}_{(x_{i})_{i\in I_{m}},\ee}(\phi(\ee,\kk,\theta))-\mathfrak{S}_{(x_{i})_{i\in I_{m}},\ee}||J_{(x_{i},e_{i})_{i\in I_{m}},\kk}| \\ + |\mathfrak{S}_{(x_{i})_{i\in I_{m}},\ee}(\phi(\ee,\kk,\theta))||J_{(x_{i},e_{i})_{i\in I_{m}},\kk}(\frac{1}{2}P^{\tilde{d}_{m}\theta})-J_{(x_{i},e_{i})_{i\in I_{m}},\kk}| \\ \ll e_{0,m}^{2+\varepsilon}\left(\prod_{j=m+1}^{r}k_{j}^{d_{j}}\right)^{2+\frac{2\sum_{j=1}^{m}b_{j}}{b_{m}+\sum_{j=1}^{m}b_{j}}+2\varepsilon}P^{\theta(2\tilde{d}_{m}-K_{m})+\varepsilon} \\ + e_{0,m}^{2+\varepsilon}\left(\prod_{j=m+1}^{r}k_{j}^{d_{j}}\right)^{2+\frac{2\sum_{j=1}^{m}b_{j}}{b_{m}+\sum_{j=1}^{m}b_{j}}+2\varepsilon}P^{\left( (1+\frac{2\sum_{j=1}^{m}b_{j}}{b_{m}+\sum_{j=1}^{m}b_{j}})\tilde{d}_{m}-\frac{b_{m}K_{m}}{b_{m}+\sum_{j=1}^{m}b_{j}}\right)\theta+\varepsilon} \\ \ll e_{0,m}^{2+\varepsilon}\left(\prod_{j=m+1}^{r}k_{j}^{d_{j}}\right)^{2+\frac{2\sum_{j=1}^{m}b_{j}}{b_{m}+\sum_{j=1}^{m}b_{j}}+2\varepsilon}P^{\left( (1+\frac{2\sum_{j=1}^{m}b_{j}}{b_{m}+\sum_{j=1}^{m}b_{j}})\tilde{d}_{m}-\frac{b_{m}K_{m}}{b_{m}+\sum_{j=1}^{m}b_{j}}\right)\theta+\varepsilon},
\end{multline*}
d'o\`u le r\'esultat. \\

Le cas o\`u la famille $ \tt $ choisie est telle que $ t_{j,k}=1 $ si $ (j,k)=(j,k^{(j)}) $ (pour un entier $ k^{(j)} $ fix\'e) et $ 0 $ sinon est particulier. Les variables auxquelles nous nous int\'eresserons sont alors les $ x_{i} $ telles que $ a_{i,j}=k^{(j)} $ pour tout $ j\in \{1,...,m\} $ et nous \'ecrivons alors le polyn\^ome $ F $ sous la forme : \[ F(\ee.\xx)= \sum_{k\in J}e_{k}A_{k}((e_{i}x_{i})_{i\in I_{m}})x_{k}+ \widetilde{F}((e_{i}x_{i})_{i\notin J}), \] o\`u \[ J=\{i\in I_{0,m}\; | \; \forall j\in \{1,...,m\},\;  a_{i,j}=k^{(j)}  \}.  \] On remarque alors que \begin{align*}
|S_{a,q,\ee}((x_{i})_{i\in I_{m}})| & \ll \sum_{(b_{k})_{k\notin J\cup I_{m}}}\left| \sum_{(b_{k})_{k\in J}}e\left(\frac{a}{q}\sum_{k\in J}e_{k}A_{k}((e_{i}x_{i})_{i\in I_{m}})b_{k}\right) \right| \\ & \ll  \sum_{(b_{k})_{k\notin J\cup I_{m}}}\prod_{k\in J}\underbrace{\left| \sum_{b_{k}\in \ZZ/q\ZZ}e\left(\frac{a}{q}\sum_{k\in J}e_{k}A_{k}((e_{i}x_{i})_{i\in I_{m}})b_{k}\right)\right|}_{ =\left\{\begin{array}{lcr} q & \mbox{si} & e_{k}A_{k}((x_{i})_{i\in I_{m}})\equiv 0 \; (q) \\ 0 & \mbox{sinon} &
\end{array}\right.}  \\ & =\left\{\begin{array}{lcr} q^{n+r-s} & \mbox{si} & e_{k}A_{k}((e_{i}x_{i})_{i\in I_{m}})\equiv 0 \; (q) \; \; \forall k\in J \\ 0 & \mbox{sinon} &
\end{array}\right.
\end{align*}
Or, pour $ q\gg e_{0,m}\prod_{j=m+1}^{r}k_{j}^{d_{j}} $, on a alors $ |e_{k}A_{k}((e_{i}x_{i})_{i\in I_{m}})|<q $ pour tout $ k\in J $, et donc la condition $  e_{k}A_{k}((e_{i}x_{i})_{i\in I_{m}})\equiv 0 \; (q) $ implique $ A_{k}((e_{i}x_{i})_{i\in I_{m}})=0 $. Or, pour $ (x_{i})_{i\in I_{m}}\in \mathcal{A}_{m}^{\lambda} $, la propri\'et\'e $ A_{k}((e_{i}x_{i})_{i\in I_{m}})=0 $ pour tout $ k\in J $ est impossible. Par cons\'equent, quitte \`a multiplier $ \phi(\ee,\kk,\theta) $ par une constante, on a que \[ \mathfrak{S}_{(x_{i})_{i\in I_{m}},\ee}(\phi(\ee,\kk,\theta))=\mathfrak{S}_{(x_{i})_{i\in I_{m}},\ee}, \] et en rempla\c{c}ant 
$ J_{(x_{i},e_{i})_{i\in I_{m}},\kk}(\frac{1}{2}P^{\tilde{d}_{m}\theta}) $ par $ J_{(x_{i},e_{i})_{i\in I_{m}},\kk} $ comme ci-dessus, on retrouve le r\'esultat du lemme.

\end{proof}
En choisissant $ \theta>0 $ tel que $ \theta<\frac{1}{5\tilde{d}_{m}}\sum_{j=1}^{m}\frac{b_{j}d_{j}}{b_{m}} $, on obtient : 
\begin{cor}\label{corollairex3}
Soit $ (x_{i})_{i\in I_{m}}\in \mathcal{A}_{m}^{\lambda}\cap \ZZ^{s} $. On suppose de plus que \[ b_{m}K_{m}>(b_{m}+3\sum_{j=1}^{m}b_{j})\tilde{d}_{m}.\] Il existe alors un r\'eel $ \delta>0 $ tel que : \begin{multline*} N_{(x_{i})_{i\in I_{m}},\ee}(P_{1},...,P_{m})=\mathfrak{S}_{(x_{i})_{i\in I_{m}},\ee}J_{(x_{i},e_{i})_{i\in I_{m}},\kk} \\ \left(\prod_{i\notin I_{m}}e_{i}\right)^{-1}\left(\prod_{j=m+1}^{r}k_{j}^{\sum_{i\notin I_{m}}a_{i,j}}\right)\left(\prod_{j=1}^{m}P_{j}^{n_{j}-d_{j}}\right) \\ + O\left(\left(e_{0,m}^{4}\left(\prod_{i\notin I_{m}}e_{i}\right)^{-1}+\left(\prod_{i\notin I_{m}}e_{i}\right)^{-\frac{1}{2}}\right)\left(\prod_{j=m+1}^{r}k_{j}^{4d_{j}+\sum_{i\notin I_{m}}a_{i,j}}\right)\left(\prod_{j=1}^{m}P_{j}^{n_{j}-d_{j}}\right)P^{-\delta}\right), \end{multline*}
uniform\'ement pour tous $ (x_{i})_{i\in I_{m}},\kk $ tels que \[ \left(e_{0,m}\prod_{j=m+1}^{r}k_{j}^{d_{j}}\right)^{2}P^{-1+3\tilde{d}_{m}\theta}<1. \]
\end{cor}

Introduisons \`a pr\'esent pour $ \bb=(b_{1},...,b_{r}) $ et $ \delta>0 $ fix\'es la fonction \begin{multline}\label{fonctgm3}
g_{m}(\bb,\delta)=\left(\frac{b_{m}+\sum_{j=1}^{m}b_{j}}{b_{m}}\right)5\tilde{d}_{m}\left(1-\sum_{j=m+1}^{r}(1+5d_{j})\frac{b_{j}}{b_{m}}-\delta\right)^{-1}\\ \left( \sum_{j=m+1}^{r}\left(1+d_{j}\left(3+\frac{\sum_{l=1}^{m}b_{l}}{b_{m}+\sum_{l=1}^{m}b_{l}}+2\varepsilon\right)\right)\frac{b_{j}}{b_{m}}+2\delta\right).
\end{multline}
D\'efinissons par ailleurs : \begin{multline}
 \tilde{N}_{\ee,m}(P_{1},...,P_{r}) = \Card\left\{ 
 \xx\in (\ZZ\setminus\{0\})^{n+r}  \; |\; \forall k\in \{m,...,r-1\}, \;(x_{i})_{i\in I_{k}}\in \mathcal{A}_{k}^{\lambda}, \; \right. \\F(\ee.\xx)=0, \;  \forall j\in \{1,...,r\}, \;  \left\lfloor|(\ee.\xx)^{E(n+j)}|\right\rfloor\leqslant P_{j}, \;  \forall i\in \{1,...,n+r\},  \;  \\ \left.|x_{i}|\leqslant \frac{1}{e_{i}}\prod_{j=1}^{r}(\left\lfloor|(\ee.\xx)^{E(n+j)}|\right\rfloor+1)^{a_{i,j}} \right\}.
\end{multline}
\begin{thm}\label{thmfinx3}
On suppose que $ (m,d_{1})\neq (1,1) $ et que  \[ b_{m}K_{m}>(b_{m}+3\sum_{j=1}^{m}b_{j})\tilde{d}_{m},\] \[ K_{m}-\frac{b_{m}+3\sum_{j=1}^{m}b_{j}}{b_{m}}\tilde{d}_{m}>g_{m}(\bb,\delta), \] et \[ \sum_{j=m+1}^{r}(1+5d_{j})\frac{b_{j}}{b_{m}}<1. \]  On a alors : \begin{multline}\label{formthmfinx3}  \tilde{N}_{\ee,m}(P_{1},...,P_{r})= \sum_{\forall j\in \{m+1,...,r\}, \;k_{j}\leqslant P_{j}}\left(\sum_{(x_{i})_{i\in I_{m}} \in \Phi_{m}(\kk)}\mathfrak{S}_{(x_{i})_{i\in I_{m}},\ee}J_{(x_{i},e_{i})_{i\in I_{m}},\kk}\right) \\ \left(\prod_{i\notin I_{m}}e_{i}\right)^{-1}\left(\prod_{j=m+1}^{r}k_{j}^{\sum_{i\notin I_{m}}a_{i,j}}\right)\left(\prod_{j=1}^{m}P_{j}^{n_{j}-d_{j}}\right) \\ + O\left(\left(e_{0,m}^{4}\left(\prod_{i=1}^{n+r}e_{i}\right)^{-1}+\left(\prod_{i=1}^{n+r}e_{i}\right)^{-1+\frac{1}{2^{m}}}\right)\left(\prod_{j=1}^{r}P_{j}^{n_{j}-d_{j}}\right)P_{m}^{-\delta}\right),\end{multline} o\`u $ \Phi_{m}(\kk) $ d\'esigne l'ensemble des $ (x_{i})_{i\in I_{m}} \in \mathcal{A}_{m}^{\lambda} $ v\'erifiant $ (x_{i})_{i\in I_{k}}\in \mathcal{A}_{k}^{\lambda}  $ pour tout $ k\in \{m,...,r-1\} $ (remarquons que si $ k\geqslant m $ alors $ I_{k}\subset I_{m} $), $ |x_{i}|\leqslant \frac{1}{e_{i}}\prod_{j=m+1}^{r}(k_{j}+1)^{a_{i,j}} $ pour tout $ i\in I_{m} $ et $ \left\lfloor|(\ee.\xx)^{E(n+j)}|\right\rfloor\leqslant P_{j} $ pour tout $ j\in \{m+1,...,r\} $ tel que $ (\ee.\xx)^{E(n+j)} $ ne d\'epend que de $ (x_{i})_{i\in I_{m}} $. 
\end{thm}
\begin{proof}
D'apr\`es la formule du lemme\;\ref{propfinx3} si l'on suppose \[ \left(e_{0,m}\prod_{j=m+1}^{r}P_{j}^{d_{j}}\right)^{2}P^{-1+3\tilde{d}_{m}\theta}<1 ,\] on a
\begin{multline*}  \tilde{N}_{\ee,m}(P_{1},...,P_{r})= \sum_{\forall j\in \{m+1,...,r\}, \;k_{j}\leqslant P_{j}}\left(\sum_{(x_{i})_{i\in I_{m}} \in \Phi_{m}(\kk)}\mathfrak{S}_{(x_{i})_{i\in I_{m}},\ee}J_{(x_{i},e_{i})_{i\in I_{m}},\kk}\right) \\ \left(\prod_{i\notin I_{m}}e_{i}\right)^{-1}\left(\prod_{j=m+1}^{r}k_{j}^{\sum_{i\notin I_{m}}a_{i,j}}\right)\left(\prod_{j=1}^{m}P_{j}^{n_{j}-d_{j}}\right) + O(\mathcal{E}_{1})+ O\left(\mathcal{E}_{2}\right)+O\left(\mathcal{E}_{3}\right),\end{multline*}
o\`u 
\begin{multline*}
\mathcal{E}_{1}\ll \sum_{\forall j\in \{m+1,...,r\}, \;k_{j}\leqslant P_{j}}\sum_{(x_{i})_{i\in I_{m}} \in \Phi_{m}(\kk)}e_{0,m}\left(\prod_{i\notin I_{m}}e_{i}\right)^{-1+\frac{1}{2^{m}}}\left(\prod_{j=m+1}^{r}k_{j}^{d_{j}+\varepsilon+\sum_{i\notin I_{m}}a_{i,j}}\right) \\ \left(\prod_{j=1}^{m}P_{j}^{n_{j}-d_{j}+1}\right)P^{2\varepsilon-\Delta(\theta,K_{m})}
\end{multline*}

\begin{multline*}
\mathcal{E}_{2}=\sum_{\forall j\in \{m+1,...,r\}, \;k_{j}\leqslant P_{j}}\sum_{(x_{i})_{i\in I_{m}} \in \Phi_{m}(\kk)}e_{0,m}^{4}\left(\prod_{i\notin I_{m}}e_{i}\right)^{-1} \\ \left(\prod_{j=m+1}^{r}k_{j}^{4d_{j}+\sum_{i\notin I_{m}}a_{i,j}}\right)\left(\prod_{j=1}^{m}P_{j}^{n_{j}-d_{j}}\right)P_{m}^{-1}P^{5\tilde{d}_{m}\theta} \\ \ll \sum_{\forall j\in \{m+1,...,r\}, \;k_{j}\leqslant P_{j}}e_{0,m}^{4}\left(\prod_{i=1 }^{n+r}e_{i}\right)^{-1}\left(\prod_{j=m+1}^{r}k_{j}^{n_{j}+4d_{j}}\right)\left(\prod_{j=1}^{m}P_{j}^{n_{j}-d_{j}}\right)P_{m}^{5\tilde{d}_{m}\theta\left(\frac{\sum_{j=1}^{m}b_{j}d_{j}}{b_{m}}\right)-1} \\ \ll e_{0,m}^{4}\left(\prod_{i=1}^{n+r}e_{i}\right)^{-1}\left(\prod_{j=m+1}^{r}P_{j}^{n_{j}+1+4d_{j}}\right)\left(\prod_{j=1}^{m}P_{j}^{n_{j}-d_{j}}\right)P_{m}^{5\tilde{d}_{m}\theta\left(\frac{\sum_{j=1}^{m}b_{j}d_{j}}{b_{m}}\right)-1}\\ \ll e_{0,m}^{4}\left(\prod_{i=1 }^{n+r}e_{i}\right)^{-1}\left(\prod_{j=1}^{r}P_{j}^{n_{j}-d_{j}}\right)P_{m}^{ -\delta_{1}}.
\end{multline*} o\`u \[ -\delta_{1}=\sum_{j=m+1}^{r}(1+5d_{j})\frac{b_{j}}{b_{m}}+5\tilde{d}_{m}\theta\left(\frac{\sum_{j=1}^{m}b_{j}d_{j}}{b_{m}}\right)-1, \]
et \begin{multline*} \mathcal{E}_{3}=\sum_{\forall j\in \{m+1,...,r\}, \;k_{j}\leqslant P_{j}}\sum_{(x_{i})_{i\in I_{m}} \in \Phi_{m}(\kk)} \left(e_{0,m}^{2+\varepsilon}\left(\prod_{i\notin I_{m}}e_{i}\right)^{-1} +\left(\prod_{i\notin I_{m}}e_{i}\right)^{-\frac{1}{2}}\right) \\ \left(\prod_{j=m+1}^{r}k_{j}^{2d_{j}\left(1+\frac{\sum_{j=1}^{m}b_{j}}{b_{m}+\sum_{j=1}^{m}b_{j}}+\varepsilon\right)+\sum_{i\notin I_{m}}a_{i,j}}\right) \\ \left(\prod_{j=1}^{m}P_{j}^{n_{j}-d_{j}}\right)P^{\left( (1+\frac{2\sum_{j=1}^{m}b_{j}}{b_{m}+\sum_{j=1}^{m}b_{j}})\tilde{d}_{m}-\frac{b_{m}K_{m}}{b_{m}+\sum_{j=1}^{m}b_{j}}\right)\theta+\varepsilon} \\ \ll  \left(e_{0,m}^{2+\varepsilon}\left(\prod_{i=1}^{n+r}e_{i}\right)^{-1} +\left(\prod_{i=1}^{n+r}e_{i}\right)^{-\frac{1}{2}}\right)\left(\prod_{j=m+1}^{r}P_{j}^{n_{j}+1+2d_{j}\left(1+\frac{\sum_{j=1}^{m}b_{j}}{b_{m}+\sum_{j=1}^{m}b_{j}}+\varepsilon\right)}\right) \\ \left(\prod_{j=1}^{m}P_{j}^{n_{j}-d_{j}}\right)P^{\left( (1+\frac{2\sum_{j=1}^{m}b_{j}}{b_{m}+\sum_{j=1}^{m}b_{j}})\tilde{d}_{m}-\frac{b_{m}K_{m}}{b_{m}+\sum_{j=1}^{m}b_{j}}\right)\theta+\varepsilon} \\ \ll  \left(e_{0,m}^{2+\varepsilon}\left(\prod_{i=1}^{n+r}e_{i}\right)^{-1} +\left(\prod_{i=1}^{n+r}e_{i}\right)^{-\frac{1}{2}}\right)\left(\prod_{j=m+1}^{r}P_{j}^{n_{j}+1+2d_{j}\left(1+\frac{\sum_{j=1}^{m}b_{j}}{b_{m}+\sum_{j=1}^{m}b_{j}}+\varepsilon\right)}\right)\left(\prod_{j=1}^{m}P_{j}^{n_{j}-d_{j}}\right) P_{m}^{-\delta_{2}} \end{multline*}
o\`u \begin{multline*} -\delta_{2}=\left( (\frac{b_{m}+3\sum_{j=1}^{m}b_{j}}{b_{m}+\sum_{j=1}^{m}b_{j}})\tilde{d}_{m}-\frac{b_{m}K_{m}}{b_{m}+\sum_{j=1}^{m}b_{j}}\right)\theta\frac{\sum_{j=1}^{m}b_{j}d_{j}}{b_{m}}+\varepsilon \\-\sum_{j=m+1}^{r}\left(1+d_{j}\left(3+\frac{\sum_{l=1}^{m}b_{l}}{b_{m}+\sum_{l=1}^{m}b_{l}}+2\varepsilon\right)\right)\frac{b_{j}}{b_{m}}. \end{multline*} Remarquons que $ \mathcal{E}_{1} $ est n\'egligeable par rapport \`a ce terme d'erreur. On choisit alors \[\theta=\frac{1}{5\tilde{d}_{m}}\left(\frac{b_{m}}{\sum_{j=1}^{m}b_{j}d_{j}}\right)\left(1-\sum_{j=m+1}^{r}(1+5d_{j})\frac{b_{j}}{b_{m}}-\delta\right), \] de sorte que $ \delta_{1}=\delta $. Par ailleurs ce choix de $ \theta $ implique \[ g_{m}(\bb,\delta)=\left(\frac{b_{m}+\sum_{j=1}^{m}b_{j}}{\sum_{j=1}^{m}b_{j}d_{j}}\right)\theta^{-1}\left( \sum_{j=m+1}^{r}\left(1+d_{j}\left(3+\frac{\sum_{l=1}^{m}b_{l}}{b_{m}+\sum_{l=1}^{m}b_{l}}+2\varepsilon\right)\right)\frac{b_{j}}{b_{m}}+2\delta\right),  \]
et donc puisque l'on a suppos\'e \[ K_{m}-\frac{b_{m}+3\sum_{j=1}^{m}b_{j}}{b_{m}}\tilde{d}_{m}>g_{m}(\bb,\delta), \]
on a alors $ \delta_{2}<\varepsilon-2\delta<-\delta $, ce qui ach\`eve la d\'emonstration de la proposition pour le cas o\`u $ \ee $ est tel que $ \left(e_{0}\prod_{j=m+1}^{r}P_{j}^{d_{j}}\right)^{2}P^{-1+3\tilde{d}_{m}\theta}<1 $. Dans le cas contraire, on remarque que l'\'egalit\'e est triviale car le terme d'erreur est alors dominant : en effet, si l'on suppose $  \left(e_{0}\prod_{j=m+1}^{r}P_{j}^{d_{j}}\right)^{2}\geqslant P^{1-3\tilde{d}_{m}\theta}  $, on a l'estimation triviale : 
\begin{align*}  \tilde{N}_{\ee,m}(P_{1},...,P_{r}) & \ll \left(\prod_{i=1}^{n+r}e_{i}\right)^{-1}\left(\prod_{j=1}^{r}P_{j}^{n_{j}}\right) \\ & \ll e_{0,m}^{4}\left(\prod_{j=1}^{r}P_{j}^{d_{j}}\right)^{4}\left(\prod_{i=1}^{n+r}e_{i}\right)^{-1}\left(\prod_{j=1}^{r}P_{j}^{n_{j}}\right)P^{-2+6\tilde{d}_{m}\theta} \\ & \ll e_{0,m}^{4}\left(\prod_{j=1}^{r}P_{j}^{d_{j}}\right)^{4}P^{-1+6\tilde{d}_{m}\theta}\left(\prod_{j=1}^{r}P_{j}^{n_{j}-d_{j}}\right).\end{align*}
Or, par le choix de $ \theta $ on a \[ P^{-1+6\tilde{d}_{m}\theta}\ll \left(\prod_{j=1}^{r}P_{j}^{d_{j}}\right)^{-4}P_{m}^{\frac{6}{5}}P^{-1} \ll \left(\prod_{j=1}^{r}P_{j}^{d_{j}}\right)^{-4}P_{m}^{-\delta}, \] car $ P=\prod_{j=1}^{m}P_{j}^{d_{j}}\geqslant P_{m}^{2} $ (puisque $ (m,d_{1})\neq (1,1) $). Par cons\'equent, on obtient \[  \tilde{N}_{\ee,m}(P_{1},...,P_{r}) \ll e_{0,m}^{4}\left(\prod_{i=1}^{n+r}e_{i}\right)^{-1}\left(\prod_{j=1}^{r}P_{j}^{n_{j}-d_{j}}\right)P_{m}^{-\delta}. \]
Par ailleurs, en utilisant les lemmes\;\ref{Jx3} et \;\ref{Sx3} on trouve : \begin{multline*}
\sum_{\forall j\in \{m+1,...,r\}, \;k_{j}\leqslant P_{j}}\left(\sum_{(x_{i})_{i\in I_{m}} \in \Phi_{m}(\kk)}\mathfrak{S}_{(x_{i})_{i\in I_{m}},\ee}J_{(x_{i},e_{i})_{i\in I_{m}},\kk}\right) \\ \left(\prod_{i\notin I_{m}}e_{i}\right)^{-1}\left(\prod_{j=m+1}^{r}k_{j}^{\sum_{i\notin I_{m}}a_{i,j}}\right)\left(\prod_{j=1}^{m}P_{j}^{n_{j}-d_{j}}\right) \\ \ll e_{0,m}^{2+\delta}\left(\prod_{i=1}^{n+r}e_{i}\right)^{-1}\left(\prod_{j=m+1}^{r}P_{j}^{d_{j}}\right)^{4+\delta}\left(\prod_{j=m+1}^{r}P_{j}^{n_{j}}\right)\left(\prod_{j=1}^{m}P_{j}^{n_{j}-d_{j}}\right) \\ \ll e_{0,m}^{4}\left(\prod_{i=1}^{n+r}e_{i}\right)^{-1}\left(\prod_{j=m+1}^{r}t_{j}^{d_{j}}\right)^{7}\left(\prod_{j=1}^{r}P_{j}^{n_{j}-d_{j}}\right)P^{-1+3\tilde{d}_{m}\theta+\delta} \\ \ll  e_{0,m}^{4}\left(\prod_{i=1}^{n+r}e_{i}\right)^{-1}\left(\prod_{j=1}^{r}P_{j}^{n_{j}-d_{j}}\right)\underbrace{P_{m}^{\frac{7}{5}-2}}_{\leqslant P_{m}^{-\delta}},
\end{multline*}
\'etant donn\'e que $ \left(\prod_{j=m+1}^{r}P_{j}^{1+5d_{j}}\right)<P_{m} $ puisque l'on a suppos\'e $  \sum_{j=m+1}^{r}(1+5d_{j})\frac{b_{j}}{b_{m}}<1 $, et que $ P^{-1+3\tilde{d}^{m}\theta}\ll \left(\prod_{j=m+1}^{r}P_{j}^{-\frac{3}{5}(1+5d_{j})}\right)P_{m}^{\frac{3}{5}}\underbrace{P^{-1}}_{\leqslant P_{m}^{-2}} $. d'o\`u le r\'esultat. 
\end{proof}

\subsection{ Cas particulier }

Nous allons ici traiter le cas particulier o\`u $ m=1 $ et $ d_{1}=1 $ (ce qui implique que la famille $ \tt=(t_{1,K}) $ est telle que $ t_{1,K}=1 $ si $ K=1 $ et $ t_{1,K}=0 $ sinon. Le polyn\^ome $ F $ est alors du type : \[ F(\xx)=\sum_{k\in J}A_{k}((x_{i})_{i\in I_{1}})x_{k} \] o\`u \[ J=\{i\in \{1,...,n+r\} \; |\; a_{i,1}=1\}. \] Il est facile de se ramener au cas o\`u $ J=I_{1} $. Nous devons alors calculer le nombre de points \`a coordonn\'ees born\'ees d'un r\'eseau hyperplan : en effet on a pour $ \kk=(k_{2},...,k_{r}) $ fix\'e, et $ (x_{i})_{i\in I_{1}}\in \Phi_{1}(\kk) $

\begin{multline*}
N_{(x_{i})_{i\in I_{1}},\ee}(P_{1})=\Card\left\{ 
 (x_{i})_{i\notin I_{1}}\in (\ZZ\setminus\{0\})^{n+r-s} \; | \; \sum_{k\in J}A_{k}((e_{i},x_{i})_{i\in I_{1}})e_{k}x_{k}=0,  \;\right.\\ \left. \forall j\in \{2,...,r\}, \;   \left\lfloor|(\ee.\xx)^{E(n+j)}|\right\rfloor= k_{j}, \;  \left\lfloor|(\ee.\xx)^{E(n+1)}|\right\rfloor\leqslant P_{1}, \; \right. \\ \left. \forall i\notin I_{1},  \; |x_{i}|\leqslant \frac{1}{e_{i}}\left(\prod_{j=2}^{r}(k_{j}+1)^{a_{i,j}}\right)\left(\left\lfloor|(\ee.\xx)^{E(n+1)}|\right\rfloor+1\right) \right\}.
\end{multline*}
Commen\c{c}ons par introduire la d\'efinition suivante : \begin{Def}
Soit $ S $ un sous-ensemble de $ \RR^{n} $, et soit $ c $ un entier tel que $ 0\leqslant c \leqslant d $. Pour $ M\in \NN $ et $ L>0 $, on dit que \emph{$ S $ appartient \`a $ \Lip(n,c,M,L) $} s'il existe $ M $ applications $ \phi : [0,1]^{n-c}\ra \RR^{d} $ v\'erifiant : \[ ||\phi(\xx)-\phi(\yy)||_{2}\leqslant L||\xx-\yy||_{2}, \] $ ||.||_{2} $ d\'esignant la norme euclidienne, telles que $ S $ soit recouvert par les images de ces applications. 
\end{Def}

On a le r\'esultat suivant (cf. \cite[Lemme 2]{MV}) : 
\begin{lemma}\label{geomnomb3}
Soit $ S\subset\RR^{n} $ un ensemble born\'e dont le bord $ \partial S $ appartient \`a $ \Lip(n,1,M,L) $. L'ensemble $ S $ est alors mesurable et si $ \Lambda $ est un r\'eseau de $ \RR^{n} $ de premier minimum successif $ \lambda_{1} $, on a \[ \left| \card(S\cap\Lambda)-\frac{\Vol(S)}{\det(\Lambda)}\right|\leqslant c(n)M\left(\frac{L}{\lambda_{1}}+1\right)^{n-1}, \] o\`u $ c(n) $ est une constante ne d\'ependant que de $ n $.
\end{lemma} 
Nous allons utiliser ce lemme pour \'evaluer, pour tout $ k_{1}\leqslant P_{1} $, le cardinal \begin{multline*}
N_{(x_{i})_{i\in I_{1}},\ee,k_{1}}=\Card\left\{ 
 (x_{i})_{i\notin I_{1}}\in (\ZZ\setminus\{0\})^{n+r-s} \; \left| \; \sum_{k\in J}A_{k}((e_{i}x_{i})_{i\in I_{1}})e_{k}x_{k}=0,  \;\right.\right.\\ \left. \forall j\in \{2,...,r\}, \;   \left\lfloor|(\ee.\xx)^{E(n+j)}|\right\rfloor= k_{j}, \;  \left\lfloor|(\ee.\xx)^{E(n+1)}|\right\rfloor\leqslant P_{1}, \; \right. \\ \left. \forall i\notin I_{1},  \; |x_{i}|\leqslant \frac{1}{e_{i}}\left(\prod_{j=1}^{r}(k_{j}+1)^{a_{i,j}}\right), \; \left\lfloor|(\ee.\xx)^{E(n+1)}|\right\rfloor\leqslant k_{1} \right\}.
\end{multline*}
Notons $ H_{\ee,(x_{i})_{i\in I_{1}}} $ l'hyperplan de $ \RR^{n+r-s} $ d\'efini par

\[ H_{\ee,(x_{i})_{i\in I_{1}}}=\{(x_{i})_{i\notin I_{1}}\in \RR^{n+r-s}\; |\; F(\ee.\xx)=\sum_{k\in J}A_{k}((e_{i}x_{i})_{i\in I_{1}})e_{k}x_{k}\} .\] On note par ailleurs $ C_{\ee,(x_{i})_{i\in I_{1}}} $ le polytope convexe $ \BB_{\ee,(x_{i})_{i\in I_{1}}}\cap H_{\ee,(x_{i})_{i\in I_{1}}} $ o\`u \[ \BB_{\ee,(x_{i})_{i\in I_{1}}}=\prod_{i\notin I_{1}}\left[-\frac{1}{e_{i}}\prod_{j=2}^{r}(k_{j}+1)^{a_{i,j}},\frac{1}{e_{i}}\prod_{j=2}^{r}(k_{j}+1)^{a_{i,j}}\right] \]    et $ \Lambda_{\ee,(x_{i})_{i\in I_{1}}} $ le r\'eseau $ \ZZ^{n+r-s}\cap H_{\ee,(x_{i})_{i\in I_{1}}}  $. Nous allons appliquer le lemme \ref{geomnomb3} \`a $ S=k_{1}C_{\ee,(x_{i})_{i\in I_{1}}} $ et $ \Lambda=\Lambda_{\ee,(x_{i})_{i\in I_{1}}} $ vus respectivement comme un sous-ensemble et un r\'eseau de $ H_{\ee,(x_{i})_{i\in I_{1}}} $ que l'on identifiera \`a $ \RR^{n+r-s+1} $. Nous allons pour cela montrer que le bord de $ S $ appartient \`a \[ \Lip(n+r-s,1,(n+r-s-1)\prod_{i\notin I_{1}}\frac{2}{e_{i}}\prod_{j=2}^{r}(k_{j}+1)^{a_{i,j}},k_{1}(n+r-s-2)\sqrt{n+r-s}) .\] Une face du polytope $ C_{\ee,(x_{i})_{i\in I_{1}}} $ est obtenue en prenant l'intersection d'une face $ \mathcal{F} $ de $ \BB_{\ee,(x_{i})_{i\in I_{1}}} $ avec $ H_{\ee,(x_{i})_{i\in I_{1}}} $. Consid\'erons par exemple l'intersection (suppos\'ee non vide) de la face $ \mathcal{F}=\{(x_{i})_{i\notin I_{1}}\in \BB_{\ee,(x_{i})_{i\in I_{1}}}\; |\; x_{i_{0}}=\frac{1}{e_{i_{0}}}\prod_{j=2}^{r}(k_{j}+1)^{a_{i_{0},j}}\} $ avec $ H_{\ee,(x_{i})_{i\in I_{1}}} $. La face $ \mathcal{F} $ peut \^etre subdivis\'ee en $ M_{0}=\prod_{\substack{i\notin I_{1}\\ i\neq i_{0}}}\frac{2}{e_{i}}\prod_{j=2}^{r}(k_{j}+1)^{a_{i,j}} $ sous-faces $ F_{1},...,F_{M_{0}} $ qui sont des cubes de taille $ 1 $ de centres not\'es $ c_{1},...,c_{M_{0}} $. Pour simplifier les notations, on pose \[ \forall  k\in J,\;  \alpha_{k}=
A_{k}((e_{i},x_{i})_{i\in I_{1}})  \] de sorte que $ H_{\ee,(x_{i})_{i\in I_{1}}} $ a pour \'equation \[\sum_{k\in J}\alpha_{k}e_{k}x_{k}=0 \] (les $ \alpha_{k} $ \'etant non tous nuls). 
On peut par cons\'equent subdiviser chaque face $ F_{l}\cap H_{\ee,(x_{i})_{i\in I_{1}}} $. Pour $ l\in \{1,...,M_{0}\} $ quelconque, on a pour tout $ (x_{i})_{i\in J}\in F_{l}\cap H_{\ee,(x_{i})_{i\in I_{1}}} $ : \[ \alpha_{i_{0}}\prod_{j=2}^{r}(k_{j}+1)^{a_{i_{0},j}}+\sum_{k\in J\setminus\{i_{0}\} }\alpha_{k}e_{k}x_{k}=0 \] avec $ \max_{k\in J\setminus \{i_{0}\}}|\alpha_{k}|\neq 0 $ puisque l'intersection $ F_{l}\cap H_{\ee,(x_{i})_{i\in I_{1}}} $ est non vide. Posons alors \[ |\alpha_{k_{0}}|=\max_{k\in J\setminus \{i_{0}\}}|\alpha_{k}| ,\] et on a $ z_{k_{0}}=-\frac{\alpha_{i_{0}}\prod_{j=2}^{r}(k_{j}+1)^{a_{i_{0},j}}}{e_{k_{0}}\alpha_{k_{0}}}-\sum_{k\in J\setminus\{i_{0},k_{0}\}}\frac{e_{k}\alpha_{k}}{e_{k_{0}}\alpha_{k_{0}}}x_{k} $, et on peut construire l'application \[ \begin{array}{rcl} 
\phi_{F_{l}} : [0,1]^{n+r-s-2}  & \ra & H_{\ee,(x_{i})_{i\in I_{1}}} \\ (t_{k})_{k\in J\setminus\{i_{0},k_{0}\}} & \mt & (x_{k})_{k\in j}
\end{array},   \]  avec \[ x_{k}=\left\{\begin{array}{rcl} \frac{1}{e_{i_{0}}}\prod_{j=2}^{r}(k_{j}+1)^{a_{i_{0},j}} & \mbox{si} & k=i_{0} \\ -\frac{\alpha_{i_{0}}\prod_{j=2}^{r}(k_{j}+1)^{a_{i_{0},j}}}{e_{k_{0}}\alpha_{k_{0}}}-\sum_{k\in J\setminus\{i_{0},k_{0}\}}\frac{e_{k}\alpha_{k}}{e_{k_{0}}\alpha_{k_{0}}}x_{k} & \mbox{si} & k=k_{0} \\ (\frac{1}{2}-t_{k})+c_{l} & \mbox{si} & k\in J\setminus\{i_{0},k_{0}\} \end{array}\right. \]
Il est alors clair que $ F_{l}\cap H_{\ee,(x_{i})_{i\in I_{1}}}\subset \phi_{F_{l}}([0,1]^{n+r-s-2}) $ et que \begin{multline*} ||\phi_{F_{l}}(\tt)-\phi_{F_{l}}(\tt')||_{2}  \leqslant \sqrt{n+r-s}||\phi_{F_{l}}(\tt)-\phi_{F_{l}}(\tt')||_{\infty} \\ \leqslant \sqrt{n+r-s}\max\left(1,\sum_{k\in J\setminus\{i_{0},k_{0}\}}\frac{|e_{k}\alpha_{k}|}{|e_{k_{0}}\alpha_{k_{0}}|}\right)||\tt-\tt'||_{\infty} \\  \leqslant (n+r-s-2)\sqrt{n+r-s}||\tt-\tt'||_{2}. \end{multline*}
Par cons\'equent, on a bien que le bord de $ C_{\ee,(x_{i})_{i\in I_{1}}} $ appartient \`a :  \[ \Lip(n+r-s,1,(n+r-s-1)\prod_{i\notin I_{1}}\frac{2}{e_{i}}\prod_{j=2}^{r}(k_{j}+1)^{a_{i,j}},(n+r-s-2)\sqrt{n+r-s}),\] et donc que le bord de $  k_{1}C_{\ee,(x_{i})_{i\in I_{1}}} $ appartient \`a  \[ \Lip(n+r-s,1,(n+r-s-1)\prod_{i\notin I_{1}}\frac{2}{e_{i}}\prod_{j=2}^{r}(k_{j}+1)^{a_{i,j}},k_{1}(n+r-s-2)\sqrt{n+r-s}) .\] De plus puisque $ \Lambda_{\ee,(x_{i})_{i\in I_{1}}}\subset \ZZ^{n+r-s} $ le premier minimum successif de ce r\'eseau est sup\'erieur ou \'egal \`a $ 1 $. Ainsi, puisque
\begin{equation}  N_{(x_{i})_{i\in I_{1}},\ee,k_{1}}=\card(\Lambda_{\ee,(x_{i})_{i\in I_{1}}}\cap k_{1}C_{\ee,(x_{i})_{i\in I_{1}}})\end{equation} le lemme \ref{geomnomb3} nous donne : \begin{multline*}\label{reseau3} N_{(x_{i})_{i\in I_{1}},\ee,k_{1}}=k_{1}^{n+r-s-1}\frac{\Vol(C_{\ee,(x_{i})_{i\in I_{1}}})}{\det(\Lambda_{\ee,(x_{i})_{i\in I_{1}}})} \\ +O_{n,r,s}\left(\left(\prod_{i\notin I_{1}}e_{i}\right)^{-1}\left(\prod_{j=2}^{r}k_{j}^{\sum_{i\notin I_{1}}a_{i,j}}\right)k_{1}^{n+r-s-2}\right), \end{multline*}  uniform\'ement pour tout $ (x_{i})_{i\in I_{1}} $. Remarquons que \[ \frac{\Vol(C_{\ee,(x_{i})_{i\in I_{1}}})}{\det(\Lambda_{\ee,(x_{i})_{i\in I_{1}}})} \ll \left(\prod_{i\notin I_{1}}e_{i}\right)^{-1}\left(\prod_{j=2}^{r}k_{j}^{\sum_{i\notin I_{1}}a_{i,j}}\right) \] (car $ \det(\Lambda_{\ee,(x_{i})_{i\in I_{1}}})\geqslant 1 $).\\

Nous pouvons alors en d\'eduire que \begin{multline*}
\Card\left\{ 
 (x_{i})_{i\notin I_{m}}\in (\ZZ\setminus\{0\})^{n+r-s} \; | \; \sum_{k\in J}A_{k}((e_{i}x_{i})_{i\in I_{1}})x_{k}=0,  \;\right.\\ \left. \forall j\in \{2,...,r\}, \;   \left\lfloor|(\ee.\xx)^{E(n+j)}|\right\rfloor= k_{j}, \;  \right. \\ \left. \forall i\notin I_{1},  \; |x_{i}|\leqslant \frac{1}{e_{i}}\left(\prod_{j=1}^{r}(k_{j}+1)^{a_{i,j}}\right), \; \left(\left\lfloor|(\ee.\xx)^{E(n+1)}|\right\rfloor= k_{1}\right) \right\} \\  =  N_{(x_{i})_{i\in I_{1}},\ee,k_{1}}-N_{(x_{i})_{i\in I_{1}},\ee,k_{1}-1} \\ =\left(k_{1}^{n+r-s-1}-(k_{1}-1)^{n+r-s-1}\right)\frac{\Vol(C_{\ee,(x_{i})_{i\in I_{1}}})}{\det(\Lambda_{\ee,(x_{i})_{i\in I_{1}}})} \\ +O_{n,r,s}\left(\left(\prod_{i\notin I_{1}}e_{i}\right)^{-1}\left(\prod_{j=2}^{r}k_{j}^{\sum_{i\notin I_{1}}a_{i,j}}\right)\left(k_{1}^{n+r-s-2}-(k_{1}-1)^{n+r-s-2}\right)\right) \\ = (n+r-s-1)k_{1}^{n+r-s-2}\frac{\Vol(C_{\ee,(x_{i})_{i\in I_{1}}})}{\det(\Lambda_{\ee,(x_{i})_{i\in I_{1}}})} \\ +O_{n,r,s}\left(\left(\prod_{i\notin I_{1}}e_{i}\right)^{-1}\left(\prod_{j=2}^{r}k_{j}^{\sum_{i\notin I_{1}}a_{i,j}}\right)k_{1}^{n+r-s-3}\right)
\end{multline*}

Par cons\'equent, en sommant sur les $ k_{1}\leqslant P_{1} $ on obtient un r\'esultat analogue \`a celui du corollaire\;\ref{corollairex3} :  \begin{multline}\label{corollairex'3}  N_{(x_{i})_{i\in I_{1}},\ee}(P_{1})= \frac{\Vol(C_{\ee,(x_{i})_{i\in I_{1}}})}{\det(\Lambda_{\ee,(x_{i})_{i\in I_{1}}})}P_{1}^{n+r-s-1} +O\left(\left(\prod_{i\notin I_{1}}e_{i}\right)^{-1}\left(\prod_{j=2}^{r}k_{j}^{\sum_{i\notin I_{1}}a_{i,j}}\right)P_{1}^{n+r-s-2}\right), \end{multline}
uniform\'ement pour tous $ \ee,(x_{i})_{i\in I_{1}},\kk=(k_{2},...,k_{r}) $. On peut en d\'eduire, en sommant sur les $ (x_{i})_{i\in I_{1}} $ : \begin{multline*}  \widetilde{N}_{\ee,1}(P_{1},...,P_{r})=\sum_{\kk\; |\; k_{j} \leqslant P_{j}}\sum_{(x_{i})_{i\in I_{1}}\in \phi_{1}(\kk)}\frac{\Vol(C_{\ee,(x_{i})_{i\in I_{1}}})}{\det(\Lambda_{\ee,(x_{i})_{i\in I_{1}}})}P_{1}^{n+r-s-1} \\ +O\left(\left(\prod_{i=1}^{n}e_{i}\right)^{-1}\sum_{\kk\; |\; k_{j} \leqslant P_{j}}\left(\prod_{j=2}^{r}k_{j}^{n_{j}}\right)P_{1}^{n_{1}-2}\right) \\ = \sum_{\kk\; |\; k_{j} \leqslant P_{j}}\sum_{(x_{i})_{i\in I_{1}}\in \phi_{1}(\kk)}\frac{\Vol(C_{\ee,(x_{i})_{i\in I_{1}}})}{\det(\Lambda_{\ee,(x_{i})_{i\in I_{1}}})}P_{1}^{n+r-s-1} \\ +O\left(\left(\prod_{i=1}^{n}e_{i}\right)^{-1}\left(\prod_{j=2}^{r}P_{j}^{n_{j}+1}\right)P_{1}^{n_{1}-2}\right). \end{multline*}
Par ailleurs, si l'on suppose $ b_{1}\geqslant \sum_{j=2}^{r}b_{j}(d_{j}+1)+\delta $, le terme d'erreur est alors du type $ O\left(\left(\prod_{i=1}^{n}e_{i}\right)^{-1}\left(\prod_{j=1}^{r}P_{j}^{n_{j}-d_{j}}\right)P_{r}^{-\delta}\right) $. D'o\`u le r\'esultat ci-dessous qui est un \'equivalent du th\'eor\`eme\;\ref{thmfinx3} : 

\begin{thm}\label{thmfinx'3}
Si l'on suppose $ P_{1}\geqslant P_{2}\geqslant ...\geqslant P_{r} $, $ P_{j}=P_{r}^{b_{j}} $, $ b_{1}\geqslant \sum_{j=2}^{r}b_{j}(d_{j}+1)+\delta $, et $ (m,d_{1})=(1,1) $, alors on a : \begin{multline*} \widetilde{N}_{\ee,1}(P_{1},...,P_{r})=
\sum_{\kk\; |\; k_{j} \leqslant P_{j}}\sum_{(x_{i})_{i\in I_{1}}\in \phi_{1}(\kk)}\frac{\Vol(C_{\ee,(x_{i})_{i\in I_{1}}})}{\det(\Lambda_{\ee,(x_{i})_{i\in I_{1}}})}P_{1}^{n+r-s-1} \\ + O\left(\left(\prod_{i=1}^{n}e_{i}\right)^{-1}\left(\prod_{j=1}^{r}P_{j}^{n_{j}-d_{j}}\right)\right).
\end{multline*}
\end{thm}

\section{Troisi\`eme \'etape}

Nous allons \`a pr\'esent utiliser les r\'esultats obtenus dans les sections pr\'ec\'edentes pour obtenir une formule asymptotique pour $ N_{U,\ee}(P_{1},...,P_{r}) $ valable pour tous $ P_{1},...,P_{r} $. Plus pr\'ecis\'ement, nous allons montrer pour un ouvert $ U $ bien choisi (voir\;\eqref{definitionouvertU3}) et une constante $ \mathfrak{m} $ que nous pr\'eciserons (voir\;\eqref{formuleprecisefrakm3}), le th\'eor\`eme ci-dessous :

\begin{thm}\label{thmNU3}
Si l'on suppose que $ n+r>\mathfrak{m}  $ alors 
\item On en d\'eduit \begin{multline}\label{formuleNU3}  N_{U,\ee}(P_{1},...,P_{r})=C_{\sigma,\ee}\left(\prod_{j=1}^{r}P_{j}^{n_{j}-d_{j}}\right) \\ +O\left(\left(e_{0}^{4+\delta}\left(\prod_{i=1}^{n+r}e_{i}\right)^{-1}+\left(\prod_{i=1}^{n+r}e_{i}\right)^{-\frac{1}{2}}\right)\left(\prod_{j=1}^{r}P_{j}^{n_{j}-d_{j}}\right)P_{r}^{-\delta}\right).\end{multline}
\end{thm}

Si $ \sigma\in \mathfrak{S}_{r} $ est tel que $ P_{\sigma(1)}\geqslant P_{\sigma(2)}\geqslant ...\geqslant P_{\sigma(r)} $ nous poserons pour tout $ j\in \{1,...,r\} $,  $ P_{\sigma(j)}=P_{\sigma(r)}^{b_{\sigma,j}} $ avec $ b_{\sigma,j}\geqslant 1 $ (et $ b_{\sigma,r}=1 $). Nous poserons alors $ \bb_{\sigma}=(b_{\sigma,1},...,b_{\sigma,r}) $. Nous noterons \'egalement pour tout $ (\sigma,m)\in \mathfrak{S}_{r}\times\{1,...,r\} $ : \[ I_{\sigma,m}=\{ i\in \{1,...,n+r\}\; |\; \forall j\in \{1,...,m\}, \; a_{i,\sigma(j)}=0 \}. \] Fixons par ailleurs un ensemble de degr\'es $ \tt^{(\sigma,m)} $ en les variables $ (x_{i})_{i\neq I_{m,\sigma}} $ pour chaque $ (\sigma,m)\in \mathfrak{S}_{r}\times\{1,...,r-1\} $ et notons

 \begin{multline}
\mathcal{A}_{\sigma,m}^{\lambda}=\left\{ (x_{i})_{i\in I_{\sigma,m}}\; | \; \forall j_{0}\in \{\sigma(1),...,\sigma(m)\},\; \forall K_{j_{0}}\in  \{1,...,d_{j_{0}}\},\; \right. \\ \left.  \dim V_{\sigma,m,(x_{i})_{i\in I_{\sigma,m}},\tt^{(\sigma,m)},(j_{0},K_{j_{0}})}^{\ast}<\dim V_{\sigma,m,\tt^{(\sigma,m)},(j_{0},K_{j_{0}})}^{\ast}-s+\lambda \right\}
\end{multline} o\`u l'on a pos\'e  \[ V_{\sigma,m,\tt^{(\sigma,m)},(j_{0},K_{j_{0}})}^{\ast}=\{ \xx\in \AA_{\CC}^{n+r}\; | \; \forall i\in J(j_{0},K_{j_{0}}), \; \frac{\partial F_{\tt^{(\sigma,m)}}}{\partial x_{i}}(\xx)=0 \}, \] \[ V_{\sigma,m,(x_{i})_{i\in I_{\sigma,m}},\tt^{(\sigma,m)},(j_{0},K_{j_{0}})}^{\ast}=\{ \xx\notin I_{\sigma,m}\; | \; \forall i\in J(j_{0},K_{j_{0}}), \; \frac{\partial F_{\tt^{(\sigma,m)}}}{\partial x_{i}}(\xx)=0 \}, \]
 avec \[ J(j_{0},K_{j_{0}})=\{ i \in \{ 1,...,n+r\} \; | \; a_{i,j_{0}}=K_{j_{0}}\}. \]

 Pour tout $ (\sigma,m)\in \mathfrak{S}_{r}\times\{1,...,r-1\} $, on note alors : \begin{multline*} g^{\sigma}_{m}(\bb_{\sigma},\delta))=\left(\frac{b_{\sigma,m}+\sum_{j=1}^{m}b_{\sigma,j}}{b_{\sigma,m}}\right)5\tilde{d}_{\sigma,m}\left(1-\sum_{j=m+1}^{r}(1+5d_{\sigma(j)})\frac{b_{\sigma,j}}{b_{\sigma,m}}-\delta\right)^{-1} \\ \left( \sum_{j=m+1}^{r}\left(1+d_{\sigma(j)}\left(3+\frac{\sum_{l=1}^{m}b_{\sigma,l}}{b_{\sigma,m}+\sum_{l=1}^{m}b_{\sigma,l}}+2\varepsilon\right)\right)\frac{b_{\sigma,j}}{b_{\sigma,m}}+2\delta\right), \end{multline*} (pour $ \sigma=\Id $, nous retrouvons $ g_{m}^{\Id}=g_{m} $ o\`u $ g_{m} $ a \'et\'e d\'efini par\;\eqref{fonctgm3}),

\begin{multline*} h^{\sigma}_{m,\tt^{(\sigma,m)}}\left((\bb_{\tau})_{\tau\in \mathfrak{S}_{r}}\right)=2^{\sum_{j=1}^{m}D_{\sigma(j)}^{(\sigma,m)}}(\frac{b_{\sigma,m}+3\sum_{j=1}^{m}b_{\sigma,j}}{b_{\sigma,m}}\tilde{d}_{\sigma,m}+ g^{\sigma}_{m}(\bb_{\sigma},\delta)) \\ +4r\max_{\tau\in \mathfrak{S}_{r}}\left(\sum_{j=1}^{r}b_{\tau,j}d_{\tau(j)}+\delta \right)+ \max_{(j_{0},K_{j_{0}})}\dim V_{\sigma,m,\tt^{(\sigma,m)},(j_{0},K_{j_{0}})}^{\ast}, \end{multline*} o\`u $ \tilde{d}_{\sigma,m}=(\sum_{j=1}^{m}d_{\sigma(j)})-1 $, et \begin{multline*} h^{\sigma}_{r,\tt^{(\sigma,r)}}( \bb_{\sigma})=2^{\sum_{j=1}^{r}D^{(\sigma,r)}_{\sigma(j)}}\max\{\left(2+5\sum_{j=1}^{r}b_{\sigma,j}\right)\tilde{d}_{\sigma}, (2\delta+1)\sum_{j=1}^{r}b_{\sigma,j}d_{\sigma(j)}\} \\ + \max_{(j_{0},K_{j_{0}})}\dim V_{\sigma,\tt^{(r,\sigma)},(j_{0},K_{j_{0}})}^{\ast}. \end{multline*} On pose alors 
\[ h_{(\tt^{(\sigma,m)})_{(\sigma,m)\in \mathfrak{S}_{r}\times\{1,...,r\}}}\left((\bb_{\tau})_{\tau\in \mathfrak{S}_{r}}\right)=\max_{m\in \{1,...,r\}}h^{\sigma}_{m,\tt^{(\sigma,m)}}\left((\bb_{\tau})_{\tau\in \mathfrak{S}_{r}}\right). \]
et on consid\`ere $ (\tilde{\bb}_{\sigma})_{\sigma\in  \mathfrak{S}_{r}} =(\tilde{b}_{\sigma,1},...,\tilde{b}_{\sigma,r})_{\sigma\in  \mathfrak{S}_{r}} $ les r\'eels $ \tilde{b}_{\sigma,j} $ minimisant la fonction $ h_{(\tt^{(\sigma,m)})_{(\sigma,m)\in \mathfrak{S}_{r}\times\{1,...,r\}}} $ sur le domaine d\'efini par : \[ \forall \sigma \in \mathfrak{S}_{r}, \; \; b_{\sigma,1}\geqslant ...\geqslant b_{\sigma,r-1}\geqslant b_{\sigma,r}=1 ,\] \begin{equation}\label{formulebsigmaj3} \forall \sigma \in \mathfrak{S}_{r}, \; \; \forall m\in \{1,...,r-1\}, \; \sum_{j=m+1}^{r}(1+5d_{\sigma(j)})\frac{b_{\sigma,j}}{b_{\sigma,m}}<1. \end{equation}On pose alors \begin{equation}\label{formuleprecisefrakm3}  \mathfrak{m}= h_{(\tt^{(\sigma,m)})_{(\sigma,m)\in \mathfrak{S}_{r}\times\{1,...,r\}}}\left((\tilde{\bb}_{\tau})_{\tau\in \mathfrak{S}_{r}}\right) .\end{equation} On choisit alors \begin{equation}\label{expressiondelambda3} \lambda=4r\max_{\sigma\in \mathfrak{S}_{r}}\left\lceil\sum_{j=1}^{r}\tilde{b}_{\sigma,j}d_{\sigma(j)}+\delta \right\rceil. \end{equation}

Remarquons qu'il est possible de majorer $ \mathfrak{m} $ par un r\'eel plus simple :
\begin{lemma}\label{bornefrakm3}
On a que \[ \mathfrak{m} \leqslant r(8.2^{\sum_{j=1}^{r}d_{j}}+4)\left(\prod_{j=1}^{r}(3+10d_{j})\right)\left(\sum_{j=1}^{r} d_{j}\right)+\max_{\substack{m\in \{1,...,r\} \\ \sigma\in \mathfrak{S}_{r} }}\max_{(j,k)\in \mathcal{C}_{m,\tau} }\dim V^{\ast}_{\sigma,m,\dd,(j,k)}. \]
\end{lemma}
\begin{proof}
Choisissons $ (b_{\sigma,j})_{\substack{j\in \{1,...,r-1\}\\ \sigma\in \mathfrak{S}_{r} }} $ tels que pour tout $ (m,\sigma)\in \{1,...,r\}\times \mathfrak{S}_{r} $ : \[ \sum_{j=m+1}^{r}(1+5d_{\sigma(j)})\frac{b_{\sigma,j}}{b_{\sigma,m}}=\frac{1}{2}. \] Ceci est possible si et seulement si pour tout $  (m,\sigma) $ : \[ b_{\sigma,m}=(2+10d_{\sigma(r)})\prod_{j=m+1}^{r-1}(3+10d_{\sigma(j)}), \] donc en particulier, pour tous $ j<m $ : \[ b_{\sigma,m}=b_{\sigma,j}\prod_{k=j}^{m}(3+10d_{\sigma(k)}). \] 
On a alors \begin{align*} \frac{b_{\sigma,m}+\sum_{j=1}^{m}b_{\sigma,j}}{b_{\sigma,m}} & = 2+\sum_{j=1}^{m-1}\prod_{k=j}^{m}(3+10d_{\sigma(k)}) \\ & \leqslant m\prod_{j=1}^{m}(3+10d_{\sigma(k)}) \end{align*}
ainsi que  \begin{equation*} \frac{b_{\sigma,m}+3\sum_{j=1}^{m}b_{\sigma,j}}{b_{\sigma,m}} \leqslant 3m\prod_{j=1}^{m}(3+10d_{\sigma(k)}). \end{equation*}

On remarque enfin que pour $ \delta>0 $ choisi assez petit : \begin{multline*}
 \sum_{j=m+1}^{r}\left(1+d_{\sigma(j)}\left(3+\frac{\sum_{l=1}^{m}b_{\sigma,l}}{b_{\sigma,m}+\sum_{l=1}^{m}b_{\sigma,l}}+2\varepsilon\right)\right)\frac{b_{\sigma,j}}{b_{\sigma,m}}+2\delta \\ \leqslant \sum_{j=m+1}^{r}(1+4d_{\sigma(j)})\frac{b_{j}}{b_{m}}\leqslant \frac{1}{2}-\delta
\end{multline*}

Ainsi : \[ g^{\sigma}_{m}(\bb_{\sigma},\delta)\leqslant 5m\tilde{d}_{\sigma,m}\prod_{j=1}^{m}(3+10d_{\sigma(j)}).
\]
Par ailleurs, on a que, pour tout $ \tau \in \mathfrak{S}_{r} $, \begin{align*}
\sum_{j=1}^{r}b_{\tau,j}d_{\tau(j)}+\delta & \leqslant d_{\tau(r)}+\sum_{m=1}^{r-1}(2+10d_{\tau(r)})d_{\tau(m)}\prod_{j=m+1}^{r-1}(3+10d_{\tau(j)}) \\ & \leqslant  (2+10d_{\tau(r)})\sum_{m=1}^{r}d_{\tau(m)}\prod_{j=m+1}^{r-1}(3+10d_{\tau(j)}) \\ & \leqslant \left(\prod_{j=1}^{r}(3+10d_{\tau(j)})\right)(\sum_{m=1}^{r}d_{\tau(m)}) \\ & \leqslant \left(\prod_{j=1}^{r}(3+10d_{\tau(j)})\right)(\sum_{m=1}^{r}d_{\tau(m)}).
\end{align*}
De plus, \[ \sum_{j=1}^{r}D_{\sigma(j)}^{(\sigma,m)}\leqslant \sum_{j=1}^{r}d_{j}. \] Nous d\'eduisons de ces calculs que pour tous $ (m,\sigma) $ \begin{multline*} h^{\sigma}_{m,\tt^{(\sigma,m)}}\left((\bb_{\tau})_{\tau\in \mathfrak{S}_{r}}\right) \\ \leqslant r(8.2^{\sum_{j=1}^{r}d_{j}}+4)\left(\prod_{j=1}^{r}(3+10d_{j})\right)\left(\sum_{j=1}^{r} d_{j}\right)+\dim V^{\ast}_{m,\sigma,\dd,(j,k)}.\end{multline*}
\end{proof}

\`A partir d'ici, on supposera, sauf indication contraire, que $ P_{1}\geqslant P_{2}\geqslant ...\geqslant P_{r} $ (les autres cas se traitant de la m\^eme mani\`ere). Pour simplifier les notations nous noterons $ b_{1}=b_{\Id,1},...,b_{r}=b_{\Id,r} $, $ \tilde{b}_{1}=\tilde{b}_{\Id,1},...,\tilde{b}_{r}=\tilde{b}_{\Id,r} $ et $ \tt^{(m)}=\tt^{(\Id,m)} $. L'objectif de ce qui va suivre est de donner trouver une formule asymptotique pour $  N_{\ee}(P_{1},...,P_{r})  $ valable pour tous $ P_{1},...,P_{r} $ tels que $ P_{1}\geqslant P_{2}\geqslant ...\geqslant P_{r} $.\\

Commen\c{c}ons par d\'emontrer le r\'esultat ci-dessous qui s'av\'erera utile pour la suite : 
 
 \begin{prop}\label{ferme3}
 Avec les notations de la section pr\'ec\'edente, l'ensemble $ \mathcal{A}_{m}^{\lambda} $ est un ouvert de Zariski de $ \AA_{\CC}^{s} $, et on a de plus que \[ \dim(\mathcal{A}_{m}^{\lambda})^{c} \leqslant \max\{0,s-\lambda \}. \]
 \end{prop}
\begin{proof}
Il suffit de montrer que pour $ (j_{0},K_{j_{0}}) $ fix\'e quelconque, l'ensemble \[ \mathcal{A}_{m,(j_{0},K_{j_{0}})}^{\lambda,c}= \left\{(x_{i})_{i\in I_{m}}\; | \;   \dim V_{m,(x_{i})_{i\in I_{m}},\tt^{(m)},(j_{0},K_{j_{0}})}^{\ast}\geqslant\dim V_{m,\tt^{(m)},(j_{0},K_{j_{0}})}^{\ast}-s+\lambda \right\} \] est un ferm\'e de Zariski de $ \AA_{\CC}^{s} $ de dimension inf\'erieure ou \'egale \`a $ s-\lambda  $. \\

Consid\'erons le sous-r\'eseau $ N' $ de $ \ZZ^{n} $ d\'efini par \[ N'=\bigoplus_{\substack{i \in \{1,...,n\} \\ i\notin I_{m}}}\ZZ v_{i} .\] Par d\'efinition de $ I_{m} $, on a que, pour tout $ j\in \{1,...,m\} $, \[ v_{n+j}=-\sum_{\substack{i \in \{1,...,n\} \\ i\notin I_{m}}}a_{i,j}v_{i}. \] On consid\`ere alors l'\'eventail $ \Delta' $ de $ N'_{\RR}=\bigoplus_{\substack{i \in \{1,...,n\} \\ i\notin I_{m}}}\RR v_{i} $ d\'efini par les c\^ones $ (\sigma\cap N'_{\RR})_{\sigma\in \Delta} $. Cet \'eventail a alors pour ar\^etes $ (\RR^{+}v_{i})_{\substack{i \in \{1,...,n\} \\ i\notin I_{m}}} $ et $ (\RR^{+}v_{n+j})_{j\in \{1,...,m\}} $. Notons alors $ X_{m} $ la vari\'et\'e torique (compl\`ete et lisse) d\'efinie par $ N' $ et $ \Delta' $. Cette vari\'et\'e est de dimension $ n+r-s-m $. Remarquons que cette vari\'et\'e est en fait isomorphe \`a l'adh\'erence des orbites d'un point de l'orbite ouverte de $ X $ sous l'action d'un sous-tore $ T_{m} $ (de dimension $ n+r-s-m $) du tore $ T $ de $ X $.\\

Notons $ Y_{(j_{0},K_{j_{0}})} $ le ferm\'e de $ \AA_{\CC}^{s}\times X_{m}(\CC) $ d\'efini par \[ Y_{(j_{0},K_{j_{0}})}=\{ ((x_{i})_{i\in I_{m}},(x_{i})_{i\notin I_{m}})\in \AA_{\CC}^{s}\times X_{m}(\CC)\; | \;  \forall i\in J(j_{0},K_{j_{0}}), \; \frac{\partial F_{\tt^{(m)}}}{\partial x_{i}}(\xx)=0 \}. \]
 La projection canonique $ \pi : Y_{(j_{0},K_{j_{0}})}\subset \AA_{\CC}^{s}\times X_{m}(\CC) \ra  \AA_{\CC}^{s} $ est un morphisme projectif donc ferm\'e. Par cons\'equent, d'apr\`es  \cite[Corollaire 13.1.5]{G-D}, \[ \{(x_{i})_{i\in I_{m}}\in \AA_{\CC}^{s} \; | \; \dim \underbrace{Y_{(j_{0},K_{j_{0}}),(x_{i})_{i\in I_{m}}}}_{=\pi^{-1}((x_{i})_{i\in I_{m}}}\geqslant \dim V_{m,\tt^{(m)},(j_{0},K_{j_{0}})}^{\ast}-s+\lambda-m \} \] est un ferm\'e, et puisque $ \dim Y_{(j_{0},K_{j_{0}}),(x_{i})_{i\in I_{m}}}=\dim V_{m,(x_{i})_{i\in I_{m}},\tt,(j_{0},K_{j_{0}})}^{\ast}-m $, on trouve bien que $ \mathcal{A}_{m,(j_{0},K_{j_{0}})}^{\lambda} $ est un ferm\'e de Zariski. \\

Remarquons \`a pr\'esent que d'une part, $ X_{m} $ est le quotient d'un ouvert $ U_{m} $ de $ \AA^{n+r-s} $ par l'action d'un tore de dimension $ m $. Par cons\'equent \[ \dim Y_{(j_{0},K_{j_{0}})}=\dim V_{m,\tt^{(m)},(j_{0},K_{j_{0}})}^{\ast}-m.\] D'autre part, \[ Y_{(j_{0},K_{j_{0}})}\cap (\mathcal{A}_{m,(j_{0},K_{j_{0}})}^{\lambda,c}\times  X_{m}(\CC))=\bigsqcup_{(x_{i})_{i\in I_{m}}\in \mathcal{A}_{m,(j_{0},K_{j_{0}})}^{\lambda,c}}\pi^{-1}((x_{i})_{i\in I_{m}}), \] et on a alors que \begin{multline*} \dim(\mathcal{A}_{m,(j_{0},K_{j_{0}})}^{\lambda,c})+ \dim V_{m,\tt^{(m)},(j_{0},K_{j_{0}})}^{\ast}-s+ \lambda -m\leqslant \dim Y_{(j_{0},K_{j_{0}})} \\=\dim V_{m,\tt^{(m)},(j_{0},K_{j_{0}})}^{\ast}-m, \end{multline*}  ce qui implique : \[ \dim\mathcal{A}_{m,(j_{0},K_{j_{0}})}^{\lambda,c}\leqslant s-\lambda, \] d'o\`u le r\'esultat. 

\end{proof}
\begin{rem}
On peut montrer de fa\c{c}on analogue que la m\^eme propri\'et\'e est vraie pour tous les $ \mathcal{A}_{m,\sigma}^{\lambda} $. 
\end{rem}

Nous aurons \'egalement besoin du lemme ci-dessous : 

\begin{lemma}\label{estimgrossiere3}
Si l'on consid\`ere $ J\subset \{1,...,n+r\} $ de cardinal not\'e $ t $, et si $ F $ est un ferm\'e de l'espace affine $ \AA_{\CC}^{t}=\{(x_{i})_{i\in J}\} $ tel que $ \dim F\leqslant t-\alpha $ (pour $ \alpha\in \NN $), on a alors \begin{multline*} \Card\{(x_{i})_{i\in J}\in F\cap \ZZ^{J}\; |\; \forall i \in J, \; |x_{i}|\leqslant T_{i}\} \\ \ll \left(\prod_{i\in J}e_{i}\right)^{-\frac{1}{2}}\left(\prod_{j=1}^{r}P_{j}^{\sum_{i\in J}a_{i,j}}\right) P_{j_{0}}^{-\frac{\alpha}{2}},
\end{multline*}
(la constante implicite ne d\'ependant que du degr\'e de $ F $) o\`u \[ j_{0}=\min\{j\in \{1,...,r\}\; |\; \forall i \in J, \; \exists l\leqslant j\; |\; a_{i,l}\neq 0 \}.  \]
\end{lemma}
\begin{proof}

En utilisant par exemple la d\'emonstration de \cite[Th\'eor\`eme 3.1]{Br}), on montre que (en remarquant que, pour tout $ i\in J $, $ T_{i}\geqslant P_{j_{0}} $) : 
\begin{multline*} \Card\{(x_{i})_{i\in J}\in F\cap \ZZ^{J}\; |\; \forall i \in J, \; |x_{i}|\leqslant T_{i}\}  \ll \max_{i_{1},...,i_{t-\alpha}\in J}\prod_{l=1}^{t-\alpha}T_{i_{l}} \\ \sum_{i_{1},...,i_{\alpha}\in J}\left(\prod_{i\in J}T_{i}\right)\left(\prod_{l=1}^{\alpha}T_{i_{l}}\right)^{-1} \\ \ll \sum_{i_{1},...,i_{\alpha}\in J}\left(\prod_{l\in \{i_{1},...,i_{\alpha}\}}e_{l}\right)\left(\prod_{i\in J}e_{i}\right)^{-1}\left(\prod_{j=1}^{r}P_{j}^{\sum_{i\in J}a_{i,j}}\right) P_{j_{0}}^{-\alpha}.
\end{multline*}
 Si l'on suppose que $ \sum_{i_{1},...,i_{\alpha}\in J}\left(\prod_{l\in \{i_{1},...,i_{\alpha}\}}e_{l}\right)\ll \left(\prod_{i\in J}e_{i}\right)^{\frac{1}{2}}P_{j_{0}}^{\frac{\alpha}{2}} $ ce terme peut \^etre major\'e par : \[ \left(\prod_{i\in J}e_{i}\right)^{-\frac{1}{2}}\left(\prod_{j=1}^{r}P_{j}^{\sum_{i\in J}a_{i,j}}\right) P_{r}^{-\alpha/2}. \] Si au contraire $ \sum_{i_{1},...,i_{\alpha}\in J}\left(\prod_{l\in \{i_{1},...,i_{\alpha}\}}e_{l}\right)\gg \left(\prod_{i\in J}e_{i}\right)^{\frac{1}{2}}P_{j_{0}}^{\frac{\alpha}{2}} $ on a alors que $ \left(\prod_{i\in J}e_{i}\right)^{\frac{1}{2}}P_{j_{0}}^{-\frac{\alpha}{2}}\gg 1 $ et on obtient trivialement la m\^eme majoration. D'o\`u le r\'esultat. 
\end{proof}

D\'emontrons le r\'esultat suivant : 
\begin{lemma}\label{init13}
On suppose que $ n+r>\mathfrak{m} $. On a alors 
\begin{multline*}
 \sum_{k_{r}\leqslant P_{r}}\left(\sum_{(x_{i})_{i\in I_{m}} \in \Phi_{r-1}(k_{r})}\mathfrak{S}_{(x_{i})_{i\in I_{r-1}},\ee}J_{(x_{i},e_{i})_{i\in I_{r-1}},\kk}\right) \left(\prod_{i\notin I_{r-1}}e_{i}\right)^{-1}k_{r}^{\sum_{i\notin I_{r-1}}a_{i,r}} \\ = C_{\sigma,\ee}P_{r}^{n_{r}-d_{r}}+ O\left(\left(e_{0}^{4+\delta}\left(\prod_{i=1}^{n+r}e_{i}\right)^{-1}+\left(\prod_{i=1}^{n+r}e_{i}\right)^{-\frac{1}{2}}\right)P_{r}^{n_{r}-d_{r}-\delta}\right).
\end{multline*}

\end{lemma}
\begin{proof}
On choisit $ P_{1},...,P_{r-1} $ tels que $ b_{j}=\tilde{b}_{j} $ pour tout $ j\in \{1,...,r\} $. On a par d\'efinition de $ K_{r-1} $ : \[ K_{r-1}=(n+r-\lambda- \max_{(j_{0},K_{j_{0}})}\dim V_{r-1,\tt^{(r-1)},(j_{0},K_{j_{0}})}^{\ast}-\varepsilon)/2^{\sum_{j=1}^{r-1}D_{j}^{(r-1)}}.\] La condition $ n+r>\mathfrak{m}\geqslant h_{r-1,\tt^{(r-1)}}(\bb,(\tilde{\bb}_{\sigma})_{\sigma\neq \Id}) $ implique alors (par d\'efinition de $ h_{r-1,\tt^{(r-1)}} $) que \[ K_{r-1}-\frac{b_{r-1}+3\sum_{j=1}^{r-1}b_{j}}{b_{r-1}}\tilde{d}_{r-1}>g_{r-1}(\bb,\delta), \] et nous pouvons alors appliquer le th\'eor\`eme\;\ref{thmfinx3}, ce qui nous donne :  \begin{multline}\label{form13}  \tilde{N}_{\ee,r-1}(P_{1},...,P_{r})= \sum_{k_{r}\leqslant P_{r}}\left(\sum_{(x_{i})_{i\in I_{m}} \in \Phi_{r-1}(k_{r})}\mathfrak{S}_{(x_{i})_{i\in I_{r-1}},\ee}J_{(x_{i},e_{i})_{i\in I_{r-1}},k_{r}}\right) \\ \left(\prod_{i\notin I_{r-1}}e_{i}\right)^{-1}k_{r}^{\sum_{i\notin I_{r-1}}a_{i,r}}\left(\prod_{j=1}^{r-1}P_{j}^{n_{j}-d_{j}}\right) \\ + O\left(\left(e_{0,r-1}^{4}\left(\prod_{i=1}^{n+r}e_{i}\right)^{-1}+\left(\prod_{i=1}^{n+r}e_{i}\right)^{-\frac{1}{2}}\right)\left(\prod_{j=1}^{r}P_{j}^{n_{j}-d_{j}}\right)P_{r-1}^{-\delta}\right).\end{multline}
En observant par ailleurs que $ n+r>\mathfrak{m}\geqslant h_{r,\tt^{(r)}}(\bb) $ implique $ K>\max\{(5\sum_{j=1}^{r}b_{j}+2)\tilde{d}, (2\delta+1)\sum_{j=1}^{r}b_{j}d_{j}\} $, on donc appliquer le th\'eor\`eme\;\ref{propfin3} et on a : \begin{multline}\label{form23} N_{\ee}(P_{1},...,P_{r})=C_{\sigma,\ee}\left(\prod_{j=1}^{r}P_{j}^{n_{j}-d_{j}}\right)
\\ +O\left(\left(e_{0}^{4+\delta}\left(\prod_{i=1}^{n+r}e_{i}\right)^{-1}+\left(\prod_{i=1}^{n+r}e_{i}\right)^{-\frac{1}{2}}\right)\left(\prod_{j=1}^{r}P_{j}^{n_{j}-d_{j}-\delta}\right)\right). \end{multline}
D'autre part, on observe que : \begin{multline*} N_{\ee}(P_{1},...,P_{r})=\tilde{N}_{\ee,r-1}(P_{1},...,P_{r}) \\ +O\left(\sum_{\substack{(x_{i})_{i\in I_{r-1}}\in (\mathcal{A}_{r-1}^{\lambda})^{c}\cap \ZZ^{s}\\ |x_{i}|\leqslant \frac{1}{e_{i}}P_{r}^{a_{i,r}}}}\left(\prod_{i\notin I_{r-1}}e_{i}\right)^{-1}\left(\prod_{j=1}^{r-1}P_{j}^{n_{j}}\right)P_{r}^{\sum_{i\notin I_{r-1}}a_{i,r}}\right). \end{multline*}
En utilisant le lemme\;\ref{estimgrossiere3}, on remarque que le terme d'erreur de la formule ci-dessus peut \^etre major\'e par \begin{multline*} \left(\prod_{i\notin I_{r-1}}e_{i}\right)^{-1}\left(\prod_{j=1}^{r-1}P_{j}^{n_{j}}\right)P_{r}^{\sum_{i\notin I_{r-1}}a_{i,r}} \left(\prod_{i\in I_{r-1}}e_{i}\right)^{-\frac{1}{2}}P_{r}^{\sum_{i\in I_{r-1}}a_{i,r}}P_{r}^{-\frac{\lambda}{2}}. \end{multline*}
 On a donc finalement :  \begin{align*} N_{\ee}(P_{1},...,P_{r}) &= \tilde{N}_{\ee,r-1}(P_{1},...,P_{r})  +O\left( \left(\prod_{i=1}^{n+r}e_{i}\right)^{-\frac{1}{2}}\left(\prod_{j=1}^{r}P_{j}^{n_{j}}\right) P_{r}^{-\lambda/2}\right)\\ & = \tilde{N}_{\ee,r-1}(P_{1},...,P_{r})  +O\left( \left(\prod_{i=1}^{n+r}e_{i}\right)^{-\frac{1}{2}}\left(\prod_{j=1}^{r}P_{j}^{n_{j}-d_{j}}\right)P_{r}^{-\delta} \right), \end{align*} puisque $  \lambda\geqslant 2r\lceil \sum_{j=1}^{r}b_{j}d_{j}+\delta\rceil  $. En rempla\c{c}ant $  N_{\ee}(P_{1},...,P_{r}) $ et $ \tilde{N}_{\ee,r-1}(P_{1},...,P_{r}) $ par leurs expressions dans\;\eqref{form13},\;\eqref{form23} et en simplifiant par $ \left(\prod_{j=1}^{r-1}P_{j}^{n_{j}-d_{j}}\right) $, on obtient le r\'esultat du lemme (qui ne d\'epend plus des valeurs de $ P_{1},...,P_{r-1} $ choisies).
\end{proof}
Nous sommes alors en mesure de d\'emontrer le r\'esultat suivant : 
\begin{lemma}\label{init23}
Si l'on suppose que $ n+r>\mathfrak{m}  $ et que $ \frac{b_{j}}{b_{r-1}}\leqslant \frac{\tilde{b}_{j}}{\tilde{b}_{r-1}} $ pour tout $ j\in \{1,...,r\} $ (en particulier $ \frac{1}{b_{r-1}}\leqslant \frac{1}{\tilde{b}_{r-1}} $) nous avons alors la formule asymptotique :  \begin{multline}  \tilde{N}_{\ee,r-1}(P_{1},...,P_{r})=C_{\sigma,\ee}\left(\prod_{j=1}^{r}P_{j}^{n_{j}-d_{j}}\right) \\ +O\left(\left(e_{0}^{4+\delta}\left(\prod_{i=1}^{n+r}e_{i}\right)^{-1}+\left(\prod_{i=1}^{n+r}e_{i}\right)^{-\frac{1}{2}}\right)\left(\prod_{j=1}^{r}P_{j}^{n_{j}-d_{j}}\right)P_{r}^{-\delta}\right).\end{multline}
\end{lemma}
\begin{proof}
 Puisque $ \frac{b_{j}}{b_{r-1}}\leqslant \frac{\tilde{b}_{j}}{\tilde{b}_{r-1}} $ pour tout $ j\in \{1,...,r\} $, on a  \begin{multline*}
g_{r-1}(\bb,\delta) =\left(\frac{b_{r-1}+\sum_{j=1}^{r-1}b_{j}}{b_{r-1}}\right)5\tilde{d}_{r-1}\left(1-(1+5d_{r})\frac{1}{b_{r-1}}-\delta\right)^{-1}\\ \left( \left(1+d_{r}\left(3+\frac{\sum_{l=1}^{r-1}b_{l}}{b_{r-1}+\sum_{l=1}^{r-1}b_{l}}+2\varepsilon\right)\right)\frac{1}{b_{r-1}}+2\delta\right) \\= 5\tilde{d}_{r-1}\left(1-(1+5d_{r})\frac{1}{b_{r-1}}-\delta\right)^{-1}\\ \left( (2\delta+\frac{1+(3+2\varepsilon)d_{r}}{b_{r-1}})\left(1+\sum_{j=1}^{r-1}\frac{b_{j}}{b_{r-1}}\right)+\frac{d_{r}}{b_{r-1}}\frac{\sum_{l=1}^{r-1}b_{l}}{b_{r-1}}\right)\\ \leqslant 5\tilde{d}_{r-1}\left(1-(1+5d_{r})\frac{1}{\tilde{b}_{r-1}}-\delta\right)^{-1}\\ \left( (2\delta+\frac{1+(3+2\varepsilon)d_{r}}{\tilde{b}_{r-1}})\left(1+\sum_{j=1}^{r-1}\frac{\tilde{b}_{j}}{\tilde{b}_{r-1}}\right)+\frac{d_{r}}{\tilde{b}_{r-1}}\frac{\sum_{l=1}^{r-1}\tilde{b}_{l}}{\tilde{b}_{r-1}}\right)=g_{r-1}(\tilde{\bb},\delta)
\end{multline*}
o\`u $ \tilde{\bb}=(\tilde{b}_{1},...,\tilde{b}_{r-1}) $. Par cons\'equent, puisque $ n+r>h_{r-1,\tt^{(r-1)}}(\tilde{\bb},(\tilde{\bb}_{\tau})_{\tau\neq \Id}) $, on a \begin{align*}
K_{r-1} & =(n+r-\lambda-\max_{(j_{0},K_{j_{0}})}\dim V_{r-1,\tt^{(r-1)},(j_{0},K_{j_{0}})}^{\ast}-\varepsilon)/2^{\sum_{j=1}^{r-1}D_{j}^{(r-1)}} \\ &>\left(1+3\sum_{j=1}^{r-1}\frac{\tilde{b}_{j}}{\tilde{b}_{r-1}}\right)\tilde{d}_{r-1}+ g_{r-1}(\tilde{\bb},\delta) \\ &> \left(1+3\sum_{j=1}^{r-1}\frac{b_{j}}{b_{r-1}}\right)\tilde{d}_{r-1}+ g_{r-1}(\bb,\delta),
\end{align*}
et puisque, d'apr\`es la formule\;\eqref{formulebsigmaj3}  \[  (1+5d_{r})\frac{b_{r}}{b_{r-1}}\leqslant (1+5d_{r})\frac{\tilde{b}_{r}}{\tilde{b}_{r-1}}<1,  \] nous pouvons donc appliquer le th\'eor\`eme\;\ref{thmfinx3} qui, combin\'e avec le lemme pr\'ec\'edent, donne : 
\begin{multline*} \tilde{N}_{\ee,r-1}(P_{1},...,P_{r})=C_{\sigma,\ee}\left(\prod_{j=1}^{r}P_{j}^{n_{j}-d_{j}}\right) \\ +O\left(\left(e_{0}^{4+\delta}\left(\prod_{i=1}^{n+r}e_{i}\right)^{-1}+\left(\prod_{i=1}^{n+r}e_{i}\right)^{-\frac{1}{2}}\right)\left(\prod_{j=1}^{r}P_{j}^{n_{j}-d_{j}}\right)P_{r}^{-\delta}\right). \end{multline*}
\end{proof}
D\'efinissons \`a pr\'esent l'ouvert $ U $ de $ \AA^{n+r} $ par : \begin{equation}\label{definitionouvertU3} U=\bigcap_{\sigma\in \mathfrak{S}_{r}}\bigcap_{m\in \{1,...,r-1\}}\{\xx\in \AA_{\CC}^{n+r}\; | \; (x_{i})_{i\in I_{m,\sigma}}\in \mathcal{A}_{m,\sigma}^{\lambda}\}.   \end{equation}
\begin{lemma}
L'ouvert $ U $ v\'erifie la propri\'et\'e suivante : \[  \forall \xx\in X_{1}, \;\;(\xx\in U)\Rightarrow( \forall \ss=(s_{1},...,s_{r})\in (\CC^{\ast})^{r},\; \ss.\xx\in U). \]
\end{lemma}
\begin{proof}
Par sym\'etrie il nous suffit de montrer que, pour tout $ m $ : \[ \{\xx\in \AA_{\CC}^{n+r}\; | \; (x_{i})_{i\in I_{m,\sigma}}\in \mathcal{A}_{m,\sigma}^{\lambda}\} \] est stable par l'action de $ (\CC^{\ast})^{r} $. Consid\'erons donc un \'el\'ement $ \xx\in \AA_{\CC}^{n+r} $ tel que $ (x_{i})_{i\in I_{m,\sigma}}\in \mathcal{A}_{m,\sigma}^{\lambda} $. Rappelons que par d\'efinition $ \mathcal{A}_{m}^{\lambda} $ est l'ensemble
\[ \left\{ (x_{i})_{i\in I_{m}}\; | \; \forall (j_{0},K_{j_{0}}),\;  \dim V_{m,(x_{i})_{i\in I_{m}},\tt^{(m)},(j_{0},K_{j_{0}})}^{\ast}<\dim V_{m,\tt^{(m)},(j_{0},K_{j_{0}})}^{\ast}-s+\lambda \right\}. \]
Consid\'erons \`a pr\'esent un \'el\'ement $ \ss=(s_{1},...,s_{r})\in (\CC^{\ast})^{r} $, et remarquons que l'application $ (x_{i})_{i\in I_{m}}\mt ((\prod_{j=m+1}^{r}s_{j}^{a_{i,j}})x_{i})_{i\in I_{m}} $ r\'ealise un isomorphisme de $ V_{m,((\prod_{j=m+1}^{r}s_{j}^{a_{i,j}})x_{i})_{i\in I_{m}},\tt^{(m)},(j_{0},K_{j_{0}})}^{\ast} $ sur $ V_{m,(x_{i})_{i\in I_{m}},\tt^{(m)},(j_{0},K_{j_{0}})}^{\ast} $. On a donc \[ \dim V_{m,((\prod_{j=1}^{r}s_{j}^{a_{i,j}})x_{i})_{i\in I_{m}},\tt^{(m)},(j_{0},K_{j_{0}})}^{\ast} = \dim V_{m,(x_{i})_{i\in I_{m}},\tt^{(m)},(j_{0},K_{j_{0}})}^{\ast}, \] et donc que $ \{\xx\in \AA_{\CC}^{n+r}\; | \; (x_{i})_{i\in I_{m,\sigma}}\in \mathcal{A}_{m,\sigma}^{\lambda}\} $ est stable par l'action de $ \ss\in(\CC^{\ast})^{r} $. \end{proof}

En notant \begin{multline}
 N_{U,\ee}(P_{1},...,P_{r})= \Card\left\{ 
 \xx\in (\ZZ\setminus\{0\})^{n+r}\cap U \; | \; F(\ee.\xx)=0, \; \right. \\ \forall j\in \{1,...,r\}, \;  \left\lfloor|(\ee.\xx)^{E(n+j)}|\right\rfloor\leqslant P_{j}, \;  \forall i\in \{1,...,n+r\},  \;  \\ \left.|x_{i}|\leqslant \frac{1}{e_{i}}\prod_{j=1}^{r}(\left\lfloor|(\ee.\xx)^{E(n+j)}|\right\rfloor+1)^{a_{i,j}} \right\},
\end{multline}
on d\'eduit du lemme pr\'ec\'edent le r\'esultat ci-dessous : 
\begin{lemma}\label{init33}
Si l'on suppose que $ n+r>\mathfrak{m}  $ et que $ \frac{b_{j}}{b_{r-1}}\leqslant \frac{\tilde{b}_{j}}{\tilde{b}_{r-1}} $ pour tout $ j\in \{1,...,r\} $ nous avons alors :  \begin{multline}  N_{U,\ee}(P_{1},...,P_{r})=C_{\sigma,\ee}\left(\prod_{j=1}^{r}P_{j}^{n_{j}-d_{j}}\right) \\ +O\left(\left(e_{0}^{4+\delta}\left(\prod_{i=1}^{n+r}e_{i}\right)^{-1}+\left(\prod_{i=1}^{n+r}e_{i}\right)^{-\frac{1}{2}}\right)\left(\prod_{j=1}^{r}P_{j}^{n_{j}-d_{j}}\right)P_{r}^{-\delta}\right).\end{multline}
\end{lemma}

% penser absolument à définir les \mathcal{A}_{m,\sigma}^{\lambda_{\sigma}} 

\begin{proof}
On observe que \[  N_{U,\ee}(P_{1},...,P_{r})=\tilde{N}_{\ee,r-1}(P_{1},...,P_{r})+O\left(\sum_{\tau\in \mathfrak{S}_{r}}\sum_{m\in \{1,...,r-1\}}T_{m,\tau}\right), \]
o\`u  \begin{multline}
T_{m,\tau}= \Card\left\{ 
 \xx\in (\ZZ\setminus\{0\})^{n+r} \; | \; (x_{i})_{i\in I_{r-1}}\in \mathcal{A}_{r-1}^{\lambda},\; (x_{i})_{i\in I_{m,\tau}}\notin \mathcal{A}_{m,\tau}^{\lambda},\; \right. \\  F(\ee.\xx)=0, \; \forall j\in \{1,...,r\}, \;  \left\lfloor|(\ee.\xx)^{E(n+j)}|\right\rfloor\leqslant P_{j}, \;  \\ \left. \forall i\in \{1,...,n+r\},  \; |x_{i}|\leqslant \frac{1}{e_{i}}\prod_{j=1}^{r}(\left\lfloor|(\ee.\xx)^{E(n+j)}|\right\rfloor+1)^{a_{i,j}} \right\}.
\end{multline}
Pour simplifier les notations, posons $ \mathcal{F}=(\mathcal{A}_{m,\tau}^{\lambda})^{c} $ et $ s=\Card I_{m,\tau} $. Comme nous l'avons vu avec le lemme\;\ref{ferme3}, on a alors $ \dim \mathcal{F}\leqslant s-\lambda $. Posons par ailleurs $ J=I_{r-1}\cap I_{m,\tau} $, $ t=\Card J $, et pour tout $ (x_{i})_{i\in J} $ fix\'e : \[ \mathcal{F}_{(x_{i})_{i\in J}}=\{ (x_{i})_{i\in I_{m,\tau}\setminus J}\; |\; (x_{i})_{i\in I_{m,\tau}}\in \mathcal{F} \}, \] de sorte que \[ \mathcal{F}=\bigsqcup_{(x_{i})_{i\in J}}\mathcal{F}_{(x_{i})_{i\in J}}. \] Notons \[ \mathcal{S}_{1}=\{ (x_{i})_{i\in J}\; |\; \dim \mathcal{F}_{(x_{i})_{i\in J}}> s-t-\frac{\lambda}{2}\}, \] \[   \mathcal{S}_{2}=\{ (x_{i})_{i\in J}\; |\; \dim \mathcal{F}_{(x_{i})_{i\in J}}\leqslant s-t-\frac{\lambda}{2}\}, \] (on a alors que $ \dim\mathcal{S}_{1}\leqslant t-\frac{\lambda}{2} $) et en notant \begin{multline*} M_{(x_{i})_{i\in J}}(P_{1},...,P_{r})=\Card\{ (x_{i})_{i\in \{1,...,n+r\}\setminus J} \; |\;  (x_{i})_{i\in I_{m,\tau}\setminus J}\in \mathcal{F}_{(x_{i})_{i\in J}}, \\  \; \forall i\in \{1,...,n+r\}\setminus J \;  |x_{i}|\leqslant \frac{1}{e_{i}}\prod_{j=1}^{r}P_{j}^{a_{i,j}}\; \et \; F(\xx)=0 \}, \end{multline*} on a donc \begin{multline*} T_{m,\tau}\ll \sum_{\substack{(x_{i})_{i\in J}\in \mathcal{S}_{1} \\ |x_{i}|\leqslant \frac{1}{e_{i}}\prod_{j=1}^{r}P_{j}^{a_{i,j}}}} M_{(x_{i})_{i\in J}}(P_{1},...,P_{r}) \\ + \sum_{\substack{(x_{i})_{i\in J}\in \mathcal{S}_{2} \\ |x_{i}|\leqslant \frac{1}{e_{i}}\prod_{j=1}^{r}P_{j}^{a_{i,j}}}} M_{(x_{i})_{i\in J}}(P_{1},...,P_{r}). \end{multline*}
On observe alors que, en appliquant le lemme\;\ref{estimgrossiere3} aux ferm\'es $ \mathcal{F}_{(x_{i})_{i\in J}} $ (de dimension inf\'erieure \`a $ s-t-\lambda/2 $ lorsque $ (x_{i})_{i\in J}\in \mathcal{S}_{2} $),   
\begin{align*}
\sum_{\substack{(x_{i})_{i\in J}\in \mathcal{S}_{2} \\ |x_{i}|\leqslant \frac{1}{e_{i}}\prod_{j=1}^{r}P_{j}^{a_{i,j}}}} M_{(x_{i})_{i\in J}}(P_{1},...,P_{r}) & \ll \sum_{\substack{(x_{i})_{i\in J}\in \mathcal{S}_{2} \\ |x_{i}|\leqslant \frac{1}{e_{i}}\prod_{j=1}^{r}P_{j}^{a_{i,j}}}}\left(\prod_{i\notin J}e_{i}\right)^{-\frac{1}{2}}\left(\prod_{j=1}^{r}P_{j}^{\sum_{i\notin J}a_{i,j}}\right)P_{r-1}^{-\lambda/4} \\   & \ll \sum_{\substack{(x_{i})_{i\in J}\in \mathcal{S}_{2} \\ |x_{i}|\leqslant \frac{1}{e_{i}}\prod_{j=1}^{r}P_{j}^{a_{i,j}}}} \left(\prod_{i\notin J}e_{i}\right)^{-\frac{1}{2}}\left(\prod_{j=1}^{r-1}P_{j}^{n_{j}}\right)P_{r}^{\sum_{i\notin J}a_{i,r}}P_{r-1}^{-\lambda/4} \\ & \ll \left(\prod_{i=1}^{n+r}e_{i}\right)^{-\frac{1}{2}}\left(\prod_{j=1}^{r}P_{j}^{n_{j}}\right)P_{r-1}^{-\lambda/4}.
\end{align*}
Or, on remarque que (puisque $ \frac{b_{j}}{b_{r-1}}\leqslant\frac{\tilde{b}_{j}}{\tilde{b}_{r-1}}  $) : \[ P_{r-1}^{-\lambda/4}=P_{r}^{-b_{r-1}\lambda/4}\leqslant P_{r}^{-4rb_{r-1}(\sum_{j=1}^{r}\tilde{b}_{j}d_{j}+\delta)/4}\leqslant P_{r}^{-\sum_{j=1}^{r}b_{j}d_{j}-\delta}=\left(\prod_{j=1}^{r}P_{j}^{-d_{j}}\right)P_{r}^{-\delta}. \]  
Par cons\'equent \[ \sum_{\substack{(x_{i})_{i\in J}\in \mathcal{S}_{2} \\ |x_{i}|\leqslant \frac{1}{e_{i}}\prod_{j=1}^{r}P_{j}^{a_{i,j}}}} M_{(x_{i})_{i\in J}}(P_{1},...,P_{r}) \ll \left(\prod_{i=1}^{n+r}e_{i}\right)^{-\frac{1}{2}}\left(\prod_{j=1}^{r}P_{j}^{n_{j}-d_{j}}\right)P_{r}^{-\delta}. \] 
Par ailleurs, pour tous $ (x_{i})_{i\in J}\in \mathcal{S}_{1} $ et $ (x_{i})_{i\in I_{r-1}\setminus J} $ fix\'es v\'erifiant $ (x_{i})_{i\in I_{r-1}}\in \mathcal{A}_{r-1}^{\lambda} $ et $ |x_{i}|\leqslant \frac{1}{e_{i}}\prod_{j=1}^{r}P_{j}^{a_{i,j}} $, par des arguments analogues \`a ceux utilis\'es pour \'etablir le th\'eor\`eme\;\ref{thmfinx3}, et en utilisant les majoration de $ \mathfrak{S}_{(x_{i})_{i\in I_{r-1}},\ee} $ et $ J_{(e_{i},x_{i})_{i\in I_{r-1}}} $ donn\'ees dans les lemmes\;\ref{Sx3} et\;\ref{Jx3}, on montre qu'il existe $ \eta\in [0,1] $ tel que, pour tout $ (x_{i})_{i\in I_{r-1}} $ fix\'e : \begin{multline}\label{formuleuniformeI3}
\Card\{ (x_{i})_{i\in \{1,...,n+r\}\setminus I_{r-1}} \; |\;  (x_{i})_{i\in I_{m',\tau}\setminus J}\in F_{(x_{i})_{i\in J}}, \; \\ \forall i\in \{1,...,n+r\}\setminus I_{r-1},  \;  |x_{i}|\leqslant \frac{1}{e_{i}}\prod_{j=1}^{r}P_{j}^{a_{i,j}}\; \et \; F(\xx)=0 \} \\ \ll \Card\{ (x_{i})_{i\in \{1,...,n+r\}\setminus I_{r-1}} \; |\;   \forall i\notin I_{r-1},  \;  |x_{i}|\leqslant \frac{1}{e_{i}}\prod_{j=1}^{r}P_{j}^{a_{i,j}}\; \et \; F(\xx)=0 \} \\ \ll e_{0}^{2}\left(\prod_{i\notin I_{r-1}}e_{i}\right)^{-1}\left(\prod_{j=1}^{r-1}P_{j}^{n_{j}-d_{j}}\right)P_{r}^{\sum_{i\notin I_{r-1}}a_{i,r}+4d_{r}+\eta}  \\ + \left(\prod_{i\notin I_{r-1}}e_{i}\right)^{-\frac{1}{2}}\left(\prod_{j=1}^{r-1}P_{j}^{n_{j}-d_{j}}\right)P_{r}^{\sum_{i\notin I_{r-1}}a_{i,r}-d_{r}-\delta}
\end{multline}

Si l'on suppose $ \left(\prod_{i}^{n+1}e_{i}\right)^{\frac{1}{2}}P_{r}^{\sum_{i\in J}a_{i,r}}P_{r}^{-\lambda/4}\leqslant 1 $, en rappelant que $ \dim\mathcal{S}_{1}\leqslant t-\frac{\lambda}{2} $, on a, d'apr\`es le lemme\;\ref{estimgrossiere3} \begin{multline*} \Card\{ (x_{i})_{i\in J}\in \mathcal{S}_{1}\; |\; \forall i\in J, \;  |x_{i}|\leqslant \frac{1}{e_{i}}\prod_{j=1}^{r}P_{j}^{a_{i,j}}\}\\ \ll \left(\prod_{i}^{n+1}e_{i}\right)^{-\frac{1}{2}}P_{r}^{\sum_{i\in J}a_{i,r}}P_{r}^{-\lambda/4}\ll \left(\prod_{i}^{n+1}e_{i}\right)^{-\frac{3}{4}}P_{r}^{\sum_{i\in J}a_{i,r}}P_{r}^{-\lambda/8},\end{multline*}
(car on a suppos\'e $ \left(\prod_{i}^{n+1}e_{i}\right)^{\frac{1}{2}}P_{r}^{\sum_{i\in J}a_{i,r}}P_{r}^{-\lambda/4}\leqslant 1 $). 
\\
Si  $ \left(\prod_{i}^{n+1}e_{i}\right)^{\frac{1}{2}}P_{r}^{\sum_{i\in J}a_{i,r}}P_{r}^{-\lambda/4}\geqslant 1 $, par une estimation triviale on trouve encore \begin{multline*} \Card\{ (x_{i})_{i\in J}\in \mathcal{S}_{1}\; |\; \forall i\in J, \;  |x_{i}|\leqslant \frac{1}{e_{i}}\prod_{j=1}^{r}P_{j}^{a_{i,j}}\}\\ \ll \left(\prod_{i\in J}T_{i}\right) \ll \left(\prod_{i\in J}T_{i}\right)\left(\prod_{i}^{n+1}e_{i}\right)^{\frac{1}{4}}P_{r}^{\sum_{i\in J}a_{i,r}}P_{r}^{-\lambda/8} \\ = \left(\prod_{i}^{n+1}e_{i}\right)^{-\frac{3}{4}}P_{r}^{\sum_{i\in J}a_{i,r}}P_{r}^{-\lambda/8}.\end{multline*}

Par cons\'equent, en sommant\;\eqref{formuleuniformeI3} sur l'ensemble des $ (x_{i})_{i\in I_{r-1}} $ consid\'er\'es, on obtient alors \begin{multline*} \sum_{\substack{(x_{i})_{i\in J}\in \mathcal{S}_{1} \\ |x_{i}|\leqslant \frac{1}{e_{i}}\prod_{j=1}^{r}P_{j}^{a_{i,j}}}} M_{(x_{i})_{i\in J}}(P_{1},...,P_{r})  \\  \ll  e_{0}^{2}\left(\prod_{i=1}^{n+r}e_{i}\right)^{-\frac{3}{4}}P_{r}^{-\lambda/8}\left(\prod_{j=1}^{r-1}P_{j}^{n_{j}-d_{j}}\right)P_{r}^{n_{r}+4d_{r}+\eta}   \\ +\left(\prod_{i=1}^{n+r}e_{i}\right)^{-\frac{1}{2}}\left(\prod_{j=1}^{r}P_{j}^{n_{j}-d_{j}}\right)P_{r}^{-\delta}.  \end{multline*}

Or, on a, par d\'efinition de $ \lambda $ (cf. formule\;\eqref{expressiondelambda3}) : \[  P_{r}^{n_{r}+4d_{r}+\eta-\frac{\lambda}{8}} \ll P_{r}^{n_{r}-d_{r}-\delta}, \]
et par ailleurs : \[  e_{0}^{2}\left(\prod_{i\notin J}^{n+r}e_{i}\right)^{-1}\left(\prod_{i\in J}^{n+r}e_{i}\right)^{-\frac{3}{4}}\ll \max\{ e_{0}^{4+\delta}\left(\prod_{i=1}^{n+r}e_{i}\right)^{-1}, \left(\prod_{i=1}^{n+r}e_{i}\right)^{-\frac{1}{2}} \}. \]

Donc finalement, \begin{multline*}
 \sum_{\substack{(x_{i})_{i\in J}\in \mathcal{S}_{1} \\ |x_{i}|\leqslant \frac{1}{e_{i}}\prod_{j=1}^{r}P_{j}^{a_{i,j}}}} M_{(x_{i})_{i\in J}}(P_{1},...,P_{r}) \\ \ll \left(  e_{0}^{4+\delta}\left(\prod_{i=1}^{n+r}e_{i}\right)^{-1}+ \left(\prod_{i=1}^{n+r}e_{i}\right)^{-\frac{1}{2}}\right)\left(\prod_{j=1}^{r}P_{j}^{n_{j}-d_{j}}\right)P_{r}^{-\delta}.
\end{multline*}

Ainsi, nous avons \'etabli que pour tout $ (m,\tau)\in \{1,...,r-1\}\times \mathfrak{S}_{r} $, \[ T_{m,\tau} \ll \left(  e_{0}^{4+\delta}\left(\prod_{i=1}^{n+r}e_{i}\right)^{-1}+ \left(\prod_{i=1}^{n+r}e_{i}\right)^{-\frac{1}{2}}\right)\left(\prod_{j=1}^{r}P_{j}^{n_{j}-d_{j}}\right)P_{r}^{-\delta}. \]
D'o\`u le r\'esultat.
\end{proof}
Nous allons \`a pr\'esent d\'emontrer, par r\'ecurrence, une g\'en\'eralisation des lemmes\;\ref{init13},\;\ref{init23} et\;\ref{init33}. 

\begin{thm}\label{rec3}
Si l'on suppose que $ n+r>\mathfrak{m}  $ et que pour un $ m\in \{1,...,r-1\} $ fix\'e $ \frac{b_{j}}{b_{m}}\leqslant \frac{\tilde{b}_{j}}{\tilde{b}_{m}} $ pour tout $ j\in \{1,...,r\} $, alors : \begin{enumerate}
\item On a d'une part \begin{multline}\label{rec13}
\sum_{\substack{k_{j}\leqslant P_{j}\\ \forall j\in \{1,...,m\}}}\sum_{(x_{i})_{i\in I_{m}}\in \Phi_{m}(\kk)}\mathfrak{S}_{(x_{i})_{i\in I_{m}},\ee}J_{(x_{i},e_{i})_{i\in I_{m}},\kk}\left(\prod_{i\notin I_{m}}e_{i}\right)^{-1}\left(\prod_{j=m+1}^{r}k_{j}^{\sum_{i\notin I_{m}}a_{i,j}}\right) \\ =C_{\sigma,\ee}\left(\prod_{j=m+1}^{r}P_{j}^{n_{j}-d_{j}}\right)+ O\left(\left(e_{0}^{4+\delta}\left(\prod_{i=1}^{n+r}e_{i}\right)^{-1}+\left(\prod_{i=1}^{n+r}e_{i}\right)^{-\frac{1}{2}}\right)\left(\prod_{j=m+1}^{r}P_{j}^{n_{j}-d_{j}}\right)\right).\end{multline}
\item  D'autre part \begin{multline}\label{rec23}  \tilde{N}_{\ee,m}(P_{1},...,P_{r})=C_{\sigma,\ee}\left(\prod_{j=1}^{r}P_{j}^{n_{j}-d_{j}}\right) \\ +O\left(\left(e_{0}^{4+\delta}\left(\prod_{i=1}^{n+r}e_{i}\right)^{-1}+\left(\prod_{i=1}^{n+r}e_{i}\right)^{-\frac{1}{2}}\right)\left(\prod_{j=1}^{r}P_{j}^{n_{j}-d_{j}}\right)P_{r}^{-\delta}\right).\end{multline}
\item On en d\'eduit \begin{multline}\label{rec33}  N_{U,\ee}(P_{1},...,P_{r})=C_{\sigma,\ee}\left(\prod_{j=1}^{r}P_{j}^{n_{j}-d_{j}}\right) \\ +O\left(\left(e_{0}^{4+\delta}\left(\prod_{i=1}^{n+r}e_{i}\right)^{-1}+\left(\prod_{i=1}^{n+r}e_{i}\right)^{-\frac{1}{2}}\right)\left(\prod_{j=1}^{r}P_{j}^{n_{j}-d_{j}}\right)P_{r}^{-\delta}\right).\end{multline}

\end{enumerate}
\end{thm}
\begin{proof}
Nous allons proc\'eder par r\'ecurrence descendante sur $ m\in \{1,...,r-1\} $. Le cas $ m=r-1 $ a d\'ej\`a \'et\'e trait\'e (cf. lemmes\;\ref{init13},\;\ref{init23} et\;\ref{init33}). Supposons que les r\'esultats $ 1,2 $ et $ 3 $ sont v\'erif\'es au rang $ m $, et montrons alors qu'ils le sont au rang $ m-1 $. \\

\'Etape $ 1 $ : Dans cette \'etape, nous allons \'etablir le r\'esultat $ 1 $ au rang $ m-1 $. Pour cela, fixons, pour toute cette \'etape,  $ b_{1},...,b_{m-1} $ tels que $ \frac{b_{1}}{\tilde{b}_{1} }=...=\frac{b_{m-1}}{\tilde{b}_{m-1}} $ (avec $ b_{m-1} $ fix\'e quelconque) et consid\'erons $ b_{m},...,b_{r} $ v\'erifiant $ \frac{b_{j}}{b_{m-1}}\leqslant \frac{\tilde{b}_{j}}{\tilde{b}_{m-1}} $. On a alors, puisque $ \frac{b_{j}}{b_{m-1}}\leqslant \frac{\tilde{b}_{j}}{\tilde{b}_{m-1}} $ pour tout $ j\in \{1,...,r\} $ : \begin{multline*}
g_{m-1}(\bb,\delta) =\left(\frac{b_{m-1}+\sum_{j=1}^{m-1}b_{j}}{b_{m-1}}\right)5\tilde{d}_{m-1}\left(1-\sum_{j=m}^{r}(1+5d_{j})\frac{b_{j}}{b_{m-1}}-\delta\right)^{-1}\\ \left( \sum_{j=m}^{r}\left(1+d_{j}\left(3+\frac{\sum_{l=1}^{m-1}b_{l}}{b_{m-1}+\sum_{l=1}^{m-1}b_{l}}+2\varepsilon\right)\right)\frac{b_{j}}{b_{m-1}}+2\delta\right) \\= 5\tilde{d}_{m-1}\left(1-\sum_{j=m}^{r}(1+5d_{j})\frac{b_{j}}{b_{m-1}}-\delta\right)^{-1}\\ \left( \sum_{j=m}^{r}\left(2\delta+\frac{(1+(3+2\varepsilon)d_{j})b_{j}}{b_{m-1}})\right)\left(1+\sum_{j=1}^{m-1}\frac{b_{j}}{b_{m-1}}\right)+d_{j}\frac{b_{j}}{b_{m-1}}\sum_{l=1}^{m-1}\frac{b_{l}}{b_{m-1}}\right)\\ \leqslant 5\tilde{d}_{m-1}\left(1-\sum_{j=m}^{r}(1+5d_{j})\frac{\tilde{b}_{j}}{\tilde{b}_{m-1}}-\delta\right)^{-1}\\ \left( \sum_{j=m}^{r}\left(2\delta+\frac{(1+(3+2\varepsilon)d_{j})\tilde{b}_{j}}{\tilde{b}_{m-1}})\right)\left(1+\sum_{j=1}^{m-1}\frac{\tilde{b}_{j}}{\tilde{b}_{m-1}}\right)+d_{j}\frac{\tilde{b}_{j}}{\tilde{b}_{m-1}}\sum_{l=1}^{m-1}\frac{\tilde{b}_{l}}{\tilde{b}_{m-1}}\right) \\ =g_{m-1}(\tilde{\bb},\delta). \end{multline*}
Ainsi, on a que (puisque $ n+r>\mathfrak{m}\geqslant h^{\Id}_{m-1,\tt^{(m-1)}}(\bb,(\tilde{\bb}_{\tau})_{\tau\neq \Id}) $) \begin{align*}
K_{m-1} & =(n+r-\lambda-\max_{(j_{0},K_{j_{0}})}\dim V_{m-1,\tt^{(m-1)},(j_{0},K_{j_{0}})}^{\ast}-\varepsilon)/2^{\sum_{j=1}^{m-1}D_{j}} \\ &>\left(1+3\sum_{j=1}^{m-1}\frac{\tilde{b}_{j}}{\tilde{b}_{m-1}}\right)\tilde{d}_{m-1}+ g_{m-1}(\tilde{\bb},\delta) \\ &> \left(1+3\sum_{j=1}^{m-1}\frac{b_{j}}{b_{m-1}}\right)\tilde{d}_{m-1}+ g_{m-1}(\bb,\delta).
\end{align*}
Par cons\'equent, puisque \[ \sum_{j=m}^{r}(1+5d_{j})\frac{b_{j}}{b_{m-1}} \leqslant \sum_{j=m}^{r}(1+5d_{j})\frac{\tilde{b}_{j}}{\tilde{b}_{m-1}}<1,\] on peut appliquer le th\'eor\`eme\;\ref{thmfinx3} et on obtient la formule asymptotique 
\begin{multline}\label{Nm3}  \tilde{N}_{\ee,m-1}(P_{1},...,P_{r})= \sum_{\forall j\in \{m,...,r\}, \;k_{j}\leqslant P_{j}}\left(\sum_{(x_{i})_{i\in I_{m-1}} \in \Phi_{m-1}(\kk)}\mathfrak{S}_{(x_{i})_{i\in I_{m-1}},\ee}J_{(x_{i},e_{i})_{i\in I_{m-1}},\kk}\right) \\ \left(\prod_{i\notin I_{m-1}}e_{i}\right)^{-1}\left(\prod_{j=m}^{r}k_{j}^{\sum_{i\notin I_{m-1}}a_{i,j}}\right)\left(\prod_{j=1}^{m-1}P_{j}^{n_{j}-d_{j}}\right) \\ + O\left(\left(e_{0,m-1}^{4}\left(\prod_{i=1}^{n+r}e_{i}\right)^{-1}+\left(\prod_{i=1}^{n+r}e_{i}\right)^{-\frac{1}{2}}\right)\left(\prod_{j=1}^{r}P_{j}^{n_{j}-d_{j}}\right)P_{m-1}^{-\delta}\right).\end{multline}

Pour terminer la preuve de l'\'etape $ 1 $, nous allons distinguer deux cas :\\

Cas $ 1 $ : supposons qu'il existe $ m'\in \{m,...,r-1\} $ tel que $ \frac{b_{m-1}}{b_{m'}}\leqslant\frac{\tilde{b}_{m-1}}{\tilde{b}_{m'}} $. Puisque $ \frac{b_{j}}{b_{m-1}}\leqslant \frac{\tilde{b}_{j}}{\tilde{b}_{m-1}} $ pour tout $ j\in \{1,...,r-1\} $, on a donc \[ \forall j\in \{1,...,r-1\}, \; \; \frac{b_{j}}{b_{m'}}\leqslant\frac{\tilde{b}_{j}}{\tilde{b}_{m'}}. \]Par hypoth\`ese de r\'ecurrence, la formule\;\eqref{rec23} (au rang $ m' $) est v\'erifi\'ee et on remarque que (par des arguments analogues \`a ceux utilis\'es pour comparer $ \tilde{N}_{\ee,r-1} $ et $ N_{\ee} $, dans la d\'emonstration du lemme\;\ref{init13})  \begin{equation*} N_{\ee,m'}(P_{1},...,P_{r})=\tilde{N}_{\ee,m-1}(P_{1},...,P_{r})  +O\left(\left(\prod_{i=1}^{n+r}e_{i}\right)^{-\frac{1}{2}}\left(\prod_{j=1}^{r}P_{j}^{n_{j}-d_{j}}\right)P_{r}^{-\delta}\right).  \end{equation*} En rempla\c{c}ant $ N_{\ee,m'}(P_{1},...,P_{r}) $ et $ \tilde{N}_{\ee,m-1}(P_{1},...,P_{r}) $ par les formules asymptotiques obtenues, puis en simplifiant par $ \left(\prod_{j=1}^{m-1}P_{j}^{n_{j}-d_{j}}\right) $, nous obtenons la formule\;\eqref{rec13} du th\'eor\`eme au rang $ m-1 $ valable pour les $ b_{m},...,b_{r} $ tels que $  \frac{b_{j}}{b_{m'}}\leqslant\frac{\tilde{b}_{j}}{\tilde{b}_{m'}} $ pour tout $ j\in \{m,...,r\} $. \\

Cas $ 2 $ : supposons \`a pr\'esent que pour tout $ m'\in \{m,...,r-1\} $, $ \frac{b_{m-1}}{b_{m'}}\geqslant\frac{\tilde{b}_{m-1}}{\tilde{b}_{m'}} $. On a alors pour tout $ j\in \{1,...,r-1\} $, $ b_{j}\leqslant \tilde{b}_{j} $. Dans ce cas, la condition $ n+r>\mathfrak{m}  $ implique que \begin{align*} K & \max\{\left(2+5\sum_{j=1}^{r}\tilde{b}_{j}\right)\tilde{d}, (2\delta+1)\sum_{j=1}^{r}\tilde{b}_{j}d_{j}\} \\ & \geqslant\max\{\left(2+5\sum_{j=1}^{r}b_{j}\right)\tilde{d}, (2\delta+1)\sum_{j=1}^{r}b_{j}d_{j}\}\end{align*} et donc que l'on peut appliquer le th\'eor\`eme\;\ref{propfin3}, ce qui nous donne : \begin{multline*}N_{\ee}(P_{1},...,P_{r})=C_{\sigma,\ee}\left(\prod_{j=1}^{r}P_{j}^{n_{j}-d_{j}}\right)
\\ +O\left(\left(e_{0}^{4+\delta}\left(\prod_{i=1}^{n+r}e_{i}\right)^{-1}+\left(\prod_{i=1}^{n+r}e_{i}\right)^{-\frac{1}{2}}\right)\left(\prod_{j=1}^{r}P_{j}^{n_{j}-d_{j}-\delta}\right)\right). \end{multline*}
o\`u l'on a not\'e \[ C_{\sigma,\ee}=\left(\prod_{i=1}^{n+r}e_{i}\right)^{-1}\mathfrak{S}_{\ee}J_{\sigma}.\] Puis en comparant $ N_{\ee} $ et $ N_{\ee,m-1} $ et en simplifiant \`a nouveau l'\'egalit\'e obtenue par $ \left(\prod_{j=1}^{m-1}P_{j}^{n_{j}-d_{j}}\right) $, nous obtenons la formule\;\eqref{rec13} pour tous les $ b_{m},...,b_{r} $ tels que $  \frac{b_{j}}{b_{m'}}\geqslant\frac{\tilde{b}_{j}}{\tilde{b}_{m'}} $ pour tout $ m'\in \{m,...,r-1\} $. 
\\
Nous avons donc \'etablit la r\'esultat $ 1 $ au rang $ m-1 $.\\

\'Etape $ 2 $ : D\'emontrons \`a pr\'esent le r\'esultat $ 2 $ au rang $ m-1 $. Nous prenons ici $ b_{1},...,b_{r-1} $ quelconques tels que $ \frac{b_{j}}{b_{m-1}}\leqslant\frac{\tilde{b}_{j}}{\tilde{b}_{m-1}} $ pour tout $ j\in \{1,...,r-1\} $. \\
Puisque $ \frac{b_{j}}{b_{m}}\leqslant \frac{\tilde{b}_{j}}{\tilde{b}_{m}} $, alors, comme pour la d\'emonstration de la formule\;\eqref{rec13} ci-dessus, on a 
\[
K_{m-1} > \left(1+3\sum_{j=1}^{m-1}\frac{b_{j}}{b_{m-1}}\right)\tilde{d}_{m-1}+ g_{m-1}(\bb,\delta).
\]
et on applique le th\'eor\`eme\;\ref{thmfinx3}, ce qui nous donne : \begin{multline*}  \tilde{N}_{\ee,m-1}(P_{1},...,P_{r})= \sum_{\forall j\in \{m,...,r\}, \;k_{j}\leqslant P_{j}}\left(\sum_{(x_{i})_{i\in I_{m-1}} \in \Phi_{m-1}(\kk)}\mathfrak{S}_{(x_{i})_{i\in I_{m-1}},\ee}J_{(x_{i},e_{i})_{i\in I_{m-1}},\kk}\right) \\ \left(\prod_{i\notin I_{m-1}}e_{i}\right)^{-1}\left(\prod_{j=m}^{r}k_{j}^{\sum_{i\notin I_{m-1}}a_{i,j}}\right)\left(\prod_{j=1}^{m-1}P_{j}^{n_{j}-d_{j}}\right) \\ + O\left(\left(e_{0,m-1}^{4}\left(\prod_{i=1}^{n+r}e_{i}\right)^{-1}+\left(\prod_{i=1}^{n+r}e_{i}\right)^{-\frac{1}{2}}\right)\left(\prod_{j=1}^{r}P_{j}^{n_{j}-d_{j}}\right)P_{m-1}^{-\delta}\right).\end{multline*} et en appliquant la formule\;\eqref{rec13} au rang $ m-1 $ on obtient la formule\;\eqref{rec23} au rang $ m-1 $.\\

\'Etape $ 3 $ : nous allons \`a pr\'esent, en employant des arguments analogues \`a ceux utilis\'es pour \'etablir le lemme\;\ref{init33}, d\'emontrer que la formule\;\eqref{rec23} implique\;\eqref{rec33} (par comparaison de $ \tilde{N}_{\ee,m-1} $ et de $ N_{U,\ee} $).
Remarquons dans un premier temps que :

\[  N_{U,\ee}(P_{1},...,P_{r})=\tilde{N}_{\ee,m}(P_{1},...,P_{r})+O\left(\sum_{\tau\in \mathfrak{S}_{r}}\sum_{m'\in \{1,...,r-1\}}T_{m',\tau}\right), \]
o\`u  \begin{multline*}
T_{m',\tau}= \Card\left\{ 
 \xx\in (\ZZ\setminus\{0\})^{n+r} \; | \; (x_{i})_{i\in I_{m',\tau}}\notin \mathcal{A}_{m',\tau}^{\lambda},\; \forall k\in \{m,...,r-1\}, \; \right. \\ (x_{i})_{i\in I_{k}}\in \mathcal{A}_{k}^{\lambda},\;  F(\ee.\xx)=0, \;\forall j\in \{1,...,r\}, \;  \left\lfloor|(\ee.\xx)^{E(n+j)}|\right\rfloor\leqslant P_{j}, \;  \\ \left.\forall i\in \{1,...,n+r\},  \;  |x_{i}|\leqslant \frac{1}{e_{i}}\prod_{j=1}^{r}(\left\lfloor|(\ee.\xx)^{E(n+j)}|\right\rfloor+1)^{a_{i,j}} \right\}.
\end{multline*}
Consid\'erons un \'el\'ement $ (m',\tau) $. Posons $ \mathcal{F}=(\mathcal{A}_{m',\tau}^{\lambda})^{c} $ et $ s=\Card I_{m',\tau} $ de sorte que $ \dim \mathcal{F}\leqslant s-\lambda $. Posons par ailleurs $ J_{m}=I_{m}\cap I_{m',\tau} $, $ s_{m}=\Card J_{m} $, et pour tout $ (x_{i})_{i\in J} $ fix\'e : \[ \mathcal{F}_{(x_{i})_{i\in J_{m}}}=\{ (x_{i})_{i\in I_{m',\tau}\setminus J_{m}}\; |\; (x_{i})_{i\in I_{m',\tau}}\in \mathcal{F} \}, \] de sorte que \[ \mathcal{F}=\bigsqcup_{(x_{i})_{i\in J_{m}}}\mathcal{F}_{(x_{i})_{i\in J_{m}}}. \] Notons alors \[ \mathcal{S}^{(m)}_{1}=\{ (x_{i})_{i\in J_{m}}\; |\; \dim \mathcal{F}_{(x_{i})_{i\in J_{m}}}> s-s_{m}-\frac{m\lambda}{r}\}, \] \[   \mathcal{S}^{(m)}_{2}=\{ (x_{i})_{i\in J_{m}}\; |\; \dim \mathcal{F}_{(x_{i})_{i\in J_{m}}}\leqslant s-s_{m}-\frac{m\lambda}{r}\} \] (on remarque que $ \dim\mathcal{S}^{(m)}_{1}\leqslant s_{m}-\frac{(r-m)\lambda}{r} $, car $ \dim\mathcal{F}\leqslant s-\lambda $), et en notant \begin{multline*} M_{(x_{i})_{i\in J_{m}}}(P_{1},...,P_{r})=\Card\{ (x_{i})_{i\in \{1,...,n+r\}\setminus J_{m}} \; |\;  (x_{i})_{i\in I_{m',\tau}\setminus J_{m}}\in \mathcal{F}_{(x_{i})_{i\in J_{m}}}, \\ \; \forall k\in \{m,...,r-1\}, \; (x_{i})_{i\in I_{k}}\in \mathcal{A}_{k}^{\lambda} \; \\ \forall i\in \{1,...,n+r\}\setminus J_{m}, \;  |x_{i}|\leqslant \frac{1}{e_{i}}\prod_{j=1}^{r}P_{j}^{a_{i,j}}\; \et \; F(\xx)=0 \}, \end{multline*} on a donc \begin{multline*} T_{m',\tau}\ll \sum_{\substack{(x_{i})_{i\in J_{m}}\in \mathcal{S}_{1}^{(m)} \\ |x_{i}|\leqslant \frac{1}{e_{i}}\prod_{j=1}^{r}P_{j}^{a_{i,j}}}} M_{(x_{i})_{i\in J_{m}}}(P_{1},...,P_{r}) \\ + \sum_{\substack{(x_{i})_{i\in J_{m}}\in \mathcal{S}_{2}^{(m)} \\ |x_{i}|\leqslant \frac{1}{e_{i}}\prod_{j=1}^{r}P_{j}^{a_{i,j}}}} M_{(x_{i})_{i\in J_{m}}}(P_{1},...,P_{r}). \end{multline*}
On observe alors que, d'apr\`es le lemme\;\ref{estimgrossiere3} 
\begin{multline*}
\sum_{\substack{(x_{i})_{i\in J_{m}}\in \mathcal{S}_{2}^{(m)} \\ |x_{i}|\leqslant \frac{1}{e_{i}}\prod_{j=1}^{r}P_{j}^{a_{i,j}}}} M_{(x_{i})_{i\in J_{m}}}(P_{1},...,P_{r}) \\ \ll \sum_{\substack{(x_{i})_{i\in J_{m}}\in \mathcal{S}_{2}^{(m)} \\ |x_{i}|\leqslant \frac{1}{e_{i}}\prod_{j=1}^{r}P_{j}^{a_{i,j}}}} \left(\prod_{i\notin J_{m}}e_{i}\right)^{-\frac{1}{2}}\left(\prod_{j=1}^{r}P_{j}^{\sum_{i\notin J_{m}}a_{i,j}}\right)P_{m}^{-\frac{m\lambda}{2r}} \\  \ll \left(\prod_{i=1}^{n+r}e_{i}\right)^{-\frac{1}{2}}\left(\prod_{j=1}^{r}P_{j}^{n_{j}}\right)P_{m}^{-\frac{m\lambda}{2r}}.
\end{multline*}
Or, puisque $ \frac{b_{j}}{b_{m}}\leqslant\frac{\tilde{b}_{j}}{\tilde{b}_{m}}  $, \[ P_{m}^{-\frac{\lambda m}{2r}}=P_{r}^{-b_{m}\frac{\lambda m}{2r}}\leqslant P_{r}^{-2rb_{m}(\sum_{j=1}^{r}\tilde{b}_{j}d_{j}+\delta)\frac{\lambda m}{2r}}\leqslant P_{r}^{-\sum_{j=1}^{r}b_{j}d_{j}-\delta}=\left(\prod_{j=1}^{r}P_{j}^{-d_{j}}\right)P_{r}^{-\delta}. \]  
Par cons\'equent \[ \sum_{\substack{(x_{i})_{i\in J_{m}}\in \mathcal{S}_{2}^{(m)} \\ |x_{i}|\leqslant \frac{1}{e_{i}}\prod_{j=1}^{r}P_{j}^{a_{i,j}}}} M_{(x_{i})_{i\in J_{m}}}(P_{1},...,P_{r}) \ll \left(\prod_{i=1}^{n+r}e_{i}\right)^{-\frac{1}{2}}\left(\prod_{j=1}^{r}P_{j}^{n_{j}-d_{j}}\right)P_{r}^{-\delta}. \] 
 Remarquons \`a pr\'esent que, par les m\^emes arguments que ceux utilis\'es pour d\'emontrer le th\'eor\`eme\;\ref{thmfinx3} (reposants sur la proposition\;\ref{propfinx3}) on a :

 \begin{multline}\label{avantlemme3}
\sum_{\substack{(x_{i})_{i\in J_{m}}\in \mathcal{S}_{1}^{(m)} \\ |x_{i}|\leqslant \frac{1}{e_{i}}\prod_{j=1}^{r}P_{j}^{a_{i,j}}}} M_{(x_{i})_{i\in J_{m}}}(P_{1},...,P_{r}) \\ = \sum_{k_{j}\leqslant P_{j}}\sum_{\substack{(x_{i})_{i\in I_{m}}\in \Phi_{m}(\kk) \\ (x_{i})_{i\in J_{m}}\in \mathcal{S}_{1}^{(m)}}}N_{\ee,(x_{i})_{i\in I_{m}}}(P_{1},...,P_{m}) \\ =\sum_{k_{j}\leqslant P_{j}}\sum_{\substack{(x_{i})_{i\in I_{m}}\in \Phi_{m}(\kk) \\ (x_{i})_{i\in J_{m}}\in S_{1}^{(m)}}}\mathfrak{S}_{(x_{i})_{i\in I_{m}},\ee}J_{(x_{i},e_{i})_{i\in I_{m}},\kk}\left(\prod_{i\notin I_{m}}e_{i}\right)^{-1}\left(\prod_{j=m+1}^{r}k_{j}^{\sum_{i\notin I_{m}}a_{i,j}}\right)\prod_{j=1}^{m}P_{j}^{n_{j}-d_{j}} \\+ O\left(\left(\prod_{i=1}^{n+r}e_{i}\right)^{-\frac{1}{2}}\left(\prod_{j=1}^{r}P_{j}^{n_{j}-d_{j}}\right)P_{r}^{-\delta}\right)
\end{multline}

Pour achever l'\'etape $ 3 $ de la d\'emonstration du th\'eor\`eme, il est n\'ecessaire de d\'emontrer le lemme ci-dessous : 

\begin{lemma}
Si l'on suppose $ n+r\geqslant \mathfrak{m} $ et que $ \frac{b_{j}}{b_{m}}\leqslant \frac{\tilde{b}_{j}}{\tilde{b}_{m}} $ pour tout $ j\in \{m+1,...,r\} $, on a alors que : \begin{multline*} \sum_{k_{j}\leqslant P_{j}}\sum_{\substack{(x_{i})_{i\in I_{m}}\in \Phi_{m}(\kk) \\ (x_{i})_{i\in J_{m}}\in S_{1}^{(m)}}}\mathfrak{S}_{(x_{i})_{i\in I_{m}},\ee}J_{(x_{i},e_{i})_{i\in I_{m}},\kk}\left(\prod_{j=m+1}^{r}k_{j}^{\sum_{i\notin I_{m}}a_{i,j}}\right) \\ \ll \left(e_{0}^{4+\delta}\left(\prod_{i\in I_{m}}e_{i}\right)^{-1}+\left(\prod_{i\in I_{m}}e_{i}\right)^{-\frac{1}{2}}\right)\left(\prod_{i\notin I_{m}}e_{i}\right)^{\frac{1}{2}}\left(\prod_{j=m+1}^{r}P_{j}^{n_{j}-d_{j}}\right)P_{r}^{-\delta}.  \end{multline*}
\end{lemma}
\begin{proof}
Nous allons d\'emontrer ce r\'esultat par r\'ecurence descendante sur $ m $. Le cas du rang $ m=r-1 $ a d\'ej\`a \'et\'e trait\'e dans la d\'emonstration du lemme\;\ref{init33}. Supposons le r\'esultat vrai au rang $ k $ pour tout $ k\in\{m+1,...,r-1\} $. Montrons alors que la propri\'et\'e est vraie au rang $ m $. \\

Supposons dans un premier temps que pour tout $ j\in \{m+1,...,r-1\} $, $ b_{j}\leqslant \tilde{b}_{j} $. On remarque alors que, puisque $ \dim \mathcal{S}_{1}^{(m)}\leqslant s_{m}-\frac{(r-m)\lambda}{r} $ 
\begin{multline*}
\sum_{k_{j}\leqslant P_{j}}\sum_{\substack{(x_{i})_{i\in I_{m}}\in \Phi_{m}(\kk) \\ (x_{i})_{i\in J_{m}}\in S_{1}^{(m)}}}\mathfrak{S}_{(x_{i})_{i\in I_{m}},\ee}J_{(x_{i},e_{i})_{i\in I_{m}},\kk}\left(\prod_{j=m+1}^{r}k_{j}^{\sum_{i\notin I_{m}}a_{i,j}}\right) \\ \ll \left(\prod_{i\in I_{m}}e_{i}\right)^{-\frac{1}{2}}\left(\prod_{j=m+1}^{r}P_{j}^{n_{j}}\right)P_{r}^{-\frac{\lambda(r-m)}{2r}}  
\end{multline*}
Or puisque pour tout $ j\in \{m+1,...,r-1\} $, $ b_{j}\leqslant \tilde{b}_{j} $, on a alors que \[  P_{r}^{-\frac{(r-m)\lambda}{2r}} \leqslant P_{r}^{-\sum_{j=1}^{r}\tilde{b}_{j}d_{j}+\delta}\leqslant P_{r}^{-\sum_{j=1}^{r}b_{j}d_{j}+\delta}  \leqslant \left(\prod_{j=m+1}^{r}P_{j}^{-d_{j}}\right)P_{r}^{-\delta},   \]
et on en d\'eduit le r\'esultat dans ce cas.\\

On suppose \`a pr\'esent qu'il existe $ j\in \{m+1,...,r\} $ tel que $ b_{j}\geqslant \tilde{b}_{j} $, autrement dit, que $ \max_{j\in\{m+1,...,r-1\}}\left( \frac{b_{j}}{\tilde{b}_{j}}\right)\geqslant 1 $. Consid\'erons $ k\in \{m+1,...,r-1\} $ tel que $ \frac{b_{k}}{\tilde{b}_{k}}=\max_{j\in\{m+1,...,r-1\}}\left( \frac{b_{j}}{\tilde{b}_{j}}\right) $. On a alors, pour tout $ j\in \{k+1,...,r\} $, $  \frac{b_{j}}{b_{k}}\leqslant \frac{\tilde{b}_{j}}{\tilde{b}_{k}} $. On remarque que, \'etant donn\'e que $ m<k $,  \[ I_{k}\subset I_{m}, \]\[  J_{k}=I_{k}\cap I_{m',\tau}\subset I_{m}\cap I_{m',\tau}=J_{m}. \]Posons \[ \mathcal{S}^{(m)}_{1}=\bigsqcup_{(x_{i})_{i\in J_{k}}}\mathcal{S}^{(m)}_{1,(x_{i})_{i\in J_{k}}} \]
o\`u \[ \mathcal{S}^{(m)}_{1,(x_{i})_{i\in J_{k}}}=\{ (x_{i})_{i\in J_{m}\setminus J_{k}}\; |\; (x_{i})_{i\in J_{m}}\in  \mathcal{S}^{(m)}_{1} \}. \]
 Posons alors \[ \tilde{\mathcal{S}}^{(k)}_{1}=\{ (x_{i})_{i\in J_{k}}\; |\; \dim \mathcal{S}^{(m)}_{1,(x_{i})_{i\in J_{k}}}> s_{m}-s_{k}-\frac{(k-m)\lambda}{r}\}, \] \[   \tilde{\mathcal{S}}^{(k)}_{2}=\{ (x_{i})_{i\in J_{m}}\; |\; \dim \mathcal{S}^{(m)}_{1,(x_{i})_{i\in J_{k}}}\leqslant  s_{m}-s_{k}-\frac{(k-m)\lambda}{r}\}. \]
Remarquons que, comme pr\'ec\'edemment, \begin{multline}\label{derlemme3}
\sum_{k_{j}\leqslant P_{j}}\sum_{\substack{(x_{i})_{i\in I_{m}}\in \Phi_{m}(\kk) \\ (x_{i})_{i\in J_{m}}\in S_{1}^{(m)} \\ (x_{i})_{i\in J_{k}}\in \tilde{\mathcal{S}}^{(k)}_{2}}}\mathfrak{S}_{(x_{i})_{i\in I_{m}},\ee}J_{(x_{i},e_{i})_{i\in I_{m}},\kk}\left(\prod_{j=m+1}^{r}k_{j}^{\sum_{i\notin I_{m}}a_{i,j}}\right) \\  \ll \left(\prod_{i\in I_{m}}e_{i}\right)^{-\frac{1}{2}}\left(\prod_{j=m+1}^{r}P_{j}^{n_{j}}\right) P_{k}^{-\frac{k\lambda}{2r}} \\ \ll \left(\prod_{i\in I_{m}}e_{i}\right)^{-\frac{1}{2}}\left(\prod_{j=m+1}^{r}P_{j}^{n_{j}}\right) P_{r}^{-b_{k}(\sum_{j=1}^{r}\tilde{b}_{j}d_{j}+\delta)} \\ \ll  \left(\prod_{i\in I_{m}}e_{i}\right)^{-\frac{1}{2}}\left(\prod_{j=m+1}^{r}P_{j}^{n_{j}-d_{j}}\right) P_{r}^{-\delta} 
\end{multline}
(car $  \frac{b_{j}}{b_{k}}\leqslant \frac{\tilde{b}_{j}}{\tilde{b}_{k}} $). D'autre part, on observe que, si $ (x_{i})_{i\in J_{k}}\in  \tilde{\mathcal{S}}^{(k)}_{1} $, puisque 
\[ F_{(x_{i})_{i\in J_{k}}}\supset\bigsqcup_{(x_{i})_{i\in J_{m}\setminus J_{k}}\in \mathcal{S}^{(m)}_{1,(x_{i})_{i\in J_{k}}}}F_{(x_{i})_{i\in J_{m}}}, \] on a, pour tout $  (x_{i})_{i\in J_{k}}\in  \tilde{\mathcal{S}}^{(k)}_{1} $, \[ \dim F_{(x_{i})_{i\in J_{k}}} \geqslant (s-s_{m}-\frac{m\lambda}{r})+(s_{m}-s_{k}-\frac{(k-m)\lambda}{r})=s-s_{k}-\frac{k\lambda}{r} \] et on en d\'eduit que \[  \tilde{\mathcal{S}}^{(k)}_{1}\subset \mathcal{S}^{(k)}_{1}. \] 
Choisissons \`a pr\'esent $ P_{1},...,P_{m} $ de sorte que $ \frac{b_{j}}{b_{k}}\leqslant \frac{\tilde{b}_{j}}{\tilde{b}_{k}} $ pour tout $ j\in \{1,...,r\} $. On a alors \begin{equation}\label{aprelemme3}
 \sum_{\substack{(x_{i})_{i\in J_{m}}\in \mathcal{S}_{1}^{(m)}\\ (x_{i})_{i\in J_{k}}\in \tilde{\mathcal{S}}_{1}^{(k)} \\ |x_{i}|\leqslant \frac{1}{e_{i}}\prod_{j=1}^{r}P_{j}^{a_{i,j}}}} M_{(x_{i})_{i\in J_{m}}}(P_{1},...,P_{r})\leqslant  \sum_{\substack{(x_{i})_{i\in J_{k}}\in \mathcal{S}_{1}^{(k)} \\ |x_{i}|\leqslant \frac{1}{e_{i}}\prod_{j=1}^{r}P_{j}^{a_{i,j}}}} M_{(x_{i})_{i\in J_{k}}}(P_{1},...,P_{r}) 
\end{equation}
Or, nous avons vu (cf.\;\eqref{avantlemme3} et\;\eqref{derlemme3}) que : \begin{multline*}
\sum_{\substack{(x_{i})_{i\in J_{m}}\in \mathcal{S}_{1}^{(m)} \\ |x_{i}|\leqslant \frac{1}{e_{i}}\prod_{j=1}^{r}P_{j}^{a_{i,j}}}} M_{(x_{i})_{i\in J_{m}}}(P_{1},...,P_{r})\\  =\sum_{k_{j}\leqslant P_{j}}\sum_{\substack{(x_{i})_{i\in I_{m}}\in \Phi_{m}(\kk) \\ (x_{i})_{i\in J_{m}}\in \mathcal{S}_{1}^{(m)}\\ (x_{i})_{i\in J_{k}}\in \tilde{\mathcal{S}}_{1}^{(k)}}}\mathfrak{S}_{(x_{i})_{i\in I_{m}},\ee}J_{(x_{i},e_{i})_{i\in I_{m}},\kk}\left(\prod_{i\notin I_{m}}e_{i}\right)^{-1}\left(\prod_{j=m+1}^{r}k_{j}^{\sum_{i\notin I_{m}}a_{i,j}}\right)\prod_{j=1}^{m}P_{j}^{n_{j}-d_{j}} \\+ O\left(\left(\prod_{i=1}^{n+r}e_{i}\right)^{-\frac{1}{2}}\left(\prod_{j=1}^{r}P_{j}^{n_{j}-d_{j}}\right)P_{r}^{-\delta}\right).
\end{multline*}
Par cons\'equent l'in\'egalit\'e\;\eqref{aprelemme3} implique : \begin{multline*}
\sum_{k_{j}\leqslant P_{j}}\sum_{\substack{(x_{i})_{i\in I_{m}}\in \Phi_{m}(\kk) \\x_{i})_{i\in J_{m}}\in \mathcal{S}_{1}^{(m)}\\ (x_{i})_{i\in J_{k}}\in \tilde{\mathcal{S}}_{1}^{(k)} }}\mathfrak{S}_{(x_{i})_{i\in I_{m}},\ee}J_{(x_{i},e_{i})_{i\in I_{m}},\kk}\left(\prod_{i\notin I_{m}}e_{i}\right)^{-1}\left(\prod_{j=m+1}^{r}k_{j}^{\sum_{i\notin I_{m}}a_{i,j}}\right)\prod_{j=1}^{m}P_{j}^{n_{j}-d_{j}} \\ \ll \sum_{\substack{(x_{i})_{i\in J_{k}}\in \mathcal{S}_{1}^{(k)} \\ |x_{i}|\leqslant \frac{1}{e_{i}}\prod_{j=1}^{r}P_{j}^{a_{i,j}}}} M_{(x_{i})_{i\in J_{k}}}(P_{1},...,P_{r})+  O\left(\left(\prod_{i=1}^{n}e_{i}\right)^{-\frac{1}{2}}\left(\prod_{j=1}^{r}P_{j}^{n_{j}-d_{j}}\right)P_{r}^{-\delta}\right) \\ \ll \sum_{k_{j}\leqslant P_{j}}\sum_{\substack{(x_{i})_{i\in I_{k}}\in \Phi_{k}(\kk) \\ (x_{i})_{i\in J_{k}}\in \mathcal{S}_{k}^{(1)}}}\mathfrak{S}_{(x_{i})_{i\in I_{k}},\ee}J_{(x_{i},e_{i})_{i\in I_{k}},\kk}\left(\prod_{i\notin I_{k}}e_{i}\right)^{-1}\left(\prod_{j=k+1}^{r}k_{j}^{\sum_{i\notin I_{k}}a_{i,j}}\right)\prod_{j=1}^{k}P_{j}^{n_{j}-d_{j}} \\+ O\left(\left(\prod_{i=1}^{n+r}e_{i}\right)^{-\frac{1}{2}}\left(\prod_{j=1}^{r}P_{j}^{n_{j}-d_{j}}\right)P_{r}^{-\delta}\right).
\end{multline*}
En utilisant l'hypoth\`ese de r\'ecurrence au rang $ k $, on obtient donc :   \begin{multline*}
\sum_{k_{j}\leqslant P_{j}}\sum_{\substack{(x_{i})_{i\in I_{m}}\in \Phi_{m}(\kk) \\x_{i})_{i\in J_{m}}\in \mathcal{S}_{1}^{(m)}\\ (x_{i})_{i\in J_{k}}\in \tilde{\mathcal{S}}_{1}^{(k)} }}\mathfrak{S}_{(x_{i})_{i\in I_{m}},\ee}J_{(x_{i},e_{i})_{i\in I_{m}},\kk}\left(\prod_{i\notin I_{m}}e_{i}\right)^{-1}\left(\prod_{j=m+1}^{r}k_{j}^{\sum_{i\notin I_{m}}a_{i,j}}\right)\prod_{j=1}^{m}P_{j}^{n_{j}-d_{j}} \\ = O\left(\left(e_{0}^{4+\delta}\left(\prod_{i=1}^{n+r}e_{i}\right)^{-1}+\left(\prod_{i=1}^{n+r}e_{i}\right)^{-\frac{1}{2}}\right)\left(\prod_{j=1}^{r}P_{j}^{n_{j}-d_{j}}\right)P_{r}^{-\delta}\right)
\end{multline*}
et on obtient donc le r\'esultat en simplifiant par $ \left(\prod_{i\notin I_{m}}e_{i}\right)^{-1}\prod_{j=1}^{m}P_{j}^{n_{j}-d_{j}} $.
\end{proof}

\textit{Fin de la d\'emonstration du th\'eor\`eme\;\ref{rec3}}
En appliquant ce lemme \`a la majoration\;\eqref{avantlemme3}, on obtient donc : \begin{multline*}
\sum_{\substack{(x_{i})_{i\in J_{m}}\in \mathcal{S}_{1}^{(m)} \\ |x_{i}|\leqslant \frac{1}{e_{i}}\prod_{j=1}^{r}P_{j}^{a_{i,j}}}} M_{(x_{i})_{i\in J_{m}}}(P_{1},...,P_{r}) \\  =O\left(\left(e_{0}^{4+\delta}\left(\prod_{i=1}^{n+r}e_{i}\right)^{-1}+\left(\prod_{i=1}^{n+r}e_{i}\right)^{-\frac{1}{2}}\right)\left(\prod_{j=1}^{r}P_{j}^{n_{j}-d_{j}}\right)P_{r}^{-\delta}\right),
\end{multline*}
ce qui cl\^ot la d\'emonstration du th\'eor\`eme.
\end{proof}
Nous d\'eduisons alors de ce th\'eor\`eme le th\'eor\`eme\;\ref{thmNU3} : 

\begin{proof}[D\'emonstration du th\'eor\`eme\;\ref{thmNU3}]
Supposons dans un premier temps que $ b_{j}\leqslant \tilde{b}_{j} $ pour tout $ j\in \{1,...,r-1\} $. On a alors que \[  N_{U,\ee}(P_{1},...,P_{r})=N_{\ee}(P_{1},...,P_{r})+O\left(\sum_{\tau\in \mathfrak{S}_{r}}\sum_{m\in \{1,...,r-1\}}T_{m,\tau}\right), \]
o\`u  \begin{multline*}
T_{m,\tau}= \Card\left\{ 
 \xx\in (\ZZ\setminus\{0\})^{n+r} \; | \; (x_{i})_{i\in I_{m,\tau}}\notin \mathcal{A}_{m,\tau}^{\lambda},\;  F(\ee.\xx)=0, \; \right. \\ \forall j\in \{1,...,r\}, \;  \left\lfloor|(\ee.\xx)^{E(n+j)}|\right\rfloor\leqslant P_{j}, \;  \forall i\in \{1,...,n+r\},  \;  \\ \left.|x_{i}|\leqslant \frac{1}{e_{i}}\prod_{j=1}^{r}(\left\lfloor|(\ee.\xx)^{E(n+j)}|\right\rfloor+1)^{a_{i,j}} \right\} \\ \ll \left(\prod_{i=1}^{n+r}e_{i}\right)^{-\frac{1}{2}}\left(\prod_{j=1}^{r}P_{j}^{n_{j}}\right)P_{r}^{-\lambda/2} \\ \ll \left(\prod_{i=1}^{n+r}e_{i}\right)^{-\frac{1}{2}}\left(\prod_{j=1}^{r}P_{j}^{n_{j}-d_{j}}\right)P_{r}^{-\delta}.
\end{multline*}
Or, puisque $ b_{j}\leqslant \tilde{b}_{j} $, et \[  n+r\geqslant \mathfrak{m}\geqslant h^{\Id}_{r,\tt^{(r)}}(\tilde{\bb})\geqslant h^{\Id}_{r,\tt^{(r)}}(\bb), \] on peut appliquer le th\'eor\`eme\;\ref{propfin3} et on a donc 

\begin{multline*} N_{\ee}(P_{1},...,P_{r})=C_{\sigma,\ee}\left(\prod_{j=1}^{r}P_{j}^{n_{j}-d_{j}}\right)
\\ +O\left(\max\left\{e_{0}^{4+\delta}, \left(\prod_{i=1}^{n+r}e_{i}\right)^{\frac{1}{2}}\right\}\left(\prod_{i=1}^{n+r}e_{i}\right)^{-1}\left(\prod_{j=1}^{r}P_{j}^{n_{j}-d_{j}-\delta}\right)\right). \end{multline*}

d'o\`u le r\'esultat. \\

Supposons \`a pr\'esent que $ \max_{j\in \{1,...,r-1\}}\frac{b_{j}}{\tilde{b}_{j}}>1 $. Posons $ m $ un \'el\'ement de $ \{1,...,r-1\} $ tel que \[ \frac{b_{m}}{\tilde{b}_{m}}=\max_{j\in \{1,...,r-1\}}\frac{b_{j}}{\tilde{b}_{j}}. \] Cette condition implique en particulier que : \[ \forall j\in \{1,...,r\}, \; \; \frac{b_{j}}{b_{m}}\leqslant \frac{\tilde{b}_{j}}{\tilde{b}_{m}}, \] et la formule \eqref{formuleNU3} est v\'erifi\'ee d'apr\`es le th\'eor\`eme pr\'ec\'edent.

\end{proof}

\section{Quatri\`eme \'etape}

Nous allons \`a pr\'esent utiliser les r\'esultats obtenus dans les sections pr\'ec\'edentes pour trouver une formule asymptotique pour $ N_{U,\sigma}(B) $. Pour cela, il sera n\'ecessaire d'\'etablir un analogue de \cite[Th\'eor\`eme 2.1]{BB}.  
Introduisons les notions suivantes :

\begin{Def}\label{definition3}
Pour $ N,r\in \NN  $ fix\'es, on consid\`ere une famille de fonctions $ (h_{\ee})_{\ee\in \NN^{N}} $ de $ \NN^{r} $ dans $ [0,\infty[ $. Soient $ \aalpha=(\alpha_{1},...,\alpha_{r})\in \NN^{r} $, $ \bbeta=(\beta_{1},...,\beta_{N})\in \NN^{N} $ et $ \delta,D,\nu>0 $ avec $ \nu< 1 $. On suppose de plus que $ D\geqslant r\max_{j}(\alpha_{j}) $. On dira que $ (h_{\ee})_{\ee\in \NN^{N}} $ est une $ (\aalpha,D,\nu,\delta,\bbeta)- $ famille si les $ h_{\ee} $ v\'erifient les conditions suivantes : \begin{enumerate}
\item Pour tout $ \ee\in \NN^{N} $ et $ \XX=(X_{1},...,X_{r})\in \NN^{r} $, il existe une constante $ c_{\ee} $ telle que : \[ \sum_{\xx\leqslant \XX}h_{\ee}(\xx)=c_{\ee}\XX^{\aalpha}+O\left(\ee^{\bbeta}\XX^{\aalpha}(\min_{1\leqslant j\leqslant r}X_{j})^{-\delta}\right), \] o\`u l'on a not\'e $ \XX^{\aalpha}=\prod_{j=1}^{r}X^{\alpha_{j}}_{j} $, $ \ee^{\bbeta}=\prod_{j=1}^{N}e_{i}^{\beta_{i}} $ et o\`u $ \xx\leqslant \XX $ signifie $ x_{j}\leqslant X_{j} $ pour tout $ j\in \{1,...,r\} $. De plus $ (c_{\ee})_{\ee\in \NN^{N}} $ est telle que $ c_{\ee}\ll \ee^{\bbeta} $
\item Pour tout $ J\varsubsetneq\{1,...,r\} $ non vide de cardinal $ k\in \{1,...,r-1\} $, il existe des fonctions $ c_{J,\ee} : \NN^{k}\ra [0,\infty[ $ telles que pour tout $ \uu=(x_{j})_{j\in J}\in \NN^{k} $, la formule asymptotique \[ \sum_{\forall j\notin J, \; x_{j}\leqslant V_{j}}h_{\ee}(\xx)=c_{J,\ee}(\uu)\left(\prod_{j\notin J}V_{j}^{\alpha_{j}}\right)+O\left(\ee^{\bbeta}|\uu|^{D}\left(\prod_{j\notin J}V_{j}^{\alpha_{j}}\right)(\min_{ j\notin J}V_{j})^{-\delta}\right),\] est v\'erifi\'ee uniform\'ement pour tout $ (V_{j})_{j\notin J}\in (\NN^{\ast})^{r-k} $ et $ |\uu|\leqslant \left(\prod_{j\notin J}V_{j}\right)^{\nu} $. 
\end{enumerate}
\end{Def}
Nous allons alors \'etablir le r\'esultat suivant  

\begin{thm}\label{resultatBB3}
Soit $ r\geqslant 2 $ et $ (h_{\ee})_{\ee\in \NN^{N}} $ une $  (\aalpha,D,\nu,\delta,\bbeta)- $ famille de fonctions arithm\'etiques. Si l'on pose \[ \Upsilon_{\ee}(P)=\sum_{\xx^{\aalpha}\leqslant P}h_{\ee}(\xx), \] on a alors \[ \Upsilon_{\ee}(P)=\frac{1}{(r-1)!}c_{\ee}P(\log P)^{r-1}+O\left( \ee^{\bbeta}P(\log P)^{r-2}\right). \]
\end{thm}
La d\'emonstration de ce th\'eor\`eme est directement inspir\'ee de celle de V. Blomer et J.Br\"{u}dern pour \cite[Th\'eor\`eme 2.1]{BB} et nous la d\'ecomposerons, comme eux, en plusieurs \'etapes. 

\subsection{Familles de fonctions arithm\'etiques}

\begin{lemma}\label{sommepart13}
Soit $ (h_{\ee})_{\ee\in \NN^{N}} $ une $  (\aalpha,D,\nu,\delta,\bbeta)- $ famille. On a alors que : \[ h_{\ee}(\xx)\ll \ee^{\bbeta}\xx^{\aalpha} \]
\end{lemma}
\begin{proof}
Imm\'ediat en utilisant la propri\'et\'e $ 1 $ de la d\'efinition\;\ref{definition3} avec $ X_{j}=x_{j} $ pour tout $ j\in \{1,...,r\} $.
\end{proof}
\begin{lemma}\label{sommepart23}
Soit $ (h_{\ee})_{\ee\in \NN^{N}} $ une $  (\aalpha,D,\nu,\delta,\bbeta)- $ famille. Pour tout $ J\varsubsetneq\{1,...,r\} $  non vide de cardinal $ l $ , l'ensemble des fonctions $ (c_{J,\ee})_{\ee\in \NN^{N}} $ forme une $  (\aalpha,D,\nu,\delta,\bbeta)- $ famille, et on a de plus : \[ \sum_{\yy\leqslant \YY}c_{J,\ee}(\yy)=c_{\ee}\left(\prod_{j\in J}Y_{j}^{\alpha_{j}}\right)+O\left(\ee^{\bbeta}\left(\prod_{j\in J}Y_{j}^{\alpha_{j}}\right)(\min_{ j\in J}Y_{j})^{-\delta}\right). \]

\end{lemma}
\begin{proof}
Posons $ m=r-l $. Pour $ Z\geqslant 1 $, la condition $ 2 $ donne (pour $ \yy=(x_{i})_{i\in J} $) : \[ \sum_{\forall j\notin J, \; x_{j}\leqslant Z}h_{\ee}(\xx)=c_{J,\ee}(\yy)\left(\prod_{j\notin J}Z^{\alpha_{j}}\right)+O\left(\ee^{\bbeta}|\yy|^{D}Z^{\sum_{j\notin J}\alpha_{j}-\delta}\right).\] Par cons\'equent, pour tout $ \YY=(Y_{j})_{j\in J } $ tel que $ Y_{j}\geqslant 1 $ pour tout $ j $, satisfaisant $ |\YY|\leqslant Z^{\nu} $ : \begin{multline}\label{partiel13} \sum_{\yy\leqslant \YY}\sum_{\forall j\notin J, \; x_{j}\leqslant Z}h_{\ee}(\xx) \\ =Z^{\sum_{j\notin J}\alpha_{j}}\sum_{\yy\leqslant \YY}c_{J,\ee}(\yy)+O\left(\ee^{\bbeta}Z^{\sum_{j\notin J}\alpha_{j}-\delta}\left(\prod_{j\in J}Y_{j}\right)^{D+1}\right). \end{multline} D'autre part, d'apr\`es la condition $ 1 $, lorsque $ Z\geqslant |\YY| $ : \begin{multline}\label{partiel2} \sum_{\yy\leqslant \YY}\sum_{\forall j\notin J, \; x_{j}\leqslant Z}h_{\ee}(\xx) \\ =c_{\ee}Z^{\sum_{j\notin J}\alpha_{j}}\left(\prod_{j\in J}Y_{j}^{\alpha_{j}}\right)+O\left(\ee^{\bbeta}Z^{\sum_{j\notin J}\alpha_{j}}\left(\prod_{j\in J}Y_{j}^{\alpha_{j}}\right)(\min_{ j\in J}Y_{j})^{-\delta}\right). \end{multline}
En regroupant les formules \eqref{partiel13} et \eqref{partiel2} puis en simplifiant par $ Z^{\sum_{j\notin J}\alpha_{j}} $, on a pour $ Z\geqslant |\YY|^{\frac{1}{\nu}} $ :\[ \sum_{\yy\leqslant \YY}c_{J,\ee}(\yy)=c_{\ee}\left(\prod_{j\in J}Y_{j}^{\alpha_{j}}\right)+O\left(\ee^{\bbeta}\left(\prod_{j\in J}Y_{j}^{\alpha_{j}}\right)(\min_{ j\in J}Y_{j})^{-\delta}\right)+O\left(\ee^{\bbeta}|\YY|^{D+1}Z^{-\delta}\right). \] En particulier la fonction $ c_{J,\ee} $ v\'erifie la condition $ 1 $ de la d\'efinition\;\ref{definition3}. \\

V\'erifions \`a pr\'esent que les $ c_{J,\ee} $ v\'erifient la condition $ 2 $. Consid\'erons $ K\subsetneq J$ de cardinal $ k\in \{1,...,l-1\} $. Si $ \xx\in \NN^{r} $, on pose $ \uu=(x_{i})_{i\in K} $, $ \vv=(x_{i})_{i\in J\setminus K } $, $ \zz=(x_{i})_{i\notin J} $, $ \yy=(x_{i})_{i\in J} $. On suppose $ |\uu|\leqslant Z $, $ |\vv|\leqslant Z^{m\nu} $, $ |\VV|\leqslant Z $ (o\`u $ \VV=(V_{j})_{j\in J\setminus K}) $, $ |\VV|\leqslant Z^{m\nu} $ et $ |\uu|\leqslant \left(\prod_{j\in J\setminus K}V_{j}\right)^{\nu}Z^{m\nu} $. On a alors par la condition $ 2 $ sur $ h_{\ee} $, pour $ \uu $ fix\'e : \begin{multline*} \sum_{\vv\leqslant \VV} \sum_{\forall i \notin J, \; z_{i}\leqslant Z}h_{\ee}(\xx)=c_{K,\ee}(\uu)Z^{\sum_{j\notin J}\alpha_{j}}\prod_{j\in J \setminus K}V_{j}^{\alpha_{j}} \\ +O\left( \ee^{\bbeta}|\uu|^{D}Z^{\sum_{j\notin J}\alpha_{j}}\left(\prod_{j\in J \setminus K}V_{j}^{\alpha_{j}}\right)(\min_{j\in J\setminus K}V_{j})^{-\delta}\right), \end{multline*} et d'autre part\begin{multline*} \sum_{\vv\leqslant \VV} \sum_{\forall j \notin J, \; z_{i}\leqslant Z}h_{\ee}(\xx)= \sum_{\vv\leqslant \VV} c_{J,\ee}(\yy)Z^{\sum_{j\notin J}\alpha_{j}} \\+O\left( \ee^{\bbeta}|\uu|^{D}\left(\prod_{j\in J\setminus K}V_{j}\right)^{D+1}Z^{\sum_{j\notin J}\alpha_{j}-\delta}\right). \end{multline*} En regroupant ces deux \'egalit\'es et en simplifiant par $ Z^{\sum_{j\notin J}\alpha_{j}} $ on trouve : \begin{multline*}\sum_{\vv\leqslant \VV} c_{J,\ee}(\yy)=c_{K,\ee}(\uu)\prod_{j\in J \setminus K}V_{j}^{\alpha_{j}}+O\left( \ee^{\bbeta}|\uu|^{D}\left(\prod_{j\in J \setminus K}V_{j}^{\alpha_{j}}\right)(\min_{j\in J\setminus K}V_{j})^{-\delta}\right) \\ +O\left( \ee^{\bbeta}|\uu|^{D}\left(\prod_{j\in J\setminus K}V_{j}\right)^{D+1}Z^{-\delta}\right). \end{multline*}
En prenant $ Z $ assez grand on obtient alors la condition $ 2 $ pour $ c_{J,\ee} $, et la fonction arithm\'etique correspondant \`a $ K $ est alors $ c_{K,\ee} $. 

\end{proof}
Consid\'erons \`a pr\'esent un r\'eel $ A $ et $ J\subset \{1,...,r\} $ de cardinal $ l $. On fixe $ \ww=(w_{i})_{i\in J} $ et on pose $ g_{\ee,\ww} $ la fonction de $ \NN^{r-l} $ dans $ [0,\infty[ $ d\'efinie par, pour tout $ \yy=(y_{j})_{j\notin J} $ : \[ g_{\ee,\ww}(\yy)=\langle \ww \rangle^{-A}h_{\ee}(\xx) \] o\`u $ \langle \ww \rangle=\prod_{i\in J}w_{i} $, $ x_{j}=w_{j} $ pour tout $ j\in J $, et $ x_{j}=y_{j} $ pour $ j\notin J $. 

\begin{lemma}\label{lemmA3}
Si $ (h_{\ee})_{\ee\in \NN^{N}} $ une $  (\aalpha,D,\nu,\delta,\bbeta)- $ famille, avec $ D\geqslant r\max_{j}(\alpha_{j})  $ et $ \delta\leqslant \min\{1,\min_{j}(\alpha_{j})\} $, et si $ A\geqslant D+(r+1)(\max_{j\in \{1,...,r\}}\alpha_{j})+\nu^{-1}(1+(\max_{j\in \{1,...,r\}}\alpha_{j})) $ alors $ (g_{\ee,\ww})_{\ee} $ est une $  (\aalpha,D,\nu,\delta,\bbeta)- $ famille. 
\end{lemma} 
\begin{proof}
Quitte \`a permuter les variables, on peut supposer que $ J=\{1,...,l\} $. On note alors $ c_{\ee,l} $ pour $ c_{\ee,J} $. D'apr\`es le lemme\;\ref{sommepart13}, on a : \begin{equation}\label{formsup3}
\left(\prod_{j=1}^{l}w_{j}\right)^{-A}c_{\ee,l}(\ww)\ll \left(\prod_{j=1}^{l}w_{j}^{-\alpha_{j}}\right)c_{\ee,l}(\ww)\ll \ee^{\bbeta}.
\end{equation} 
Montrons que pour tout $ \ww $, en posant $ \langle \ww \rangle=\prod_{i=1}^{l}w_{i} $, on a : \begin{equation}\label{cond1g3}
\sum_{\yy\leqslant \YY}\frac{h_{\ee}(\ww,\yy)}{\langle \ww \rangle^{A}}=\frac{c_{\ee,l}(\ww)}{\langle \ww \rangle^{A}}\left(\prod_{j=l+1}^{r}Y_{j}^{\alpha_{j}}\right)+O\left(\ee^{\bbeta}\left(\prod_{j=l+1}^{r}Y_{j}^{\alpha_{j}}\right)(\min Y_{j})^{-\delta}\right)
\end{equation}
(ce qui impliquera que $ (g_{\ee,\ww})_{\ee} $ v\'erifie la condition $ 1. $ avec $ \frac{c_{\ee,l}(\ww)}{\langle \ww \rangle^{A}} $ \`a la place de $ c_{\ee} $). Supposons dans un premier temps que $ |\ww|\leqslant \langle \YY \rangle^{\nu} $. Puisque $ A>D $, la formule\;\eqref{cond1g3} d\'ecoule directement de la condition $ 2 $ pour $ (h_{\ee})_{\ee} $. Supposons \`a pr\'esent que $ |\ww|> \langle \YY \rangle^{\nu} $. On remarque que : \begin{align*}  \frac{c_{\ee,l}(\ww)}{\langle \ww \rangle^{A}}\left(\prod_{j=l+1}^{r}Y_{j}^{\alpha_{j}}\right) & \ll \ee^{\bbeta}\left(\prod_{j=1}^{l}w_{j}^{\alpha_{j}-A}\right)\left(\prod_{j=l+1}^{r}Y_{j}^{\alpha_{j}}\right) \\ & \ll \ee^{\bbeta}|\ww|^{r(\max \alpha_{j})-A}\left(\prod_{j=l+1}^{r}Y_{j}\right)^{\max\alpha_{j}} \\ & \ll \ee^{\bbeta}|\ww|^{\max\alpha_{j}(r+\nu^{-1})-A} \ll \ee^{\bbeta}. \end{align*}
De m\^eme, d'apr\`es le lemme\;\ref{sommepart13} : \begin{align*} \sum_{\yy\leqslant \YY}\frac{h_{\ee}(\ww,\yy)}{\langle \ww \rangle^{A}} & \ll \ee^{\bbeta}\left(\prod_{j=1}^{l}w_{j}^{\alpha_{j}-A}\right)\left(\prod_{j=l+1}^{r}Y_{j}^{\alpha_{j}+1}\right) \\ & \ll |\ww|^{(\max\alpha_{j})(r+1+\nu^{-1})-A}\leqslant \ee^{\bbeta}. 
\end{align*}
\'Etant donn\'e que $ \delta\leqslant \min\{1,\min_{j}(\alpha_{j})\} $, on a que $ \left(\prod_{j=l+1}^{r}Y_{j}^{\alpha_{j}}\right)(\min Y_{j})^{-\delta}\geqslant 1 $, et l'\'egalit\'e\;\eqref{cond1g3} est donc vraie. \\

Nous allons montrer de fa\c{c}on analogue que $ (g_{\ee,\ww})_{\ee} $ satisfait la condition $ 2 $. Consid\'erons alors $ J'\subset \{l+1,...,r\} $. Quitte \`a permuter les variables, on peut supposer que $ J'=\{l+1,...,l+k\} $ pour un certain $ k\in \{1,...,r-l-1\} $. Nous allons chercher \`a montrer que \begin{equation}\label{cond2g3}
\sum_{\zz\leqslant \ZZ}\frac{h_{\ee}(\ww,\vv,\zz)}{\langle \ww \rangle^{A}}=\frac{c_{\ee,l+k}(\ww,\vv)}{\langle \ww \rangle^{A}}\left(\prod_{j=l+k+1}^{r}Z_{j}^{\alpha_{j}}\right)+O\left(\ee^{\bbeta}|\vv|^{D}\left(\prod_{j=l+k+1}^{r}Z_{j}^{\alpha_{j}}\right)(\min Z_{j})^{-\delta}\right)
\end{equation} uniform\'ement pour $ |\vv|\leqslant \langle \ZZ \rangle^{\nu} $. Comme pr\'ec\'edemment, nous commen\c{c}ons par le cas $ |\ww|\leqslant \langle \ZZ \rangle^{\nu} $. Puisque $ A>D $, on a $ |(\ww,\vv)|^{D}\langle \ww \rangle^{-A}\leqslant |\vv|^{D} $ et la formule\;\eqref{cond2g3} d\'ecoule \`a nouveau de la condition $ 2 $ pour $ (h_{\ee})_{\ee} $. Si $ |\ww|> \langle \YY \rangle^{\nu} $, d'apr\`es le lemme\;\ref{sommepart13}, on a comme pr\'ec\'edemment : \begin{align*}  \frac{c_{\ee,l+k}(\ww)}{\langle \ww \rangle^{A}}\left(\prod_{j=l+k+1}^{r}Z_{j}^{\alpha_{j}}\right) & \ll \ee^{\bbeta}\left(\prod_{j=1}^{l}w_{j}^{\alpha_{j}-A}\right)\left(\prod_{j=1+1}^{l+k}v_{j}^{\alpha_{j}}\right)\left(\prod_{j=l+k+1}^{r}Z_{j}^{\alpha_{j}}\right) \\ & \ll \ee^{\bbeta}|\ww|^{\max\alpha_{j}(r+\nu^{-1})-A}\left(\prod_{j=1+1}^{l+k}v_{j}^{\alpha_{j}}\right) \ll \ee^{\bbeta}\left(\prod_{j=1+1}^{l+k}v_{j}^{\alpha_{j}}\right) \end{align*}
ainsi que \begin{align*} \sum_{\zz\leqslant \ZZ}\frac{h_{\ee}(\ww,\vv,\zz)}{\langle \ww \rangle^{A}} & \ll \ee^{\bbeta}\left(\prod_{j=1}^{l}w_{j}^{\alpha_{j}-A}\right)\left(\prod_{j=1+1}^{l+k}v_{j}^{\alpha_{j}}\right)\left(\prod_{j=l+k+1}^{r}Y_{j}^{\alpha_{j}+1}\right) \\ &\ll  \ee^{\bbeta}\left(\prod_{j=1+1}^{l+k}v_{j}^{\alpha_{j}}\right). 
\end{align*}
d'o\`u le r\'esultat (puisque $ D\geqslant r\max \alpha_{j} $). 
\end{proof}

\subsection{Lemmes pr\'eliminaires}

Pour $ X>1 $, on note \begin{equation}
\Delta^{(r)}=\{\tt\in \RR^{r} \; |\; 1<t_{1}<t_{2}<...<t_{r} \},
\end{equation}\begin{equation}
\Delta^{(r)}(X)=\{\tt\in \Delta^{(r)}\; |\; t_{r}\leqslant X \}.
\end{equation}
Nous allons chercher \`a \'evaluer les int\'egrales \[ I_{\ee,j}(X)=\int_{\Delta^{(r)}(X)}\frac{(\log\pt)^{j}}{\tt^{\aalpha+\1}}\sum_{\xx\leqslant \tt}h_{\ee}(\xx)d\tt, \] o\`u $ \aalpha+\1=(\alpha_{1}+1,...,\alpha_{r}+1) $ et $ \pt=\prod_{j=1}^{r}t_{j} $. 
\begin{lemma}\label{intj'3}
Soit $ (h_{\ee})_{\ee\in \NN^{N}} $ une $  (\aalpha,D,\nu,\delta,\bbeta)- $ famille et $ j\in \NN^{\ast} $. Il existe un r\'eel $ \eta>0  $ et un polyn\^ome $ q_{\ee,j} $ \`a coefficients de taille $ O(\ee^{\bbeta}) $ tels que : \[  I_{\ee,j}(X)=q_{\ee,j}(\log(X))+O(\ee^{\bbeta}X^{-\eta}). \]
\end{lemma}
\begin{proof}
Commen\c{c}ons par traiter le cas $ r=1 $. Dans ce cas, on peut r\'e\'ecrire la condition $ 1. $ sous la forme : \[ \sum_{x\leqslant t}h_{\ee}(\xx)=c_{\ee}t^{\alpha}+E_{\ee}(t) \] o\`u  $ E_{\ee} $ est une fonction continue par morceaux telle que $ E_{\ee}(t)\ll \ee^{\bbeta}t^{\alpha-\delta} $. On a alors que \[\int_{1}^{X}\frac{(\log t)^{j}}{t^{\alpha+1}}\sum_{x\leqslant t}h_{\ee}(x)dt=\frac{c_{\ee}}{(j+1)}(\log X)^{j+1}+ D_{\ee,j}+O(\ee^{\bbeta}X^{-\delta}\log X) \] avec \[ D_{\ee,j}=\int_{1}^{\infty}E_{\ee}(t)t^{-\alpha-1}(\log t)^{j}dt \ll \ee^{\bbeta}. \] D'o\`u le r\'esultat lorsque $ r=1 $. \\

Nous allons \`a pr\'esent proc\'eder par r\'ecurrence sur $ r $. Nous supposons ici que $ r>1 $ et que le lemme est \'etabli pour toute valeur strictement inf\'erieure \`a $ r $. Nous allons s\'eparer l'ensemble $ \Delta^{(r)} $ en une union de $ r $ ensembles disjoints : on pose $ t_{0}=1 $ et $ \beta=\min\{\nu,\delta(2D+4r)^{-1}\} $ (en particulier $ \beta<1 $, puisque $ \nu<1 $). Pour $ l\in \{0,...,r-1\} $, notons 
\begin{equation}
\Delta^{(r,l)}=\{\tt\in \Delta^{(r)} \; |\; t_{l}\leqslant t_{l+1}^{\beta}, \; \forall i\in \{l+1,...,r-1\},\; t_{i}>t_{i+1}^{\beta} \},
\end{equation}\begin{equation}
\Delta^{(r,l)}(X)=\Delta^{(r,l)}\cap\Delta^{(r)}(X).
\end{equation}
On a alors que \[ \Delta^{(r)}(X)=\bigsqcup_{l\in \{0,...,r-1\}}\Delta^{(r,l)}(X), \] et on a donc \[ I_{\ee,j}(X)=\sum_{l=0}^{r-1}I_{\ee,j,l}(X), \] o\`u l'on a not\'e : \begin{equation}
I_{\ee,j,l}(X)=\int_{\Delta^{(r,l)}(X)}\frac{(\log\pt)^{j}}{\tt^{\aalpha+\1}}\sum_{\xx\leqslant \tt}h_{\ee}(\xx)d\tt.
\end{equation}
Traitons d'abord le cas des int\'egrales $ I_{\ee,j,0}(X) $. Pour $ \tt\in \Delta^{(r,0)}(X) $, on pose : \[ \sum_{\xx\leqslant \tt}h_{\ee}(\xx)=c_{\ee}\tt^{\aalpha}+E_{\ee}(\tt) \] avec \[ E_{\ee}(\tt)\ll \ee^{\bbeta}\tt^{\aalpha}t_{1}^{-\delta}. \] Pour tout $ Z\geqslant 1 $ : 
\begin{multline*}  \int_{\Delta^{(r,0)}(2Z)\setminus\Delta^{(r,0)}(Z)}\frac{(\log\pt)^{j}}{\tt^{\aalpha+\1}} E_{\ee}(\tt)d\tt \\  \ll \ee^{\bbeta}\int_{Z}^{2Z}\int_{t_{r}^{\beta}}^{t_{r}}\int_{t_{r-1}^{\beta}}^{t_{r-1}}...\int_{t_{2}^{\beta}}^{t_{2}}\pt^{-1}t_{1}^{-\delta}(\log t_{r})^{j}d\tt \ll \ee^{\bbeta}Z^{-\delta\beta^{r}}.  \end{multline*}
On en d\'eduit, en sommant sur les intervalles dyadiques, que l'int\'egrale  \[ D_{\ee,j,0}=\int_{\Delta^{(r,0)}}\frac{(\log\pt)^{j}}{\tt^{\aalpha+\1}} E_{\ee}(\tt)d\tt \] existe et diff\`ere de l'int\'egrale sur $ \Delta^{(r,0)}(X) $ d'un terme d'erreur $ O(\ee^{\bbeta}X^{-\delta\beta^{r}}) $. On a donc que : \[ I_{\ee,j,0}(X)=c_{\ee}\int_{\Delta^{(r,0)}(X)}\frac{(\log\pt)^{j}}{\pt}d\tt+ \underbrace{D_{\ee,j,0}}_{\ll \ee^{\bbeta}}+O(\ee^{\bbeta}X^{-\delta\beta^{r}}). \]
Consid\'erons \`a pr\'esent les int\'egrales $ I_{\ee,j,l}(X) $ pour $ l\in \{1,...,r-1\} $. Pour $ \tt\in \RR^{r} $, on pose $ \tt=(\tt',\tt'') $ avec $ \tt'=(t_{1},...,t_{l}) $ et $ \tt''=(t_{l+1},...,t_{r}) $. On a alors $ |\tt'|\leqslant |\tt''|^{\nu} $ (car $ \beta\leqslant \nu $) pour tout $ \tt\in \Delta^{(r,l)} $. Par cons\'equent, d'apr\`es la condition $ 2 $, on a uniform\'ement pour tout $ \xx'\leqslant \tt' $. \[ \sum_{\xx''\leqslant \tt''}h_{\ee}(\xx',\xx'')=c_{\ee,J_{l}}(\xx')\prod_{j=l+1}^{r}t_{j}^{\alpha_{j}}+E_{\ee,J_{l}}(\xx',\tt''), \] o\`u \[ J_{l}=\{1,...,l\}, \] \[ E_{\ee,J_{l}}(\xx',\tt'') \ll \ee^{\bbeta}|\xx'|^{D}t_{l+1}^{-\delta}\prod_{j=l+1}^{r}t_{j}^{\alpha_{j}}. \] Posons \[ R_{\ee,l}(\tt)=\sum_{\xx'\leqslant \tt'}E_{\ee,J_{l}}(\xx',\tt''). \] Ceci d\'efinit une fonction sur $ \Delta^{(r,l)} $ satisfaisant \[R_{\ee,l}(\tt)\ll \ee^{\bbeta}\langle \tt'\rangle t_{l}^{D}t_{l+1}^{-\delta} \prod_{j=l+1}^{r}t_{j}^{\alpha_{j}}. \]
Pour $ Z\geqslant 1 $ on voit alors que :  

\begin{multline*}  \int_{\Delta^{(r,l)}(2Z)\setminus\Delta^{(r,l)}(Z)}\frac{(\log\pt)^{j}}{\tt^{\aalpha+\1}} R_{\ee,l}(\tt)d\tt \\  \ll \ee^{\bbeta}\int_{\mathcal{R}_{l}(Z)}\langle \tt''\rangle^{-1} t_{l+1}^{-\delta}\int_{\Delta^{(l)}(t_{l+1}^{\beta})}t_{l}^{D}\prod_{j=1}^{l}t_{j}^{-\alpha_{j}}d\tt'd\tt'',  \end{multline*}
o\`u \[ \mathcal{R}_{l}(Z)=\{\tt''\; |\; Z<t_{r}<2Z,\; \forall i\in \{l+1,...,r-1\},\;  t_{i+1}^{\beta}\leqslant t_{i}\leqslant t_{i+1}\}. \] On remarque que $ t_{l}^{D}\prod_{j=1}^{l}t_{j}^{-\alpha_{j}}\ll t_{l+1}^{D\beta} $ pour tout $ \tt'\in \Delta^{(l)}(t_{l+1}^{\beta}) $, et que la mesure de $ \Delta^{(l)}(t_{l+1}^{\beta}) $ est major\'ee par $ t_{l+1}^{l\beta} $ et on en d\'eduit, puisque $ D\beta+l\beta\leqslant \frac{\delta}{2} $, que l'int\'egrale ci-dessus est born\'ee par \[ \ee^{\bbeta}\int_{\mathcal{R}_{l}(Z)}\langle \tt''\rangle^{-1} t_{l+1}^{-\delta/2}d\tt'' \ll \ee^{\bbeta}Z^{-\delta\beta^{r}/2}. \] Ainsi, en sommant sur les intervalles dyadiques, il s'ensuit que \[ D_{\ee,j,l}=\int_{\Delta^{(r,l)}}\frac{(\log\pt)^{j}}{\tt^{\aalpha+\1}} R_{\ee,l}(\tt)d\tt \] est du type $ O(\ee^{\bbeta}) $ et diff\`ere de l'int\'egrale sur $ \Delta^{(r,l)}(X) $ de $ O(\ee^{\bbeta}X^{-\delta\beta^{r}/2}) $. On a donc que : \[ I_{\ee,j,l}(X) =K_{\ee,j,l}+D_{\ee,j,l}+O(\ee^{\bbeta}X^{-\delta\beta^{r}/2}), \] o\`u \[ K_{\ee,j,l}=\int_{\Delta^{(r,l)}(X)}\frac{(\log\pt)^{j}}{\langle \tt''\rangle\prod_{j=1}^{l}t_{j}^{\alpha_{j}+1}}\sum_{\xx'\leqslant \tt'}c_{\ee,J_{l}}(\xx')d\tt. \] On observe que, par la formule du bin\^ome de Newton : \[ K_{\ee,j,l}=\sum_{k=0}^{j}\left(\begin{array}{l}
j \\ k 
\end{array}\right)\int_{\mathcal{S}_{l}(X)}\frac{(\log\langle \tt''\rangle )^{j-k}}{\langle \tt''\rangle}\int_{\Delta^{(l)}(t_{l+1}^{\beta})}\frac{(\log\langle \tt'\rangle)^{k}}{\prod_{j=1}^{l}t_{j}^{\alpha_{j}+1}}\sum_{\xx'\leqslant \tt'}c_{\ee,J_{l}}(\xx')d\tt \] o\`u l'on a not\'e : \[  \mathcal{S}_{l}(X)=\{ \tt''\; |\; t_{r}\leqslant X, \;  \et \forall i\in \{l+1,...,r-1\}, \; t_{i+1}^{\beta}\leqslant t_{i} \leqslant t_{i+1} \}. \] En appliquant alors l'hypoth\`ese de r\'ecurrence ainsi que le lemme\;\ref{sommepart23}, on trouve alors :  \begin{equation*} \int_{\Delta^{(l)}(t_{l+1}^{\beta})}\frac{(\log\langle \tt'\rangle)^{k}}{\prod_{j=1}^{l}t_{j}^{\alpha_{j}+1}}\sum_{\xx'\leqslant \tt'}c_{\ee,J_{l}}(\xx')d\tt'  =Q_{\ee,l,k}(\log t_{l+1}^{\beta})+F_{\ee,l,k}(t_{l+1}^{\beta}), \end{equation*} o\`u $ Q_{\ee,l,k} $ est un polyn\^ome \`a coefficients de type $ O(\ee^{\bbeta}) $ et $ F_{\ee,l,k} $ une fonction telle que $ F_{\ee,l,k}(t)\ll \ee^{\bbeta}t^{-\eta(l)} $ pour un certain $ \eta(l)>0 $. Comme nous l'avons fait pr\'ec\'edemment, nous en d\'eduisons qu'il existe un r\'eel $ E_{\ee,j,l,k} $ tel que $ E_{\ee,j,l,k}\ll \ee^{\bbeta} $ et \[ \int_{\mathcal{S}_{l}(X)}\frac{(\log\langle \tt''\rangle )^{j-k}}{\langle \tt''\rangle}F_{\ee,l,k}(t_{l+1}^{\beta})d\tt''=E_{\ee,j,l,k}+O(\ee^{\bbeta}X^{-\eta(l)\beta^{r}/2}). \] En sommant sur $ k $, on a alors qu'il existe un r\'eel, not\'e $ E_{\ee,j,l} $ tel que : \[ K_{\ee,j,l}=E_{\ee,j,l}+\int_{S_{l}(X)}\sum_{k=0}^{j}\left(\begin{array}{l}
j \\ k 
\end{array}\right)\frac{(\log\langle \tt''\rangle )^{j-k}}{\langle \tt''\rangle}Q_{\ee,l,k}(\log t_{l+1}^{\beta})d\tt''+O(\ee^{\bbeta}X^{-\eta(l)\beta^{k}/2}). \]
On obtient donc finalement : \begin{multline*}
 I_{\ee,j,l}(X) =E_{\ee,j,l}+\int_{S_{l}(X)}\sum_{k=0}^{j}\left(\begin{array}{l}
j \\ k 
\end{array}\right)\frac{(\log\langle \tt''\rangle )^{j-k}}{\langle \tt''\rangle}Q_{\ee,l,k}(\log t_{l+1}^{\beta})d\tt'' \\ +O(\ee^{\bbeta}X^{-\eta(l)\beta^{k}/2})+D_{\ee,j,l}+O(\ee^{\bbeta}X^{-\delta\beta^{r}/2}),
\end{multline*}
On peut par ailleurs d\'evelopper $ (\log\langle \tt''\rangle )^{j-k}Q_{\ee,l,k}(\log t_{l+1}^{\beta}) $ en un polyn\^ome en $ \log(t_{i}) $ pour $ i\in \{l+1,...,r\} $ \`a coefficients $ O(\ee^{\bbeta}) $, et on peut donc \'ecrire l'int\'egrale de la formule ci-dessus comme une combinaison lin\'eaire (\`a coefficients de type $ O(\ee^{\bbeta}) $) d'int\'egrales du type : \[  \int_{S_{l}(X)}\frac{\prod_{i=l+1}^{r}(\log t_{i})^{b_{i}}}{\prod_{i=l+1}^{r}t_{i}}d\tt'' \] avec $ b_{l+1},...,b_{r}\in \NN\setminus \{0\} $. Ces int\'egrales peuvent \^etre calcul\'ees explicitement et sont des polyn\^omes en $ \log(X) $ \`a un terme d'erreur de type $ O(X^{-\beta^{r}/2}) $ pr\`es, d'o\`u le r\'esultat.

\end{proof}
On d\'eduit directement de ce lemme le r\'esultat suivant : 
\begin{lemma}\label{intj3}
Soit $ (h_{\ee})_{\ee\in \NN^{N}} $ une $  (\aalpha,D,\nu,\delta,\bbeta)- $ famille et $ j\in \NN^{\ast} $. Il existe un r\'eel $ \eta>0  $ et un polyn\^ome $ Q_{\ee,j} $ \`a coefficients de taille $ O(\ee^{\bbeta}) $ tels que : \[  \int_{[1,X]^{r}}\frac{(\log\pt)^{j}}{\tt^{\aalpha+\1}}\sum_{\xx\leqslant \tt}h_{\ee}(\xx)d\tt=Q_{\ee,j}(\log(X))+O(\ee^{\bbeta}X^{-\eta}). \]
\end{lemma}

\subsection{Un r\'esultat interm\'ediaire}

On pose pour $ r\in \NN $, $ j\in \NN^{\ast} $ : \begin{equation}
V_{r,j}=\int_{[0,1]^{r}}(\xi_{1}+...+\xi_{r})^{j}d\xxi.
\end{equation}
On remarque que : \[ V_{r,j}=\sum_{\substack{a_{1}+...+a_{r}=j\\\forall i \in 
\{1,...,r\}, \; a_{i}\geqslant 0}}\left( \begin{array}{c} j \\ a_{1}\;... \;a_{k}\end{array}\right) \frac{1}{(a_{1}+1)...(a_{r}+1)}. \] 
\begin{thm}\label{thmintermediaire3}
Soit $ (h_{\ee})_{\ee} $ une $  (\aalpha,D,\nu,\delta,\bbeta)- $ famille et $ j\in \NN^{\ast} $. Il existe alors un r\'eel $ \eta>0 $ et une famille de polyn\^omes $ (p_{\ee,j}) $ de degr\'e au plus $  r+j $ dont les coefficients sont de type $ O(\ee^{\bbeta}) $ tels que : \[ \sum_{\forall i, \; x_{i}\leqslant X}\frac{(\log\px)^{j}}{\xx^{\aalpha}}h_{\ee}(\xx)=p_{\ee,j}(\log X)+O(\ee^{\bbeta}X^{-\eta}). \] De plus, si $ c_{\ee}\neq 0 $, alors le degr\'e de $ p_{\ee,j} $ est exactement $ r+j $ et son coefficient dominant est $ \left(\prod_{i=1}^{r}\alpha_{i}\right)c_{\ee}V_{r,j} $. 
\end{thm}
\begin{proof}
Commen\c{c}ons par remarquer que \[ \frac{\partial}{\partial x}\frac{\log(xy)^{j}}{x^{\alpha}y^{\beta}}=-\frac{\alpha q_{j}(\log(xy))}{x^{\alpha+1}y^{\beta}}, \] o\`u $ q_{j}(t)=t^{j}+j\alpha^{-1}t^{j-1} $. En appliquant cette \'egalit\'e r\'ecursivement on obtient alors par sommation par parties : \[  \sum_{\forall i, \; x_{i}\leqslant X}\frac{(\log\px)^{j}}{\xx^{\aalpha}}h_{\ee}(\xx)=\sum_{J\subset \{1,...,r\}}\left(\prod_{i\in J}\alpha_{i}\right)\Xi_{J}, \] o\`u \[ \Xi_{\emptyset}=(r\log X)^{j}X^{-\sum_{i=1}^{r}\alpha_{i}}\sum_{\forall i, \; x_{i}\leqslant X}h_{\ee}(\xx), \] et pour $ J\neq \emptyset $ de cardinal $ n $ et $ m=r-n $ : \[ \Xi_{J}=X^{-\sum_{i\notin J}\alpha_{i}}\int_{[1,X]^{n}}\frac{q_{J}(\log(X^{m}\prod_{i\in J}t_{i}))}{\prod_{i\in J}t_{i}^{\alpha_{i}+1}}\sum_{i\in J, \; x_{i}\leqslant t_{i}}\sum_{i\notin J, \; x_{i}\leqslant X}h_{\ee}(\xx)d(t_{i})_{i\in J}, \] o\`u $ q_{J} $ est un polyn\^ome unitaire de degr\'e $ j $. L'objectif est \`a pr\'esent de montrer que pour tout $ J\subset \{1,...,r\} $, il existe un polyn\^ome $ p_{\ee,J} $ dont les coefficients sont du type $ O(\ee^{\bbeta}) $, et satisfaisant la propri\'et\'e : \begin{equation}\label{formulecruciale3}
\Xi_{J}=p_{J}(\log X)+O\left( \ee^{\bbeta}X^{-\eta}\right).
\end{equation}
Si $ J=\{1,...,r\} $, la formule\;\eqref{formulecruciale3} r\'esulte du lemme\;\ref{intj3}. D'autre part, si $ J=\emptyset $, par d\'efinition de $ \Xi_{\emptyset} $, on a \[ \Xi_{\emptyset}=(r\log X)^{j}(c_{\ee}+O(\ee^{\bbeta}X^{-\delta})), \] d'o\`u le r\'esultat. Il nous reste \`a traiter le cas o\`u $ n\in \{1,...,r-1\} $. Quitte \`a permuter les variables $ x_{i} $, on peut se ramener au cas o\`u $ J=\{1,...,n\} $. \'Ecrivons alors $ \Xi_{n}=\Xi_{\{1,...,n\}} $. Soit $ \tt=(t_{1},...,t_{n}) $. Par d\'eveloppement multinomial, on remarque qu'il existe des constantes $ \gamma_{k,s} $ telles que : \[ q_{\{1,...,n\}}(\log(X^{m}\prod_{i\notin J}t_{i}))=\sum_{k+s\leqslant j}\gamma_{k,s}(\log(\prod_{i=1}^{n}t_{i}))^{k}(\log X)^{s}, \] et par cons\'equent, \begin{multline}\label{Xin3} \Xi_{n}=X^{-\sum_{i=n+1}^{r}\alpha_{i}}\sum_{k+s\leqslant j}\gamma_{k,s}(\log X)^{s} \\ \int_{[1,X]^{n}}\frac{(\log(\prod_{i=1}^{n}t_{i}))^{k}}{\prod_{i=1}^{n}t_{i}^{\alpha_{i}+1}}\sum_{\substack{ x_{i}\leqslant t_{i} \\ i\in \{1,...,n\} }}\sum_{\substack{ x_{i}\leqslant X \\ i\in \{n+1,...r\}}}h_{\ee}(\xx)d\tt. \end{multline}
Premi\`ere \'etape : remarquons que \begin{multline}\label{int1X3}
 \int_{[1,X]^{n}}\frac{(\log(\prod_{i=1}^{n}t_{i}))^{k}}{\prod_{i=1}^{n}t_{i}^{\alpha_{i}+1}}\sum_{\substack{ x_{i}\leqslant t_{i} \\ i\in \{1,...,n\} }}\sum_{\substack{ x_{i}\leqslant X \\ i\in \{n+1,...r\}}}h_{\ee}(\xx)d\tt \\ =\sum_{\sigma\in \mathfrak{S}_{n}}\int_{\Delta^{(n)}(X)}\frac{(\log(\prod_{i=1}^{n}t_{i}))^{k}}{\prod_{i=1}^{n}t_{i}^{\alpha_{i}}}\sum_{\substack{ x_{\sigma(i)}\leqslant t_{i} \\ i\in \{1,...,n\} }}\sum_{\substack{ x_{i}\leqslant X \\ i\in \{n+1,...r\}}}h_{\ee}(\xx)d\tt.
\end{multline}
Comme dans la section pr\'ec\'edente, nous allons partitionner l'ensemble $ \Delta^{(n)}(X) $. Prenons \`a nouveau $ \beta=\min\{\nu,\delta(2D+4r)^{-1}\} $, et consid\'erons les intervalles : \[ \forall i\in \{0,...,n-1\}, \; \mathfrak{I}_{i}=]X^{\beta^{i+1}},X^{\beta^{i}}], \; \; \mathfrak{I}_{n}=[1,X^{\beta^{n}}], \] On a alors \[ [1,X]=\bigsqcup_{0\leqslant i\leqslant n}\mathfrak{I}_{i}. \]
\'Etant donn\'e $ \tt\in \Delta^{(n)}(X) $, il existe au moins un entier $ i\in \{0,...,n\} $ tel que $ \mathfrak{I}_{i} $ ne contient aucune des coordonn\'ees de $ \tt $. Notons $ i(\tt) $ le plus petit de ces entiers $ i $. On pose alors $ l(\tt)=0 $ si $ t_{1}>X^{\beta^{i(\tt)}} $, et sinon $ l(\tt) $ d\'esigne le plus grand entier $ l $ tel que $ t_{l}\leqslant X^{\beta^{i(\tt)+1}} $. On note alors \[ \Delta^{(n)}_{i,l}(X)=\{\tt\in \Delta^{(n)}(X)\; | \; i(\tt)=i, \; l(\tt)=l \}. \] On a alors  \[ \Delta^{(n)}(X)=\bigsqcup_{0\leqslant i\leqslant n}\Delta^{(n)}_{i,l}(X), \] et on remarque de plus que $ \Delta^{(n)}_{i,l}(X) $ est non vide uniquement lorsque $ i+l\leqslant n $. On pose \`a pr\'esent \[ J_{k,n,i,l}^{\sigma}(\ee)=X^{-\sum_{\rho=n+1}^{r}\alpha_{\rho}}\int_{\Delta^{(n)}_{i,l}(X)}\frac{(\log(\prod_{\rho=1}^{n}t_{\rho}))^{k}}{\prod_{\rho=1}^{n}t_{\rho}^{\alpha_{\rho}+1}}\sum_{\substack{ x_{\sigma(\rho)}\leqslant t_{\rho} \\ \rho\in \{1,...,n\} }}\sum_{\substack{ x_{\rho}\leqslant X \\ \rho\in \{n+1,...r\}}}h_{\ee}(\xx)d\tt, \] et on a alors, par les formules\;\eqref{Xin3} et\;\eqref{int1X3} \[ \Xi_{n}=\sum_{\sigma\in \mathfrak{S}_{n}}\sum_{k+s\leqslant j}\gamma_{k,s}(\log X)^{s}\sum_{i+l\leqslant n}J_{k,n,i,l}^{\sigma}(\ee). \]
Deuxi\`eme \'etape : nous allons \`a pr\'esent chercher \`a \'evaluer $ J_{k,n,i,l}^{\sigma}(\ee) $. Quitte \`a permuter les variables, nous nous ramenons \`a l'estimation de $ J_{k,n,i,l}(\ee)=J_{k,n,i,l}^{Id}(\ee) $. Commen\c{c}ons par traiter le cas de $ J_{k,n,i,0}(\ee) $. Remarquons que $ \Delta^{(n)}_{0,0}(X) $ est vide, et que l'on peut donc supposer $ i\geqslant 1 $. Pour $ \tt\in \Delta^{(n)}_{i,0}(X) $, on a $ t_{1}> X^{\beta^{i}}\geqslant X^{\beta^{n}} $ et la condition $ 1. $ donne : \[ \sum_{\substack{ x_{\rho}\leqslant t_{\rho} \\ \rho\in \{1,...,n\} }}\sum_{\substack{ x_{\rho}\leqslant X \\ \rho\in \{n+1,...r\}}}h_{\ee}(\xx)=c_{\ee}X^{\sum_{\rho=n+1}^{r}\alpha_{\rho}}\prod_{\rho=1}^{n}t_{\rho}^{\alpha_{\rho}}+O\left( \ee^{\bbeta}X^{\sum_{\rho=n+1}^{r}\alpha_{\rho}-\delta\beta^{n}}\prod_{\rho=1}^{n}t_{\rho}^{\alpha_{\rho}}\right). \] Par cons\'equent, \[  J_{k,n,i,0}(\ee)=c_{\ee}\int_{ \Delta^{(n)}_{i,0}(X)}\frac{(\log(\prod_{\rho=1}^{n}t_{\rho}))^{k}}{\prod_{\rho=1}^{n}t_{\rho}}d\tt +O\left( \ee^{\bbeta}X^{-\delta\beta^{n}/2}\right). \]
Montrons \`a pr\'esent qu'il existe un polyn\^ome $ Q $ ne d\'ependant que de $ n,i,\beta,k $ tel que : \begin{equation}\label{formutile3} \int_{ \Delta^{(n)}_{i,0}(X)}\frac{(\log(\prod_{\rho=1}^{n}t_{\rho}))^{k}}{\prod_{\rho=1}^{n}t_{\rho}}d\tt=Q(\log X). \end{equation} Pour cela, consid\'erons des entiers naturels $ n=u_{1}>u_{2}>...>u_{i}>u_{i+1}=1 $ et posons $ \uu=(u_{1},...,u_{i+1}) $. Soit \[ \Gamma_{i,\uu}=\{\tt\in  \Delta^{(n)}(X)\; | \; \forall \lambda\in \{0,...,i-1\}, \; \forall \rho\in \{u_{\lambda +2}+1,...,u_{\lambda+1}\}, \; t_{\rho}\in \mathfrak{I}_{\lambda} \}. \] Par construction   \[ \Delta^{(n)}_{i,0}(X)=\bigsqcup_{\uu}\Gamma_{i,\uu}. \] Par d\'eveloppement multinomial, on a : \begin{multline*} \int_{ \Delta^{(n)}_{i,0}(X)}\frac{(\log(\prod_{\rho=1}^{n}t_{\rho}))^{k}}{\prod_{\rho=1}^{n}t_{\rho}}d\tt \\ =\sum_{\uu}\sum_{a_{1}+...+a_{n}=k}\left( \begin{array}{c}
k \\ a_{1}\; ...\; a_{n}
\end{array} \right) \int_{\Gamma_{i,\uu}}\frac{\prod_{\nu=1}^{n}(\log t_{\rho})^{a_{\rho}}}{\prod_{\rho=1}^{n}t_{\rho}}d\tt.
\end{multline*}
Par d\'efinition de $ \Gamma_{i,\uu} $, cette derni\`ere int\'egrale se factorise en produit d'int\'egrale sur $ t_{\rho}\in  \mathfrak{I}_{\lambda} $ avec $ \rho\in \{u_{\lambda +2}+1,...,u_{\lambda+1}\} $. Une telle int\'egrale prend la forme : \[ \int_{Y^{\beta}\leqslant v_{1}<...<v_{s}\leqslant Y}\frac{\prod_{\rho=1}^{s}(\log v_{\rho})^{b_{\rho}}}{\prod_{\rho=1}^{s}v_{\rho}}d\vv, \] o\`u $ Y=X^{\beta^{\lambda}} $. Ces int\'egrales peuvent \^etre calcul\'ees explicitement et donnent un polyn\^ome en $ \log(X) $, d'o\`u le r\'esultat. Nous avons donc montr\'e que \[ J_{k,n,i,0}(\ee) =c_{\ee}Q(\log X)+O\left(\ee^{\bbeta}X^{-\delta\beta^{n}/2}\right). \] Traitons maintenant le cas $ l=n $. Puisque l'on a suppos\'e $ l+i\leqslant n $, on doit avoir $ i=0 $, et on observe que $ \Delta^{(n)}_{0,n}(X)=\Delta^{(n)}(X^{\beta}) $. Rappelons que $ \beta\leqslant \nu $, et par cons\'equent, en \'ecrivant $ \xx'=(x_{1},...,x_{l}) $, on d\'eduit de la condition $ 2 $ que \[ \sum_{\substack{x_{j}\leqslant X\\ j\in \{n+1,...,r\}}}h_{\ee}=c_{\ee,\{1,...,l\}}(\xx')X^{\sum_{j=n+1}^{r}\alpha_{j}}+O\left(\ee^{\bbeta}|\xx'|^{D}X^{\sum_{j=n+1}^{r}\alpha_{j}-\delta}\right). \] Pour $ \tt\in \Delta^{(n)}(X^{\beta}) $, en sommant sur $ \xx'\leqslant \tt $, on obtient :  \[ \sum_{\xx'\leqslant \tt}\sum_{\substack{x_{j}\leqslant X\\ j\in \{n+1,...,r\}}}h_{\ee}=X^{\sum_{j=n+1}^{r}\alpha_{j}}\sum_{\xx'\leqslant \tt}c_{\ee,\{1,...,l\}}(\xx')+O\left(\ee^{\bbeta}|\tt|^{n+D}X^{\sum_{j=n+1}^{r}\alpha_{j}-\delta}\right). \] En rappelant que $ \beta\leqslant \delta/(4n+2D) $, et en utilisant la d\'efinition de $ J_{k,n,0,n}(\ee) $ on trouve : \[ J_{k,n,0,n}(\ee)=\int_{\Delta^{(n)}(X^{\beta})}\frac{(\log(\prod_{\rho=1}^{n}t_{\rho}))^{k}}{\prod_{\rho=1}^{n}t_{\rho}^{\alpha_{\rho}+1}}\sum_{\xx'\leqslant \tt}c_{\ee,\{1,...,l\}}(\xx')d\tt+O\left(\ee^{\bbeta}X^{-\delta/2}\right). \] D'apr\`es le lemme\;\ref{sommepart23}, on peut appliquer le lemme\;\ref{intj'3} avec $ c_{\ee,\{1,...,l\}} $ \`a la place de $ h_{\ee} $, et il s'ensuit que $ J_{k,n,0,n}(\ee) $ est un polyn\^ome en $ \log X $ \`a coefficients de type $ O(\ee^{\bbeta}) $ \`a un terme d'erreur pr\`es de type $ O(\ee^{\bbeta}X^{-\eta}) $ pour un certain $ \eta>0 $. Il nous reste \`a traiter le cas o\`u $ l\in \{1,...,n-1\} $. Posons $ \tt=(\tt',\tt'') $, avec $ \tt'=(t_{1},...,t_{l}) $ et $ \tt''=(t_{l+1},...,t_{n}) $. Remarquons que $ \tt\in \Delta_{i,l}^{(n)}(X) $ si et seulement si $ \tt'\in \Delta^{(l)}(X^{\beta^{i+1}}) $ et $ \tt''\in \Delta_{i,0}^{(n-l)}(X) $. De plus $ \beta\leqslant \nu $, et par la condition $ 2 $ on a : \begin{multline*} \sum_{\substack{x_{i}\leqslant t_{i} \\ l<i\leqslant n}}\sum_{\substack{x_{i}\leqslant X \\ n<i\leqslant r}}h_{\ee}(\xx)=c_{\ee,\{1,...,l\}}(\xx')X^{\sum_{j=n+1}^{r}\alpha_{j}}\prod_{\rho=l+1}^{n}t_{\rho}^{\alpha_{\rho}} \\ +O\left( \ee^{\bbeta}|\xx'|^{D}X^{\sum_{\rho=n+1}^{r}\alpha_{\rho}}\left(\prod_{\rho=l+1}^{n}t_{\rho}^{\alpha_{\rho}}\right)t_{l+1}^{-\delta}\right). \end{multline*}
Puis en sommant sur les $ \xx'\leqslant \tt' $, on trouve (puisque $ t_{l+1}\geqslant X^{\beta^{i}} $) :  \begin{multline*} \sum_{\substack{x_{i}\leqslant t_{i} \\ 1\leqslant i\leqslant n}}\sum_{\substack{x_{i}\leqslant X \\ n<i\leqslant r}}h_{\ee}(\xx)=X^{\sum_{j=n+1}^{r}\alpha_{j}}\prod_{\rho=l+1}^{n}t_{\rho}^{\alpha_{\rho}}\sum_{\xx'\leqslant \tt'}c_{\ee,\{1,...,l\}}(\xx') \\ +O\left( \ee^{\bbeta}|\tt'|^{D+n}X^{\sum_{\rho=n+1}^{r}\alpha_{\rho}-\delta\beta^{i}}\left(\prod_{\rho=l+1}^{n}t_{\rho}^{\alpha_{\rho}}\right)\right). \end{multline*}
En multipliant ce r\'esultat par $ \left(\prod_{\rho=1}^{n}t_{\rho}^{-\alpha_{\rho}-1}\right)(\log\prod_{\rho=1}^{n}t_{\rho})^{k} $ et en int\'egrant sur $ \Delta^{(n)}_{i,l}(X) $, le terme d'erreur peut \^etre major\'e par : 
\begin{multline*}
\ee^{\bbeta}X^{\sum_{\rho=n+1}^{r}\alpha_{\rho}-\delta\beta^{i}}(\log X)^{k}\int_{\Delta^{(n)}_{i,l}(X)} \left(\prod_{\rho=1}^{n}t_{\rho}\right)^{-1}|\tt'|^{n+D}d\tt \\ \ll \ee^{\bbeta}X^{\sum_{\rho=n+1}^{r}\alpha_{\rho}-\delta\beta^{i}}(\log X)^{k+n}\underbrace{\int_{\Delta^{(l)}(X^{\beta^{i+1}})} |\tt'|^{n+D}d\tt'}_{\ll X^{\beta^{i+1}(2n+D)}\ll X^{\delta\beta^{i}/2}}  \ll \ee^{\bbeta}X^{\sum_{\rho=n+1}^{r}\alpha_{\rho}-\delta\beta^{n}/2}.
\end{multline*}
Ainsi, \`a un terme d'erreur $ O(\ee^{\bbeta}X^{\sum_{\rho=n+1}^{r}\alpha_{\rho}-\delta\beta^{n}/2}) $ pr\`es, l'int\'egrale $ J_{k,n,l,n}(\ee) $ vaut : \begin{multline*} \sum_{k'+k''=k}\left(\begin{array}{c} k \\ k' 
\end{array}\right) \int_{ \Delta^{(n-l)}_{i,0}(X)}\frac{(\log\prod_{\rho=l+1}^{n}t_{\rho})^{k''}}{\prod_{\rho=l+1}^{n}t_{\rho}}
\\ \int_{\Delta^{(l)}(X^{\beta^{i+1}})}\frac{(\log\prod_{\rho=1}^{l}t_{\rho})^{k'}}{\prod_{\rho=1}^{l}t_{\rho}^{\alpha_{\rho}+1}}\sum_{\xx'\leqslant \tt'}c_{\ee,\{1,...,l\}}(\xx')d\tt'\end{multline*}
Par application des lemmes\;\ref{sommepart23} et\;\ref{intj'3} et de la formule\;\eqref{formutile3}, on montre directement que ceci est, \`a un terme d'erreur $ O(\ee^{\bbeta}X^{-\eta}) $ pr\`es, un polyn\^ome en $ \log(X) $ \`a coefficients $ O(\ee^{\bbeta}) $. \\

Troisi\`eme \'etape :  Le degr\'e et le coefficient dominant du polyn\^ome $ p_{\ee,j} $ peuvent \^etre calcul\'es de la mani\`ere suivante. On a par d\'efinition de $ \Xi_{r} $ et par la condition $ 1. $ :  \[ \Xi_{r}=c_{\ee}\int_{[1,X]^{r}}\frac{q_{\{1,...,r\}}(\log(\prod_{i=1}^{r}t_{i}))}{\prod_{i=1}^{r}t_{i}}d\tt+O\left(\ee^{\bbeta}\int_{[1,X]^{r}}\frac{(\log(\prod_{i=1}^{r}t_{i}))^{j}}{\prod_{i=1}^{r}t_{i}\min_{i\in \{1,...,r\}}t_{i}^{\delta}}d\tt\right). \]

Puisque $ q_{\{1,...,r\}} $ est unitaire de degr\'e $ j $, il s'ensuit que : \[ \Xi_{r}=c_{\ee}\int_{[1,X]^{r}}\frac{(\log(\prod_{i=1}^{r}t_{i}))^{j}}{\prod_{i=1}^{r}t_{i}}d\tt+O\left(\ee^{\bbeta}(\log X)^{r+j-1}\right). \] Par d\'eveloppment multinomial, on obtient : \begin{align*}
\Xi_{r} & =c_{\ee}\sum_{a_{1}+...+a_{r}=j}\left(\begin{array}{c}
j \\ a_{1} \; ...\; a_{r}
\end{array}\right)\int_{[1,X]^{r}}\frac{\prod_{i=1}^{r}(\log t_{i})^{a_{i}}}{\prod_{i=1}^{r}t_{i}}d\tt+O\left(\ee^{\bbeta}(\log X)^{r+j-1}\right) \\ & =c_{\ee}\sum_{a_{1}+...+a_{r}=j}\left(\begin{array}{c}
j \\ a_{1} \; ...\; a_{r}
\end{array}\right)\frac{(\log X)^{r+j}}{\prod_{i=1}^{r}(a_{i}+1)}+ O\left(\ee^{\bbeta}(\log X)^{r+j-1}\right) \\ & = c_{\ee}V_{r,j}(\log X)^{r+j}+ O\left(\ee^{\bbeta}(\log X)^{r+j-1}\right).
\end{align*}
En utilisant \`a nouveau la condition $ 1 $, on montre par ailleurs que pour tout $ J\subsetneq\{1,...,r\} $ de cardinal $ n $ : \[ \Xi_{J} \ll \ee^{\bbeta}(\log X)^{n+j}. \] Anisi, on obtient \[  \sum_{\forall i, \; x_{i}\leqslant X}\frac{(\log\px)^{j}}{\xx^{\aalpha+\1}}h_{\ee}(\xx)= \left(\prod_{i=1}^{r}\alpha_{i}\right)V_{r,j}c_{\ee}(\log X)^{r+j}+O\left(\ee^{\bbeta}(\log X)^{r+j-1}\right), \] ce qui ach\`eve la d\'emonstration du th\'eor\`eme. 
\end{proof}

\subsection{Sommation sur les fibres}

Nous allons commencer par \'evaluer $ \sum_{\xx^{\aalpha}\leqslant P }h_{\ee}(\xx) $ pour des $ x_{j} $ assez grands. Plus pr\'ecis\'ement, fixons un \'el\'ement $ \WW\in [1,P]^{r} $ et consid\'erons : \[ \Upsilon_{\ee}(P,\WW)=\sum_{\substack{\ww^{\aalpha}\leqslant P \\ \ww>\WW}}h_{\ee}(\ww), \] \[ p_{r}(t)=\sum_{l=0}^{^{r-1}}\frac{(-1)^{r+1+l}}{l!}t^{l}. \]
\begin{lemma}\label{sommeW3}
Soit $ (h_{\ee})_{\ee} $ une $  (\aalpha,D,\nu,\delta,\bbeta)- $ famille. Supposons de plus que $ \WW^{\aalpha}\leqslant P^{\frac{1}{2}} $ et que $ \min_{j\in \{1,...,r\}}W_{j}\geqslant \log(P)^{\frac{2r}{\delta}} $. On a alors : \[ \Upsilon_{\ee}(P,\WW)=c_{\ee}p_{r}\left(\log\left(\frac{P}{\WW^{\aalpha}}\right)\right)P+O\left(\ee^{\bbeta}P(\log P)^{r}(\min_{j\in \{1,...,r\}}W_{j})^{-\frac{\delta}{2r}}\right). \]
\end{lemma}
Pour d\'emontrer ce lemme, nous aurons besoin du r\'esultat ci-dessous : 
\begin{lemma}\label{identitepol3}
Soient $ r $ et $ J $ deux entiers naturels. Pour tout $ t\in \CC $, on a : \[ (1-t)^{r}\sum_{\substack{j_{1}+...+j_{r}\leqslant J \\ \forall k, \; j_{k}\geqslant 0}}t^{j_{1}+...+j_{r}}=1-t^{J+1}\sum_{l=0}^{r-1}\left(\begin{array}{c} J+l \\ l\end{array}\right) (1-t)^{l}. \]
\end{lemma}
\begin{proof}
Voir\;\cite[Lemme 2.9]{BB}.
\end{proof}
\begin{proof}[D\'emonstration du lemme\;\ref{sommeW3}]
Nous allons nous ramener \`a l'\'evaluation des sommes \[ H_{\ee}(\UU^{+},\UU^{-})=\sum_{\UU^{-}<\uu\leqslant \UU^{+}}h_{\ee}(\uu). \]
Soit $ \theta $ un r\'eel, et $ J $ un entier naturel que nous pr\'eciserons ult\'erieurement. On suppose que $ 1<\theta<3 $ et que $ \theta^{J}=P/\WW^{\aalpha} $. Pour $ j\geqslant 0 $, nous poserons \[ U_{k,j}=W_{k}\theta^{\frac{j}{\alpha_{k}}}, \] et notons $ \UU_{\jj}=(U_{1,j_{1}},...,U_{r,j_{r}}) $. Nous consid\'erons alors les bo\^ites $ \UU_{\jj}<\uu\leqslant \UU_{\jj+\1} $ qui sont incluses dans $ \{\uu \; |\; \uu^{\aalpha}\leqslant P\} $ lorsque $ \WW^{\aalpha}\theta^{(\sum_{k=1}^{r}j_{k})+r}=\UU_{\jj+\1}^{\aalpha}\leqslant P $, ce qui est vrai si et seulement si $ |\jj|_{1}=j_{1}+...+j_{r}\leqslant J-r $. Soit $ \uu $ tel que $ \uu>\WW $ et $ \uu^{\aalpha}\leqslant P $. Il existe alors un unique $ \jj \in \NN^{r} $ tel que $ \UU_{\jj}<\uu\leqslant \UU_{\jj+\1}  $. Les in\'egalit\'es $ \UU_{\jj}^{\aalpha}<\uu^{\aalpha}\leqslant P $ impliquent $ |\jj|_{1}\leqslant J $. On a donc que \begin{equation}\label{HeeU3} \sum_{|\jj|_{1}\leqslant J-r} H_{\ee}(\UU_{\jj+\1},\UU_{\jj})\leqslant \Upsilon_{\ee}(P,\WW) \leqslant \sum_{|\jj|_{1}\leqslant J} H_{\ee}(\UU_{\jj+\1},\UU_{\jj}). \end{equation} En notant pour tout $ \XX\in \NN^{r} $ : \[ H_{\ee}(\XX)=\sum_{\xx\leqslant \XX}h_{\ee}(\XX), \] on obtient alors l'identit\'e : \begin{equation} H_{\ee}(\UU_{\jj+\1},\UU_{\jj})=\sum_{\ss\in \{0,1\}^{r}}(-1)^{r-|\ss|_{1}}H_{\ee}(\UU_{\jj+\ss}). \end{equation}
Par ailleurs, d'apr\`es la condition $ 1. $ : 
\[ H_{\ee}(\UU_{\jj+\ss})=c_{\ee}\WW^{\aalpha}\theta^{|\jj+\ss|_{1}}+O\left(\ee^{\bbeta}\WW^{\aalpha}(\min W_{j})^{-\delta}\theta^{|\jj+\ss|_{1}}\right). \] Posons \`a pr\'esent $ J^{+}=J $ et $ J^{-}=J-r $ et \'etudions les sommes : \[ \Upsilon_{\ee}^{+}=\sum_{|\jj|_{1}\leqslant J^{+}} H_{\ee}(\UU_{\jj+\1},\UU_{\jj}), \]\[ \Upsilon_{\ee}^{-}=\sum_{|\jj|_{1}\leqslant J^{-}} H_{\ee}(\UU_{\jj+\1},\UU_{\jj}). \] Les \'egalit\'es pr\'ec\'edentes donnent alors : \[ \Upsilon_{\ee}^{\pm}=c_{\ee}\WW^{\aalpha}\sum_{|\jj|_{1}\leqslant J^{\pm}}\sum_{\ss\in \{0,1\}^{r}}(-1)^{r-|\ss|_{1}}\theta^{|\jj+\ss|_{1}}+O\left(\ee^{\bbeta}\WW^{\aalpha}(\min W_{j})^{-\delta}\sum_{|\jj|_{1}\leqslant J^{\pm}}\theta^{|\jj|_{1}}\right). \] En utilisant l'\'egalit\'e : \[\sum_{\ss\in \{0,1\}^{r}}(-1)^{r-|\ss|_{1}}T^{|\ss|_{1}}=(T-1)^{r}, \] on a \begin{equation}\label{upsilonpm3} \Upsilon_{\ee}^{\pm}=c_{\ee}(\theta-1)^{r}\WW^{\aalpha}\sum_{|\jj|_{1}\leqslant J^{\pm}}\theta^{|\jj|_{1}}+O\left(\ee^{\bbeta}\WW^{\aalpha}W_{j_{0}}^{-\delta}\sum_{|\jj|_{1}\leqslant J^{\pm}}\theta^{|\jj|_{1}}\right), \end{equation} avec $ W_{j_{0}}=\min W_{j} $. Le terme d'erreur peut \^etre major\'e par \[ \ee^{\bbeta}\WW^{\aalpha}W_{j_{0}}^{-\delta}\theta^{J}\Card\{\jj\in \NN^{r}\; |\; |\jj|_{1}\leqslant J \}\ll \ee^{\bbeta}PJ^{r}W_{j_{0}}^{-\delta}. \] En utilisant l'\'egalit\'e du lemme\;\ref{identitepol3} avec $ t=\theta $ dans la formule\;\eqref{upsilonpm3}, et en multipliant par $ (-1)^{r} $ on trouve alors (en rappelant que $ \theta^{J}=P/\WW^{\aalpha} $) : \[ \Upsilon_{\ee}^{+}=c_{\ee}P\theta\sum_{l=0}^{r-1}\left(\begin{array}{c} J+l \\ l\end{array}\right) (-1)^{r+1+l}(\theta-1)^{l}+O\left(\ee^{\bbeta}\WW^{\aalpha}\right)+O\left(\ee^{\bbeta}PJ^{r}W_{j_{0}}^{-\delta}\right), \] et de m\^eme : \[ \Upsilon_{\ee}^{-}=c_{\ee}P\theta^{1-r}\sum_{l=0}^{r-1}\left(\begin{array}{c} J-r+l \\ l\end{array}\right) (-1)^{r+1+l}(\theta-1)^{l}+O\left(\ee^{\bbeta}\WW^{\aalpha}\right)+O\left(\ee^{\bbeta}PJ^{r}W_{j_{0}}^{-\delta}\right). \] Choisissons \`a pr\'esent $ J=\lfloor W_{j_{0}}^{\delta/2r}\rfloor $, de sorte que $ J\geqslant \log P $ et que $ 1<\theta\leqslant e $. On a alors \[\theta=\exp\left(J^{-1}\log\left(\frac{P}{\WW^{\aalpha}}\right)\right)=1+ J^{-1}\log\left(\frac{P}{\WW^{\aalpha}}\right)+O\left(J^{-2}(\log P)^{2}\right) \](car on a suppos\'e $ \WW^{\aalpha}\leqslant P^{\frac{1}{2}} $). D'autre part, on remarque que \[ \left(\begin{array}{c} J+l \\ l\end{array}\right) =\frac{J^{l}}{l!}+O(J^{l-1}) \] et par cons\'equent \[ \left(\begin{array}{c} J+l \\ l\end{array}\right)(\theta-1)^{l}=\frac{1}{l!}\log\left(\frac{P}{\WW^{\aalpha}}\right)^{l}+O\left(J^{-1}(\log P)^{l+1}\right). \]
 On multiplie cette \'egalit\'e par $ (-1)^{r+l+1} $ et on somme sur $ l\in \{0,...,r-1\} $. En rappelant que \[ p_{r}(t)=\sum_{l=0}^{r-1}\frac{(-1)^{r+l+1}}{l!}t^{l}, \] on obtient donc que : \begin{align} \Upsilon_{\ee}^{+} &=c_{\ee}Pp_{k}\left(\log\frac{P}{\WW^{\aalpha}}\right)+O\left(\ee^{\bbeta}\left(P(\log P)^{r}J^{-1}+\WW^{\aalpha}+PW_{j_{0}}^{-\delta}J^{r}\right)\right) \\ \label{sommeWplus3} &= c_{\ee}Pp_{k}\left(\log\frac{P}{\WW^{\aalpha}}\right)+O\left(\ee^{\bbeta}P(\log P)^{r}W_{j_{0}}^{-\frac{\delta}{2r}}\right).\end{align}
De fa\c{c}on analogue on obtient exactement la m\^eme formule asymptotique pour $ \Upsilon_{\ee}^{-} $ et donc pour $ \Upsilon_{\ee}(P,\WW) $ d'apr\`es\;\eqref{HeeU3}. 
\end{proof}
\subsection{D\'emonstration du th\'eor\`eme\;\ref{resultatBB3}}

Nous sommes \`a pr\'esent en mesure de d\'emontrer le th\'eor\`eme\;\ref{resultatBB3}. Dans ce qui va suivre, nous consid\'ererons $ (h_{\ee})_{\ee} $ une $  (\aalpha,D,\nu,\delta,\bbeta)- $ famille. On v\'erifie facilement qu'alors 
$ (h_{\ee})_{\ee} $ est une $  (\aalpha,D,\nu',\delta',\bbeta)- $ famille pour tous $ \nu'\leqslant \nu $, $ \delta'\leqslant \delta $, et on peut donc supposer dor\'enavant que $ \nu\leqslant \frac{1}{2\sum_{j=1}^{r}\alpha_{j}} $ et $ \delta\leqslant \frac{1}{2} $. Fixons par ailleurs les param\`etres \begin{equation}
A=D+(r+1)\max_{j\in \{1,...,r\}}(\alpha_{j})+\nu^{-1}(1+\max_{j\in \{1,...,r\}}(\alpha_{j})),
\end{equation}
\begin{equation}
B=4Ar^{2}/\delta.
\end{equation}
Introduisons un param\`etre $ V\geqslant 4 $ et supposons que $ P $ est tel que $ V^{B^{r}}\leqslant P^{\nu} $. On pose alors \[ V_{0}=1, \; V_{1}=V, \; \forall k\in \{1,...,r\}, \;  V_{k}=V^{B^{k-1}},\;  V_{r+1}=P, \] et \[ \mathcal{V}_{0}=[V_{0},V_{1}],\; \forall k \in \{1,...,r\},\; \mathcal{V}_{k}=]V_{k},V_{k+1}] \] de sorte que \[ [1,P]=\bigsqcup_{k=0}^{r}\mathcal{V}_{k}. \] Par le principe des tiroirs, pour tout $ \xx=(x_{1},...,x_{r})\in \NN^{r} $ tel que $ \xx^{\aalpha}\leqslant P $, il existe au moins un $ l\in \{0,...,r\} $ tel que $ x_{j}\notin \mathcal{V}_{l} $ pour tout $ j\in \{1,...,r\} $. On note alors $ l(\xx) $ le plus grand de ces entiers $ l $. Posons ensuite pour tout $ l\in \{0,...,r\} $ : \[ \Upsilon_{\ee,l}(P)=\sum_{\substack{\xx^{\aalpha}\leqslant P \\ l(\xx)=l   }}h_{\ee}(\xx) \] (on remarque que $ \Upsilon_{\ee}(P)=\sum_{l=0}^{r}\Upsilon_{\ee,l}(P) $). Remarquons que la condition $ l(\xx)=r $ \'equivaut \`a dire que $ |x_{j}|\leqslant V_{r} $ pour tout $ j\in \{1,...,r\} $. Donc, d'apr\`es la condition $ 1 $, on obtient la borne : \[
\Upsilon_{\ee,r}(P)  \leqslant H_{\ee}(V_{r},...,V_{r}) \ll \ee^{\bbeta}\prod_{j=1}^{r}V_{r}^{\alpha_{j}}\ll \ee^{\bbeta}V^{B^{r}\sum_{j=1}^{r}\alpha_{j}} \ll \ee^{\bbeta}P^{\frac{1}{2}}.
\]
Consid\'erons \`a pr\'esent $ \xx\in \NN^{r} $ tel que $ \xx^{\aalpha}\leqslant P $ et $ l(\xx)=l $ avec $ l\in \{0,...,r-1\} $. On peut associer \`a un tel $ \xx $ les ensembles \[ \mathcal{J}(\xx)=\{j\in \{1,...,r\}\; |\; x_{j}\leqslant V_{l}\}, \] \[ \forall m\in \{l+1,...,r\}, \; \mathcal{L}_{m}(\xx)=\{j\in \{1,...,r\}\; |\; x_{j}\in \mathcal{V}_{m}\}. \]
On observe que $ \{1,...,r\}=\mathcal{J}(\xx)\sqcup\left(\bigsqcup_{m=l+1}^{r}\mathcal{L}_{m}\right) $, et on note \[ \mathcal{A}(\xx)=(\mathcal{J}(\xx),\mathcal{L}_{l+1}(\xx),...,\mathcal{L}_{r}(\xx)). \]
Une partition $ \mathcal{A}=(\mathcal{J},\mathcal{L}_{l+1},...,\mathcal{L}_{r}) $ de $ \{1,...,r\} $ telle que $ \mathcal{L}_{m}\neq \emptyset $ pour tout $ m\in \{l+1,...,r\} $ sera dite \emph{admissible pour $ l $}. On a alors pour tout $ l\in \{0,...,r-1\} $ :  \[ \Upsilon_{\ee,l}(P)=\sum_{\substack{\mathcal{A }\; \admissible \\ \pour \; l}}\Upsilon_{\ee,\mathcal{A }}(P), \] o\`u \[ \Upsilon_{\ee,\mathcal{A }}(P)=\sum_{\substack{\xx^{\aalpha}\leqslant P \\ \mathcal{A}(\xx)=\mathcal{A}}}h_{\ee}(\xx). \]
Nous allons \`a pr\'esent utiliser cette d\'ecomposition de $ \Upsilon_{\ee}(P) $ pour d\'emontrer le r\'esultat ci-dessous : 
\begin{lemma}\label{lemmeavantthm3}
Soit $ (h_{\ee})_{\ee} $ une $  (\aalpha,D,\nu,\delta,\bbeta)- $ famille. On a alors pour tout $ \ee $ : \[  \Upsilon_{\ee}(P)=\frac{c_{\ee}}{(r-1)!}P(\log P)^{r-1}+O\left(\ee^{\bbeta}P(\log P)^{r-2}\log(\log P)\right). \]
\end{lemma}
\begin{proof}
Le cas $ r=1 $ est imm\'ediat d'apr\`es la condition $ 1 $. Nous supposerons donc dor\'enavant $ r\geqslant 2 $, et nous choisissons $ V=(\log P)^{B} $.\\

 On consid\`ere l'ensemble $ \mathcal{A}=(\emptyset,\{1,...,r\}) $ qui est admissible pour $ r-1 $. Si $ \mathcal{A}(\xx)=\mathcal{A} $, alors pour tout $ j\in \{1,...,r\} $, $ x_{j}\geqslant V_{r-1} $ et $ x_{j}\in \mathcal{V}_{r}=]V_{r},P] $. On en d\'eduit que \[ \Upsilon_{\ee,\mathcal{A }}(P)= \Upsilon_{\ee}(P,(V_{r},...,V_{r}). \] Puisque $ V_{r}=(\log P)^{B^{r}} $ et que $ B^{r}=(\frac{4Ar^{2}}{\delta})^{r}\geqslant \frac{2r}{\delta}\geqslant 1 $, on peut appliquer le lemme\;\ref{sommeW3}, et on obtient : \begin{align*} \Upsilon_{\ee}(P,(V_{r},...,V_{r})) & =c_{\ee}p_{r}\left(\log\left(\frac{P}{\prod_{j=1}^{r}V_{r}^{\alpha_{j}}}\right)\right)P+O\left(\ee^{\bbeta}P(\log P)^{r}(\log P)^{-\frac{B^{r}\delta}{2r}}\right) \\& =\frac{c_{\ee}}{(r-1)!}P(\log P)^{r-1}+O\left(\ee^{\bbeta}P(\log P)^{r-2}\log(\log P)\right).
\end{align*}
Il reste \`a montrer que pour tout autre ensemble $ \mathcal{A} $ admissible pour $ l $, la somme $  \Upsilon_{\ee,\mathcal{A }}(P) $ est du type $ \left(\ee^{\bbeta}P(\log P)^{r-2}\log(\log P)\right) $. Commen\c{c}ons par traiter le cas o\`u $ \mathcal{J}=\emptyset $. Consid\'erons un ensemble $ (\emptyset, \mathcal{L}_{l+1},...,\mathcal{L}_{r}) $ admissible pour $ l $. Nous avons d\'ej\`a trait\'e les cas $ l=r $ et $ l=r-1 $. Nous supposons donc que $ l\in \{0,...,r-2\} $. Soit $ \xx\in \NN^{r} $ tel que $ \mathcal{A}(\xx)=\mathcal{A} $. Puisque $ \mathcal{L}_{l+1}\neq \emptyset $, il existe un certain $ j_{0}\in \{1,...,r\} $ tel que $ x_{j_{0}}\in \mathcal{V}_{l+1} $. Quitte \`a permuter les variables, nous pouvons supposer, pour simplifier, que $ j_{0}=1 $. On a alors, pour tout $ j\neq 1 $, $ x_{j}\geqslant V_{l+1} $. En utilisant le lemme \ref{sommeW3} (en remarquant que $ V_{l+1}>(\log P)^{2r/\delta} $ puisque $ B>\frac{2r}{\delta} $) on en d\'eduit l'estimation : \begin{align*} \Upsilon_{\ee,\mathcal{A }}(P) & \leqslant \sum_{V_{l+1}<x_{1}\leqslant V_{l+2}}\sum_{\substack{\forall j\neq 1, \; x_{j}\geqslant V_{l+1} \\ \xx^{\aalpha}\leqslant P }}h_{\ee}(\xx) \\ & = \Upsilon_{\ee}(P,(V_{l+1},V_{l+1},...,V_{l+1}))-\Upsilon_{\ee}(P,(V_{l+2},V_{l+1},...,V_{l+1})) \\ & = c_{\ee}P\left( p_{r}\left(\log \frac{P}{\prod_{j=1}^{r}V_{l+1}^{\alpha_{j}}}\right)-p_{r}\left(\log \frac{P}{V_{l+2}^{\alpha_{1}}\prod_{j=2}^{r}V_{l+1}^{\alpha_{j}}}\right)\right)+O(\ee^{\bbeta}P) \\ & \ll \ee^{\bbeta}\alpha_{1}P\log\left(\frac{V_{l+2}}{V_{l+1}}\right)(\log P)^{r-2}+O\left( \ee^{\bbeta}P(\log P)^{r-2}\log(\log P)\right) \\ & \ll \ee^{\bbeta}P(\log P)^{r-2}\log(\log P). \end{align*}

Supposons \`a pr\'esent que $ \mathcal{A}=(\mathcal{J},\mathcal{L}_{l+1},...,\mathcal{L}_{r}) $ est admissible pour $ l $ avec $ \mathcal{J}\neq \emptyset $. Puisque tous les $ \mathcal{L}_{m} $ sont non vides, on a $ l\geqslant 1 $. Quitte \`a permuter les variables, on peut supposer que $ \mathcal{J}=\{1,...,k\} $ pour un certain $ k\geqslant 1 $. Soit $ \xx\in \NN^{r} $ tel que $ \mathcal{A}(\xx)=\mathcal{A} $. On pose $ \ww=(x_{1},...,x_{k}) $ et $ \yy=(x_{k+1},...,x_{r}) $. On a alors que $ y_{j}>V_{l+1} $ pour tout $ j $ et $ w_{j}\leqslant V_{l} $, et donc : \[ \Upsilon_{\ee,\mathcal{A }}(P)  \leqslant \sum_{\forall j\in \{1,...,k\}, \; x_{j}\leqslant V_{l} }\sum_{\substack{ \prod_{j=k+1}^{r}y_{j}^{\alpha_{j}}\leqslant \frac{P}{\prod_{j=1}^{k}x_{j}^{\alpha_{j}}} }}h_{\ee}(\ww,\yy). \] Or, d'apr\`es le lemme\;\ref{sommeW3} on a :   \begin{multline*} \sum_{\substack{ \prod_{j=k+1}^{r}y_{j}^{\alpha_{j}}\leqslant \frac{P}{\prod_{j=1}^{k}x_{j}^{\alpha_{j}}} }}h_{\ee}(\ww,\yy) \\  =\frac{ c_{\ee,\{1,...,k\}}(\ww)}{\prod_{j=1}^{k}x_{j}^{\alpha_{j}}}P p_{r-k}\left(\log \frac{P}{\prod_{j=1}^{k}x_{j}^{\alpha_{j}}\prod_{j=k+1}^{r}V_{l+1}^{\alpha_{j}}}\right)\\ +O(\ee^{\bbeta}PV_{l+1}^{-\delta}(\log P)^{r}(\prod_{j=1}^{k}x_{j}^{\alpha_{j}})^{A}) \\  \ll \frac{ c_{\ee,\{1,...,k\}}(\ww)}{\prod_{j=1}^{k}x_{j}^{\alpha_{j}}}P(\log P)^{r-k-1} +\ee^{\bbeta}PV_{l+1}^{-\delta}(\log P)^{r}(\prod_{j=1}^{k}x_{j}^{\alpha_{j}})^{A}. \end{multline*}
Puis, en sommant sur les $ \ww $ tels que $ |\ww|\leqslant V_{l}  $ et en utilisant le th\'eor\`eme\;\ref{thmintermediaire3}, on obtient \begin{multline*}   \Upsilon_{\ee,\mathcal{A }}(P)  \ll \ee^{\bbeta}P (\log P)^{r-k-1}(\log V_{l})^{k}  + \ee^{\bbeta}\prod_{j=1}^{k}V_{l}^{A+1}V_{l+1}^{-\delta/2r}(\log P )^{r}\\ \ll \ee^{\bbeta}P (\log P)^{r-k-1}(\log(\log P))^{k}  + \ee^{\bbeta}\prod_{j=1}^{k}(\log P )^{r+kB^{l}(A+1)-B^{l+1}\delta/2r}\\ \ll \ee^{\bbeta}P (\log P)^{r-2}\log(\log P)  + \ee^{\bbeta}\prod_{j=1}^{k}(\log P )^{r-2} \end{multline*}
(par d\'efinition $ B $, on a en effet $ kB^{l}(A+1)-B^{l+1}\delta/2r\leqslant-2 $). 
\end{proof}
Nous sommes \`a pr\'esent en mesure de d\'emontrer le th\'eor\`eme\;\ref{resultatBB3} : 
\begin{proof}[D\'emonstration du th\'eor\`eme\;\ref{resultatBB3}]
Posons $ \kappa=\frac{\nu}{(\sum_{j=1}^{r}\alpha_{j})B^{r}} $ et choisissons $ V=P^{\kappa} $. Nous allons montrer que pour un ensemble $ \mathcal{A} $ admissible pour $ l\in \{0,...,r-1\} $, il existe un polyn\^ome $ p_{\mathcal{A},\ee} $ de degr\'e au plus $ r-1 $ \`a coefficients du type $ O(\ee^{\bbeta}) $ tel que : \begin{equation}\label{formpol3}
 \Upsilon_{\ee,\mathcal{A }}(P)=Pp_{\mathcal{A},\ee}(\log P)+O(\ee^{\bbeta}P^{1-\eta})
\end{equation} pour un certain $ \eta>0 $. Commen\c{c}ons par consid\'erer le cas o\`u $ \mathcal{A}=\{\emptyset,\mathcal{L}_{l+1},...,\mathcal{L}_{r}\} $ admissible pour $ l $. Si $ l=r-1 $, alors $ \mathcal{L}_{r}=\{1,...,r\} $ et $ \Upsilon_{\ee,\mathcal{A }}(P)=\Upsilon_{\ee}(P,(V_{r},...,V_{r})) $ et on a donc : \begin{align*}  \Upsilon_{\ee,\mathcal{A }}(P) & =c_{\ee}Pp_{r}\left(\log \frac{P}{V_{r}^{\sum_{j=1}^{r}\alpha_{j}}}\right)+O(\ee^{\bbeta}PV_{r}^{-\delta/2r}) \\ &=c_{\ee}Pp_{r}\left(\log \frac{P}{V_{r}^{\sum_{j=1}^{r}\alpha_{j}}}\right)+O(\ee^{\bbeta}P^{1-\eta})   \end{align*}
pour $ 0<\eta<\frac{\nu}{rB} $. Si \`a pr\'esent $ l\in \{0,...,r-2\} $, on a alors si $ \mathcal{A}(\xx)=\mathcal{A} $ pour tout $ m\in \{l+1,...,r-1\} $, $ V_{m}<x_{j}\leqslant V_{m+1} $, $ \forall j\in \mathcal{L}_{m} $ et on en d\'eduit : \[  \Upsilon_{\ee,\mathcal{A }}(P)=\sum_{\VV=(V^{(1)},...,V^{(r)})}(-1)^{\varepsilon(\VV)}\Upsilon_{\ee}(P,\VV), \] avec $ V^{(j)}\in \{V_{m},V_{m+1}\} $ pour $ j\in \mathcal{L}_{m} $ et $ V^{(j)}=V_{r} $ si $ j\in \mathcal{L}_{r} $, et avec $ \varepsilon(\VV)\in \{0,1\} $ pour tout $ \VV $. Puisque $ V=P^{\kappa} $, d'apr\`es le lemme\;\ref{sommeW3} : \[ \Upsilon_{\ee}(P,\VV)=c_{\ee}P p_{r}\left(\log \frac{P}{\prod_{j=1}^{r}(V^{(j)})^{\alpha_{j}}}\right)+O(\ee^{\bbeta}P^{1-\eta}). \] Puisque tous les $ V^{(j)} $ sont des puissances de $ P $, on a que $ \log \frac{P}{\prod_{j=1}^{r}(V^{(j)})^{\alpha_{j}}} $ est un multiple de $ \log P $, et on obtient donc bien la formule\;\eqref{formpol3} pour $  \Upsilon_{\ee,\mathcal{A }}(P) $.\\

Consid\'erons \`a pr\'esent $ \mathcal{A}=\{\mathcal{J},\mathcal{L}_{l+1},...,\mathcal{L}_{r}\} $ est admissible pour $ l $ avec $ \mathcal{J}\neq \emptyset $. Quitte \`a permuter les variables, on peut supposer $ \mathcal{J}=\{1,...,k\} $. Nous noterons alors $ \xx=(\ww,\yy) $ pour tout $ \xx $ tel que $ \mathcal{A}(\xx)=\mathcal{A} $, comme dans la d\'emonstration du th\'eor\`eme pr\'ec\'edent. On obtient alors \[ \Upsilon_{\ee,\mathcal{A}}(P)=\sum_{|\ww|\leqslant V_{l}}\sum_{\substack{\prod_{j=k+1}^{r}y_{j}^{\alpha_{j}}\leqslant \frac{P}{\prod_{j=1}^{k}x_{j}^{\alpha_{j}}} \\ \forall j\in \mathcal{L}_{m}, \; y_{j}\in \mathcal{V}_{m}}}h_{\ee}(\ww,\yy). \] Comme pr\'ec\'edemment, nous pouvons \'ecrire \[\Upsilon_{\ee,\mathcal{A}}(P)=\sum_{\substack{\forall j\in \{1,...,k\} \\ x_{j}\leqslant V_{l} }}\sum_{\VV=(V^{(1)},...,V^{(r))}}(-1)^{\varepsilon(\VV)}\sum_{\substack{ \prod_{j=k+1}^{r}y_{j}^{\alpha_{j}}\leqslant \frac{P}{\prod_{j=1}^{k}x_{j}^{\alpha_{j}}} \\ \yy\geqslant \VV }}h_{\ee}(\ww,\yy). \]
Or, d'apr\`es le lemme\;\ref{lemmA3} et le lemme\;\ref{sommeW3} \begin{multline*}
\sum_{\substack{ \prod_{j=k+1}^{r}y_{j}^{\alpha_{j}}\leqslant \frac{P}{\prod_{j=1}^{k}x_{j}^{\alpha_{j}}} \\ \yy\geqslant \VV }}h_{\ee}(\ww,\yy) \\ =\frac{ c_{\ee,\{1,...,k\}}(\ww)}{\prod_{j=1}^{k}x_{j}^{\alpha_{j}}}P p_{r-k}\left(\log \frac{P}{\prod_{j=1}^{k}x_{j}^{\alpha_{j}}\prod_{j=k+1}^{r}(V^{(j)})^{\alpha_{j}}}\right)\\ +O(\ee^{\bbeta}P^{1}V_{l+1}^{-\delta}(\log P)^{r}(\prod_{j=1}^{k}x_{j}^{\alpha_{j}})^{A}).
\end{multline*}
Remarquons que $ \prod_{j=k+1}^{r}(V^{(j)})^{\alpha_{j}}=P^{\rho} $ pour un certain $ \rho>0 $, et donc que \[ \log \frac{P}{\prod_{j=1}^{k}x_{j}^{\alpha_{j}}\prod_{j=k+1}^{r}(V^{(j)})^{\alpha_{j}}}=(1-\rho)\log P-\log \prod_{j=1}^{k}x_{j}^{\alpha_{j}}.  \] Le polyn\^ome $ p_{r-k}\left(\log \frac{P}{\prod_{j=1}^{k}x_{j}^{\alpha_{j}}\prod_{j=k+1}^{r}(V^{(j)})^{\alpha_{j}}}\right) $ peut donc se r\'e\'ecrire comme un polyn\^ome en $ \log P $ (\`a coefficients $ O(1) $). En sommant sur $ \VV $ et sur $ \ww $, on obtient alors qu'il existe un polyn\^ome $ p_{\mathcal{A},\ee} $ \`a coefficients $ O(\ee^{\bbeta}) $ tels que : \begin{align*} \Upsilon_{\ee,\mathcal{A }}(P) & =p_{\mathcal{A},\ee}(\log P)P+O(\ee^{\bbeta}P^{1}V_{l+1}^{-\delta}(\log P)^{r}(\prod_{j=1}^{k}V_{l}^{\alpha_{j}+1})^{A}) \\ & =p_{\mathcal{A},\ee}(\log P)P+O(\ee^{\bbeta}P^{1-\eta}), \end{align*} pour un certain $ \eta>0 $, d'o\`u le r\'esultat.
Nous avons donc \'etabli la formule\;\eqref{formpol3} pour tout ensemble $ \mathcal{A} $ admissible pour $ l\in \{0,...,r-1\} $. Nous en d\'eduisons qu'il existe un polyn\^ome $ p_{\ee}^{\ast} $ de degr\'e au plus $ r-1 $ \`a coefficients $ O(\ee^{\bbeta}) $ tel que \[ \Upsilon_{\ee}(P)=p_{\ee}^{\ast}(\log P)+O(\ee^{\bbeta}P^{1-\eta}). \] Or, d'apr\`es le lemme\;\ref{lemmeavantthm3}, nous savons alors que $ p_{\ee}^{\ast}(t) $ est du type $ \frac{c_{\ee}}{(r-1)!}t^{r-1}+\sum_{k=0}^{r-2}\alpha_{k}t^{k} $, avec $ \alpha_{k}\ll \ee^{\bbeta} $ pour tout $ k $, ce qui cl\^ot la d\'emonstration du th\'eor\`eme. 
\end{proof}

\subsection{Application du th\'eor\`eme\;\ref{resultatBB3} }
Comme nous l'avons annonc\'e \`a la fin de la section\;\ref{preliminaires}, nous allons chercher \`a appliquer le th\'eor\`eme\;\ref{resultatBB3} aux familles de fonctions $ (\overline{h}_{\ee})_{\ee\in \NN^{n+r}} $ et $ (\underline{h}_{\ee})_{\ee\in \NN^{n+r}} $. Pour cela, il convient de v\'erifier que ces familles de fonctions sont bien des $  (\aalpha,D,\nu,\delta,\bbeta)- $ famille pour des param\`etres  $  (\aalpha,D,\nu,\delta,\bbeta) $ bien choisis. \\

Rappelons que par d\'efinition :  \begin{multline*}\overline{h}_{\ee}(k_{1},...,k_{r})=\Card\left\{ 
 \xx\in (\ZZ\setminus\{0\})^{n+r} \; | \;  \ee.\xx\in U, \;F(\ee.\xx)=0, \;\right. \\ \left. \forall j\in \{1,...,r\}, \;  \left\lfloor|(\ee.\xx)^{E(n+j)}|\right\rfloor=k_{j}, \;  \forall i\in \{1,...,n+r\},  \; |x_{i}|\leqslant \frac{1}{e_{i}}\prod_{j=1}^{r}(k_{j}+1)^{a_{i,j}} \right\}.
\end{multline*}
Or, nous avons vu que, puisque $ \sum_{\forall j\in \{1,...,r\}, k_{j}\leqslant P_{j}}=N_{U,\ee}(P_{1},...,P_{r}) $, d'apr\`es le th\'eor\`eme\;\ref{thmNU3}, $ (\underline{h}_{\ee})_{\ee\in \NN^{n+r}} $ v\'erifie la condition $ 1. $ pour $ \alpha_{j}=n_{j}-d_{j} $ pour tout $ j\in \{1,...,r\} $, $ c_{\ee}=C_{\sigma,\ee} $ et $ \bbeta $ d\'efini par \[ \ee^{\bbeta}=\max\left(e_{0}^{4+\delta}\left(\prod_{i=1}^{n+r}e_{i}\right)^{-1},\left(\prod_{i=1}^{n+r}e_{i}\right)^{-\frac{1}{2}}\right), \] (le fait que $ C_{\sigma,\ee}\ll \ee^{\bbeta} $ d\'ecoule de la remarque\;\ref{remsigmae3}). \\

On remarque que, pour tous $ \kk=(k_{m+1},...,k_{r}) $ et tous $ P_{1},...,P_{m} $, pour \begin{multline*} \phi_{m}(\kk)= \{(x_{i})_{i\in I_{m}}\in \mathcal{A}_{m}^{\lambda}\cap \ZZ^{s},\; |\;  \forall j\in \{1,...,r\}, \;  \left\lfloor|(\ee.\xx)^{E(n+j)}|\right\rfloor=k_{j}, \; \\ \forall i\in \{1,...,n+r\},  \; |x_{i}|\leqslant \frac{1}{e_{i}}\prod_{j=1}^{r}(k_{j}+1)^{a_{i,j}}\}  \end{multline*}
\begin{multline*}
\sum_{\forall j\in \{1,...,m\}, k_{j}\leqslant P_{j}}\overline{h}_{\ee}(k_{1},...,k_{r}) \\=\sum_{(x_{i})_{i\in I_{m}}\in \phi_{m}(\kk) }N_{(x_{i})_{i\in I_{m}},\ee}(P_{1},...,P_{m})+O(E)
\end{multline*}
o\`u  
\begin{align*}
E & = \sum_{\sigma\in\mathfrak{S}_{r} }\sum_{m'=1}^{r-1}\sum_{(x_{i})_{i\in I_{m',\sigma}}\notin \mathcal{A}_{m',\sigma}^{\lambda_{\sigma}}}\left(\prod_{i\notin I_{m',\sigma}}e_{i}\right)^{-1}\left(\prod_{j=1}^{m}P_{j}^{n_{\sigma(j)}}\right)\left(\prod_{j=m+1}^{r}k_{j}^{\sum_{i\notin I_{m,\sigma}}a_{i,\sigma(j)}}\right) \\ & \ll \left(\prod_{i=1}^{n+r}e_{i}\right)^{-\frac{1}{2}}\left(\prod_{j=1}^{m}P_{j}^{n_{j}}\right)\left(\prod_{j=m+1}^{r}k_{j}^{n_{j}+4d_{j}}\right) P_{\sigma(r)}^{-\lambda_{\sigma}/2}.
\end{align*} 
D'autre part, la formule du corollaire\;\ref{corollairex3}, est valable pour tous $ (x_{i})_{i\in I_{m}},\kk $ tels que $ \left(e_{0,m}\prod_{j=m+1}^{r}k_{j}^{d_{j}}\right)^{2}P^{-1+3\tilde{d}_{m}\theta}<1 $, o\`u $ P=\prod_{j=1}^{m}P_{j}^{d_{j}} $. Par cons\'equent, on obtient, uniform\'ement pour tous $ (k_{m+1},...,k_{r}) $ tels que $ \left(e_{0}\prod_{j=m+1}^{r}k_{j}^{d_{j}}\right)^{2}<P^{1-3\tilde{d}_{m}\theta} $ donc en particulier pour $ \left(e_{0}\prod_{j=m+1}^{r}k_{j}^{d_{j}}\right)^{2}<\left(\prod_{j=1}^{m}P_{j}^{d_{j}}\right)^{\frac{7}{10}} $ (puisque $ \theta<\frac{1}{10\tilde{d}_{m}} $) :

\begin{multline}\label{cond2k3}
 \sum_{\forall j\in \{1,...,m\}, k_{j}\leqslant P_{j}}\overline{h}_{\ee}(k_{1},...,k_{r})=\sum_{(x_{i})_{i\in I_{m}}\in \phi_{m}(\kk) }\mathfrak{S}_{(x_{i})_{i\in I_{m}},\ee}J_{(x_{i},e_{i})_{i\in I_{m}},\kk} \\ \left(\prod_{i\notin I_{m}}e_{i}\right)^{-1}\left(\prod_{j=m+1}^{r}k_{j}^{\sum_{i\notin I_{m}}a_{i,j}}\right)\left(\prod_{j=1}^{m}P_{j}^{n_{j}-d_{j}}\right) \\ + O\left(\left(e_{0}^{4}\left(\prod_{i=1}^{n+r}e_{i}\right)^{-1}+\left(\prod_{i=1}^{n+r}e_{i}\right)^{-\frac{1}{2}}\right)\left(\prod_{j=m+1}^{r}k_{j}^{n_{j}+4d_{j}}\right)\left(\prod_{j=1}^{m}P_{j}^{n_{j}-d_{j}}\right)P^{-\delta}\right).
\end{multline}

On peut en fait remplacer la condition $ \left(e_{0}\prod_{j=m+1}^{r}k_{j}^{d_{j}}\right)<\left(\prod_{j=1}^{m}P_{j}^{d_{j}}\right)^{\frac{7}{20}} $ par \begin{equation}\label{condunif3} \prod_{j=m+1}^{r}k_{j}^{d_{j}}<\left(\prod_{j=1}^{m}P_{j}^{d_{j}}\right)^{\frac{1}{10}-\delta}, \end{equation} avec $ \delta>0 $ arbitrairement petit. En effet, si l'on suppose $ e_{0}\leqslant \left(\prod_{j=1}^{m}P_{j}^{d_{j}}\right)^{\frac{1}{4}+\delta}  $, alors la condition\;\eqref{condunif3} implique : \[ \left(e_{0}\prod_{j=m+1}^{r}k_{j}^{d_{j}}\right)<\left(\prod_{j=1}^{m}P_{j}^{d_{j}}\right)^{\frac{1}{10}+\frac{1}{4}}=\left(\prod_{j=1}^{m}P_{j}^{d_{j}}\right)^{\frac{7}{20}}, \] et donc la formule\;\eqref{cond2k3} est v\'erifi\'ee.\\

Si l'on suppose \`a pr\'esent que $ e_{0}> \left(\prod_{j=1}^{m}P_{j}^{d_{j}}\right)^{\frac{1}{4}+\delta}  $ on a alors que, d'apr\`es les lemmes\;\ref{Sx3} et\;\ref{Jx3} : \begin{multline*}
\sum_{(x_{i})_{i\in I_{m}}\in \phi_{m}(\kk) }\mathfrak{S}_{(x_{i})_{i\in I_{m}},\ee}J_{(x_{i},e_{i})_{i\in I_{m}},\kk} \left(\prod_{i\notin I_{m}}e_{i}\right)^{-1}\left(\prod_{j=m+1}^{r}k_{j}^{\sum_{i\notin I_{m}}a_{i,j}}\right)\left(\prod_{j=1}^{m}P_{j}^{n_{j}-d_{j}}\right) \\ \ll  e_{0,m}^{2+\delta}\left(\prod_{i\notin I_{m}}e_{i}\right)^{-1}\left(\prod_{j=m+1}^{r}k_{j}^{\sum_{i\notin I_{m}}a_{i,j}+4d_{j}+\delta}\right)\left(\prod_{j=1}^{m}P_{j}^{n_{j}-d_{j}}\right)\Card\phi_{m}(\kk) \\ \ll e_{0}^{4}\left(\prod_{i=1}^{n+r}e_{i}\right)^{-1}\left(\prod_{j=m+1}^{r}k_{j}^{n_{j}+4d_{j}}\right)\left(\prod_{j=1}^{m}P_{j}^{n_{j}-d_{j}}\right)P^{-\delta},
\end{multline*}
et d'autre part, $  \sum_{\forall j\in \{1,...,m\}, k_{j}\leqslant P_{j}}\overline{h}_{\ee}(k_{1},...,k_{r}) $ peut \^etre major\'e trivialement par \begin{multline*} \left(\prod_{i=1}^{n+r}e_{i}\right)^{-1}\left(\prod_{j=m+1}^{r}k_{j}^{n_{j}}\right)\left(\prod_{j=1}^{m}P_{j}^{n_{j}}\right) \\ \ll  e_{0}^{4}\left(\prod_{i=1}^{n+r}e_{i}\right)^{-1}\left(\prod_{j=m+1}^{r}k_{j}^{n_{j}+4d_{j}}\right)\left(\prod_{j=1}^{m}P_{j}^{n_{j}-d_{j}-\delta}\right). 
\end{multline*}
La formule\;\eqref{cond2k3} est donc v\'erifi\'ee. Ainsi, on a que la condition $ 2. $ est v\'erifi\'ee par $ (\overline{h}_{\ee})_{\ee} $ pour $ D=\max_{j\in \{1,...,r\}}\{n_{j}+4d_{j}\} $ et $ \nu=\frac{1}{\max_{j\in \{1,...,r\}}d_{j}}\left(\frac{1}{10}-\delta\right) $. Nous pouvons donc bien appliquer le th\'eor\`eme \`a la famille $ (\overline{h}_{\ee})_{\ee} $ et nous obtenons : \begin{multline*} \sum_{\prod_{j=1}^{r}k_{j}^{n_{j}-d_{j}}\leqslant B}\overline{h}_{\ee}(k_{1},...,k_{r}) \\ =C_{\sigma,\ee}B(\log B)^{r-1}+O\left( \left(e_{0}^{4+\delta}\left(\prod_{i=1}^{n+r}e_{i}\right)^{-1}+\left(\prod_{i=1}^{n+r}e_{i}\right)^{-\frac{1}{2}}\right)B(\log B)^{r-2}\right). \end{multline*}
Par les m\^emes calculs, on trouve exactement le m\^eme r\'esultat avec $ (\underline{h}_{\ee})_{\ee} $, et on en d\'eduit le th\'eor\`eme ci-dessous : 

\begin{thm}\label{thmNsigmaU3} Si l'on a $ n+r>\mathfrak{m}  $, alors pour tout $ \sigma\in \Delta_{\max} $ :
\[ N_{\ee,\sigma,U}(B)=\frac{1}{(r-1)!}C_{\sigma,\ee}B\log(B)^{r-1}+O\left(\left(e_{0}^{4+\delta}\left(\prod_{i=1}^{n+r}e_{i}\right)^{-1}+\left(\prod_{i=1}^{n+r}e_{i}\right)^{-\frac{1}{2}}\right)B\log(B)^{r-2}\right). \] 
\end{thm}

\begin{rem}\label{remnF3}
On remarque que, d'apr\`es le lemme\;\ref{bornefrakm3} on a  
\[ \mathfrak{m} \leqslant r(8.2^{\sum_{j=1}^{r}d_{j}}+4)\left(\prod_{j=1}^{r}(3+10d_{j})\right)\left(\sum_{j=1}^{r} d_{j}\right)+\max_{\substack{m\in \{1,...,r\} \\ \sigma\in \mathfrak{S}_{r} }}\max_{(j,k)\in \mathcal{C}_{m,\tau} }\dim V^{\ast}_{\tau,m,\tt,(j,k)}. \]
Par cons\'equent, si l'on note 
 \[ n(F)= n+r-\max_{\substack{m\in \{1,...,r\} \\ \tau\in \mathfrak{S}_{r} }}\min_{\tt^{(m,\tau)}\in \mathcal{D}_{m,\tau}}\max_{(j,k)\in \mathcal{C}_{m,\tau} }\dim V^{\ast}_{\tau,m,\tt^{(m,\tau)},(j,k)},   \] o\`u $ \mathcal{D}^{\ast}_{m,\tau}=\{\dd=(d_{j,k})_{(j,k)\in\mathcal{C}_{m,\tau} }\in \NN^{\mathcal{C}_{m,\tau}}\;|\; F_{\dd}\neq 0 \} $, alors le th\'eor\`eme\;\ref{thmNsigmaU3} est v\'erifi\'e lorsque 
\[  n(F)\geqslant r(8.2^{\sum_{j=1}^{r}d_{j}}+4)\left(\prod_{j=1}^{r}(3+10d_{j})\right)\left(\sum_{j=1}^{r} d_{j}\right). \]
\end{rem}

\section{Conclusion et interpr\'etation des constantes }

Nous allons \`a pr\'esent conclure quant \`a la formule asymptotique pour \[ \mathcal{N}_{V}(B)=\Card\{P=\pi(\xx)\in Y(\QQ) \; | \; \xx\in V(\QQ)\; |\; H_{D_{0}}(P)\leqslant B\}, \] o\`u $ V $ est l'ouvert $ V=\pi(U) $ de $ X $. Nous avons vu dans la section\;\ref{hauteur} que \[ \mathcal{N}_{V}(B)=\frac{1}{2^{r}} \mathcal{N}_{0,U}(B), \]
o\`u 
\begin{multline*} \mathcal{N}_{0,U}(B)  =\Card\{ \xx\in (\ZZ\setminus \{0\})^{n+r}\cap U\; |\;  \\ \; F(\xx)=0, \; \PGCD_{\sigma\in \Delta_{\max}}(\xx^{\underline{\sigma}})=1, \; H_{0}(\xx)\leqslant B \}. \end{multline*}
Nous avons par ailleurs : \[ \mathcal{N}_{0,U}(B)=\sum_{\sigma \in \Delta_{\max}}N_{U,\sigma}(B)+O\left(B(\log B)^{r-2}\right) \] o\`u 
\begin{multline*}
N_{U,\sigma}(B)=\Card\{\xx\in (\ZZ\setminus\{0\})^{n+r}\cap C_{0,\sigma}\cap U  \; |\; \PGCD_{\sigma\in \Delta_{\max}}(\xx^{\underline{\sigma}})=1, \\ \; F(\xx)=0, \; H_{0}(\xx)\leqslant B \}.
\end{multline*} 
Nous allons \`a pr\'esent relier le cardinal $ N_{U,\sigma}(B) $ \`a $ N_{\ee,\sigma,U} $ via une inversion de M\"{o}bius afin d'en d\'eduire une formule asymptotique pour $ \mathcal{N}_{V}(B) $ \`a partir du th\'eor\`eme\;\ref{thmNsigmaU3}.

\subsection{Sommation de M\"{o}bius}\label{sommemobius}

D\'efinissons \`a pr\'esent la fonction arithm\'etique $ \mu $ (cf. \cite[Proposition 11.9]{Sa}) : 
Pour tout $ \ee,\dd\in \NN^{n+r} $ on pose : \[ \chi(\ee)=\left\{ \begin{array}{lcr } 1 & \mbox{si} & \PGCD_{\sigma\in \Delta_{\max}}(\ee^{\underline{\sigma}})=1 \\ 0 & \mbox{sinon}
\end{array}\right. , \]\[ \chi_{\dd}(\ee)=\left\{ \begin{array}{lcr } 1 & \mbox{si} & \forall i\in \{1,...,n+r\}, \; d_{i}|e_{i} \\ 0 & \mbox{sinon}\end{array}\right.. \] Il existe alors (voir \cite[Proposition 11.9]{Sa}) une unique fonction arithm\'etique $ \mu : \NN^{n+r}\ra \RR $ telle que pour tout $ \ee\in \NN^{n+r} $ : \[ \chi(\ee)=\sum_{\dd\in (\NN^{\ast})^{n+r}}\mu(\dd)\chi_{\dd}(\ee). \]
Donnons quelques propri\'et\'es de cette fonction $ \mu $. 
\begin{lemma}
Si, pour $ p $ premier, \[ \chi^{(p)}(\ee)=\left\{ \begin{array}{lcr } 1 & \mbox{si} & p\nmid \PGCD_{\sigma\in \Delta_{\max}}(\ee^{\underline{\sigma}}) \\ 0 & \mbox{sinon}\end{array}\right.,\]
alors on a \[ \chi^{(p)}=\sum_{\substack{\dd=(p^{\nu_{1}},...,p^{\nu_{n+r}}) \\ \nu_{1},...,\nu_{n+r}\in \NN^{n+r}}}\mu(\dd)\chi_{\dd}(\ee). \]
De plus, si $ \dd=(p^{\nu_{1}},...,p^{\nu_{n+r}}) $, avec l'un des $ \nu_{i} $ sup\'erieur ou \'egal \`a $ 2 $, alors $ \mu(\dd)=0 $. 

\end{lemma}
\begin{lemma}[\bf{Salberger}]\label{fonctmu3}
\begin{enumerate}
\item Soient $ \dd,\ee \in( \NN^{\ast})^{n+r} $ et $ \delta=\PGCD_{\sigma\in \Delta_{\max}}(\dd^{\underline{\sigma}}) $, $ \varepsilon=\PGCD_{\sigma\in \Delta_{\max}}(\ee^{\underline{\sigma}}) $. Si $ \PGCD(\delta,\varepsilon)=1 $, alors \[ \mu(\dd.\ee)=\mu(\dd)\mu(\ee). \]
\item Soit $ f $ le plus petit entier tel qu'il existe $ f $ ar\^etes de $ \Delta $ non contenues dans un c\^one de $ \Delta $. Alors le produit eul\'erien : \[ \prod_{p\in \mathcal{P}}\left(\sum_{\substack{\dd=(p^{\nu_{1}},...,p^{\nu_{n+r}}) \\ \nu_{1},...,\nu_{n+r}\in \NN^{n+r}}}\frac{|\mu(\dd)|}{\left(\prod_{i=1}^{n+r}d_{i}\right)^{s}}\right) \] est absolument convergent pour $ s>\frac{1}{f} $. 
\end{enumerate}

\end{lemma}
\begin{proof}
Voire \cite[Lemme 11.15]{Sa}.
\end{proof}

On a alors : \[ N_{U,\sigma}(B)=\sum_{\ee\in (\NN^{\ast})^{n+r}}\mu(\ee) N_{\ee,\sigma,U}. \] Donc, en utilisant le th\'eor\`eme\;\ref{thmNsigmaU3}, on obtient : \begin{multline*} N_{U,\sigma}(B)=\left(\frac{1}{(r-1)!}\sum_{\ee\in (\NN^{\ast})^{n+r}}\mu(\ee)C_{\sigma,\ee}\right)B(\log B)^{r-1} \\ +O\left(\sum_{\ee\in (\NN^{\ast})^{n+r}}|\mu(\ee)|\left(e_{0}^{4+\delta}\left(\prod_{i=1}^{n+r}e_{i}\right)^{-1}+\left(\prod_{i=1}^{n+r}e_{i}\right)^{-\frac{1}{2}}\right)B(\log B)^{r-2}\right).
\end{multline*}
Notons \`a pr\'esent $ I_{1},...,I_{N} $ les ensembles $ I $ minimaux pour l'inclusion tels que \[  \forall \sigma \in \Delta, \; \sum_{i\in I} \RR^{+}v_{i}\nsubseteq \sigma.  \] Par hypoth\`ese, pour tout $ k $, $ \Card I_{k}\geqslant 6 $. Dans ce cas, si l'on observe que, pour tout $ \ee $ tel que $ \mu(\ee)\neq 0 $, on a \[ e_{0}\leqslant \PPCM_{i=1}^{n+r}(e_{i})\leqslant \prod_{\substack{p\in \mathcal{P} \\ \exists i \; |\; p|e_{i} }}p \leqslant \left(\prod_{i=1}^{n+r}e_{i}\right)^{\frac{1}{\min_{k} \Card I_{k}}}\leqslant \left(\prod_{i=1}^{n+r}e_{i}\right)^{\frac{1}{6}}.  \]
Ainsi, on a \begin{multline*} \sum_{\ee\in (\NN^{\ast})^{n+r}}|\mu(\ee)|\left(e_{0}^{4+\delta}\left(\prod_{i=1}^{n+r}e_{i}\right)^{-1}+\left(\prod_{i=1}^{n+r}e_{i}\right)^{-\frac{1}{2}}\right) \\ \ll 
 \sum_{\ee\in (\NN^{\ast})^{n+r}}|\mu(\ee)|\left(\prod_{i=1}^{n+r}e_{i}\right)^{-\frac{1}{3}+\delta},
\end{multline*}et cette s\'erie converge d'apr\`es le lemme\;\ref{fonctmu3}.\\

Ainsi, \begin{align*} N_{U,\sigma}(B) & =\left(\frac{1}{(r-1)!}\sum_{\ee\in (\NN^{\ast})^{n+r}}\mu(\ee)C_{\sigma,\ee}\right)B(\log B)^{r-1} +O(B(\log B)^{r-2}) \\ & =\frac{1}{(r-1)!} \left(\sum_{\ee\in (\NN^{\ast})^{n+r}}\frac{\mu(\ee)\mathfrak{S}_{\ee}}{\prod_{i=1}^{n+r}e_{i}}\right)J_{\sigma}B(\log B)^{r-1} +O(B(\log B)^{r-2})\end{align*}
o\`u \[  \mathfrak{S}_{\ee}=\sum_{q=1}^{\infty}q^{-(n+r)}\sum_{\substack{0\leqslant a<q\\ \PGCD(a,q)=1}}\sum_{\bb\in (\ZZ/q\ZZ)^{n+r}}e\left(\frac{a}{q}F(\ee.\bb)\right)  \]
\[ J_{\sigma}=\int_{\RR}\int_{C_{\sigma}(\RR)\cap [-1,1]^{n+r}}e\left(\beta F(\uu)\right)d\uu d\beta. \]

On obtient donc finalement, en remarquant que $ \bigcup_{\sigma\in \Delta_{\max}}C_{\sigma}(\RR)=\RR^{n+r} $ et que l'intersection des $ C_{\sigma}(\RR) $ est de mesure nulle (en effet, si $ \xx\in  C_{\sigma}(\RR)\cap  C_{\tau}(\RR) $, alors $ |\xx^{D(\sigma)}|=|\xx^{D(\tau)}| $)
\begin{thm}\label{Conclusion3}
Si l'on a \[  n(F)\geqslant r(8.2^{\sum_{j=1}^{r}d_{j}}+4)\left(\prod_{j=1}^{r}(3+10d_{j})\right)\left(\sum_{j=1}^{r} d_{j}\right), \] et si $ \Delta $ est tel que, pour tout $ k\in \NN $, $ \Card I_{k}\geqslant 6 $, alors 
\begin{equation}\label{formconcl3} \mathcal{N}_{V}(B)=\frac{1}{(r-1)!}\frac{1}{2^{r}}\mathfrak{S}J B(\log B)^{r-1} +O(B(\log B)^{r-2}), \end{equation} o\`u  \begin{equation} \mathfrak{S}=\sum_{\ee\in (\NN^{\ast})^{n+r}}\frac{\mu(\ee)\mathfrak{S}_{\ee}}{\prod_{i=1}^{n+r}e_{i}} \end{equation} et \begin{equation} J  = \sum_{\sigma\in \Delta_{\max}}J_{\sigma} =\int_{\RR}\int_{\substack{\uu\in \RR^{n+r}\; |\; \\ \max_{\sigma\in \Delta_{\max}}|\uu^{D(\sigma)}|\leqslant 1}}e\left(\beta F(\uu)\right)d\uu d\beta. \end{equation} 
\end{thm}
Nous allons \`a pr\'esent chercher \`a interpr\'eter les constantes intervenant dans la formule\;\eqref{formconcl3} en terme des mesures de Tamagawa et d\'emontrer ainsi que la constante $ \frac{1}{(r-1)!}\frac{1}{2^{r}}\mathfrak{S}J $ est bien celle conjectur\'ee par Peyre, ce qui ach\`evera la d\'emonstration du th\'eor\`eme\;\ref{thm3b3}

Avant d'aller plus loin, rappelons la d\'efinition des formes de Leray $ \omega_{L,\nu} $ et des mesures de Tamagawa $ \omega_{\nu} $ pour $ \nu \in \Val(\QQ) $.\\
 
Consid\'erons un point $ \xx_{0}\in Y_{0} $, tel que $ \frac{\partial F}{\partial x_{i_{0}}}(\xx_{0})\neq 0 $ pour un certain $ i_{0}\in \{1,...,n+r\} $, et on note $ P_{0}=\pi(\xx_{0}) $. La forme de Leray $ \omega_{L} $ est d\'efinie sur un voisinage de $ \xx_{0} $ sur lequel $ \frac{\partial F}{\partial x_{i_{0}}}\neq 0 $ par : \[ \omega_{L} (\xx)=\frac{(-1)^{n+r-i_{0}}}{\frac{\partial F}{\partial x_{i_{0}}}(\xx)}dx_{1}\wedge...\wedge \widehat{dx_{i_{0}}} \wedge ... \wedge dx_{n+r}(\xx). \]
Pour tout $ \nu\in \Val(\QQ) $, la forme de Leray induit une mesure locale $ \omega_{L,\nu} $ sur $ Y_{0}(\QQ_{\nu}) $. On suppose \`a pr\'esent que $ \xx_{0} $ est tel que $ x_{0,i}\neq 0 $ pour tout $ i\in \{1,...,n+r\} $. Pour $ \nu \in \Val(\QQ) $, on consid\`ere le morphisme : \[ \begin{array}{rcl}
\rho : Y_{\QQ_{\nu}} &  \ra  & \AA_{\QQ}^{n-1} \\ \xx & \mt & \left(\frac{x_{i}}{\prod_{j=1}^{r}x_{n+j}^{a_{i,j}}}\right)_{i\neq i_{0}}
\end{array} \]

Par le th\'eor\`eme d'inversion locale, il existe un voisinage ouvert de $ P_{0} $ not\'e $ V_{0} $ sur lequel $ \rho $ est bien d\'efini et induit un diff\'eomorphisme analytique sur $ \rho(V_{0}) $. On note alors $ W_{0}=\pi^{-1}(V_{0}) $ et $ \uu=((u_{i})_{i\in \{1,...,n\}},\underbrace{1,...,1}_{r\; \elements}) $. La mesure de Tamagawa $ \omega_{\nu} $ est d\'efinie par : \[ \rho_{\ast}\omega_{\nu}=\frac{du_{1,\nu}...\widehat{du_{i_{0},\nu}}...du_{n,\nu}}{h_{\nu}(\uu)\left|\frac{\partial F}{\partial x_{i_{0}}}(\uu)\right|_{\nu}} \] (o\`u $ u_{i_{0}} $ est d\'efini implicitement par $ F(\uu)=0 $), avec \[ h_{\nu}(\uu)=\max_{\sigma\in \Delta_{\max}}|\uu^{D(\sigma)}|_{\nu}. \]

\subsection{\'Etude de la s\'erie singuli\`ere $ \mathfrak{S} $}

 Posons \[ A_{\ee}(q)=q^{-(n+r)}\sum_{\substack{0\leqslant a<q\\ \PGCD(a,q)=1}}\sum_{\bb\in (\ZZ/q\ZZ)^{n+r}}e\left(\frac{a}{q}F(\ee.\bb)\right). \] On a alors $ \mathfrak{S}_{\ee}=\sum_{q=1}^{\infty}A_{\ee}(q) $.
\begin{lemma}\label{fonctionmultiplicative3}
Pour tout \'el\'ement $ \ee\in \NN^{n+r} $, la fonction $ A_{\ee} $ est multiplicative. 
\end{lemma}
\begin{proof}
Consid\'erons $ q_{1},q_{2}\in \NN $ tels que $ \PGCD(q_{1},q_{2})=1 $ et $ q=q_{1}q_{2} $. On remarque que  \[ A_{\ee}(q_{1})A_{\ee}(q_{2})=q^{-(n+r)}\sum_{\substack{a_{1}\in (\ZZ/q_{1}\ZZ)^{\ast} \\a_{2}\in (\ZZ/q_{2}\ZZ)^{\ast}  }}\sum_{\substack{\bb_{1}\in (\ZZ/q_{1}\ZZ)^{n+r} \\ \bb_{2}\in (\ZZ/q_{2}\ZZ)^{n+r}}}e\left(\frac{a_{1}q_{2}F(\ee.\bb_{1})+a_{2}q_{1}F(\ee.\bb_{2})}{q_{1}q_{2}}\right). \] Or, si l'on consid\`ere l'unique \'el\'ement $ \bb\in (\ZZ/q\ZZ)^{n+r} $ tel que $ \bb\equiv \bb_{1}\;(q_{1}) $ et $ \bb\equiv \bb_{2}\;(q_{2}) $.  On a alors \[ \left\{\begin{array}{l}
q_{2}F(\ee.\bb_{1})\equiv q_{2}F(\ee.\bb)\; (q) \\ q_{1}F(\ee.\bb_{2})\equiv q_{1}F(\ee.\bb)\; (q)
\end{array}\right., \]
et par cons\'equent \[ A_{\ee}(q_{1})A_{\ee}(q_{2})=q^{-(n+r)}\sum_{\substack{a_{1}\in (\ZZ/q_{1}\ZZ)^{\ast} \\a_{2}\in (\ZZ/q_{2}\ZZ)^{\ast}  }}\sum_{\substack{\bb\in (\ZZ/q\ZZ)^{n+r} }}e\left(\frac{a_{1}q_{2}+a_{2}q_{1}}{q}F(\ee.\bb)\right). \]
Par ailleurs, l'application \[ \begin{array}{rcl}
(\ZZ/q_{1}\ZZ)^{\ast}\times (\ZZ/q_{2}\ZZ)^{\ast} & \ra & (\ZZ/q\ZZ)^{\ast} \\ (a_{1},a_{2}) & \mt & a_{1}q_{2}+a_{2}q_{1}
\end{array} \] est bijective . On obtient donc finalement \[ A_{\ee}(q_{1})A_{\ee}(q_{2})=q^{-(n+r)}\sum_{a\in (\ZZ/q\ZZ)^{\ast}}\sum_{\bb\in (\ZZ/q\ZZ)^{n+r}}e\left(\frac{a}{q}F(\ee.\bb)\right)=A_{\ee}(q). \]
\end{proof}
Puisque nous avons vu avec le lemme\;\ref{S3} que la s\'erie $ \mathfrak{S}_{\ee} $ est absolument convergente, on a donc : \begin{equation}\label{prodeuler3}
\mathfrak{S}_{\ee}=\prod_{p\in \mathcal{P}}\left(\sum_{k=0}^{\infty}A_{\ee}(p^{k})\right). \end{equation}
Remarquons \`a pr\'esent que pour tous $ \ee\in \NN^{n+r} $ et $ k\in \NN $ : \[ \sum_{\bb\in (\ZZ/p^{k}\ZZ)^{n+r} }e\left( \frac{a}{p^{k}}F(\ee.\bb)\right)= \sum_{\bb\in (\ZZ/p^{k}\ZZ)^{n+r} }e\left( \frac{a}{p^{k}}F((p^{v_{p}(e_{i})}b_{i})_{i\in \{1,...,n+r\}})\right) \] (en effet si $ \PGCD(q,p)=1 $, $ b\mt qb $  est une bijection de $ \ZZ/p^{k}\ZZ $ sur  $ \ZZ/p^{k}\ZZ $). Par cons\'equent, on a \[ A_{\ee}(p^{k})=A_{(p^{v_{p}(e_{i})})_{i\in \{1,...,n+r\}}}(p^{k}), \] et en utilisant la formule\;\eqref{prodeuler3}, on trouve : \[ \frac{\mu(\ee)\mathfrak{S}_{\ee}}{\prod_{i=1}^{n+r}e_{i}}=\prod_{p\in \mathcal{P}}B_{(p^{v_{p}(e_{i})})_{i\in \{1,...,n+r\}}} \] o\`u pour tous $ \nu_{1},...,\nu_{n+r}\in \NN^{n+r} $ : \[ B_{(p^{\nu_{i}})_{i\in \{1,...,n+r\}}}=\frac{\mu(p^{\nu_{1}},...,p^{\nu_{n+r}})}{p^{\sum_{i=1}^{n+r}\nu_{i}}}A_{(p^{\nu_{i}})_{i\in \{1,...,n+r\}}}(p^{k}). \]

La s\'erie $ \mathfrak{S}=\sum_{\ee\in (\NN^{\ast})^{n+r}}\frac{\mu(\ee)\mathfrak{S}_{\ee}}{\prod_{i=1}^{n+r}e_{i}} $ \'etant absolument convergente, on a alors : \begin{align*}
\mathfrak{S} & = \prod_{p\in \mathcal{P}}\left(\sum_{\nu_{1},...,\nu_{n+r}\in \NN}B_{(p^{\nu_{i}})_{i\in \{1,...,n+r\}}}\right) \\ & =\prod_{p\in \mathcal{P}}\left(\sum_{(\nu_{1},...,\nu_{n+r})\in \{0,1\}^{n+r}}B_{(p^{\nu_{i}})_{i\in \{1,...,n+r\}}}\right) \\ & =\prod_{p\in \mathcal{P}}\left(\sum_{k=0}^{\infty}\sum_{(\nu_{1},...,\nu_{n+r})\in \{0,1\}^{n+r}}\frac{\mu(p^{\nu_{1}},...,p^{\nu_{n+r}})}{p^{\sum_{i=1}^{n+r}\nu_{i}}}A_{(p^{\nu_{i}})_{i\in \{1,...,n+r\}}}(p^{k})\right).
\end{align*}
Notons \`a pr\'esent : \begin{equation} M_{p}^{\ast}(k)=\Card\{\xx\in (\ZZ/p^{k}\ZZ)^{n+r}\; |\; p\nmid \PGCD_{\sigma\in \Delta_{\max}}(\xx^{\underline{\sigma}}) \; \et \; F(\xx)\equiv 0 \; (p^{k}) \}, \end{equation} et \begin{equation}\label{prodeulerpoursigmap3}
\sigma_{p}=\left(\sum_{k=0}^{\infty}\sum_{(\nu_{1},...,\nu_{n+r})\in \{0,1\}^{n+r}}\frac{\mu(p^{\nu_{1}},...,p^{\nu_{n+r}})}{p^{\sum_{i=1}^{n+r}\nu_{i}}}A_{(p^{\nu_{i}})_{i\in \{1,...,n+r\}}}(p^{k})\right).
\end{equation}
\begin{lemma}
Pour tout entier $ m>0 $, on a \[ \frac{M_{p}^{\ast}(m)}{p^{m(n+r-1)}}=\sum_{k=0}^{m}\sum_{\nu_{1},...,\nu_{n+r}\in \{0,1\}}\frac{\mu(p^{\nu_{1}},...,p^{\nu_{n+r}})}{p^{\sum_{i=1}^{n+r}\nu_{i}}}A_{(p^{\nu_{i}})_{i\in \{1,...,n+r\}}}(p^{k}),  \]
et donc \[ \sigma_{p}=\lim_{m\ra \infty}\frac{M_{p}^{\ast}(m)}{p^{m(n+r-1)}}. \]
\end{lemma}
\begin{proof}
On pose $ q=p^{m} $. Il est alors imm\'ediat que : \[ q^{-1}\sum_{t=0}^{q-1}\sum_{\bb\in (\ZZ/q\ZZ)^{n+r}}e\left( \frac{t}{q}F(\bb)\right)=\Card\{\xx\in \ZZ^{n+r}\; |\; F(\xx)\equiv 0 \; (q) \}, \] et plus g\'en\'eralement pour tout $ \nu_{1},...,\nu_{n+r}\in \{0,1\} $ : \begin{multline*} q^{-1}\sum_{t=0}^{q-1}\sum_{\bb\in (\ZZ/q\ZZ)^{n+r}}e\left( \frac{t}{q}F((p^{\nu_{i}}b_{i})_{i\in \{1,...,n+r\}})\right) \\ =\left(\prod_{l=1}^{n+r}p^{\nu_{l}}\right)\Card\{\xx\in \ZZ^{n+r}\; |\; x_{i}\equiv 0 \; (p^{\nu_{i}}), \; F(\xx)\equiv 0 \; (q) \}. \end{multline*} On a donc que, en utilisant le lemme\;\ref{fonctionmultiplicative3} et la formule\;\eqref{prodeulerpoursigmap3} : \begin{multline*} M_{p}^{\ast}(m)  =q^{-1}\sum_{t=0}^{q-1}\sum_{\bb\in (\ZZ/q\ZZ)^{n+r}}\sum_{(\nu_{1},...,\nu_{n+r})\in \{0,1\}^{n+r}} \\ \frac{\mu(p^{\nu_{1}},...,p^{\nu_{n+r}})}{p^{\sum_{i=1}^{n+r}\nu_{i}}}e\left( \frac{t}{q}F((p^{\nu_{i}}b_{i})_{i\in \{1,...,n+r\}})\right)
 \\  =q^{-1}\sum_{q_{1}|q}\sum_{\substack{0\leqslant a<q_{1} \\ \PGCD(a,q_{1})=1 }}\sum_{\bb\in (\ZZ/q\ZZ)^{n+r}}\sum_{(\nu_{1},...,\nu_{n+r})\in \{0,1\}^{n+r}} \\ \frac{\mu(p^{\nu_{1}},...,p^{\nu_{n+r}})}{p^{\sum_{i=1}^{n+r}\nu_{i}}}e\left( \frac{a}{q_{1}}F((p^{\nu_{i}}b_{i})_{i\in \{1,...,n+r\}})\right)
 \\  = p^{-m}\sum_{k=0}^{m}\frac{p^{m(n+r)}}{p^{k(n+r)}}\sum_{a\in (\ZZ/p^{k}\ZZ)^{\ast}}\sum_{\bb\in (\ZZ/p^{k}\ZZ)^{n+r}}\sum_{(\nu_{1},...,\nu_{n+r})\in \{0,1\}^{n+r}} \\ \frac{\mu(p^{\nu_{1}},...,p^{\nu_{n+r}})}{p^{\sum_{i=1}^{n+r}\nu_{i}}}e\left( \frac{a}{p^{k}}F((p^{\nu_{i}}b_{i})_{i\in \{1,...,n+r\}})\right)
\\  = p^{m(n+r-1)}\sum_{k=0}^{m}\sum_{(\nu_{1},...,\nu_{n+r})\in \{0,1\}^{n+r}} \frac{\mu(p^{\nu_{1}},...,p^{\nu_{n+r}})}{p^{\sum_{i=1}^{n+r}\nu_{i}}}A_{(p^{\nu_{i}})_{i\in \{1,...,n+r\}}}(p^{k}) \\  =p^{m(n+r-1)}\sigma_{p}.
\end{multline*}

\end{proof}

% à compléter

\begin{lemma}\label{lemmeinter13}
Pour tout $ N\in \NN^{\ast} $, on note \begin{multline*} W_{p}^{\ast}(N)=\{ \xx\in (\ZZ_{p}/p^{N})^{n+r} \; | \; \PGCD_{\sigma\in \Delta_{\max}}\xx^{\underline{\sigma}}\not\equiv \0\; (p), \;  F(\xx)\equiv 0 \; (p^{N}) \} \end{multline*} de sorte que $ M_{p}^{\ast}(N)=\Card W_{p}^{\ast}(N) $. Il existe alors un entier $ N_{0} $ tel que pour tout $ N\geqslant N_{0} $ : \[ \int_{\substack{\xx\in \ZZ_{p}^{n+r} \; |\;   F(\xx)=0  \\ \PGCD_{\sigma\in \Delta_{\max}}\xx^{\underline{\sigma}}\not\equiv \0\; (p), \;  }}\omega_{L,p}=\frac{M_{p}^{\ast}(N)}{p^{N(n+r-1)}}. \]
\end{lemma}
\begin{proof}
Soit $ \xx\in \ZZ_{p}^{n+r} $. Dans tout ce qui suit, on note \[ [\xx]_{N}=\xx \mod p^{N} .\] On \'ecrit alors : \begin{align}
\int_{\substack{\xx\in \ZZ_{p}^{n+2} \;|\; F(\xx)=0 \\ \PGCD_{\sigma\in \Delta_{\max}}\xx^{\underline{\sigma}}\not\equiv \0\; (p)}}\omega_{L,p} & =\sum_{\substack{\xx\mod p^{N}\; |\; F(\xx)\equiv 0 \; (p^{N}) \\ \PGCD_{\sigma\in \Delta_{\max}}\xx^{\underline{\sigma}}\not\equiv \0\; (p) }}\int_{\substack{\uu\in \ZZ_{p}^{n+r} \; |\; F(\uu)=0 \\ [\uu]_{N}=\xx    }}\omega_{L,p}(\uu) \\ & =\sum_{\xx\in W_{p}^{\ast}(N)}\int_{\substack{\uu,\in \ZZ_{p}^{n+r}\; |\; F(\uu)=0 \\ [\uu]_{N}=\xx    }}\omega_{L,p}(\uu).
\end{align}
Puisque $ Y $ est lisse, il existe un entier $ N>0 $ assez grand tel que, pour tout $ \xx\in (\ZZ_{p}/p^{N})^{n+r} $ et tout  $ \uu\in \ZZ_{p}^{n+r} $ tel que $ [\uu]_{N}=\xx $, $ \PGCD_{\sigma\in \Delta_{\max}}\uu^{\underline{\sigma}}\not\equiv \0\; (p) $ et $ F(\uu)=0 $  : \[ c=\inf_{i}\lbrace v_{p}\left(\frac{\partial F}{\partial x_{i}}(\uu)\right)\rbrace \] soit non nul et constant sur la classe d\'efinie par $ \xx $. On peut supposer que $ N>c $ et que $ c=v_{p}\left(\frac{\partial F}{\partial x_{i_{0}}}(\xx)\right) $ pour un $ i_{0}\in \{1,...,n+r\} $ fix\'e. On consid\`ere $ \uu\in \ZZ_{p}^{n+r} $ tel que $ [\uu]_{N}=\xx $, et $ \uu'\in \ZZ_{p}^{n+r} $ quelconque. On a alors \[
F(\uu+\uu')=F(\uu)+\sum_{i=0}^{n+r}\frac{\partial F}{\partial x_{i}}(\uu)u_{i}'+G(\uu,\uu'), 
\] o\`u $ G(\uu,\uu') $ est une somme de termes contenant au moins deux facteurs $ u_{i}' $. Ainsi, on a donc, si $ \uu'\in (p^{N}\ZZ_{p})^{n+r} $ : \[ F(\uu+\uu')\equiv F(\uu) \; (p^{N+c}). \] Par cons\'equent, l'image de $ F(\uu) $ dans $ \ZZ_{p}/p^{N+c} $ d\'epend uniquement de $ \uu\mod p^{N}=\xx $, on note alors $ F^{\ast}(\xx) $ cette image. \\

 Si $ F^{\ast}(\xx)\neq 0 $, alors l'int\'egrale \[ \int_{\substack{\uu\in \ZZ_{p}^{n+r}\; |\; F(\uu,\vv,\ww)=0 \\ [\uu]_{N}=\xx  }}\omega_{L,p}(\uu) \] est nulle, et l'ensemble \[ \{\uu \mod p^{N+c} \; | \; [\uu]_{N}=\xx, \; F(\uu)\equiv 0\; (p^{N+c}) \} \] est vide. \\

Si $ F^{\ast}(\xx)= 0 $ alors, par le lemme de Hensel, les applications coordonn\'ees $ X_{1},...,X_{i_{0}-1},X_{i_{0}+1},...,X_{n+r} $ d\'efinissent un isomorphisme de \[ \{\uu\in \ZZ_{p}^{n+r} \; | \; [\uu]_{N}=\xx, \; F(\uu)=0 \} \] sur \[\hat{\xx}+(p^{N}\ZZ_{p})^{n+r-1}, \] o\`u $ \hat{\xx}=(x_{1},...,x_{i_{0}-1},x_{i_{0}+1},...,x_{n+r}) $. Par cons\'equent, on a : \begin{multline*}
\int_{\substack{\uu\in \ZZ_{p}^{n+r}\; |\;  F(\uu)=0\\ [\uu]_{N}=\xx    }}\omega_{L,p}(\uu)  \\ =\int_{\hat{\xx}+(p^{N}\ZZ_{p})^{n+r-1}}p^{c}dx_{1}...dx_{i_{0}-1}dx_{i_{0}+1}...dx_{n+r} =p^{c-N(n+r-1)}.
\end{multline*} 
On a d'autre part, puisque $ F(\uu) \mod p^{N+c} $ ne d\'epend que de $ \xx $ : \begin{multline*}
p^{-(N+c)(n+r-1)}\card\{\uu\mod p^{N+c}\; | \; [\uu]_{N}=\xx,\; F(\uu)\equiv 0(p^{N+c}) \}\\ =p^{-(N+c)(n+r-1)}p^{(n+r-1)c}  = p^{c-N(n+r-1)}.
\end{multline*}
On a donc finalement : \begin{multline*}
\int_{\substack{\xx\in \ZZ_{p}^{n+r} \; |\;  F(\xx)=0\\ \PGCD_{\sigma\in \Delta_{\max}}\xx^{\underline{\sigma}}\not\equiv \0\; (p)}}\omega_{L,p}  =\sum_{\substack{\xx\in W_{p}^{\ast}(N) \\ F^{\ast}(\xx)=0}}p^{c-N(n+r-1)} \\  = \sum_{\substack{\xx,\in W_{p}^{\ast}(N) }}p^{-(N+c)(n+r-1)}\card\{\uu\mod p^{N+c}\; | \; [\uu]_{N}=\xx,\; F(\uu)\equiv 0(p^{N+c}) \}\\   = \frac{M_{p}^{\ast}(N+c)}{p^{(N+c)(n+r-1)}}.
\end{multline*}

D'o\`u le r\'esultat.

\end{proof}

Remarquons \`a pr\'esent que, puisque l'on a suppos\'e qu'une puissance de $ \omega_{Y}^{-1} $ est engendr\'ee par ses sections globales, d'apr\`es la proposition\;\ref{effectif2}, pour tout $ \sigma \in \Delta_{\max} $, $ D(\sigma) $ a pour support $ \bigcup_{i\notin \sigma(1)} D_{i} $ et il existe des entiers $ (a_{i})_{i\notin \sigma(1)} $ strictement positifs tels que $ D(\sigma)=\sum_{i\notin\sigma(1)}a_{i}D_{i} $. On en d\'eduit donc que \[  \PGCD_{\sigma\in \Delta_{\max}}\xx^{\underline{\sigma}}\not\equiv \0\; (p) \Leftrightarrow \PGCD_{\sigma\in \Delta_{\max}}\xx^{D(\sigma)}\not\equiv \0\; (p) \Leftrightarrow h_{p}(\xx)=1. \]

\begin{lemma}
On l'\'egalit\'e  \[ \sigma_{p}=\int_{\substack{\xx\in \ZZ_{p}^{n+r}\; |\; F(\xx)=0 \\  \PGCD_{\sigma\in \Delta_{\max}}\xx^{\underline{\sigma}}\not\equiv \0\; (p) }}\omega_{L,p}=\left(1-\frac{1}{p}\right)^{r}\omega_{p}(Y(\QQ_{p})). \]
\end{lemma}
\begin{proof}
D'apr\`es ce qui pr\'ec\`ede, il suffit de montrer que pour tout ouvert $ V_{0} $ de $ Y $ sur lequel $ \rho $ est bien d\'efini et induit un diff\'eomorphisme analytique sur $ \rho(V_{0}) $ : \[ \int_{\substack{\xx\in \ZZ_{p}^{n+r}\cap W\; |\; \\  \PGCD_{\sigma\in \Delta_{\max}}\xx^{\underline{\sigma}}\not\equiv \0\; (p) }}\omega_{L,p}=\left(1-\frac{1}{p}\right)^{r}\omega_{p}(V_{0}(\QQ_{p})), \] o\`u $ W_{0}=\pi^{-1}(V_{0}) $.\\

On peut \'ecrire : \[
\int_{\substack{\xx\in \ZZ_{p}^{n+r}\cap W_{0}\; |\;  \\  \PGCD_{\sigma\in \Delta_{\max}}\xx^{\underline{\sigma}}\not\equiv \0\; (p) }}\omega_{L,p}(\xx) = \sum_{(\alpha_{1},...,\alpha_{r})\in \NN^{r}}
\int_{\substack{\xx\in \ZZ_{p}^{n+r}\cap W_{0}\; |\;  \\ \forall j, \; |x_{n+j}|_{p}=p^{-\alpha_{j}} \\  h_{p}(\xx)=1}}\omega_{L,p}(\xx).
\]
En effectuant le changement de variables $ x_{i}=\prod_{j=1}^{r}p^{\alpha_{j}a_{i,j}}u_{i} $ pour tout $ i\in \{1,...,n+r\} $, on obtient alors \begin{multline*} \int_{\substack{\xx\in \ZZ_{p}^{n+r}\cap W_{0}\; |\;  \\  \PGCD_{\sigma\in \Delta_{\max}}\xx^{\underline{\sigma}}\not\equiv \0\; (p) }}\omega_{L,p}(\xx) \\  =\sum_{(\alpha_{1},...,\alpha_{r})\in \NN^{r}}
\int_{\substack{\uu\in \QQ_{p}^{n+r}\cap W_{0}\; | \; (p^{\alpha_{1}},...,p^{\alpha_{r}}).\uu\in \ZZ_{p}^{n+r} \\h_{p}(\uu)=\left(\prod_{j=1}^{r}p^{\alpha_{j}(n_{j}-d_{j})}\right)^{-1} \\  \forall j, \; |u_{n+j}|_{p}=1 }}\frac{d u_{1}\wedge...\wedge \widehat{du_{i_{0}}}\wedge...\wedge du_{n+r}}{\left(\prod_{j=1}^{r}p^{\alpha_{j}(n_{j}-d_{j})}\right)^{-1}\left|\frac{\partial F}{\partial x_{i_{0}}}(\uu)\right|_{p}} \\  =\sum_{(\alpha_{1},...,\alpha_{r})\in \NN^{r}}
\int_{\substack{\uu\in \QQ_{p}^{n+r}\cap W_{0}\; | \; (p^{\alpha_{1}},...,p^{\alpha_{r}}).\uu\in \ZZ_{p}^{n+r} \\  h_{p}(\uu)=\left(\prod_{j=1}^{r}p^{\alpha_{j}(n_{j}-d_{j})}\right)^{-1}\\ \forall j, \; |u_{n+j}|_{p}=1 }}\frac{d u_{1}\wedge...\wedge \widehat{du_{i_{0}}}\wedge...\wedge du_{n+r}}{h_{p}(\uu)\left|\frac{\partial F}{\partial x_{i_{0}}}(\uu)\right|_{p}}.
\end{multline*}
En effectuant le changement de variables $ u_{i}=\prod_{j=1}^{r}u_{n+j}^{a_{i,j}} v_{i} $ pour tout $ i\in \{1,...,n\} $ et en remarquant que pour tout $ j\in \{1,...,r\} $ \[ \int_{u_{n+j}\; |\; |u_{n+j}|_{p}=1}du_{n+j}=1-\frac{1}{p}, \] on obtient finalement \begin{multline*}   \left(1-\frac{1}{p}\right)^{r}\sum_{(\alpha_{1},...,\alpha_{r})\in \NN^{r}}\int_{\substack{\vv=(v_{1},...,v_{n})\in  V_{0} \; |\; \\ (p^{\alpha_{1}},...,p^{\alpha_{r}}).\vv\in \ZZ_{p}^{n} }}\frac{d v_{1}\wedge...\wedge \widehat{dv_{i_{0}}}\wedge...\wedge dv_{n}}{h_{p}((\vv,1,...,1))\left|\frac{\partial F}{\partial x_{i_{0}}}(\vv,1,...,1)\right|_{p}}\\  =\left(1-\frac{1}{p}\right)^{r}\omega_{p}(V_{0}(\QQ_{p})). \end{multline*}
\end{proof}

\subsection{\'Etude de l'int\'egrale singuli\`ere $ J $}

Rappelons que l'on a \[ J=\int_{\RR}\int_{\substack{\xx\in \Phi}}e\left(\beta F(\xx)\right)d\xx d\beta. \]
o\`u $ \Phi =\{\RR^{n+r}\; |\;  \max_{\sigma\in \Delta_{\max}}|\xx^{D(\sigma)}|\leqslant 1\} $. Posons $ \tau_{\infty}=\omega_{\infty} $ et montrons que : \[ \tau_{\infty}(Y)=\int_{\pi^{-1}(Y)\cap \Phi}\omega_{L,\infty}=\frac{\prod_{j=1}^{r}(n_{j}-d_{j})}{2^{r}}J. \] 
Pour d\'emontrer cela nous allons utiliser des r\'esultats issus de \cite{Pe2}. Introduisons les notations suivantes : \\

Soit $ X $ une vari\'et\'e d\'efinie sur un corps de nombres $ k $. On pose $ \mathcal{G}=\Gal(\bar{k}/k) $ et \[ \AA_{C_{\eff}(\bar{X}),k}=\Spec(\bar{k}[C_{\eff}(\bar{X})\cap X^{\ast}(T_{NS})]^{\mathcal{G}}), \] et si $ \mathcal{T}_{X} $ est un torseur universel sur $ X $, on note \[ \widehat{\mathcal{T}}_{X}=\mathcal{T}_{X} \times^{T_{NS}} \AA_{-C_{\eff}(\bar{X}),k}. \] On dit qu'une hypersurface $ Y $ de $ X $ \emph{v\'erifie l'hypoth\`ese $ (G) $} s'il existe une d\'esingularisation $ \AA_{\Sigma} $ de $  \AA_{C_{\eff}(\bar{X}),k} $ \'equivariante sous l'action du tore $ T_{NS} $. D'autre part, si $ L $ est un diviseur de $ Y $ appartenant \`a $ C_{\eff}(Y) $, on pose \[ \delta(L)=\inf\{ \langle x,L\rangle, \; x\in C_{\eff}(Y)^{\vee}\cap \Pic(Y)^{\vee}\setminus \{0\}. \] On note $ \delta(Y)=\delta(\omega_{Y}^{-1})\geqslant 1 $. On a alors (cf. \cite[Proposition 3.5.1]{Pe2}) : 

\begin{prop}\label{propdePeyre2}
Si $ Y $ est une hypersurface presque de Fano v\'erifiant l'hypoth\`ese $ (G) $ telle que $ \delta(Y)\geqslant 1 $, $ C_{\eff}(\bar{Y}) $ est poly\'edrique rationnel et si $ \delta(\omega_{X}^{-1}-3L)>0 $ (o\`u $ L $ est le diviseur de $ X $ associ\'e \`a l'hypersurface $ Y $), alors pour toute place archim\'edienne $ \nu \in \Val(k) $, pour toute fonction $ \phi $ complexe d\'efinie sur $ \widehat{\mathcal{T}}_{X}(k_{\nu}) $, $ C^{\infty} $ \`a support compact on a la relation : \[ \int_{\mathcal{T}_{Y}(k_{\nu}) }\phi(y)\omega_{\mathcal{T}_{Y},\nu}(y)=\int_{k_{\nu}}\int_{\mathcal{T}_{X}(k_{\nu}) }\phi(y)e_{\nu}(\xi_{\nu}f(y))\omega_{\mathcal{T}_{X},\nu}(y)d\xi_{\nu}, \] o\`u $ f $ est telle que $ Y=\{y\in X \; |\; f(y)=0\} $, et $ e_{\nu} $ est le caract\`ere $ \xi\mt e^{2i\pi\Lambda_{\nu}(\xi)} $, avec $ \Lambda_{\nu}(\xi)=\lfloor \Tr_{k_{\nu}/\RR}(\xi) \rfloor $. 
\end{prop}
 Nous allons appliquer cette proposition aux vari\'et\'es $ X $ et $ Y $, avec $ k=\QQ $, $ \nu=\infty $ et $ f=F $. Si l'on montre que la proposition s'applique dans ce cadre, on aura alors pour tout fonction $ \phi $, $ C^{\infty} $ \`a support compact sur $ \mathcal{T}_{X}(\RR) $ : \[ \int_{\mathcal{T}_{Y}(\RR)}\phi(\xx)\omega_{\mathcal{T}_{Y},\infty}(\xx)=\int_{\RR}\int_{\mathcal{T}_{X}(\RR) }\phi(\xx)e(\beta F(\xx))\omega_{\mathcal{T}_{X},\infty}(\xx)d\beta. \] Quitte \`a approximer l'indicatrice du domaine $ \Phi $ par des fonctions $ \phi\in C^{\infty}(\mathcal{T}_{X}(\RR))=C^{\infty}(\RR^{n+r}) $, on obtient alors \[ \sigma_{\infty}(Y)=\int_{\xx\in \Phi}\omega_{\mathcal{T}_{Y},\infty}(\xx)=\int_{\RR}\int_{\Phi }e(\beta F(\xx))d\xx d\beta=J. \]

Nous allons donc montrer que la proposition s'applique bien au cas qui nous int\'eresse. \\

Dans le cas pr\'esent nous avons : \[ C_{\eff}(\bar{X})=\sum_{j=1}^{r}\RR^{+}.[D_{n+j}] \] \[ X^{\ast}(T_{NS})=\bigoplus_{j=1}^{r}\ZZ.[D_{n+j}], \] et donc \[ \AA_{C_{\eff}(\bar{X}),\QQ}=\Spec(\QQ[\sum_{j=1}^{r}[D_{n+j}]])\simeq \AA^{r}_{\QQ}. \] Donc en particulier $ \AA_{C_{\eff}(\bar{X}),\QQ} $ est non singuli\`ere, donc v\'erifie bien $ (G) $. D'autre part, le diviseur anticanonique de $ Y $ est, rappelons-le : \[ [\omega_{Y}^{-1}]=\sum_{j=1}^{r}(n_{j}-d_{j})[\tilde{D}_{n+j} ]  \] donc \[ \delta(\omega_{Y}^{-1})=\inf_{(x_{1},...,x_{r})\in \NN^{r}\setminus\{\0\}}\left\{\sum_{j=1}^{r}x_{j}(n_{j}-d_{j})\right\}= \inf_{j\in \{1,...,r\}}(n_{j}-d_{j})\geqslant 1. \] Le diviseur $ L $ associ\'e \`a $ Y $ est $ \sum_{j=1}^{r}d_{j}[D_{n+j}] $ et on obtient donc : \[ \delta(\omega_{X}^{-1}-3L)=\inf_{(x_{1},...,x_{r})\in \NN^{r}\setminus\{\0\}}\left\{\sum_{j=1}^{r}x_{j}(n_{j}-3d_{j})\right\}= \inf_{j\in \{1,...,r\}}(n_{j}-3d_{j})\geqslant 1. \] De plus, en remarquant que \[  C_{\eff}(\bar{Y})=\sum_{j=1}^{r}\RR^{+}.[\tilde{D}_{n+j}] \] qui est poly\'edrique rationnel. Les conditions de la proposition\;\ref{propdePeyre2} sont donc toutes bien v\'erifi\'ees.

\subsection{Conclusion}

Si $ S $ est un ensemble fini de places sur $ \QQ $ contenant la place infinie, et $ V $ une vari\'et\'e alg\'ebrique sur $ \QQ $, on note : \begin{equation}
L_{S}(s,\Pic(\bar{V}))=\prod_{p\in \Val(\QQ)\setminus S}L_{p}(s,\Pic(\bar{V}))
\end{equation}
\begin{equation}
L_{p}(s,\Pic(\bar{V}))=\frac{1}{\det(1-p^{-s}\Fr_{p}|\Pic(V_{\FF_{p}}\otimes \QQ))}
\end{equation}
\begin{equation}
\lambda_{\nu}=\left\{ \begin{array}{lcr} L_{\nu}(1,\Pic(\bar{V})) & \mbox{si} & \nu \in \Val(\QQ)\setminus S \\ 1 & \mbox{sinon} & 
\end{array}\right. 
\end{equation}

La formule conjectur\'ee par Peyre, Batyrev et Tschinkel pour $ \mathcal{N}_{V}(B) $ (cf. \cite{Pe2}) est : \begin{equation}
\mathcal{N}_{V}(B)\sim_{B\ra \infty}\alpha(Y)\beta(Y)\tau_{H}(Y)B(\log B)^{\rg(\Pic(X))-1}
\end{equation}
o\`u l'on a pos\'e : \begin{equation}
\alpha(Y)=\frac{1}{(\rg(\Pic(X))-1)!}\int_{C_{\eff}(Y)^{\vee}}e^{-\langle\omega_{Y}^{-1},\yy\rangle }d\yy,
\end{equation}
\begin{equation}
\beta(Y)=\Card(H^{1}(\QQ,\Pic(\bar{Y}))),
\end{equation}
\begin{equation}
\tau_{H}(Y)=\lim_{s\ra 1}(s-1)^{\rg(\Pic(\bar{Y}))}L_{S}(s,\Pic(\bar{Y}))\prod_{\nu\in \Val(\QQ)}\lambda_{\nu}^{-1}\omega_{\nu},
\end{equation}
avec \begin{equation}
C_{\eff}(Y)^{\vee}=\{\yy\in (\Pic(Y)\otimes \RR) \; |\; \forall \xx\in C_{\eff}(Y), \; \langle \xx,\yy\rangle\geqslant 0 \}.
\end{equation}
Dans le cas pr\'esent on a : \[ \Pic(Y)=\bigoplus_{j=1}^{r}\ZZ.[\tilde{D}_{n+j}] \] donc en particulier $ \rg(\Pic(Y))=3 $, et d'autre part : \[ [-K_{Y}]=\sum_{j=1}^{r}(n_{j}-d_{j}). [\tilde{D}_{n+j}], \] \[ C_{\eff}(Y)=\sum_{j=1}^{r}\RR^{+}. [\tilde{D}_{n+j}]\simeq (\RR^{+})^{3}. \] En choisissant $ S=\{\infty\} $, \[ \prod_{p\in \Val(\QQ)\setminus S}L_{p}(s,\Pic(\bar{Y}))=\zeta(s)^{r} \] et donc \[ \lim_{s\ra \infty }(s-1)^{\rg(\Pic(Y))}L(s,\Pic(\bar{Y}))=1, \] et on a pour tout $ p  $ premier, \[ \lambda_{p}^{-1}=\prod_{j=1}^{r}\left(1-\frac{1}{p}\right)=\left(1-\frac{1}{p}\right)^{r}. \] On a donc que \[ \tau_{H}(Y)=\omega_{\infty}\prod_{p\in \mathcal{P}}\left(1-\frac{1}{p}\right)^{r}\omega_{p}(Y)=\frac{1}{2^{r}}(\prod_{j=1}^{r}n_{j}-d_{j})\mathfrak{S}J. \] Par ailleurs, puisque $ \Pic(\bar{Y})\simeq \ZZ^{r} $, $ \Gal(\bar{\QQ}/\QQ) $ agit trivialement sur $ \Pic(\bar{Y}) $ et on a donc que $ \beta(Y)=1 $. Enfin on a que : \begin{align*} \alpha(Y) & =\frac{1}{(r-1)!}\int_{[0,\infty]^{r}}e^{-\sum_{j=1}^{r}(n_{j}-d_{j})t_{j}}dt_{1}...dt_{r} \\ & =\frac{1}{(r-1)!}\prod_{j=1}^{r}\int_{0}^{\infty}e^{-(n_{j}-d_{j})t}dt\\ & = \frac{1}{(r-1)!}\frac{1}{\prod_{j=1}^{r}(n_{j}-d_{j})}. \end{align*}
Par cons\'equent, on obtient : \[ \alpha(Y)\beta(Y)\tau_{H}(Y)B(\log B)^{\rg(\Pic(X))-1}=\frac{1}{2^{r}}\frac{1}{(r-1)!}\mathfrak{S}JB(\log B)^{r-1}, \] et on retrouve donc bien la formule du th\'eor\`eme\;\ref{Conclusion3}.

\end{document}